\documentclass[11pt,reqno]{amsart}

\usepackage{geometry}

\title{The Plectic Weight Filtration on Cohomology of Shimura Varieties and Partial Frobenius}
\author{Zhiyou Wu}

\usepackage{amsmath,amssymb,fancyhdr,url,amsthm,mathtools,mathrsfs,eufrak,setspace}
\usepackage[pdftex]{graphicx}
\usepackage{tikz-cd}
\usetikzlibrary{cd}

\usepackage{tikz}
\usepackage{dsfont}
\usetikzlibrary{matrix}
\newcommand*\Z{\mathds{Z}}

\newcommand*{\isoarrow}[1]{\arrow[#1,"\rotatebox{90}{\(\sim\)}"
]}

\makeatletter
\newcommand*{\rom}[1]{\expandafter\@slowromancap\romannumeral #1@}
\makeatother

\makeatletter
\newenvironment{innerproof}[1][\proofname]
  {\par\normalfont \topsep6\p@ \@plus6\p@\relax
  \trivlist
  \item[\hskip\labelsep\itshape#1\@addpunct{.}]\ignorespaces}
  {\endtrivlist\@endpefalse}
\makeatother

\newtheorem{proposition}{Proposition}[section]
\newtheorem{theorem}[proposition]{Theorem}

\newtheorem{corollary}[proposition]{Corollary}
\newtheorem{definition}[proposition]{Definition}
\newtheorem{remark}[proposition]{Remark}
\newtheorem{lemma}[proposition]{Lemma}
\newtheorem{warning}[proposition]{Warning}
\newtheorem{dt}[proposition]{Definition/Theorem}

\begin{document}

\maketitle

\begin{abstract}
 We prove that there is a natural plectic weight filtration on the cohomology of Hilbert modular varieties in the spirit of Nekov\'a\v r and Scholl. This is achieved with the help of Morel's work on weight t-structures and  a detailed study of partial Frobenius. We prove in particular that the partial Frobenius extends to toroidal and minimal compactifications.  

\end{abstract}

\tableofcontents

\section{Introduction}  

Nekov\'a\v r and Scholl recently proposed in \cite{MR3502988} a program on plectic theory, which is about some hidden symmetries of Shimura varieties. The theme of this paper is to exploit some of these hidden symmetries, and provide evidence for their conjectures. More precisely, Nekov\'a\v r and Scholl observed that when the group of a Shimura variety $X$ is of the form $Res_{F/\mathbb{Q}} G$ with $F$ totally real, the cohomology of $X$ has extra structures. This is most easily observed in the (Betti) intersection cohomology of (minimal compactification of) Shimura varieties, in which case we have
\[
IH^*(X^{\text{min}}(\mathbb{C}), \mathbb{C}) = H^*_{(2)}(X(\mathbb{C}), \mathbb{C}) = \underset{\pi}{\oplus}  H^*( \mathfrak{g} , K_{\infty}; \pi_{\infty}) \otimes \pi_f^{K_f}  =  \underset{\pi}{\oplus} \underset{v | \infty}{\otimes} H^*( \mathfrak{g}_v , K_{\infty, v}; \pi_{v}) \otimes \pi_f^{K_f}\]
where the first equality is the (proven) Zucker's conjecture, $\pi$ ranges over irreducible $L^2$ automorphic representations of the group $Res_{F/\mathbb{Q}}G$, and the last equality follows by applying the Kunneth theorem for $(\mathfrak{g}, K)$-cohomology to $\pi_{\infty} = \underset{v | \infty}{\otimes} \pi_{v}$. 
As each $(\mathfrak{g}, K)$-cohomology
$H^*( \mathfrak{g}_v , K_{\infty, v}; \pi_{v})$
equips with a Hodge structure of type $(p_v,q_v)$, we see that $IH^*(X^{\text{min}}(\mathbb{C}), \mathbb{C})$
is a sum of  refined Hodge structures of type 
$\underset{v | \infty}{\otimes} (p_v,q_v) $, 
i.e.  plectic Hodge structures. A remarkably similar structure appears in the etale cohomology, at least in the case of Hilbert modular varieties, which suggests that it is motivic in nature. This motivates the question of explaining this extra structure.

 Nekov\'a\v r and Scholl proposed that the Shimura variety prolongs to a variety defined over $Spec(k_{plec})$, where $Spec(k_{plec})$ is a (product of) symmetric product of $Spec(k)$ over $\mathbb{F}_1$, the field with one element. Obviously, this does not make sense as we do not have a good theory of $\mathbb{F}_1$. However, this heuristic allows us to guess what extra structures we can expect on the cohomology, which sometimes can be established directly. In particular, we expect that for noncompact Shimura varieties of type $Res_{F/\mathbb{Q}}G$, the Betti cohomology has a natural plectic weight filtration, which is a $\mathbb{Z}^d$-indexed filtration whose graded pieces have pure plectic Hodge strutures as we observed using $(\mathfrak{g}, K)$-cohomology. What we prove in this paper is that this is true in the special case of Hilbert modular varieties. Before explaining more about the results, we remark that the plectic conjectures have powerful arithmetic consequences on special values of L-functions. 

Let us first recall how we detect the classical weight filtration on a smooth non-proper complex variety $X$. Using Nagata embedding and resolution of singularities, we can find an open embedding 
$j: X \hookrightarrow \overline{X}$ 
into a proper smooth variety with 
$\overline{X}\setminus X$ union of normal crossing divisors. Then, as observed by Deligne, the weight filtration is detected using the filtration on 
$Rj_* \mathbb{C}$
induced by the standard truncation
$\tau_{\leq a} Rj_* \mathbb{C}$, 
and the graded pieces of the weight filtration is detected using cohomology of strata of $\overline{X}$. More precisely, we have a spectral sequence induced by the filtration
$\tau_{\leq a} Rj_* \mathbb{C}$, 
\begin{equation} \label{2}
E^{p,q}_1 =   \mathbb{H}^{p+q}(\overline{X}(\mathbb{C}), \tau_{\geq -p} \tau_{\leq -p} Rj_*\mathbb{C}) \Rightarrow H^{p+q}(X(\mathbb{C}), \mathbb{C})
\end{equation}
which is nothing but the (reindexed) Leray spectral sequence for $j$. The graded sheaves
$\tau_{\geq -p} \tau_{\leq -p} Rj_*\mathbb{C}$
are supported on the strata defined by intersections of boundary divisors, and the weight filtration is a shift of the converging filtration of the spectral sequence.  

When $X$ is a Shimura variety of type $Res_{F/\mathbb{Q}}G$,  we can find an explicit $\overline{X}$ using toroidal compactifications. However, we can not use them to detect the plectic weight filtration since toroidal compactifications are not "plectic", in particular the strata of them possess no plectic structures on their cohomology. Our strategy is to look at the minimal compactification $X^{min}$ of $X$ instead, and what we gain is that the strata are now again Shimura varieties of type $Res_{F/\mathbb{Q}} G$, hence "plectic". This is a highly singular proper variety, and 
$\tau_{\leq a} Rj_* \mathbb{C}$ 
is not a reasonable object to consider. The way to approach it is to use Morel's weight t-structures (\cite{MR2862060}) in place of the standard t-structures. The formalism gives us a new truncation
$w_{\leq a} Rj_* \mathbb{C}$, 
giving rise to a spectral sequence of Hodge structures 
\begin{equation} \label{1}
E_1^{p,q}= \mathbb{H}^{p+q}(X^{min}(\mathbb{C}), w_{\geq -p} w_{\leq -p} Rj_*\mathbb{C}) \Rightarrow H^{p+q}(X(\mathbb{C}), \mathbb{C})
\end{equation}
first observed by Nair in \cite{nairmixed}. Note that Morel's formalism only makes sense in a theory with good notion of weights and perverse sheaves, and we have to use the derived category of mixed Hodge modules here. It is not hard to see that 
$w_{\geq -p} w_{\leq -p} Rj_*\mathbb{C}$
decomposes into shifted simple Hodge modules strictly supported on (closure of) strata of $X^{min}$, and can be made explicit with the help of Burgos and Wildeshaus' results (\cite{burgos2004hodge}). Moreover, these simple summands are automorphic in the sense they are associated to algebraic representations of the groups associated to the strata they support. Now $E^{p,q}_1$ is a sum of intersection cohomology of "plectic" Shimura varieties with automorphic coefficients, the same computation as before using $(\mathfrak{g}, K)$-cohomology on twisted automorphic representations shows that it possesses plectic structures. 

To proceed further, we have to know whether the spectral sequence detects the weight filtration and how we can extract the plectic weight filtration from it. Unfortunately, the answer to the first question is no in general, though it is true in the Hilbert modular case. The problem is that the graded pieces of the filtration are not necessarily pure, but direct sums of pure Hodge structures possibly of different weights. It is a coincidence that in the Hilbert modular case, this does not happen. On the other hand, to find the plectic weight filtration, it is not necessary to know the weight filtration a priori, and the spectral sequence does help with our purpose. 

To motivate the strategy, let us recall that there is another way to detect weights, namely using Frobenius weights. By spreading out the variety, we can assume that it is defined over a finitely generated $\mathbb{Z}$-algebra, and  reduce it to a variety defined over a finite field, then the Weil conjecture proved by Deligne  tells us that the $l$-adic cohomology has a weight filtration defined by Archemdean places of Frobenius eigenvalues. Using comparison theorems and base change or nearby cycles, we can find the weight filtration on Betti cohomology using finite fields. It is necessary to check that the new weight filtration is the same as the previously defined one, and this is proved by observing that the Frobenius acts on the spectral sequence (\ref{2}) through the comparison isomorphism, and has the right Frobenius weight on each $E_1^{p,q}$. 

In the plectic case, we expect that there are plectic Frobenius weights in some reasonable sense, and the above classical method can be applied to find the plectic weight filtration. Fortunately, morphisms called partial Frobenius have been defined and studied in the literature (\cite{JanNekovar2018}). These are decompositions of the usual Frobenius, and their eigenvalues are naturally expected to give plectic Frobenius weights, hence the plectic weight filtration. To fulfill the expectation, we have to prove that the partial Frobenius extends to the minimal compactification, and induces a morphism on the spectral sequence (\ref{1}). This is achieved through toroidal compactifications. Indeed, we prove firstly the partial Frobenius extends to toroidal compactifications, using Lan's universal property of toroidal compactifications (\cite{lan2013arithmetic}). To check the universal property, we have to make full use of the degeneration data of semiabelian varieties constructed by Faltings-Chai and Lan. Then we prove that the extended partial Frobenius morphism descends to the minimal compactification, which is a standard argument adapted from Lan (\cite{lan2013arithmetic}). 

\begin{theorem}
Let $M_n$ be a (similitude) PEL Shimura variety  with principal level $n$ structure, and $M_{n,\Sigma}^{\text{tor}}$ its toroidal compactification associated to a cone decomposition $\Sigma$. We assume that $M_n$ is defined over a finite field over which we have a well-defined partial Feobenius map 
$F_{\mathfrak{p}_i} : M_{n} \rightarrow M_{n}$,
then $F_{\mathfrak{p}_i}$ extends to a map
\[
F_{\mathfrak{p}_i}: M_{n,\Sigma}^{\text{tor}} \longrightarrow M_{n,\Sigma'}^{\text{tor}}
\]
with a different choice of $\Sigma'$. 
\end{theorem}

\begin{corollary}
$F_{\mathfrak{p}_i}$ extends to the minimal compactification 
\[
F_{\mathfrak{p}_i}: M_{n}^{\text{min}} \longrightarrow M_{n}^{\text{min}}.
\]
\end{corollary}

Now the partial Frobenius acts on each summand of $E_1^{p,q}$, which as we have already seen is the intersection cohomology of (closure of) strata of the minimal compactifications with automorphic coefficients, and have plectic Hodge structures given by $(\mathfrak{g}, K)$-cohomology. A subtle point here is that we have to pass to special fibers of integral models of Shimura varieties and use the spectral sequence (\ref{1}) in the $l$-adic setting in order to have the action of the partial Frobenius,  and then compare it with the one in the Hodge theory setting. This can be done with some technical input from Huber and Morel's horizontal mixed complexes in \cite{morel2019mixed}  (a simpler proof in the special case of Hilbert modular varieties exists). 

Now, similar to the classical case,  we have to check that the eigenvalues of the partial Frobenius on each summand are Weil numbers with absolute value compatible with the multi-weights of the plectic Hodge structures.  In the case of Hilbert modular varieties, we have two different types of summands. The first type is when the summand is the cohomology of cusps with automorphic coefficients, which can be checked by direct computations. 

The second is when it is the intersection cohomology of (minimal compactification of) Hilbert modular variety with trivial coefficients. This is decomposed into Hecke equivariant isotypic components indexed by discrete cohomological automorphic representations. If the automorphic representation is cuspidal, we know that it corresponds to a holomorphic Hilbert modular form $f$ of parallel weight 2, and the plectic Hodge type is 
\[
\underset{v | \infty}{\otimes} ((1,0) \oplus (0,1)),
\]
which is of plectic weights $(1,\cdots, 1)$. We have to show that each partial Frobenius acts with eigenvalues of absolute value $p^{\frac{1}{2}}$. This follows from the Eichler-Shimura relation of the partial Frobenius proved by Nekov\'a\v r in \cite{JanNekovar2018}. Indeed, it tells that the eigenvalues of the partial Frobenius is the same as the eigenvalues of the Frobenius $Frob_{\mathfrak{p}} \in Gal(\overline{F}/F)$ on the Galois representation $\rho_f$ associated to $f$, where $\mathfrak{p}$ ranges over primes of $F$ above $p$. We know that $\rho_f$ is pure of weight 1 by Blasius (\cite{blasius2006hilbert}) and Blasius-Rogawski (\cite{blasius1993motives}), proving the claim. If the automorphic representation is discrete but not cuspidal, we know that they are one dimensional, and have plectic Hodge types (wedge products of) the sum of $\underset{v | \infty}{\otimes} (p_v,q_v)$, with $(p_v, q_v)= (1,1)$ for one $v$, and $(p_v, q_v)= (0,0)$ for the rest, which is of plectic weight
\[
(0, \cdots,0, 2,0, \cdots, 0). 
\]
This forces us to show that the partial Frobenius corresponding to $v$ (under the embedding $ \mathbb{Q}_p \hookrightarrow \mathbb{C}$ implicitly fixed in the comparison theorem) has eigenvalues with absolute value $p$, and the rest have  eigenvalues with absolute value 1. This is shown by observing that these cohomology spaces are spanned by first Chern classes of the natural line bundles $L_v$ whose sections are modular forms of weight $(0, \cdots,0, 2,0, \cdots, 0)$, 
and the partial Frobenius acts on them in the expected way ($F_v^*L_v = L_v^{\otimes p}$ and $F_{v'}^*L_v = L_v$). Note that here we use a motivic explanation of the plectic strucutures to compare the plectic Frobenius weights and plectic Hodge weights, and this is the main reason we restrict to Hilbert modular varieties. 

Now we have a $\mathbb{Z}^d$-filtration defined by eigenvalues of the partial Frobenius, and the previous proof shows that the graded pieces have natural plectic Hodge structures given by $(\mathfrak{g}, K)$-cohomology in a compatible way.  
This finishes the construction of the plectic weight filtration, and gives a conceptual explanation of the ad-hoc construction of the plectic weight filtration by Nekov\'a\v r and Scholl in \cite{nekovar2017plectic}. Moreover, the proof has the potential to extend to more general situations where the naive construction of Nekov\'a\v r and Scholl fails. Indeed, most ingredients we use are proved for general PEL type Shimura varieties. The only serious obstacle for the general case is the use of motivic explanation as remarked above. To summarize, we have

\begin{theorem}
Let $\mathscr{M}$ be a Hilbert modular variety, 
there is an increasing $\mathbb{Z}^d$-filtration $W_{\underline{a}}$ (defined over $\mathbb{C}$) on 
$H^*(\mathscr{M}(\mathbb{C}), \mathbb{C})$
with $\underline{a}=(a_1, \cdots, a_d) \in \mathbb{Z}^d$, defined by 
\[
W_{\underline{a}}= \underset{\begin{subarray}{c}
  |\beta_i|= p^{\frac{k_i}{2}} \\
  k_i \leq a_i
  \end{subarray}}{\bigoplus}
  V_{(\beta_1, \cdots, \beta_d)}
\]
where $V_{(\beta_1, \cdots, \beta_d)}$ is the generalized eigenspace of $F_i$ with eigenvalue $\beta_i$ for all $i$. The action of $F_i$ on 
$H^*(\mathscr{M}(\mathbb{C}), \mathbb{C})$
is through the natural comparison isomorphism 
$H^*(\mathscr{M}(\mathbb{C}), \mathbb{C}) \cong 
\imath_*H^*(\mathscr{M}_{\bar{\mathbb{F}}_p}, \overline{\mathbb{Q}_l})$ 
for some fixed isomorphism $\imath: \overline{\mathbb{Q}_l} \cong \mathbb{C}$. 

The filtration is plectic in the sense that there is a natural plectic Hodge structure on 
$Gr_{\underline{a}}^W$
with plectic weight $\underline{a}$.
\end{theorem}

The reader is warned that  the construction does not a priori give the plectic mixed Hodge structure in the sense of Nekov\'a\v r and Scholl (\cite{nekovar2017plectic}) since we have not proved that the plectic Hodge filtration is compatible with plectic weight filtration. This is left to future works. 

We now give a summary of each section. In section 2, we review Morel's work on the weight t-structures,  and prove a comparison theorem between two spectral sequences obtained using mixed Hodge modules and etale cohomology respectively. In section 3, we define the PEL moduli varieties and the partial Frobenius. This section is mostly to fix notations. In section 4, we use the results on partial Frobenius proved in section 5 and the weight spectral sequences in section 2 to prove the existence of plectic weight structures on cohomology of Hilbert modular varieties. In particular, we make the weight spectral sequence in this case explicit in section 4.2,  and carry out the computations of the eigenvalues of the partial Frobenius in section 4.3. section 5 is a largely independent section, in which we prove the partial Frobenius extends to toroidal compactifications and minimal compactifications. Following Lan, we review the construction of toroidal compactifications in sections 5.1 to 5.3. In particular, we review with some details on  the degeneration data, and how to construct it from degenerating abelian varieties. This is then used to construct the formal boundary strata of the toroidal compactifications, and provides fundamental local formal models of the boundary strata. We use those constructions to prove the extension of partial Frobenius to toroidal and minimal compactifications in section 5.4 and 5.5 respectively.     
  
\subsection*{Acknowledgments}
This work is the author's PhD thesis at the University of Cambridge. The author would like to thank his PhD supervisor Tony Scholl for the guidance and encouragement he has received during the past four years. The author is grateful to Marius Leonhardt who raised his interest in plectic theory. The author would also like to thank Jack Thorne and Ana Caraiani for a careful reading of this paper and suggestions of improvement.

\section{Morel's weight t-structure}

\subsection{Formalism}

We review Morel's weight t-structures in this section. Everything in this section is due to Morel and Nair (the Hodge module case is due to Nair). The references we follow are \cite{nairmixed} and \cite{MR2862060}.

In this section, $X$ denotes a separated scheme of finite type over a field $k$. We assume that $k$ is either finitely generated over its prime field, or $k = \mathbb{C}$. Let $l$ be a prime number different from the characteristic of $k$, and
$D^b_c(X, \Bar{\mathbb{Q}_l})$
be the usual constructible derived category. We use $H^i$ to denote the cohomology with respect to the usual constructible t-structure and $\prescript{p}{}{H}^i$ for the cohomology with respect to the perverse t-structure. For Hodge modules,  $\prescript{p}{}{H}^i$ 
will denote the usual cohomology of complexes of Hodge modules. They correspond to perverse cohomology under $rat$, see below for explanation of the terminology. 

We denote both
$D^b_m(X, \overline{\mathbb{Q}_l})$ and 
$D^bMHM(X(\mathbb{C}))$
by $D^b_m(X)$, where 
$D^b_m(X, \overline{\mathbb{Q}_l})$
is the bounded derived category of horizontal mixed complexes with weight filtrations  as defined in \cite{morel2019mixed} when $k$ is finitely generated, and 
$D^bMHM(X(\mathbb{C}))$ 
is the bounded derived category of Saito's mixed Hodge modules when $k = \mathbb{C}$. Note that $m$ here means "mixed". The key property of 
$D^b_m(X)$ is that they have the notion of weights and perverse t-structures, giving rise to canonical weight filtrations on perverse sheaves in $D^b_m(X)$. Further, morphisms between perverse sheaves strictly preserve weight filtrations. Under this abuse of notation, perverse sheaves refers to the usual perverse sheaves in the $l$-adic case, and Hodge modules in the complex case.    

When $k$ is a finite field,  
$D^b_m(X, \overline{\mathbb{Q}_l})$
is the usual derived category of mixed sheaves defined by Deligne, i.e. 
$D^b_m(X, \overline{\mathbb{Q}_l}) \subset D^b_c(X, \overline{\mathbb{Q}_l})$
is the full subcategory defined by 
$K \in D^b_m(X, \overline{\mathbb{Q}_l})$
if and only if for every $i \in \mathbb{Z}$,  $ H^i(K)$ has a finite filtration $W$ whose graded pieces are pure in the sense that for every closed point 
$i_x: Spec(k(x)) \hookrightarrow X$ 
and $n \in \mathbb{Z}$, 
$i_x^* Gr_n^W H^i(K) $,
as a representation of $Gal(\Bar{k}/ k(x))$, has algebraic Frobenius eigenvalues whose absolute value are $(\# k(x))^{-n/2}$ for every Archimedean place. 

We review Morel's construction in \cite{morel2019mixed} of 
$D^b_m(X, \overline{\mathbb{Q}_l})$
for $k$ finitely generated. We can write $k$ as a direct limit of regular finite type $\mathbb{Z}$-algebras $A$ sitting inside $k$ and having fraction field $k$. The standard spreading argument shows that for $A$ as above (possibly passing to a localization), there is a flat finite type $A$-scheme $ \mathscr{X}_A $ such that $(\mathscr{X}_A)_k \cong X$. The $\{\mathscr{X}_A\}_A$ forms a direct system and induces natural functors between constructible derived categories. We define the derived category $D^b_h(X, \overline{\mathbb{Q}_l})$
of horizontal constructible sheaves on $X$ to be the 2-limit of the category 
$D^b_c(\mathscr{X}_A, \overline{\mathbb{Q}_l})$
indexed by $A$ as above. The perverse t-structures on 
$D^b_c(\mathscr{X}_A, \overline{\mathbb{Q}_l})$
induces a t-structure on 
$D^b_h(X, \overline{\mathbb{Q}_l})$, whose heart $\text{Perv}_h(X)$ is called the category of horizontal perverse sheaves. The usual t-structures also induce a t-strucutre on 
$D^b_h(X, \overline{\mathbb{Q}_l})$ whose heart are called horizontal constructible sheaves. 
$K \in D^b_h(X, \overline{\mathbb{Q}_l})$ 
is called mixed if $H^i(K)$ has a finite filtration whose graded pieces can be represented by a construtible sheaf $F_A$ on $\mathscr{X}_A$
such that for every closed point $x \in \text{Spec}(A)$ (necessarily of finite residue field),
$(F_A)_x$ is pure of some weight as discussed in the previous paragraph. Mixed horizontal complexes define a triangulated subcategory of 
$D^b_h(X, \overline{\mathbb{Q}_l})$, 
and the perverse t-structure on 
$D^b_h(X, \overline{\mathbb{Q}_l})$
induces a t-structure on it, whose heart 
$\text{Perv}_m(X)$
is called the category of mixed horizontal perverse sheaves. The problem is that an element of $\text{Perv}_m(X)$
does not necessarily have a weight filtration. However, the weight filtration is unique if it exists. We can define the subcategory
$\text{Perv}_{mf}(X)$
of $\text{Perv}_m(X)$ consisting of those with a weight filtration. The uniqueness shows that morphisms in 
$\text{Perv}_{mf}(X)$ 
is strict with respect to the weight filtration. Finally, we define the derived category of mixed horizontal perverse sheaves to be 
\[
D^b_m(X, \overline{\mathbb{Q}_l}) := D^b(\text{Perv}_{mf}(X))
\]
Morel proves that the six functors can be defined on 
$D^b_m(X, \overline{\mathbb{Q}_l})$. 
Note that for $k$ a finite field, $A=k$ and every mixed perverse sheaf has a weight filtration, proving that
$D^b_m(X, \overline{\mathbb{Q}_l})$
is identical to the category in the previous paragraph, see BBD (\cite{MR751966}) for details. 

\begin{remark}
The constructions, especially the six functors,  depend fundamentally on the finiteness results of Gabber, see \cite{MR3309086}. If we restrict to $k$ with transcendental dimension smaller than 2, which is the only case we need, then the older finiteness results of Deligne in SGA suffices. Moreover, Morel's proof uses sophisticated homological algebra results, including Beilinson's reconstruction of constructible t-structures from perverse ones, and Ayoub's work on crossed functors.    
\end{remark}

For mixed Hodge modules, we will not give a precise review. We only remind the reader that a mixed Hodge module consists of a good filtered regular holonomic $D$-module together with a perverse $\mathbb{Q}$-sheaf which is isomorphic after tensoring with $\mathbb{C}$ to the D-module under the Riemann Hilbert correspondence. The precise conditions to put on these data is through a delicate induction process where vanishing cycles play an important role. It can be proved that admissible graded polarizable variations of Hodge structures are mixed Hodge modules, and they (their intermediate extension) constitute the simple mixed Hodge modules in a way similar to locally systems and perverse sheaves. Forgetting about the $D$-modules gives a faithful functor 
\[
rat : D^bMHM(X) \rightarrow D^b_c(X, \mathbb{Q})
\]
where we use the classical topology on $X(\mathbb{C})$ to define the right hand side. An important property is that $rat$ commutes with the six functors. The comparison theorem gives an $l$-adic perverse sheaf for each Hodge module. We will only use $\mathbb{C}$-Hodge modules, in which case the extra choices of perverse sheaves are redundant.  

We will only need the cases when $k$ is a finite field, a number field or the complex numbers. Indeed, we will be primarily concerned with complex numbers, and finite fields come into play by reducing the complex situation to the finite fields cases. The reduction step will be achieved through number fields.   

We now introduce Morel's fundamental weight t-structures. 

\begin{dt} 
(\cite{MR2862060} Proposition 3.1.1)
With notations as above, for
$a \in \mathbb{Z} \cup \{\infty\}$ 
there is a $t$-structure 
\[
(\prescript{w}{}{D}^{\leq a},  \prescript{w}{}{D}^{\geq a+1})
\]
on  $D^b_m(X)$ defined by 
$K \in \prescript{w}{}{D}^{\leq a}$ 
(resp. $ K \in \prescript{w}{}{D}^{\geq a+1}$)
if and only if for all $i \in \mathbb{Z}$, 
$\prescript{p}{}{H}^i(K)$
has weights $\leq a$ (resp. $\geq a+1$). Moreover, $\prescript{w}{}{D}^{\leq a}$ 
and  
$\prescript{w}{}{D}^{\geq a+1}$
are traingulated subcategories and are stable under extensions. For 
$K \in \prescript{w}{}{D}^{\leq a}$
and 
$L \in \prescript{w}{}{D}^{\geq a+1}$, 
we have 
\[
RHom(K,L)=0
\]
Note that this is stronger than being given by a $t$-structures. We have 
$ \prescript{w}{}{D}^{\leq a}(1) =  \prescript{w}{}{D}^{\leq a-2}$
and 
$ \prescript{w}{}{D}^{\geq a}(1) =  \prescript{w}{}{D}^{\geq a-2}$,
where $(1)$ is the Tate twist. 
\end{dt}

\begin{remark}
The t-structure is unusual in that it has trivial heart, and stable under shift $[1]$ in the triangulated category. Note that a complex 
$K \in \prescript{w}{}{D}^{\leq a} \cap \prescript{w}{}{D}^{\geq a}$
is not a pure complex of weight $a$ in the sense of Deligne, which means $H^i(K)$ has weight $a+i$ (or equivalently $\prescript{p}{}{H}^i(K)$ has weight i+a). 
\end{remark}

Recall that (over finite fields) a pure complex is a direct sum of its shifted perverse cohomology after base change to the algebraic closure, and the decomposition does not hold before base change. This fact plays an important role in the proof of the decomposition theorem. The next proposition gives a variant of this fact in complete generality. In particular, we do not need to pass to algebraic closure. 

\begin{proposition} \label{splitting}
(\cite{nairmixed} lemma 2.2.3)
If 
$K \in \prescript{w}{}{D}^{\geq a} \cap \prescript{w}{}{D}^{\leq a}$, 
we have an isomorphism
\[
K \cong \underset{i}{\oplus} \prescript{p}{}{H}^i(K)[-i] 
\]
The constituents $\prescript{p}{}{H}^i(K)$ are pure, and they decompose by supports into intersection complexes, i.e. intermediate extension of smooth sheaves on a smooth locally closed subscheme. 

Moreover, this isomorphism is canonical and the constituents are semisimple if we are in the mixed Hodge modules case. 
\end{proposition}

\begin{remark}
The corresponding statement is not true in the $l$-adic case. 
\end{remark}

The t-structure gives us functors 
$w_{\leq a} : D^b_m(X) \rightarrow \prescript{w}{}{D}^{\leq a}$
(resp. 
$w_{\geq a} : D^b_m(X) \rightarrow \prescript{w}{}{D}^{\geq a}$)
such that for every 
$K \in D^b_m(X)$, 
we have a distinguished triangle
\[
w_{\leq a}K \longrightarrow K \longrightarrow w_{\geq a+1}K \overset{+1}{\longrightarrow}  \cdot
\]
If $K$ is written as a complex of perverse sheaves $K_i$, which is always possible, then 
$w_{\leq a} K $ is the complex represented by $w_{\leq a} K_i$, where $w_{\leq a} K_i$ is the weight filtration on $K_i$. 
We have the following proposition on the behaviour of $w_{\leq a}$. 

\begin{proposition}
(\cite{MR2862060} proposition 3.1.3)
Let $K \in D^b_m(X)$, we have that $w_{\leq a}$ (resp. $w_{\geq a}$ ) is exact with respect to the perverse t-structure, i.e. 
\[
w_{\leq a} \prescript{p}{}{H}^i(K) = \prescript{p}{}{H}^i (w_{\leq a}K)
\]
\[
w_{\geq a} \prescript{p}{}{H}^i(K) = \prescript{p}{}{H}^i (w_{\geq a}K)
\]
Moreover, the distinguished triangle 
\[
w_{\leq a}K \longrightarrow K \longrightarrow w_{\geq a+1}K \overset{+1}{\longrightarrow}  \cdot
\]
induces a short exact sequence of perverse sheaves
\[
0 \longrightarrow \prescript{p}{}{H}^i (w_{\leq a}K) \longrightarrow \prescript{p}{}{H}^i ( K) \longrightarrow \prescript{p}{}{H}^i (w_{\geq a+1}K) \longrightarrow 0 
\]
\end{proposition}

The four functors interacts with the weight t-structure as described in the following proposition. 

\begin{proposition} \label{functors}
(\cite{MR2862060} proposition 3.1.3)
Let $f : X \rightarrow Y$ be a morphism with dimension of the fibers less than or equal to $d$, then 
\[
Rf_{!}(\prescript{w}{}{D}^{\leq a}(X)) \subset \prescript{w}{}{D}^{\leq a+d}(Y)
\]
\[
f^{*}(\prescript{w}{}{D}^{\leq a}(Y)) \subset \prescript{w}{}{D}^{\leq a+d}(X)
\]
\[
Rf_{*}(\prescript{w}{}{D}^{\geq a}(X)) \subset \prescript{w}{}{D}^{\geq a-d}(Y)
\]
\[
f^{!}(\prescript{w}{}{D}^{\geq a}(Y)) \subset \prescript{w}{}{D}^{\geq a-d}(X)
\]
 The duality functor 
$D := RHom( -, \omega_X)$ 
($\omega_X$ is the dualizing complex)
exchanges $\prescript{w}{}{D}^{\leq a}(X)$
and 
$\prescript{w}{}{D}^{\geq -a}(X)$, i.e. 
$D(\prescript{w}{}{D}^{\leq a}(X)) = \prescript{w}{}{D}^{\geq -a}(X)
$
so 
\[
D\circ w_{\leq a} = w_{\geq -a} \circ D
\]

\end{proposition}

The most important property of $w_{\leq a}$ is its relation with intermediate extension functor.

\begin{theorem} \label{ie}
(\cite{MR2862060} theorem 3.1.4)
Let $j : U \rightarrow X$ be a nonempty open embedding, and $K \in D^b_m(X)$ a pure perverse sheaf of weight $a$ on U, then we have natural isomorphisms
\[
w_{\geq a}j_!K = j_{!*}K = w_{\leq a}Rj_*K
\]
\end{theorem}

We now introduce a refined version of the weight t-structure, taking a specified stratification into consideration. Let 
$X = \underset{0 \leq i \leq n}{\cup} S_i $
be a stratification such that each $S_i$ is locally closed in $X$, and $S_k$ is open in 
$\underset{k \leq i \leq n}{\cup} S_i$ 
for every $k \in [0,n]$. Let 
$\underline{a} = (a_0, \cdots, a_n)$
with each 
$a_i \in \mathbb{Z}\cup \{\infty\}$
and 
$i_k : S_k \hookrightarrow X$ be the inclusion. 

\begin{dt}
(\cite{MR2862060} proposition 3.3.2)
Let
$\prescript{w}{}{D}^{\leq \underline{a}}$
(resp. $\prescript{w}{}{D}^{\geq \underline{a}}$)
be the subcategory of $D^b_m(X)$ defined by 
$K \in \prescript{w}{}{D}^{\leq \underline{a}}$
(resp. $K \in \prescript{w}{}{D}^{\geq \underline{a}}$)
if and only if 
$i_k^* K \in \prescript{w}{}{D}^{\leq a_k}(S_k)$
(resp. $i_k^! K \in \prescript{w}{}{D}^{\geq a_k}(S_k)$) for every $k$. Then 
\[
(\prescript{w}{}{D}^{\leq \underline{a}}, \prescript{w}{}{D}^{\geq \underline{a} + \underline{1}})
\]
defines a t-structure on $D^b_m(X)$, giving rise to functors 
\[
w_{\leq \underline{a}} : D^b_m(X) \rightarrow \prescript{w}{}{D}^{\leq \underline{a}}
\]
and
\[
w_{\geq \underline{a}} : D^b_m(X) \rightarrow \prescript{w}{}{D}^{\geq \underline{a}}
\]
such that for every $K \in D^b_m(X)$, there is a distinguished triangle 
\[
w_{\leq \underline{a}}K \longrightarrow K \longrightarrow w_{\geq \underline{a}+\underline{1}}K \overset{+1}{\longrightarrow} \cdot
\]
Moreover, we have 
$RHom(L, K) =0 $
for 
$L \in \prescript{w}{}{D}^{\leq \underline{a}}$
and 
$K \in \prescript{w}{}{D}^{\geq \underline{a} +\underline{1}}$.
\end{dt}

Most of the properties of $\prescript{w}{}{D}^{\leq a}$ (resp. $\prescript{w}{}{D}^{\geq a}$)
generalizes to 
$\prescript{w}{}{D}^{\leq \underline{a}}$ (resp. $\prescript{w}{}{D}^{\geq \underline{a}}$) , we summarize them as follows. 

\begin{theorem}
(\cite{MR2862060} proposition 3.4.1)
$\prescript{w}{}{D}^{\leq \underline{a}}$
and 
$\prescript{w}{}{D}^{\geq \underline{a}}$
are triangulated subcategories of $D^b_m(X)$  that are stable under extensions. If $\underline{a}= (a, \cdots , a)$, then 
$\prescript{w}{}{D}^{\geq \underline{a}} = \prescript{w}{}{D}^{\geq a}$
and 
$\prescript{w}{}{D}^{\leq \underline{a}} = \prescript{w}{}{D}^{\leq a}$.

For $Y$ another scheme with strata 
$\{S'_i\}_{0 \leq i \leq n}$ 
satisfying the same condition as before, and $f : X \rightarrow Y$ a morphism such that $f(S_k) \subset S'_k$, assume the dimension of the fibers of $f$ is smaller than or equal to $d$, then we have 
\[
Rf_{!}(\prescript{w}{}{D}^{\leq \underline{a}}(X)) \subset \prescript{w}{}{D}^{\leq \underline{a}+\underline{d}}(Y)
\]
\[
f^{*}(\prescript{w}{}{D}^{\leq \underline{a}}(Y)) \subset \prescript{w}{}{D}^{\leq \underline{a}+\underline{d}}(X)
\]
\[
Rf_{*}(\prescript{w}{}{D}^{\geq \underline{a}}(X)) \subset \prescript{w}{}{D}^{\geq \underline{a}-\underline{d}}(Y)
\]
\[
f^{!}(\prescript{w}{}{D}^{\geq \underline{a}}(Y)) \subset \prescript{w}{}{D}^{\geq \underline{a}-\underline{d}}(X)
\]
Further, we have 
\[
D \circ w_{\leq \underline{a}} = w_{\geq - \underline{a}} \circ D
\]
\end{theorem}

The next proposition tells us how to compute $w_{\leq \underline{a}}$ and $w_{\geq \underline{a}}$ in terms of $w_{\leq a}$ and $w_{\geq a}$. 

\begin{proposition} \label{split}
(\cite{MR2862060} proposition 3.3.4)
Let $k \in \{0 \cdots n\}$ and $a \in \mathbb{Z}\cup \{\infty \}$, we denote 
\[
w^k_{\leq a} := w_{\leq (\infty, \cdots,\infty, a,\infty, \cdots, \infty) }
\]
\[
w^k_{\geq a} := w_{\geq (\infty, \cdots,\infty, a,\infty, \cdots, \infty) }
\]
where $a$ sits in the $k$-th position. We have 
\[
w_{\leq \underline{a}} = w^n_{\leq a_n } \circ \cdots \circ w^0_{\leq a_0}  
\]
\[
w_{\geq \underline{a}} = w^n_{\geq a_n } \circ \cdots \circ w^0_{\geq a_0}  
\]
For $K \in D^b_m(X)$, we have distinguished triangles 
\[
w_{\leq a}^kK \longrightarrow K \longrightarrow Ri_{k*} w_{\geq a+1} i_k^*K \overset{+1}{\longrightarrow} \cdot  
\]
\[
i_{k!} w_{\leq a-1} i_k^!K \longrightarrow K \longrightarrow w_{\geq a}^k K \overset{+1}{\longrightarrow} \cdot  
\]
\end{proposition}

\begin{corollary} \label{co}
We have natural isomorphisms 
\[
i_k^* \circ w_{\leq a}^k = w_{\leq a} \circ i_k^*
\]
\[
i_k^! \circ w_{\geq a}^k = w_{\geq a} \circ i_k^!
\]
and 
\[
i_j^* \circ w_{\leq a}^k = i_j^*
\]
\[
i_j^! \circ w_{\geq a}^k = i_j^! 
\]
for $j < k$. 
\end{corollary}

\begin{theorem}
(\cite{MR2862060} proposition 3.4.2)
Let $U := S_0$ and $j= i_0: U \hookrightarrow X $ be the inclusion of the open stratum, then for $K\in D_m^b(U)$ a pure perverse sheaf of weight $a$ we have 
\[
w_{\geq (a, a+1, \cdots, a+1)} j_!K = j_{!*}K = w_{\leq (a, a-1, \cdots, a-1)}Rj_*K
\]
\end{theorem}

\subsection{Applications to Shimura varieties} \label{app Sh}

In this section, we take $X$ to be a Shimura variety associated to a Shimura datum $(G, \mathscr{X})$, where $G$ is a reductive group over $\mathbb{Q}$ and $\mathscr{X}$ is a conjugacy class of cocharacters 
$Res_{\mathbb{C}/\mathbb{R}} \textbf{G}_m \longrightarrow G_{\mathbb{R}}$.
The pair has to satisfy a list of axioms to be a Shimura datum which we will not review.
We assume that $X$ is smooth, which can always be achieved if we take a small enough level structure. An important property of Shimura varieties is that they have a canonical model over a number field $F$, called the reflex field of $(G,\mathscr{X})$. For simplicity, we assume that $G$ is simple. 

An algebraic rational representation of $G$ gives naturally an admissible variation of Hodge structure on $X$, whence a mixed Hodge module. The representation creates a smooth $l$-adic sheaf on $X$ as well. However, unlike Hodge modules,  the $l$-adic sheaf is not known to be mixed in general, although this is expected to be the case. Fortunately, we know that the associated $l$-adic sheaves are of geometric origin, hence mixed,  if the Shimura variety is of abelian type. We will only need to work with Shimura varieties of PEL type (up to similitude) in this paper, so we make this assumption from now on. We note that PEL type Shimura varieties have the hereditary property that strata of the minimal compactification are also of PEL type.

Let $X^{\text{min}}$ be the minimal compactification of $X$, it has a natural stratification 
$X^{\text{min}} = \underset{0 \leq i \leq n}{\cup} S_i $
with $S_0 =X$ and $S_k$ open in 
$\underset{k \leq i \leq n}{\cup} S_i$ 
for each $k$. Each $S_i$ is the union of standard strata corresponding to parabolic subgroups of $G$ of a fixed type. We will not give an explicit description of $S_i$ here, see Nair for details.

Let $V$ be a rational algebraic representation of $G$, and
$\mathcal{F}V \in D^b_m(X)$  
the corresponding sheaf. We note that $\mathcal{F}V$ is concentrated in degree 0 and smooth. Let 
$j : X \hookrightarrow X^{\text{min}}$
be the open embedding, applying 
$R\Gamma (X^{\text{min}}, -)$ 
to the weight truncations 
$w_{\leq a}Rj_* (\mathcal{F}V)$
of 
$Rj_* (\mathcal{F}V)$
induces a spectral sequence 
\[
E_1^{p,q}= H^{p+q}(X^{\text{min}}, w_{\geq -p} w_{\leq -p} Rj_*(\mathcal{F}V)) \Rightarrow H^{p+q}(X, \mathcal{F}V)
\]
Since $G$ is simple reductive, we can assume that $V$ is irreducible and pure of weight $-a$. Note that the weight of $V$ is the weight of the representation 
$ \textbf{G}_{m\mathbb{R}} \hookrightarrow \text{Res}_{\mathbb{C}/\mathbb{R}}\textbf{G}_m \overset{h}{\rightarrow} G_{\mathbb{R}} \rightarrow \text{End}(V_{\mathbb{R}})$
for one (and hence any) $h \in \mathscr{X}$. Then $\mathcal{F}V$ is pure of weight $a$, and the first nontrivial truncation of 
$Rj_*(\mathcal{F}V))$
is 
\begin{equation} \label{a1}
w_{\geq a} w_{\leq a} Rj_*(\mathcal{F}V) =w_{\leq a}Rj_*(\mathcal{F}V) = j_{!*}(\mathcal{F}V)
\end{equation}
by proposition \ref{functors} and theorem \ref{ie}. It completes into a distinguished triangle 
\[
j_{!*}(\mathcal{F}V) \longrightarrow Rj_*(\mathcal{F}V) \longrightarrow w_{\geq a+1}Rj_*(\mathcal{F}V) \overset{+1}{\longrightarrow} \cdot
\]
which shows that  $w_{\geq a+1}Rj_*(\mathcal{F}V)$
has support in the complement of $X$ as 
$j^*j_{!*} = j^*Rj_*=id$.
Let 
$i : \underset{1 \leq i \leq n}{\cup} S_i \hookrightarrow X^{\text{min}}$ 
be the complement of $X$, then 
\[
\prescript{w}{}{D}^{\geq a+1} \ni
w_{\geq a+1}Rj_*(\mathcal{F}V) = i_* (i^* w_{\geq a+1}Rj_*(\mathcal{F}V))  
\]
Since $i_* = i_!$ is exact with respect to the weight $t$-structure by proposition \ref{functors}, 
$i^* w_{\geq a+1}Rj_*(\mathcal{F}V) \in \prescript{w}{}{D}^{\geq a+1}$. 
Applying $i^*$ to the distinguished triangle 
\[
w_{\leq a}Rj_*(\mathcal{F}V) \longrightarrow Rj_*(\mathcal{F}V) \longrightarrow w_{\geq a+1}Rj_*(\mathcal{F}V) \overset{+1}{\longrightarrow} \cdot
\]
we have 
\[
\begin{tikzcd}
i^*w_{\leq a}Rj_*(\mathcal{F}V) \arrow[r] & i^*Rj_*(\mathcal{F}V) \arrow[r] & 
i^* w_{\geq a+1}Rj_*(\mathcal{F}V) \arrow[r, "+1"] &
\cdot 
\end{tikzcd}
\]
We see by proposition \ref{functors} that 
$i^*w_{\leq a}Rj_*(\mathcal{F}V) \in \prescript{w}{}{D}^{\leq a}$.
Together with 
$i^* w_{\geq a+1}Rj_*(\mathcal{F}V) \in \prescript{w}{}{D}^{\geq a+1}$ 
that we have just observed, we obtain 
\[
i^* w_{\geq a+1}Rj_*(\mathcal{F}V) = w_{\geq a+1} i^*Rj_*(\mathcal{F}V)
\]
Therefore 
\[
w_{\leq a+1} w_{\geq a+1} Rj_*(\mathcal{F}V) = w_{\leq a+1} i_* w_{\geq a+1} i^*Rj_*(\mathcal{F}V)= i_* w_{\leq a+1} w_{\geq a+1} i^*Rj_*(\mathcal{F}V)
\]
and similarly 
\begin{equation} \label{b1}
w_{\leq a+k} w_{\geq a+k} Rj_*(\mathcal{F}V)=
i_* w_{\leq a+k} w_{\geq a+k} i^*Rj_*(\mathcal{F}V)
\end{equation}
for all $k>0$ (applying $w_{\leq a+k} w_{\geq a+k}$ to $w_{\geq a+   1} Rj_*(\mathcal{F}V)$ and use that $w_{\geq a +k} w_{\geq a+1} =w_{\geq a+k}$). 

It is shown that
$i^*Rj_*(\mathcal{F}V)$
is constructible with respect to the standard stratification (and in particular for $\{S_i\}$) by Burgos and Wildeshaus (\cite{burgos2004hodge}) in the Hodge Module case, and Pink (\cite{MR1149032}) in the $l$-adic case. Moreover, the restriction of 
$i^*Rj_*(\mathcal{F}V)$
to strata have automorphic cohomology sheaves in the sense that they are associated to algebraic representations of the group corresponding to the strata as a Shimura variety. We claim that 
$w_{\leq a+k} w_{\geq a+k} i^*Rj_*(\mathcal{F}V)$
is also constructible with respect to the standard stratification, and even automorphic when restricted to each stratum. Indeed, by proposition \ref{split}
\[
w_{\leq a+k}i^*Rj_*(\mathcal{F}V) = w_{\leq (a+k, \cdots, a+k)}i^*Rj_*(\mathcal{F}V) = w^n_{\leq a+k } \circ \cdots \circ w^1_{\leq a+k} i^*Rj_*(\mathcal{F}V)
\]
and there is a distinguished triangle
\[
w^1_{\leq a+k} i^*Rj_*(\mathcal{F}V) \longrightarrow i^*Rj_*(\mathcal{F}V) \longrightarrow i_{1*} w_{\geq a+k+1} i_1^* i^*Rj_*(\mathcal{F}V) \overset{+1}{\longrightarrow} \cdot
\]
Since both 
$i^*Rj_*(\mathcal{F}V)$
and 
$i_{1*} w_{\geq a+k+1} i_1^* i^*Rj_*(\mathcal{F}V)$ 
are constructible and automorphic with respect to the standard stratification  (using $w_{\leq a}\mathcal{F}V = \mathcal{F}(w_{\geq dimX-a}V)$, see \cite{MR2862060} 4.1.2),
so is 
$w^1_{\leq a+k} i^*Rj_*(\mathcal{F}V)$. 
The same argument applies to   
$w_{\leq a+k}^2$
by replacing 
$i^*Rj_*(\mathcal{F}V)$
to
$w^1_{\leq a+k} i^*Rj_*(\mathcal{F}V)$,
and an easy induction proves that 
$w_{\leq a+k}i^*Rj_*(\mathcal{F}V)$
is constructible and automorphic. The claim follows from the distinguished traingle
\[
w_{\leq a+k-1}i^*Rj_*(\mathcal{F}V) \longrightarrow w_{\leq a+k}i^*Rj_*(\mathcal{F}V) 
\longrightarrow w_{\geq a+k} w_{\leq a+k}i^*Rj_*(\mathcal{F}V)
\overset{+1}{\longrightarrow} \cdot
\]
and what we have just proved for the first two terms. Note that $w_{\geq a+k} w_{\leq a+k} =w_{\leq a+k} w_{\geq a+k}  $, see \cite{nairmixed} 2.2.3. 

Recall that proposition \ref{splitting} tells us that 
$w_{\geq a+k} w_{\leq a+k}i^*Rj_*(\mathcal{F}V)$
decompose into shifts of pure perverse sheaves, and the claim we have just proved shows that these perverse sheaves are intermediate extensions of automorphic sheaves on the standard strata.
This is also true for 
$w_{\geq a+k}w_{\leq a+k} Rj_*(\mathcal{F}V)$ 
using (\ref{a1}) and (\ref{b1}). 
We know that the normalization of the closure of a strata is 
the minimal compactification of the strata, and intersection cohomology is invariant under normalization, hence
$\mathbb{H}^{p+q}(X^{\text{min}}, w_{\geq -p} w_{\leq -p} Rj_*(\mathcal{F}V))$
is a sum of intersection cohomology of the minimal compactification of the strata with coefficients automorphic sheaves. We now summarize what we have proved. 

\begin{theorem} (Nair \cite{nairmixed}) \label{w.s.s.}
For $X$ a Shimura variety of PEL type with Shimura data $(G, \mathscr{X})$, and $V$ a representation of $G$, we have a spectral sequence 
\[
E_1^{p,q}= \mathbb{H}^{p+q}(X^{\text{min}}, w_{\geq -p} w_{\leq -p} Rj_*(\mathcal{F}V)) \Rightarrow H^{p+q}(X, \mathcal{F}V)
\]
where 
$\mathbb{H}^{p+q}(X^{\text{min}}, w_{\geq -p} w_{\leq -p} Rj_*(\mathcal{F}V))$
is a sum of 
$IH^*(Y^{\text{min}}, \mathcal{F}W) := \mathbb{H}^*(Y^{\text{min}}, j_{!*}\mathcal{F}W)$ 
with $Y \subset X^{\text{min}}$ a standard strata, and $W$ an algebraic representation of the group associated to $Y$. 
\end{theorem}

\begin{remark}
It is possible to write $E_1^{p,q}$ more explicitly, using Pink (\cite{MR1149032}) or Burgos and Wildeshaus' (\cite{burgos2004hodge}) results. We will do that with Hilbert modular varieties later. 
\end{remark}

We know that the PEL type Shimura variety $X$ has a natural smooth integral model $\mathfrak{X} $ over an open subset $\mathcal{U}$ of $\text{Spec}\mathcal{O}_F $, and the automorphic sheaf $\mathcal{F}V$ extends to $\mathfrak{X}$, which we still denote by $\mathcal{F}V$. Let $Spec (k)$ be a closed point of $\mathcal{U}$, hence $k$ is a finite field. The above theorem gives us two spectral sequences
\[
{}_H E_1^{p,q}= \mathbb{H}^{p+q}(X^{min}(\mathbb{C}), w_{\geq -p} w_{\leq -p} Rj_*({}_H \mathcal{F}V)) \Rightarrow H^{p+q}(X(\mathbb{C}), {}_H \mathcal{F}V)
\]
and 
\[
{}_l E_1^{p,q}= \mathbb{H}^{p+q}(\mathfrak{X}_{\bar{k}}^{min}, w_{\geq -p} w_{\leq -p} Rj_*({}_l \mathcal{F}V)) \Rightarrow H^{p+q}(\mathfrak{X}_{\bar{k}}, {}_l \mathcal{F}V)
\]
where ${}_H \mathcal{F}V$ is the ($\mathbb{C}$-) Hodge module associated to $V$ (it is normalized so that 
$rat({}_H \mathcal{F}V) = \mathcal{F}V[0] \in D^b_c(X(\mathbb{C}), \mathbb{Q})$, 
in other words, 
${}_H \mathcal{F}V \in D^bMHM(X(\mathbb{C}))$ sit in degree $\text{dim}X$ ) in the first spectral sequence and 
${}_H E_1^{p,q}$ 
is obtained from the weight truncation in $D^b MHM(X(\mathbb{C}))$. Similarly, ${}_l \mathcal{F}V$ is the mixed $l$-adic lisse sheaf assoicated to $V$ in the second one, and the spectral sequence is obtained by looking at the weight truncation in $D^b_m(\mathfrak{X}_k, \overline{\mathbb{Q}_l})$
and then passing to the algebraic closure of $k$. Note that the first spectral sequence takes values in (complex) mixed Hodge structures, while the second takes values in $\text{Gal}(\overline{k}/k)$-modules. 
The next theorem provides a comparision between the two spectral sequences. Since it seems not to be in the literature, we give a proof. 

\begin{theorem} \label{comparison thm}
Fix an isomorphism 
$\imath: \mathbb{C} \cong \overline{\mathbb{Q}_l}$, 
then for all but finitely many $\text{Spec}(k) \subset \mathcal{U}$, 
there is a natural isomorphism 
\[
\imath_* H^n(X(\mathbb{C}),{}_H \mathcal{F}V) \cong H^n(\mathfrak{X}_{\overline{k}}, {}_l \mathcal{F}V)
\]
as $\overline{\mathbb{Q}_l}$-vector spaces, and the filtrations induced by ${}_H E_1^{p,q}$ and ${}_l E_1^{p,q}$ are identified through the isomorphism. 
\end{theorem}

\begin{proof}
Recall that
$D^b_m(X/F, \overline{\mathbb{Q}_l})$
is the derived category of horizontal mixed complexes on $X/F$, which is defined by the direct limit of suitable subcategories of $D^b_c(\mathfrak{X}_{\mathcal{V}}, \overline{\mathbb{Q}_l})$, 
indexed by open subsets
$\mathcal{V} \subset \mathcal{U}$. 
Since ${}_l \mathcal{F}V$ extends to $\mathfrak{X}$, it defines an element 
${}_l \mathcal{F}V \in D^b_m(X/F, \overline{\mathbb{Q}_l})$. As $X^{\text{min} }$ also descends to a canonical model
$\mathfrak{X}^{\text{min}}$
over $\mathcal{U}$, 
which is a compactification of $\mathfrak{X}$, 
we have that 
$Rj_* ({}_l \mathcal{F}V) \in D^b_c(\mathfrak{X}^{\text{min}}, \overline{\mathbb{Q}_l})$
defines an element of 
$D^b_m(X^{\text{min}}/F, \overline{\mathbb{Q}_l})$. 

The weight t-structure on 
$D^b_m(X^{\text{min}}/F, \overline{\mathbb{Q}_l})$
gives the truncations 
$w_{\leq a} Rj_* ({}_l \mathcal{F}V) \in D^b_m(X^{\text{min}}/F, \overline{\mathbb{Q}_l})$, which are represented by complexes on $\mathfrak{X}^{\text{min}}_{\mathcal{V}}$ 
for some nonempty open subset 
$\mathcal{V} \subset \mathcal{U} $
by definition of the horizontal complexes. Since there are only finitely many truncations, we can assume that $\mathcal{V}$ is chosen such that all the truncations are represented by complexes on $\mathfrak{X}^{\text{min}}_{\mathcal{V}}$,
which we still denote by 
$w_{\leq a} Rj_* ({}_l \mathcal{F}V) \in D^b_c(\mathfrak{X}^{\text{min}}_{\mathcal{V}}, \overline{\mathbb{Q}_l})$.

Recall that weights on  
$D^b_m(X^{\text{min}}/F, \overline{\mathbb{Q}_l})$
are defined by first reducing to finite fields and then taking the weights there. We have basically from definition that 
\[
(w_{\leq a} Rj_* ({}_l \mathcal{F}V))_k = w_{\leq a} ((Rj_{*} ({}_l \mathcal{F}V))_k)\in D^b_m(\mathfrak{X}^{\text{min}}_k, \overline{\mathbb{Q}_l})
\]
By the lemma below, we have 
$(Rj_{*} ({}_l \mathcal{F}V))_k = Rj_{k*} ({}_l \mathcal{F}V|_{\mathfrak{X}_k})$, 
where 
$j_k : \mathfrak{X}_k \hookrightarrow \mathfrak{X}^{\text{min}}_k$
is the base change of $j$ to $k$. Thus we have 
\begin{equation} \label{err}
(w_{\leq a} Rj_* ({}_l \mathcal{F}V))_k = w_{\leq a}Rj_{k*} ({}_l \mathcal{F}V|_{\mathfrak{X}_k})
\end{equation}
We base change 
$w_{\leq a} Rj_* ({}_l \mathcal{F}V) \in D^b_c(\mathfrak{X}^{\text{min}}_{\mathcal{V}}, \overline{\mathbb{Q}_l})$
to a complex points of $\mathcal{V}$, then the comparison between etale and classical sites and that $\mathcal{F}V$ is of geometric origin provide us with a natural isomorphism
\begin{equation} \label{afjkd}
(w_{\leq a} Rj_* ({}_l \mathcal{F}V))_{\mathbb{C}} \cong \imath_* rat(w_{\leq a} Rj_{\mathbb{C}*} ({}_H \mathcal{F}V))
\end{equation}

Let $\mathcal{V}_{(k)}$ be the etale localization of $\mathcal{V}$ at $spec(k)$ and $\bar{\eta}$ the geometric generic point of $\mathcal{V}_{(k)}$.   
By properness of 
$g: \mathfrak{X}^{\text{min}}_{\mathcal{V}_{(k)}} \rightarrow \mathcal{V}_{(k)}$,
we have 
\[
R\Gamma(\mathfrak{X}_{\bar{\eta}}^{\text{min}}, (w_{\leq a} Rj_* ({}_l \mathcal{F}V))_{\bar{\eta}})=
R\Gamma(\bar{\eta}, (Rg_*w_{\leq a} Rj_* ({}_l \mathcal{F}V))_{\bar{\eta}})=
R\Gamma(\mathcal{V}_{(k)}, Rg_*w_{\leq a} Rj_* ({}_l \mathcal{F}V))
\]
\[
= (Rg_*w_{\leq a} Rj_* ({}_l \mathcal{F}V))_{\bar{k}}
=R\Gamma(\mathfrak{X}^{\text{min}}_{\bar{k}} , (w_{\leq a} Rj_* ({}_l \mathcal{F}V))_{\bar{k}})
\]
Together with eqution (\ref{err}), we have 
\begin{equation} \label{akfkdjfdfsdf}
R\Gamma(\mathfrak{X}_{\bar{\eta}}^{\text{min}}, (w_{\leq a} Rj_* ({}_l \mathcal{F}V))_{\bar{\eta}})=
R\Gamma(\mathfrak{X}^{\text{min}}_{\bar{k}} , w_{\leq a}Rj_{k*} ({}_l \mathcal{F}V|_{\mathfrak{X}_k}))
\end{equation}
Choose an embedding of $\bar{\eta}$ into $\mathbb{C}$, then (\ref{afjkd}) and (\ref{akfkdjfdfsdf}) gives us 
\begin{equation}  \label{jkjkjkl}
\imath_* H^n(X^{\text{min}}(\mathbb{C}),w_{\leq a} Rj_{\mathbb{C}*} ({}_H \mathcal{F}V)) \cong 
H^n(X^{\text{min}}_{\mathbb{C}}, (w_{\leq a} Rj_* ({}_l \mathcal{F}V))_{\mathbb{C}}) 
\end{equation}
\[
=H^n(\mathfrak{X}_{\bar{\eta}}^{\text{min}}, (w_{\leq a} Rj_* ({}_l \mathcal{F}V))_{\bar{\eta}}) 
=H^n(\mathfrak{X}^{\text{min}}_{\bar{k}} , w_{\leq a}Rj_{k*} ({}_l \mathcal{F}V|_{\mathfrak{X}_k}))
\]

We know by definition of the spectral sequence ${}_H E_1^{p,q}$ that the image of 
$\imath_* H^n(X^{\text{min}}(\mathbb{C}),w_{\leq a} Rj_{\mathbb{C}*} ({}_H \mathcal{F}V))$
in 
\[
\imath_* H^n(X^{\text{min}}(\mathbb{C}), Rj_{\mathbb{C}*} ({}_H \mathcal{F}V)) = 
\imath_* H^n(X(\mathbb{C}), {}_H \mathcal{F}V)
\]
is the filtration corresponding to 
${}_H E_1^{p,q}$, and similarly for
$H^n(\mathfrak{X}^{\text{min}}_{\bar{k}} , w_{\leq a}Rj_{k*} ({}_l \mathcal{F}V|_{\mathfrak{X}_k}))$. 
The isomorphism (\ref{jkjkjkl}) for $a$ large enough defines the isomorphism in the statement of the theorem, and it respects the filtration by what we have just observed. 
\end{proof}

\begin{lemma}
Let $j: \mathfrak{X} \hookrightarrow \mathfrak{X}^{\text{min}}$
be the inclusion of a PEL Shimura variety into its minimal compactification, defined over 
$\mathcal{U} \subset \mathcal{O}_F$, and $V$ an algebraic representation of the group $G$ associated to the Shimura variety. Then for $Spec(k) \subset \mathcal{U}$ a closed point, we have an isomorphism
\[
(Rj_{*} ({}_l \mathcal{F}V))_k \cong Rj_{k*} ({}_l \mathcal{F}V|_{\mathfrak{X}_k})
\]
induced by the base change map.  

\end{lemma}

\begin{proof}
We know that ${}_l \mathcal{F}V$ is up to a Tate twist a summand of 
$Rf_* \overline{\mathbb{Q}_l}$,
where 
$f: \mathcal{A}^{\times n}  \rightarrow \mathfrak{X}$
is the structure map of the $n$-th fiber product of the universal abelian scheme 
$\mathcal{A} \rightarrow \mathfrak{X}$,
i.e. $ \mathcal{A}^{\times n} := \mathcal{A} \times_{\mathfrak{X}} \cdots \times_{\mathfrak{X}} \mathcal{A}$, 
for some integer $n$. By \cite{MR3719245} 4.1, we have that 
$\mathcal{A}^{\times n}$ is 
$\mathbb{Z}^{\times}_{(l)}$-isogenous to another abelian scheme $Y $ over $\mathfrak{X}$ such that $Y$ extends to a proper scheme $\overline{Y}$ over $\mathfrak{X}^{\text{tor}}_{\Sigma}$ 
for some choice of smooth projective toroidal compactification 
$\mathfrak{X}^{\text{tor}}_{\Sigma} /\mathcal{U}$ 
with 
$\overline{Y} \setminus Y$ 
union of normal corssing diviosrs over $\mathcal{U}$. Since $\mathbb{Z}^{\times}_{(l)}$-isogeny does not change Tate modules, we see that ${}_l \mathcal{F}V$ is (up to a Tate twist) a summand of 
$R\pi_* \overline{\mathbb{Q}_l}$,
where 
$\pi : Y \rightarrow \mathfrak{X}$ the structure map. Thus it suffices to show 
\[
(Rj_{*} R\pi_* \overline{\mathbb{Q}_l})_k \cong Rj_{k*} (R\pi_* \overline{\mathbb{Q}_l}|_{\mathfrak{X}_k})
\]

Let 
$\overline{\pi} : \overline{Y} \rightarrow \mathfrak{X}^{\text{tor}}_{\Sigma}$
be the extension of $\pi$, 
$J: \mathfrak{X} \rightarrow \mathfrak{X}^{\text{tor}}_{\Sigma}$,
$J_Y : Y \rightarrow \overline{Y}$
the inclusion, which form a catesian diagram
\[
\begin{tikzcd}
Y \arrow[r, hook, "J_Y"] \arrow[d,"\pi"] & 
\overline{Y} \arrow[d, "\overline{\pi}"]
\\
\mathfrak{X} \arrow[r ,hook,  "J"] & 
\mathfrak{X}^{tor}_{\Sigma} 
\end{tikzcd}
\]
Let 
$\phi:  \mathfrak{X}^{tor}_{\Sigma} \rightarrow \mathfrak{X}^{min}$ 
be the natural proper projection map. We denote by $\pi_k $ for the base change of $\pi$ to $k$, and similarly for the other maps. We know that 
$j = \phi \circ J$, 
so 
\[
(Rj_{*} R\pi_* \overline{\mathbb{Q}_l})_k =
(R\phi_{*} RJ_* R\pi_* \overline{\mathbb{Q}_l})_k=
R\phi_{k*} (RJ_* R\pi_* \overline{\mathbb{Q}_l})_k 
\]
\[
=R\phi_{k*}(R\overline{\pi}_* RJ_{Y*}  \overline{\mathbb{Q}_l})_k =
R\phi_{k*} R\overline{\pi}_{k*} ( RJ_{Y*}  \overline{\mathbb{Q}_l})_k
\]
by proper base change. Moreover, 
\[
( RJ_{Y*}  \overline{\mathbb{Q}_l})_k =
  RJ_{Yk*}  \overline{\mathbb{Q}_l}
\]
by 5.1.3 in 7.5 of SGA 4.5 (\cite{MR463174}), where we use that 
$\overline{Y} \setminus Y$ 
are union of normal crossing divisors over $\mathcal{U}$. This gives 
\[
(Rj_{*} R\pi_* \overline{\mathbb{Q}_l})_k =
R\phi_{k*} R\overline{\pi}_{k*} RJ_{Yk*}  \overline{\mathbb{Q}_l}=
R\phi_{k*} RJ_{k*} R\pi_{k*}\overline{\mathbb{Q}_l}=
Rj_{k*} (R\pi_* \overline{\mathbb{Q}_l}|_{\mathfrak{X}_k})
\]
by proper base change again, proving the claim. 
\end{proof}

Lastly, we record the functoriality of the spectral sequence $E^{p,q}_1$. 

\begin{proposition} \label{func.s.s}
Let $\overline{X}$ and $\overline{Y}$ be varieties defined over a field $k$ which is either finitely generated over its prime field or the complex number, as in the previous section. Let  $X \subset \overline{X}$ and 
$Y \subset \overline{Y}$ 
be nonempty open subvarieties,  
$\overline{f}: \overline{X} \rightarrow \overline{Y}$
a finite morphism which restricts to a morphism 
$f : X \rightarrow Y$ making the following diagram cartesian 
\[
\begin{tikzcd}
X \arrow[r, hook, "j_X"] \arrow[d, "f"] &
\overline{X} \arrow[d,"\overline{f}"] \\
Y \arrow[r,hook, "j_Y"] &
\overline{Y} 
\end{tikzcd}
\]
Let $\mathcal{F} \in D^b_m(X)$
and 
$\mathcal{G} \in D^b_m(Y)$ 
together with a morphism 
$h: \mathcal{G} \rightarrow Rf_*\mathcal{F}$, 
then $h$ induces a morphism 
\[
{}_Y E_1^{p,q} \longrightarrow {}_X E_1^{p,q}
\]
between the weight spectral sequences
\[
{}_X E_1^{p,q}= \mathbb{H}^{p+q}(\overline{X}, w_{\geq -p} w_{\leq -p} Rj_{X*}\mathcal{F}) \Rightarrow \mathbb{H}^{p+q}(X, \mathcal{F})
\]
\[
{}_Y E_1^{p,q}= \mathbb{H}^{p+q}(\overline{Y}, w_{\geq -p} w_{\leq -p} Rj_{Y*}\mathcal{G}) \Rightarrow \mathbb{H}^{p+q}(Y, \mathcal{G})
\]
In particular the morphism 
$R^{p+q}\Gamma(Y, -) (h)  : \mathbb{H}^{p+q}(Y, \mathcal{G}) \rightarrow \mathbb{H}^{p+q}(X, \mathcal{F})$
respects filtrations induced by the spectral sequences. 
\end{proposition}

\begin{proof}
Observe that $h$ induces a morphism
\begin{equation} \label{ kjkl}
Rj_{Y*} \mathcal{G} \overset{Rj_{Y*} h}{ \longrightarrow} Rj_{Y*} Rf_* \mathcal{F} = R\overline{f}_* Rj_{X*} \mathcal{F}
\end{equation}
and applying
$R\Gamma(\overline{Y}, -)$ 
to it recovers the usual morphism induced by $h$
\[
R\Gamma(Y, -) (h)  : R\Gamma(Y, \mathcal{G}) \longrightarrow R\Gamma(X, \mathcal{F})
\]
which is the sought-after morphism on $E_{\infty}$. 
Applying the functor $w_{\leq a}$ to (\ref{ kjkl}) gives us 
\begin{equation} \label{aaaa}
w_{\leq a} Rj_{Y*} \mathcal{G} \longrightarrow w_{\leq a} R\overline{f}_* Rj_{X*} \mathcal{F} 
= R\overline{f}_* w_{\leq a} Rj_{X*} \mathcal{F}
\end{equation}
where in the last equality we use that $\overline{f}$ is finite, hence $\overline{f}_* =\overline{f}_!$ preserves both $\prescript{w}{}{D}^{\leq a}$ 
and 
$\prescript{w}{}{D}^{\geq a}$ 
by proposition \ref{functors} ($d=0$ as $f$ is finite). Now (\ref{aaaa}) shows that the morphism (\ref{ kjkl}) preserves the filtration induced by the weight truncation $w_{\leq a}$, hence defining a morphism between spectral sequences as desired. 
\end{proof}

\section{PEL moduli problems}

\subsection{Kottwitz's PEL moduli problems} \label{kpelmp}

We begin by recalling the definition of PEL moduli problems given by Kottwitz in \cite{MR1124982}. We follow the notation of Lan (\cite{lan2013arithmetic}). 

Let $B$ be a finite dimensional simple algebra over $\mathbb{Q}$ with a positive involution $*$, and $\mathcal{O}$ a $\mathbb{Z}$-order in $B$ that is invariant under $*$ and maximal at $p$, where $p$ is a rational prime that is unramified in B, i.e. $B_{\mathbb{Q}_p} \cong M_n(K)$ for some finite unramified extension $K$ of $\mathbb{Q}_p$. Let $L$ be an $\mathcal{O}$-lattice in a finite dimensional $B$-module $V$, and $\langle \cdot, \cdot\rangle: L \times L \rightarrow \mathbb{Z}(1) $ an alternating nondegenerate bilinear form on $L$ which satisfies $\langle\alpha x,y\rangle = \langle x, \alpha^*y\rangle$ for $\alpha \in \mathcal{O}$ and $x,y \in L$. We also assume that when localized at $p$, $L$ is self-dual with respect to $\langle\cdot, \cdot\rangle$. Here we denote $\mathbb{Z}(1) := \text{Ker}(\text{exp}: \mathbb{C} \rightarrow \mathbb{C}^*)$. A choice of $\sqrt{-1}$ gives an identification of it with $\mathbb{Z}$, but we do not fix such an identification. 

We assume that there is an $\mathbb{R}$-algebra homormorphism $h : \mathbb{C} \rightarrow \text{End}_{\mathcal{O}_{\mathbb{R}}}(L_{\mathbb{R}})$ such that $\langle h(z)x,y\rangle_{\mathbb{R}}=\langle x,h(\bar{z})y\rangle_{\mathbb{R}}$ and $\langle\cdot, h(\sqrt{-1}) \cdot\rangle_{\mathbb{R}}$  is symmetric and positive definite, if we fix an identification $\mathbb{Z}(1) \cong \mathbb{Z}$ so that $\langle\cdot, \cdot\rangle_{\mathbb{R}}$ takes values in $\mathbb{R}$. Let 
\[G^*(R) := \{(g,r) \in GL_{\mathcal{O}_R} (L_R) \times \textbf{G}_m(R): \langle gx,gy\rangle = r\langle x,y\rangle, \forall x,y \in L_R \} \]    
for an $\mathbb{Z}$-algebra $R$, this defines an algebraic group $G^*$ over $\mathbb{Z}$. We assume that the derived group has type $\textbf{A}$ or 
$\textbf{C}$ in the classification. 

The morphism $h$ defines a decomposition $L \otimes \mathbb{C} = V_0 \oplus V_0^c$, where $h(z)$ acts as $1 \otimes z$ on $V_0$ and $1 \otimes \bar{z}$ on $V_0^c$. We know that $V_0 $ is an $\mathcal{O} \otimes \mathbb{C}$-module since $h(z)$ commutes with $\mathcal{O}_{\mathbb{R}}$ by definition. The reflex field $F_0$ is defined to be the field of definition of $V_0$ as an $\mathcal{O} \otimes \mathbb{C}$-module, see \cite{lan2013arithmetic} 1.2.5.4 for more details.  

\begin{definition} \label{isogeny}
Let $\mathcal{H}$ be an open compact subgroup of $G^*(\mathbb{A}^{p\infty})$, $M_{\mathcal{H}}^{rat}$ is defined to be the category fibered in groupoids over the category of locally Noetherian schemes defined over $\mathcal{O}_{F_0} \otimes \mathbb{Z}_{(p)}$, whose fiber over $S$ consists of tuples 
\[
(A, \lambda, i, [\hat{\alpha}]_{\mathcal{H}})
\]
where $A$ is an abelian scheme over S, 
\[
\lambda : A \rightarrow A^{\vee}
\]
is a prime-to-$p$ quasi-polarization of $A$, and 
\[
i : \mathcal{O} \otimes \mathbb{Z}_{(p)} \rightarrow End_S(A) \otimes \mathbb{Z}_{(p)}
\]
is a ring homomorphism such that 
\[
i(b)^{\vee} \circ \lambda= \lambda \circ i(b^*) 
\]
for every $b \in \mathcal{O} \otimes \mathbb{Z}_{(p)}$, and  $Lie_{A/S}$ satisfies the determinant condition specified by $h$, see \cite{lan2013arithmetic} 1.3.4.1 for a precise formulation. Moreover, if we choose a geometric point $\bar{s}$ in each connected component of $S$, $[\hat{\alpha}]_{\mathcal{H}}$ is an assignment to each $\bar{s}$ a  $\pi_1(S,\bar{s})$-invariant $\mathcal{H}$-orbit of $\mathcal{O} \otimes \mathbb{A}^{p\infty}$-equivariant isomorphisms 
\[\hat{\alpha} : L \otimes \mathbb{A}^{p \infty} \overset{\sim}{\rightarrow} V^pA_{\bar{s}}\]
together with an isomorphism 
\[ \nu(\hat{\alpha}) : \mathbb{A}^{p\infty}(1) \overset{\sim}{\rightarrow} V^p(\textbf{G}_{m,\bar{s}}) \]
such that 
\[
\langle \hat{\alpha}(x), \hat{\alpha}(y) \rangle_{\lambda} = \nu(\hat{\alpha}) \circ \langle x,y \rangle
\]
where $x,y \in L \otimes \mathbb{A}^{p\infty}, \langle \cdot, \cdot\rangle_{\lambda}$ is the Weil pairing associated to the polarization $\lambda$, and $V^p$ is the prime to $p$ rational Tate module of either $A$ or $\textbf{G}_m$. 

The isomorphisms in the groupoid are defined to be $(A, \lambda, i, [\hat{\alpha}]_{\mathcal{H}}) \sim (A', \lambda', i', [\hat{\alpha}']_{\mathcal{H}})$ if and only if there is a prime to $p$ quasi-isogeny $f : A \rightarrow A'$ such that over each connected component of $S$, 
\[
\lambda = r f^{\vee} \circ \lambda' \circ f
\]
for some $r \in \mathbb{Z}_{(p), \rangle0}^{\times}$, $f\circ i(b) = i'(b) \circ f $ for all $b \in \mathcal{O} \otimes \mathbb{Z}_{(p)}$. Moreover, we require that for each geometric point $\bar{s}$ of $S$, $\hat{\alpha'}^{-1} \circ V^p(f) \circ \hat{\alpha} \in \mathcal{H}$, and $\nu(\hat{\alpha}')^{-1} \circ \nu(\hat{\alpha}) \in \nu(\mathcal{H})r \subset \mathbb{A}^{p\infty,\times}$ for the $r$ specified by $\lambda = r f^{\vee} \circ \lambda' \circ f$ at $\bar{s}$. 
\end{definition}

\begin{remark} \label{r1}
Note that the moduli problem depends only on the $B$-module $V$, different choices of $L$ will only affect the choices of maximal compact subgroup of $G^*(\mathbb{A}^{p\infty})$. This will be helpful when we consider moduli problems defined by isomorphism classes where the choice of $L$ is important. We can compare different moduli problems defined by different choices of $L$ by identifying the moduli problems with the above one using isogeny classes.  
\end{remark}

\begin{remark} \label{r2}
We work with locally Noetherian test schemes rather than arbitrary schemes because etale fundamental groups do not behave well for general schemes. General (affine) schemes can be written as inverse limits of locally Noetherian schemes, and we can extend the moduli functor to the general case by taking limits.   
\end{remark}

\begin{remark} \label{r3}
$\nu(\hat{\alpha})$ is a rigidification of Kottwitz's definition of PEL moduli problems, where he allows the ambiguity that the Weil pairing is equal to the fixed pairing $\langle\cdot,\cdot\rangle$ up to a similitude factor. If $L \neq 0$, $\nu(\hat{\alpha})$ is uniquely determined by $\hat{\alpha}$, hence the two definitions are equivalent. If $L =0$, $\nu(\hat{\alpha})$ is the only non-trivial data. We include $L=0$ case because it will appear in the boundary of the minimal compactification of PEL Shimura varieties. 
\end{remark}

It is not hard to see that the moduli space is represented by an algebraic stack that is smooth of finite type over $\mathcal{O}_{F_0} \otimes \mathbb{Z}_{(p)}$, and it is even represented by a finite type smooth scheme over $\mathcal{O}_{F_0} \otimes \mathbb{Z}_{(p)}$ if $\mathcal{H}$ is small enough. We will use the same symbol $M_{\mathcal{H}}^{rat}$ to denote the stack or scheme it represents. 

The above definition uses isogeny classes of abelian varieties, we will next define another moduli problem using isomorphism classes. This is necessary for the toroidal compactifications because semiabelian varieties do not behave well under isogeny.

We will only define moduli problems for principal level structures, the general level structures can be defined by taking orbits of the principal ones, but we choose to ignore them for reasons to be explained later.  

\begin{definition} \label{isomorphism}
Let $n$ be a natural number prime to $p$, and define $M_n$ to be the category fibered in groupoids over the category of schemes over $\mathcal{O}_{F_0} \otimes \mathbb{Z}_{(p)}$, whose fiber over $S$ is  the groupoids with objects tuples 
\[
(A, \lambda, i, (\alpha_n, \nu_n))
\]
where $A$ is an abelian scheme over $S$, 
\[
\lambda : A \rightarrow A^{\vee}
\]
is a prime-to-$p$ polarization, 
\[
i : \mathcal{O} \rightarrow End_S(A)
\]
a ring homomorphism such that $i(b)^{\vee} \circ \lambda= \lambda \circ i(b^*) $ for every $b \in \mathcal{O}$. We require that $Lie_{A/S}$ satisfies the determinant condition given by $h$. The principal level-n structure is an $\mathcal{O}$-equivariant isomorphism 
\[\alpha_n : (L/nL)_S \overset{\sim}{\rightarrow} A[n]\]
together with an isomorphism 
\[\nu_n : (\mathbb{Z}/n \mathbb{Z}(1))_S \overset{\sim}{\rightarrow} \mu_{n,S}\]
of group schemes over S such that $\langle \alpha_n(x),\alpha_n(y) \rangle_{\lambda} = \nu_n \circ \langle x,y \rangle$ for $x,y \in (L/nL)_S$. The $(\alpha_n, \nu_n)$ has to satisfy a sympletic-liftablity condition which roughly says that it can be lifted to a level-m structure for arbitary $m$ that is prime to p and divisible by $n$, see \cite{lan2013arithmetic} 1.3.6.2 for precise definitions. 

The isomorphisms in the groupoid are defined to be 
\[
(A, \lambda, i, (\alpha_n, \nu_n)) \sim (A', \lambda', i', (\alpha_n’, \nu_n’))
\]
if and only if there is an isomorphism $f : A \rightarrow A'$ such that 
\[
\lambda = f^{\vee} \circ \lambda' \circ f,
\]
$f\circ i(b) = i'(b) \circ f $ for all $b \in \mathcal{O}$, and $\alpha_n’ = f \circ \alpha_n$, $\nu_n = \nu_n’$.
\end{definition}

Let 
\[\mathcal{U}(n) := Ker(G^*(\hat{\mathbb{Z}}^p) \rightarrow G^*(\mathbb{Z}/n\mathbb{Z}))\]
then we can show that $M_n \cong M_{\mathcal{U}(n)}^{\text{rat}}$, the map being the obvious one sending PEL abelian varieties to their isogeny classes. The inverse map is to choose an abelian variety in each isogeny class, determined by the choice of the $\mathcal{O}$-lattice $L$ inside $B$-module $V$, see \cite{lan2013arithmetic} 1.4.3 for a careful proof of the isomorphism. One subtle point is that $M_n$ and $M_{\mathcal{U}(n)}^{\text{rat}}$ are defined over different category of test schemes. We can show that $M_n$ is determined by its value on locally Noetherian schemes, by writing any (affine) scheme as an inverse limit of locally Noetherian ones, and note that the moduli functor $M_n$, being fintely presented, commutes with inverse limits. See also remark \ref{r2}.

\subsection{Similitude PEL moduli problems}

Let $F^c$ be the center of $B$, which is a number field by simplicity of $B$. Let $F:= F^{c,*=1}$, and we assume that $\mathcal{O}_F \subset \mathcal{O}$. We define a group scheme $H$ over $\mathcal{O}_F$ by 
\[H(R) := \{(g,r) \in GL_{\mathcal{O} \otimes_{\mathcal{O}_F }R} (L \otimes_{\mathcal{O}_F}R) \times \textbf{G}_m(R): \langle gx,gy\rangle = r\langle x,y\rangle, \forall x,y \in L\otimes_{\mathcal{O}_F}R\}\]
for an $\mathcal{O}_F$-algebra $R$. Let
\[G := Res_{\mathcal{O}_F/\mathbb{Z}} H\]
and we have the similitude map $\nu: G \rightarrow Res_{\mathcal{O}_{F}/\mathbb{Z}} \textbf{G}_m$, then 
\[G^* = \nu^{-1} (G_{m, \mathbb{Z}}) \subset G\]
Note that 
$h : \mathbb{C} \rightarrow End_{\mathcal{O}_{\mathbb{R}}}(L_{\mathbb{R}})$ 
defines a Deligne cocharacter 
$Res_{\mathbb{C}/\mathbb{R}} \textbf{G}_m \rightarrow G^*$, 
hence also 
$Res_{\mathbb{C}/\mathbb{R}} \textbf{G}_m \rightarrow G$. The conjugacy classes of them define Shimura varieties associated to $G$ and $G^*$, which we will denote by $Sh_K(G,h)$ and $Sh_{\mathcal{H}}(G^*,h)$ for compact open subgroups
$K \subset G(\hat{\mathbb{Z}}^p)$
and $\mathcal{H} \subset G^*(\hat{\mathbb{Z}}^p)$. This is abbreviated notions for $Sh_{K G(\mathbb{Z}_p)} (G,h)$ and $Sh_{\mathcal{H} G^*(\mathbb{Z}_p)} (G^*,h)$, which might be more standard.   

We have made the assumption that our PEL datum has type $\textbf{A}$ or $\textbf{C}$, then $M_{\mathcal{H}}^{rat}$ is an integral model of the Shimura variety  $Sh_{\mathcal{H}}(G^*, h)$ in case $(\textbf{A}, even)$ or $\textbf{C}$. In the case $(\textbf{A}, odd)$, $M_{\mathcal{H}}^{rat}$ is a disjoint union of integral models of the Shimura variety  $Sh_{\mathcal{H}}(G^*,h)$, due to the failure of Hasse principle. 

We will be working with the Shimura variety associated to $G$ instead of $G^*$, and one advantage of $G$ is that it always satisfies the Hasse principle. For our purpose, the more important reason are that Shimura varieties associated to $G$ is plectic, while Shimura varieties of $G^*$ are only plectic in positive dimension. More precisely, the difference between the two Shimura varieties is that they have different sets of connected components, and we have $\pi_0(Sh_K(G,h))$ is plectic while  $\pi_0(Sh_{\mathcal{H}}(G^*,h))$ is not, i.e. $\pi_0(Sh_K(G,h))$ is the zero dimensional Shimura variety associated to 
$\text{Res}_{F/\mathbb{Q}} \textbf{G}_m $ 
(plectic), 
while
$\pi_0(Sh_{\mathcal{H}}(G^*,h))$
has group $\textbf{G}_m$ (not plectic).  The plectic nature of $Sh_K(G,h)$ will give rise to the so called partial Frobenius, which will play an important role in our study. 

On the other hand, the price to pay when changing to $G$ is that we do not have a good fine moduli problem represented by $Sh_K(G,h)$. Instead, $Sh_K(G,h)$ will only be a coarse moduli space. We now give the moduli problem of $Sh_K(G,h)$. We will follow  Nekov\'a\v r's approach as in \cite{JanNekovar2018}. 

\begin{definition} \label{similitude}
Fix $\alpha \in (F \otimes \mathbb{A}^{p \infty})^{\times} $ and $K \subset G(\hat{\mathbb{Z}}^p)$ open compact subgroup, let us define $\mathscr{M}_{\alpha, K}$ to be the category fibered in groupoids over the category of locally Noetherian schemes over $\mathcal{O}_{F_0} \otimes \mathbb{Z}_{(p)}$, whose objects over $S$ are quadruples 
\[
(A, \lambda, i, \overline{(\eta, u)})
\]
where $A$ is an abelian scheme over S, 
\[
\lambda : A \rightarrow A^{\vee}
\]
is a prime-to-$p$ quasi-polarization of $A$, and 
\[
i : \mathcal{O}  \rightarrow End_S(A) 
\]
is a ring homomorphism such that $i(b)^{\vee} \circ \lambda= \lambda \circ i(b^*) $ for every $b \in \mathcal{O}$, and  $Lie_{A/S}$ satisfies the determinant condition specified by $h$, see \cite{lan2013arithmetic} 1.3.4.1 for a precise formulation. Moreover, if we choose a geometric point $\bar{s}$ in each connected component of $S$, the level structure $\overline{(\eta, u)}$ is an assignment to each $\bar{s}$ a  $\pi_1(S,\bar{s})$-invariant $K$-orbit of $\mathcal{O} \otimes \mathbb{A}^{p\infty}$-equivariant isomorphisms 
\[\eta : L \otimes \mathbb{A}^{p \infty} \overset{\sim}{\rightarrow} V^pA_{\bar{s}}\]
together with an $\mathcal{O}_F \otimes \hat{\mathbb{Z}}^{p}$-equivariant isomorphism 
\[ u : \mathfrak{d}_F^{-1} \otimes \hat{\mathbb{Z}}^{p}(1) \overset{\sim}{\rightarrow} T^p(\mathfrak{d}_F^{-1} \otimes_{\mathbb{Z}} \textbf{G}_{m,\bar{s}}) \]
such that 
\[\langle \eta(x), \eta(y) \rangle_{\lambda} =Tr_{\mathcal{O}_F/\mathbb{Z}} ( u \circ (\alpha \langle x,y\rangle_F))\]
where $x,y \in L \otimes \mathbb{A}^{p\infty}$, 
$\mathfrak{d}_F^{-1}$
is the inverse different of $F$, 
and $u$ extends naturally from 
$\mathfrak{d}_F^{-1} \otimes \mathbb{A}^{p\infty}(1)$ to the rational Tate module. Here 
$\mathfrak{d}_F^{-1} \otimes_{\mathbb{Z}} \textbf{G}_{m,\bar{s}} $ is defined in the category of fppf sheaf of abelian groups, which can be easily seen to be representable. The $\mathcal{O}_F$-action on the first factor equips $\mathfrak{d}_F^{-1} \otimes_{\mathbb{Z}} \textbf{G}_{m,\bar{s}}$ with an action of $\mathcal{O}_F$, hence defines a $\mathcal{O}_F \otimes \hat{\mathbb{Z}}^{p}$-module structure on the Tate module 
$T^p(\mathfrak{d}_F^{-1} \otimes_{\mathbb{Z}} \textbf{G}_{m,\bar{s}})$. The $Tr_{\mathcal{O}_F/\mathbb{Z}} :\mathfrak{d}_F^{-1} \otimes_{\mathbb{Z}} \textbf{G}_{m,\bar{s}} \rightarrow \mathbb{Z} \otimes_{\mathbb{Z}} \textbf{G}_{m,\bar{s}}$ is the Trace map on the first factor. Moreover, 
\[
\langle \cdot, \cdot \rangle_F: L \times L \rightarrow \mathfrak{d}_F^{-1} \otimes \mathbb{Z}(1)
\]
is the unique 
$\mathcal{O}_F$-linear pairing such that 
\[
Tr_{\mathcal{O}_F/\mathbb{Z}} \circ \langle \cdot, \cdot \rangle_F = \langle \cdot, \cdot \rangle.
\]
The action of 
$K$ on $(\eta, u)$ is given by 
$(\eta, u)g = (\eta \circ g, u \circ \nu(g))$ for $g\in K$, where $\nu(g) \in (\mathcal{O}_F \otimes \hat{\mathbb{Z}}^{p})^{\times}$ acts on $\mathfrak{d}_F^{-1} \otimes \hat{\mathbb{Z}}^{p}(1)$ in the obvious way. 

The isomorphisms in the groupoid are defined to be 
\[
(A, \lambda, i, \overline{(\eta, u)}) \sim (A', \lambda', i', \overline{(\eta’, u’)})
\]
if and only if there is a prime-to-$p$ quasi-isogeny $f : A \rightarrow A'$ such that over each connected component of $S$, 
\[
\lambda =  f^{\vee} \circ \lambda' \circ f,
\]
$f\circ i(b) = i'(b) \circ f $ for all $b \in \mathcal{O} \otimes \mathbb{Z}_{(p)}$, and $\overline{(\eta’,u’)} = \overline{(f \circ \eta, u)}$.  
\end{definition}

\begin{remark}
The above moduli problem is in some sense in between the isogeny classes and the isomorphism classes moduli problems, in that it uses isogeny as morphisms in the groupoid and rational Tate modules for level structures, while also encoding integral structures in endomorphism structures and $u$. This has the effect that the morphism $f$ in the groupoid is required to strictly preserve the polarization, $\lambda =  f^{\vee} \circ \lambda' \circ f$, without the factor $r$ in definition \ref{isogeny}. 
\end{remark}

It is not hard to see that $\mathscr{M}_{\alpha, K}$ is representable by a smooth quasi-projective scheme $M_{\alpha,K}$ over $\mathcal{O}_{F_0} \otimes \mathbb{Z}_{(p)}$. 

There is an action of totally positive prime to $p$ units $(\mathcal{O}_{F})_+^{\times}$ on $\mathscr{M}_{\alpha, K}$ given by the formula 
\[\epsilon \cdot (A, \lambda, i, \overline{(\eta, u)})= (A, i(\epsilon)\lambda, i, \overline{(\eta, \epsilon u)})\]
which factors through the finite quotient group $\Delta := (\mathcal{O}_{F})_+^{\times}/Nm_{F^c/F}((\mathcal{O}_{F^c})^{\times} \cap K)$. The quotient $M_{\alpha, K}/\Delta$ always exists. Let $\Omega =\{\alpha\} \subset (F \otimes \mathbb{A}^{(p\infty)})^{\times}$ be a set of representatives of the double cosets
\[(F \otimes \mathbb{A}^{(p\infty)})^{\times} = \underset{\alpha \in \Omega}{\coprod}(\mathcal{O}_F \otimes \mathbb{Z}_{(p)})^{\times}_{+} \alpha (\mathcal{O}_F \otimes \hat{\mathbb{Z}}^p)^{\times}\]
then we have 
\begin{equation} \label{alpha}
    Sh_K(G,h) = \underset{\alpha \in \Omega}{\coprod} M_{\alpha, K} /\Delta = M_K /\Delta
\end{equation} 
where $M_K := \underset{\alpha \in \Omega }{\coprod} M_{\alpha, K}$. It means that $M_{K}/\Delta$ has the same complex points as $Sh_K(G,h)$, hence defining an integral model of $Sh_K(G,h)$. This is a consequence of the fact that $G$ satisfies Hasse principle, see \cite{MR2055355} 7.1.5.

\begin{remark}
We can show that $M_K/\Delta$ is the coarse moduli space of the functor sending $S$ to quadruples 
\[
(A, \lambda, i, \overline{(\eta, u)})
\]
where $A$ is an abelian scheme over S, 
\[
\lambda : A \rightarrow A^{\vee}
\]
is a prime to $p$ quasi-polarization of $A$, and 
\[
i : \mathcal{O} \otimes \mathbb{Z}_{(p)} \rightarrow End_S(A) \otimes \mathbb{Z}_{(p)}
\]
is a ring homomorphism such that $i(b)^{\vee} \circ \lambda= \lambda \circ i(b^*) $ for every $b \in \mathcal{O} \otimes \mathbb{Z}_{(p)}$, and  $Lie_{A/S}$ satisfies the determinant condition specified by $h$, see \cite{lan2013arithmetic} 1.3.4.1 for a precise formulation. Moreover, if we choose a geometric point $\bar{s}$ in each connected component of $S$, the level structure $\overline{(\eta, u)}$ is an assignment to each $\bar{s}$ a  $\pi_1(S,\bar{s})$-invariant $K$-orbit of $\mathcal{O} \otimes \mathbb{A}^{p\infty}$-equivariant isomorphisms 
\[\eta : L \otimes \mathbb{A}^{p \infty} \overset{\sim}{\rightarrow} V^pA_{\bar{s}}\]
together with an $F \otimes \mathbb{A}^{p\infty}$-equivariant isomorphism 
\[ u : F \otimes \mathbb{A}^{p\infty}(1) \overset{\sim}{\rightarrow} V^p(\mathfrak{d}_F^{-1} \otimes_{\mathbb{Z}} \textbf{G}_{m,\bar{s}}) \]
such that 
\[\langle\eta(x), \eta(y)\rangle_{\lambda} =Tr_{\mathcal{O}_F/\mathbb{Z}} ( u \circ  \langle x,y\rangle_F)\]
where $x,y \in L \otimes \mathbb{A}^{p\infty}$.

The isomorphisms in the groupoid are defined to be 
\[
(A, \lambda, i, \overline{(\eta, u)}) \sim (A', \lambda', i', \overline{(\eta’, u’)})
\]
if and only if there is a prime to $p$ quasi-isogeny $f : A \rightarrow A'$ such that over each connected component of $S$, 
\[
\lambda \circ i(a) =  f^{\vee} \circ \lambda' \circ f
\]
for some $a \in (\mathcal{O}_F \otimes \mathbb{Z}_{(p)})^{\times}_+$, $f\circ i(b) = i'(b) \circ f $ for all $b \in \mathcal{O} \otimes \mathbb{Z}_{(p)}$, and $\overline{(\eta’,u’)} = \overline{(f \circ \eta, u \circ a)}$. 

Note that this functor kills all the integral structures in definition 3, see remark 4. Moreover, it enlarges the domain of ambiguity factor $r$ in definition 1 from $(\mathbb{Z}_{(p)})^{\times}_+$ to $(\mathcal{O}_F \otimes \mathbb{Z}_{(p)})^{\times}_+$. 

See \cite{MR2055355} 7.1.3 for more details on this moduli problem.  

\end{remark}

We will work with integral toroidal and minimal compactifications of $Sh_K(G,h)$. However, this has only been constructed by Lan for the PEL moduli problems in definition \ref{isogeny} and \ref{isomorphism}. Fortunately, $Sh_K(G,h)$ is not very different from $Sh_{\mathcal{H}}(G^*,h)$. The precise relation is that for each $\mathcal{H} \subset G^*(\hat{\mathbb{Z}}^p)$, there exists an open compact subgroup $K \subset G(\hat{\mathbb{Z}}^p)$ containing $\mathcal{H}$ such that the natural map $Sh_{\mathcal{H}}(G^*,h) \rightarrow Sh_K(G,h)$
induced by $G^* \subset G$ is an open immersion containing the identity component of $Sh_K(G,h)$, and the Hecke translates of $Sh_{\mathcal{H}}(G^*,h)$ cover $Sh_K(G,h)$, see for example \cite{MR0498581} 1.15. We need a more explicit description of $Sh_K(G,h)$ in terms of $Sh_{\mathcal{H}}(G^*,h)$, so we focus on principal level structures from now on. 

Suppose that $n$ is prime to $p$, let 
\[K(n) := Ker(G(\hat{\mathbb{Z}}^p) \rightarrow G(\mathbb{Z}/n\mathbb{Z}))
\]
observe that 
\begin{equation} \label{a}
\nu(K(n) )= Ker( (\mathcal{O}_F \otimes_{\mathbb{Z}} \textbf{G}_m) (\hat{\mathbb{Z}}^p) \rightarrow  (\mathcal{O}_F \otimes_{\mathbb{Z}} \textbf{G}_m ) (\mathbb{Z}/n \mathbb{Z})). 
\end{equation}
Recall that 
\[\mathcal{U}(n) := Ker(G^*(\hat{\mathbb{Z}}^p) \rightarrow G^*(\mathbb{Z}/n\mathbb{Z}))
\]
and 
\begin{equation} \label{b}
\nu(\mathcal{U}(n) )= Ker( \textbf{G}_m(\hat{\mathbb{Z}}^p) \rightarrow \textbf{G}_m(\mathbb{Z}/n \mathbb{Z}))
\end{equation}
Choose a set $\Lambda$ of representatives of the double quotient
\begin{equation} \label{O}
    (\mathcal{O}_F \otimes \hat{\mathbb{Z}}^p)^{\times} = \underset{\delta \in \Lambda} {\coprod} (\mathcal{O}_F)^{\times}_+ \delta (\nu(K(n)) \hat{\mathbb{Z}}^{p,\times} )
\end{equation}
then together with the representatives $\Omega$, 
\begin{equation} \label{A}
(F \otimes \mathbb{A}^{(p\infty)})^{\times} = \underset{\alpha \in \Omega}{\coprod}(\mathcal{O}_F \otimes \mathbb{Z}_{(p)})^{\times}_{+} \alpha (\mathcal{O}_F \otimes \hat{\mathbb{Z}}^p)^{\times}
\end{equation}
we have the decomposition 
\begin{equation} \label{adele}
    (F \otimes \mathbb{A}^{(p\infty)})^{\times} = \underset{\begin{subarray}{c}
  \alpha \in \Omega \\
  \delta \in \Lambda
  \end{subarray}}{\coprod}(\mathcal{O}_F \otimes \mathbb{Z}_{(p)})^{\times}_{+} \alpha \delta (\nu(K(n)) \hat{\mathbb{Z}}^{p,\times} )
\end{equation}

Let us change the notation $M_n$ in definition \ref{isomorphism} into $M_n(L, \langle \cdot, \cdot\rangle)$ to emphaisze the dependence of $M_n$ on $L$ and $\langle \cdot, \cdot\rangle$. Then there is a natural embedding
\[M_n(L, Tr_{\mathcal{O}_F/ \mathbb{Z}} \circ (\alpha \delta \langle \cdot, \cdot\rangle_F)) \hookrightarrow M_{\alpha, K(n)} / \Delta 
\]
sending $(A, \lambda, i, (\alpha_n, \nu_n))$ to $(A, \lambda, i, \overline{(\eta, u)})$, where $\eta : L \otimes \mathbb{A}^{p \infty} \overset{\sim}{\rightarrow} V^pA_{\bar{s}}$ is defined by a lifting of $\alpha_n$ to 
\[\hat{\alpha}: L \otimes \hat{\mathbb{Z}}^{p} \overset{\sim}{\rightarrow} T^pA_{\bar{s}}
\]
whose existence is a condition on the level structure in definition \ref{isomorphism}, then inverting all prime to p integers. 
$u: \mathfrak{d}_F^{-1} \otimes \hat{\mathbb{Z}}^{p}(1) \overset{\sim}{\rightarrow} T^p(\mathfrak{d}_F^{-1} \otimes_{\mathbb{Z}} \textbf{G}_{m,\bar{s}}) $
is defined by 
\[\langle\eta(x), \eta(y) \rangle_{\lambda} = Tr_{\mathcal{O}_F/\mathbb{Z}} (u \circ (\alpha \delta\langle x,y\rangle_F)) 
\]
for $x,y \in L \otimes \mathbb{A}^{p \infty}$. Here we abuse notation by denoting both $u$ and $u \otimes id_{\mathbb{Z}_{(p)}}$ by $u$. The $K(n)$-class of $(\eta,u)$ does not depend on the choice of the lifting $\hat{\alpha}$. 

\begin{warning} \label{W}
$Tr_{\mathcal{O}_F/ \mathbb{Z}} \circ (\alpha \delta \langle \cdot, \cdot\rangle_F))$
is not defined on $L$, but rather on $L \otimes \hat{\mathbb{Z}}^{p}$ 
since $\alpha \delta \in \mathcal{O}_F \otimes \hat{\mathbb{Z}}^p$ if we choose 
$\alpha \in \mathcal{O}_F \otimes \hat{\mathbb{Z}}^p$, which we assume form now on. This does not affect the moduli problem since the moduli problem in definition \ref{isomorphism} only depends on $L  \otimes \hat{\mathbb{Z}}^p$, $L \otimes \mathbb{R}$ and the corresponding pairing on them. 

The appropriate notation for 
$M_n(L, Tr_{\mathcal{O}_F/ \mathbb{Z}} \circ (\alpha \delta \langle \cdot, \cdot\rangle_F))$
would be 
$M_n(L,\langle \cdot, \cdot\rangle_1) $
for some pairing $\langle \cdot, \cdot\rangle_1$ on $L$ that is  isomorphic to $Tr_{\mathcal{O}_F/ \mathbb{Z}} \circ (\alpha \delta \langle \cdot, \cdot\rangle_F)$ on  $L \otimes \hat{\mathbb{Z}}^{p}$ , 
perfect on $L \otimes \mathbb{Z}_p$, 
and compatible with $h$ on $L \otimes \mathbb{R}$.  
Such a pairing   exists if the moduli problem is nonempty. See \cite{lan2013arithmetic} 1.4.3.14 for more explainations. 

We use the wrong notation $M_n(L, Tr_{\mathcal{O}_F/ \mathbb{Z}} \circ (\alpha \delta \langle \cdot, \cdot\rangle_F))$ for simplicity. 
\end{warning}

By definition of the moduli problem, the decomposition (\ref{O}) gives 
\[M_{\alpha, K(n)} / \Delta = \underset{\delta \in \Lambda}{\coprod} M_n(L, Tr_{\mathcal{O}_F/ \mathbb{Z}} \circ (\alpha \delta \langle \cdot, \cdot\rangle_F)). 
\]
Then it follows from (\ref{adele}) and the definition of $M_{K(n)}/\Delta$ that 
\begin{equation} \label{main}
     M_{ K(n)} / \Delta = \underset{\begin{subarray}{c}
  \alpha \in \Omega \\
  \delta \in \Lambda
  \end{subarray}}{\coprod} 
  M_n(L, Tr_{\mathcal{O}_F/ \mathbb{Z}} \circ (\alpha \delta \langle \cdot, \cdot\rangle_F)). 
\end{equation}
This will help us constructing toroidal and minimal compactifications of $M_{K(n)}/ \Delta$ 
from those constructed by Lan. We briefly recall Lan's results on minimal compactifications. 

\begin{theorem} (Lan \cite{lan2013arithmetic}) \label{Lan}
There exists a compactification $M_n(L, \langle\cdot, \cdot\rangle)^{min}$ of $M_n(L, \langle\cdot, \cdot\rangle)$ together with a stratification by locally closed subschemes
\[M_n(L, \langle\cdot, \cdot\rangle)^{min} = \underset{[(Z_n, \Phi_n, \delta_n)]}{\coprod} M_n(L^{Z_n}, \langle \cdot, \cdot\rangle^{Z_n})
\]
where 

(1) $Z_n$ is a $\mathcal{O}$-invariant filtration on $L/nL$, 
\[0 \subset Z_{n,-2} \subset Z_{n,-1} \subset Z_{n,0} = L/nL
\]
which can be lifted to a $\mathcal{O}$-invariant filtration $Z$ on $L \otimes \hat{\mathbb{Z}}^p$
\[0 \subset Z_{-2} \subset Z_{-1} \subset Z_{0} = L \otimes \hat{\mathbb{Z}}^p
\]
such that $Z$ is the restriction of a split $\mathcal{O}$-invariant filtration $Z_{\mathbb{A}^p}$ on $L \otimes \mathbb{A}^p$ satisfying 
$Z_{\mathbb{A}^p, -2}^{\bot} = Z_{\mathbb{A}^p, -1}$ 
and 
$Gr_i Z_{\mathbb{A}^p} \cong L_i \otimes \mathbb{A}^p$ 
for some $\mathcal{O}$-lattice $L_i$. Let 
$L^{\mathbb{Z}_n} := L_{-1}$ 
and $\langle \cdot, \cdot\rangle^{\mathbb{Z}_n}$ a pairing on $L^{\mathbb{Z}_n}$ which induces  $\langle \cdot, \cdot\rangle$ on $Gr_{-1}Z_{\mathbb{A}^p}$. 
There exists an 
$h^{Z_n} : \mathbb{C} \rightarrow End_{\mathcal{O}_{\mathbb{R}}}(L_{\mathbb{R}}^{Z_n})$ 
that makes $(L^{Z_n}, \langle\cdot, \cdot\rangle^{Z_n}, h^{Z_n})$  a PEL data defining the moduli problem  
$M_n(L^{Z_n}, \langle \cdot, \cdot\rangle^{Z_n})$. See \cite{lan2013arithmetic} 5.2.7.5 for details. 

(2) 
$\Phi_n$ is a tuple 
$(X,Y,\phi, \varphi_{-2,n}, \varphi_{0,n}) $, where $X,Y$ are $\mathcal{O}$-lattices that are isomorphic as $B$-modules after tensoring with $\mathbb{Q}$, $\phi : Y \hookrightarrow X$ is an $\mathcal{O}$-invariant embedding. 
\[\varphi_{-2,n} : Gr_{-2}^{Z_n} \overset{\sim}{\rightarrow} Hom(X/nX, (\mathbb{Z}/n \mathbb{Z})(1))\] 
and 
\[\varphi_{0,n} : Gr_{0}^{Z_n} \overset{\sim}{\rightarrow} Y/nY\]
are isomorphisms that are reduction modulo n of $\mathcal{O}$-equivariant isomorphisms
$\varphi_{-2}: Gr_{-2}^Z \overset{\sim}{\rightarrow} Hom_{\hat{\mathbb{Z}}^p} (X \otimes \hat{\mathbb{Z}}^p , \hat{\mathbb{Z}}^p(1))$ and 
$\varphi_0 : Gr_0^Z \overset{\sim}{\rightarrow} Y \otimes \hat{\mathbb{Z}}^p$ such that 
\[\varphi_{-2}(x)(\phi(\varphi_0(y))) = \langle x,y\rangle   
\]
for $x \in Gr_{-2}^Z$ and $y \in Gr_0^Z$. 

(3) 
$\delta_n : \underset{i}{\oplus} Gr^{Z_n}_i \overset{\sim }{\rightarrow} L/nL$ 
is a splitting that is reduction modulo n of a splitting $\underset{i}{\oplus} Gr^{Z}_i \overset{\sim }{\rightarrow} L \otimes \hat{\mathbb{Z}}^p$. 

The tuple $(Z_n, \Phi_n, \delta_n)$ is called a \textbf{cusp label} at principal level n, and  $[(Z_n, \Phi_n, \delta_n)]$ is the equivalence classes of the cusp label, see \cite{lan2013arithmetic} 5.4.1.9 for the precise definition of equivalences. 

There is a precise description of closure relations of strata in terms of the cusp labels parametrizing them, see \cite{lan2013arithmetic} 5.4.1.14 for details.
\end{theorem}

\begin{remark} \label{explaination}
There are also toroidal compactifications of $M_n$ together with universal semi-abelian varieties over them, which parametrize how abelian varieties degenerate into semi-abelian varieties. The toric part of the universal semi-abelian variety is parametrized by the cusp labels, which is discrete in nature. The minimal compactification is roughly obtained by contracting the isomorphic toric part, so it keeps track of only information on the abelian part, which is where the strata in thoerem  \ref{Lan} come from. 

In other words, the toric part of toroidal compactifications degenerates into discrete indexing sets of the strata, and the abelian part is remembered in the strata themselves.  What is lost by passing to minimal compactifications is the extension between torus and abelian varieties. 

\end{remark}

\begin{corollary} \label{min}
$M_{K(n)}/ \Delta$ has a compactification $(M_{K(n)}/ \Delta)^{min}$ together with a stratification by locally closed subschemes
\[(M_{K(n)}/ \Delta)^{min} =\underset{\begin{subarray}{c}
  \alpha \in \Omega \\
  \delta \in \Lambda
  \end{subarray}}{\coprod}
  M_n(L, Tr_{\mathcal{O}_F/ \mathbb{Z}} \circ (\alpha \delta \langle \cdot, \cdot\rangle_F)) ^{min}
  \]
\[= \underset{\begin{subarray}{c}
  \alpha \in \Omega \\
  \delta \in \Lambda
  \end{subarray}}{\coprod}
  \underset{[(Z_{\alpha \delta,n}, \Phi_{\alpha \delta,n}, \delta_{\alpha \delta,n})]}{\coprod} 
  M_n(L^{Z_{\alpha \delta,n}},\langle\cdot, \cdot\rangle^{Z_{\alpha \delta,n}})
\]
where $(Z_{\alpha \delta,n}, \Phi_{\alpha \delta,n}, \delta_{\alpha \delta,n})$ are cusp labels of $M_n(L, Tr_{\mathcal{O}_F/ \mathbb{Z}} \circ (\alpha \delta \langle \cdot, \cdot\rangle_F)) $, and $M_n(L^{Z_{\alpha \delta,n}},\langle\cdot, \cdot\rangle^{Z_{\alpha \delta,n}})$ are as in the theorem. See warning \ref{W} for clarifications. 
\end{corollary}

\subsection{The partial Frobenius} \label{thpartialf}

From now on, we assume that $p$ satisfies the following condition,  
\[p \text{ splits completely in the center } F^c \text{ of } B.
\]
This implies that $\mathcal{O} \otimes \mathbb{Z}_p \cong \prod M_{n_i}( \mathbb{Z}_p)$, $G$ and $G^*$ splits over $\mathbb{Q}_p$ and 
$G(\mathbb{Q}_p) =\underset{\mathfrak{p}_i}{\prod} H(\mathbb{Q}_p)$, where $\mathfrak{p}_i$ are prime ideals of  $F = F^{c, *=1}$ such that $p = \underset{i}{\prod} \mathfrak{p}_i$. 

In this section, all moduli problems are defined over $\mathcal{O}_{F_0} \otimes_{\mathbb{Z}} \mathbb{F}_p$, i.e. $M_n $ or $M_{K(n)}$ in this section denotes 
$M_n \times_{\mathcal{O}_{F_0} \otimes \mathbb{Z}_{(p)}} (\mathcal{O}_{F_0} \otimes \mathbb{F}_p)$ 
or
$M_{K(n)} \times_{\mathcal{O}_{F_0} \otimes \mathbb{Z}_{(p)}} (\mathcal{O}_{F_0} \otimes \mathbb{F}_p)$ using notations in previous sections. We follow  Nekov\'a\v r's approach as in \cite{JanNekovar2018}.  

\begin{definition} \label{pFdef}
We fix a $\xi \in F_+^{\times}$ satisfying 
$v_{\mathfrak{p}_i} (\xi)=1$ 
and $v_{\mathfrak{p}_{i'}}(c) =0$ for $i' \neq i$. 
The partial Frobenius 
$F_{\mathfrak{p}_i} : M_{K(n)}/ \Delta \rightarrow M_{K(n)}/ \Delta $ 
is defined by disoint union of maps 
\[M_{\alpha, K(n)}/ \Delta \rightarrow M_{\alpha', K(n)}/ \Delta 
\]
sending 
$(A, \lambda, i, \overline{(\eta, u)})$ 
in definition \ref{similitude} to 
$(A', \lambda', i', \overline{(\eta', u')})$
\footnote{$A'$ satisfies the determinant condition because $Lie_{A'} = Lie_{A'[p]} = \underset{j}{\oplus} Lie_{A'[\mathfrak{p}_j]} = \underset{j \neq i}{\oplus} Lie_{A[\mathfrak{p}_j]} \oplus F^*Lie_{A[\mathfrak{p}_i]}$
as $\mathcal{O} \otimes \mathbb{F}_p$-modules, showing that $Lie_{A'}$ has the same 
$\mathcal{O} \otimes \mathbb{F}_p$ 
structure as $Lie_A$, which satisfies the determinant condition by our choice.}, 
where 
\[A' := A/ (Ker(F)[\mathfrak{p}_i])\]
with $F$ the usual Frobenius and 
$Ker(F)[\mathfrak{p}_i] := \{ x \in Ker(F)| ax =0, \forall a \in \mathfrak{p}_i \}$,  
$i'$ is induced by the quotient map $\pi_{\mathfrak{p}_i} : A \rightarrow A'$,  $\lambda'$ is a prime to $p$ quasi-isogeny characterized by 
$\xi \lambda = \pi_{\mathfrak{p}_i}^{\vee} \circ \lambda' \circ \pi_{\mathfrak{p}_i}$, 
$\eta' = \pi_{\mathfrak{p}_i} \circ \eta $,
and $\alpha'$ is defined by 
\[\xi \alpha  = \epsilon \alpha' \lambda
\]
where 
$\alpha' \in \Omega$, 
$\epsilon \in (\mathcal{O}_F \otimes \mathbb{Z}_{(p)})^{\times}_{+}$
and 
$\lambda \in (\mathcal{O}_F \otimes \hat{\mathbb{Z}}^p)^{\times}$ 
as in decomposition (\ref{A}). Lastly, 
$u'$ is the comoposition of $\mathcal{O}_F \otimes \hat{\mathbb{Z}}^{p}$-equivariant isomorphisms
\[ u' : \mathfrak{d}_F^{-1} \otimes \hat{\mathbb{Z}}^{p}(1) \overset{\lambda}{\underset{\sim}{\longrightarrow}} \mathfrak{d}_F^{-1} \otimes \hat{\mathbb{Z}}^{p}(1) \overset{u}{\underset{\sim}{\longrightarrow}} T^p(\mathfrak{d}_F^{-1} \otimes_{\mathbb{Z}} \textbf{G}_{m,\bar{s}}).
\]

\end{definition}

\begin{remark}
It is easy to see that $F_{\mathfrak{p}_i} $ is independent of the choice of $\xi$. Moreover, the same definition works for $p$ not necessarily split in $F^c$. We make the assumption because that is the only case we will use. 
\end{remark}

We observe that 
\[F_{\mathfrak{p}_i} F_{\mathfrak{p}_j} = F_{\mathfrak{p}_j} F_{\mathfrak{p}_i}  
\]
and
\[\underset{i}{\prod} F_{\mathfrak{p}_i} = F 
\]
where F is the usual Frobenius, explaining the name partial Frobenius. 

It is helpful to write the partial Frobenius in terms of the decomposition (\ref{main}). We will use the description to prove that the partial Frobenius extends to minimal compactifications and toroidal compactifications. 

The equation 
\[\xi \alpha  = \epsilon \alpha' \lambda
\]
with 
$\alpha' \in \Omega$, 
$\epsilon \in (\mathcal{O}_F \otimes \mathbb{Z}_{(p)})^{\times}_{+}$
and 
$\lambda \in (\mathcal{O}_F \otimes \hat{\mathbb{Z}}^p)^{\times}$
plays an important role in the definition of the partial Frobenius. In particular, it determines how the partial Frobenius permutes components parametrized by $\alpha \in \Omega$ as in (\ref{alpha}). We refine the description by using the finer decomposition (\ref{main}) parametrized by $\alpha \delta$. 

With notations as in the previous paragraph, let 
\[\lambda \delta = \epsilon_0 \delta' \gamma
\]
where 
$\delta, \delta' \in \Lambda$, $\epsilon_0 \in \mathcal{O}_{F,+}^{\times}$ 
and 
$\gamma \in (\nu(K(n)) \hat{\mathbb{Z}}^{p,\times} )$, as in the decomposition (\ref{O}). From equations (\ref{a}) and (\ref{b}), we observe that 
\[\nu(K(n)) \hat{\mathbb{Z}}^{p,\times} = \underset{\kappa }{\coprod} \nu(K(n)) \kappa
\]
where  $\kappa $ ranges over a complete set of representatives of  
$\hat{\mathbb{Z}}^{p,\times} /\nu(\mathcal{U}(n)) \cong (\mathbb{Z}/ n\mathbb{Z})^{\times}$
in $\hat{\mathbb{Z}}^{p,\times}$. 
Let 
\[\gamma = \beta \kappa
\]
with $\beta \in \nu(K(n))$ and $\kappa$ as above be the decomposition of $\gamma$. 

The partial Frobenius 
$F_{\mathfrak{p}_i}$ induces a map 
\[M_n(L, Tr_{\mathcal{O}_F/ \mathbb{Z}} \circ (\alpha \delta \langle \cdot, \cdot\rangle_F)) \rightarrow 
M_n(L, Tr_{\mathcal{O}_F/ \mathbb{Z}} \circ (\alpha' \delta' \langle \cdot, \cdot\rangle_F))
\]
sending 
$(A, \lambda, i, (\alpha_n, \nu_n))$
in definition \ref{isomorphism} to 
$(A', \lambda', i', (\alpha'_n, \nu'_n))$, where 
$A' := A/ (Ker(F)[\mathfrak{p}_i])$,
$i'$ is induced by the quotient map $\pi_{\mathfrak{p}_i} : A \rightarrow A'$,  $\lambda'$ is characterized by 
$\xi \lambda = \pi_{\mathfrak{p}_i}^{\vee} \circ \lambda' \circ \pi_{\mathfrak{p}_i}$ which defines a quasi-isogeny $\lambda'$, 
$\alpha_n' = \pi_{\mathfrak{p}_i} \circ \alpha_n$
and 
$\nu_n' = \nu_n \circ \kappa $. 
In the last equation, we view $\kappa$ as an element of $(\mathbb{Z}/n\mathbb{Z})^{\times}$ which acts on $\mathbb{Z}/n\mathbb{Z}(1)$, and 
$\nu_n'$ is defined to be
\[\nu_n' : (\mathbb{Z}/n\mathbb{Z}(1))_S \overset{\kappa}{\underset{\sim}{\longrightarrow}} (\mathbb{Z}/n\mathbb{Z}(1))_S \overset{\nu_n}{\underset{\sim}{\longrightarrow}} \mu_{n,S}.
\]

A subtle point in the above description is that  in definition \ref{isomorphism}, $\lambda'$ should not only be a prime to $p$ quasi-isogeny, but an actual isogeny. The characterization  
$\xi \lambda = \pi_{\mathfrak{p}_i}^{\vee} \circ \lambda' \circ \pi_{\mathfrak{p}_i}$ defines a quasi-isogeny $\lambda'$, 
but does not give an isogeny $\lambda'$ a priori. We have to check that $\lambda'$ is indeed a prime to p isogeny to make the above a well-defined map. 

Before giving the proof, let us introduce some more suggestive notations. Let
$A^{(\mathfrak{p}_i)} := A/ (Ker(F)[\mathfrak{p}_i])$, 
and
$F^{(\mathfrak{p}_i)} := \pi_{\mathfrak{p}_i} : A \rightarrow A^{(\mathfrak{p}_i)}$. 
Then we observe that there is a natural map 
$V^{(\mathfrak{p}_i)} : A^{(\mathfrak{p}_i)} \rightarrow A \otimes_{\mathcal{O}_F} \mathfrak{p}_i^{-1}$
such that the composition 
\[ A \overset{F^{(\mathfrak{p}_i)} }{\longrightarrow} A^{(\mathfrak{p}_i)} \overset{V^{(\mathfrak{p}_i)} }{\longrightarrow} A \otimes_{\mathcal{O}_F} \mathfrak{p}_i^{-1}
\]
is the map 
$id_A \otimes_{\mathcal{O}_F} (\mathcal{O}_F \hookrightarrow \mathfrak{p}_i^{-1})$, which has kernel $A[\mathfrak{p}_i]$. Here $ \mathfrak{p}_i^{-1}$ is the inverse  of $ \mathfrak{p}_i$ as fractional ideals and 
$A \otimes_{\mathcal{O}_F} \mathfrak{p}_i^{-1}$
is defined in the category of fppf sheaf of $\mathcal{O}_F$-modules, which can be easily seen to be represented by an abelian scheme isogenious to $A$. Here $F^ {\mathfrak{p}_i}$ and $V^{ \mathfrak{p}_i} $ should be viewed as partial Frobenius and Verschiebung, whose products over all $i$ will be the usual ones.  

We have a commutative diagram
\[\begin{tikzcd}
[
  arrow symbol/.style = {draw=none,"\textstyle#1" description,sloped},
  isomorphic/.style = {arrow symbol={\simeq}},
  ]
&
A \arrow[r,"F^{(\mathfrak{p}_i)}"] \arrow[d,"\lambda" ] & A^{(\mathfrak{p}_i)} \arrow [r, "V^{(\mathfrak{p}_i)}"] \arrow[d,"\lambda^{(\mathfrak{p}_i)}"' ] 
\arrow[ddd, bend left =40, dashrightarrow,  "\lambda'" pos=.2]&  
A  \otimes_{\mathcal{O}_F} \mathfrak{p}_i^{-1} \arrow[d,"\lambda \otimes id" ]
\\
A^{\vee} \otimes_{\mathcal{O}_F} \mathfrak{p}_i 
\arrow[r, "id \otimes ( \mathfrak{p}_i \hookrightarrow \mathcal{O}_F)"] 
\arrow[rd,"\xi \otimes id_{\mathfrak{p}_i}"' ]&
A^{\vee} \arrow[r, "F_{A^{\vee}}^{(\mathfrak{p}_i)}"]  \arrow[rrd, "\xi"' near start] 
\arrow[d, dashrightarrow] & 
 (A^{\vee})^{(\mathfrak{p}_i)} \arrow[r,"V_{A^{\vee}}^{(\mathfrak{p}_i)}"]
 \arrow[d, dashrightarrow]& 
 A^{\vee} \otimes_{\mathcal{O}_F} \mathfrak{p}_i^{-1}
 \arrow[d, dashrightarrow]
 \\
&
A^{\vee} \otimes_{\mathcal{O}_F} \mathfrak{p}_i  
\arrow[r, "(V^{(\mathfrak{p}_i)})^{ \vee}"'] & (A^{\vee})^{(\mathfrak{p}_i)} \otimes_{\mathcal{O}_F} \mathfrak{p}_i  \arrow[r,"(F^{(\mathfrak{p}_i)})^{ \vee}"'] & 
 A^{\vee}
 \\
 & & (A^{(\mathfrak{p}_i)})^{\vee} \arrow[u, isomorphic]&
\end{tikzcd} 
\]
which induces the dashed arrows. For example, the left dashed arrow is induced by $A^{\vee} [\mathfrak{p}_i] \subset A^{\vee} [\xi]$, and similarly for the other two. We define $\lambda' $ to be composition of the middle vertical maps, which is an actual isogeny and satisfies 
$\xi \lambda = \pi_{\mathfrak{p}_i}^{\vee} \circ \lambda' \circ \pi_{\mathfrak{p}_i}$
as the diagram shows. 

The only non-trivial arrow in the above diagram is the isomorphism 
\[(A^{(\mathfrak{p}_i)})^{\vee} \cong (A^{\vee})^{(\mathfrak{p}_i)} \otimes_{\mathcal{O}_F} \mathfrak{p}_i
\]
We give a proof here. 

\begin{lemma} \label{uuuuyyyyy}
With notations as above, for any abelian scheme $A/S$ over a characteristic $p$ scheme $S$, together with a ring homomorphism $\mathcal{O} \rightarrow End_S(A)$, we have a canonical isomorphism
\[
(A^{(\mathfrak{p}_i)})^{\vee} \cong (A^{\vee})^{(\mathfrak{p}_i)} \otimes_{\mathcal{O}_F} \mathfrak{p}_i.
\]
\end{lemma}

\begin{proof}
Applying 
$Hom_{fppf}(-,\textbf{G}_m)$ to the short exact sequence 
\[0 \rightarrow A[\mathfrak{p}_i] / Ker(F)[\mathfrak{p}_i] \rightarrow A^{(\mathfrak{p}_i)} \overset{V^{(\mathfrak{p}_i)}}{\longrightarrow} A \otimes_{\mathcal{O}_F} \mathfrak{p}_i^{-1} \rightarrow 0 
\]
and using that 
$Ext^1_{fppf} (A, \textbf{G}_m) \cong A^{\vee}$, we have 
\[0 \rightarrow Hom_{fppf}(A[\mathfrak{p}_i] /Ker(F)[\mathfrak{p}_i], \textbf{G}_m) \rightarrow A^{\vee} \otimes_{\mathcal{O}_F} \mathfrak{p}_i \rightarrow (A^{(\mathfrak{p}_i)})^{\vee} \rightarrow 0 
\]
We know that 
$Hom_{fppf}(A[\mathfrak{p}_i] / Ker(F)[\mathfrak{p}_i], \textbf{G}_m)$ 
is the Cartiar dual 
$(A[\mathfrak{p}_i] / Ker(F)[\mathfrak{p}_i])^{\vee}$ of 
$A[\mathfrak{p}_i] / Ker(F)[\mathfrak{p}_i]$, 
so the dual of the short exact sequence
\[0 \rightarrow Ker(V) \rightarrow Ker(V) \otimes_{\mathcal{O}_F} \mathfrak{p}_i^{-1}  \rightarrow
A[\mathfrak{p}_i] / Ker(F)[\mathfrak{p}_i] \rightarrow 0
\]
gives 
\[
(A[\mathfrak{p}_i] / Ker(F)[\mathfrak{p}_i])^{\vee} \cong Ker( Ker(V)^{\vee} \otimes_{\mathcal{O}_F} \mathfrak{p}_i  \rightarrow Ker(V)^{\vee})
\]
\[
\cong Ker(Ker(F_{A^{\vee}}) \otimes_{\mathcal{O}_F} \mathfrak{p}_i  \rightarrow Ker(F_{A^{\vee}}))
\]
\[\cong Ker(F_{A^{\vee}})[\mathfrak{p}_i] \otimes_{\mathcal{O}_F} \mathfrak{p}_i
\]
which is the kernel of 
\[
A^{\vee} \otimes_{\mathcal{O}_F} \mathfrak{p}_i  \overset{F^{(\mathfrak{p}_i)}\otimes id}{\longrightarrow} (A^{\vee})^{(\mathfrak{p}_i)} \otimes_{\mathcal{O}_F} \mathfrak{p}_i 
\]
proving 
$(A^{(\mathfrak{p}_i)})^{\vee} \cong (A^{\vee})^{(\mathfrak{p}_i)} \otimes_{\mathcal{O}_F} \mathfrak{p}_i$. 
\end{proof}

We now state the main technical result of the paper, which claims that the partial Frobenius extends to the minimal compactification in Corallary  \ref{min}. 

\begin{theorem} \label{pFextends}
$F_{\mathfrak{p}_i}$ extends to a morphism
\[
F_{\mathfrak{p}_i} : (M_{K(n)}/ \Delta)^{min} \longrightarrow (M_{K(n)}/ \Delta)^{min}
\]
sending the strata 
$M_n(L^{Z_{\alpha \delta,n}},\langle\cdot, \cdot\rangle^{Z_{\alpha \delta,n}})$
associated to 
$\alpha \in \Omega,  
  \delta \in \Lambda$
and the cusp label
$[(Z_{\alpha \delta,n}, \Phi_{\alpha \delta,n}, \delta_{\alpha \delta,n})]$ 
to the strata 
$M_n(L^{Z_{\alpha' \delta',n}},\langle\cdot, \cdot\rangle^{Z_{\alpha' \delta',n}})$
associated to 
$\alpha' \in \Omega,  
  \delta' \in \Lambda$
as in the above description of the partial Frobenius, and the cusp label
$[(Z_{\alpha' \delta',n}, \Phi_{\alpha' \delta',n}, \delta_{\alpha' \delta',n})]$ 
defined as follows, 
\[
Z_{\alpha' \delta',n} = Z_{\alpha \delta,n}. 
\]
If 
$ \Phi_{\alpha \delta,n} =(X,Y,\phi, \varphi_{-2,n}, \varphi_{0,n}) $,
then 
\[
\Phi_{\alpha' \delta',n} = (X \otimes_{\mathcal{O}_F} \mathfrak{p}_i,Y,\phi', \varphi_{-2,n}', \varphi'_{0,n}) \]
where 
\[
\varphi'_{-2,n} : Gr_{-2}^{Z_{\alpha' \delta',n} } = Gr_{-2}^{Z_{\alpha \delta,n} }  \overset{\varphi_{-2,n}}{\longrightarrow} Hom(X/nX, (\mathbb{Z}/n \mathbb{Z})(1)) 
\]
\[
\overset{\sim}{\longrightarrow} Hom(X \otimes \mathfrak{p}_i/n (X \otimes \mathfrak{p}_i), (\mathbb{Z}/n \mathbb{Z})(1))
\]
and 
\[
\varphi_{0,n}' : Gr_{0}^{Z_{\alpha' \delta',n} } = Gr_{0}^{Z_{\alpha \delta,n} }  \overset{\varphi_{0,n}}{\longrightarrow} Y/nY.
\]
Lastly, $\phi'$ is defined by the following diagram similar to the above diagram defining $\lambda'$,
\[\begin{tikzcd}
[
  arrow symbol/.style = {draw=none,"\textstyle#1" description,sloped},
  isomorphic/.style = {arrow symbol={\simeq}},
  ]
&
X  & 
X \otimes_{\mathcal{O}_F} \mathfrak{p}_i 
\arrow [l, "id \otimes (  \mathcal{O}_F \hookleftarrow \mathfrak{p}_i)"] 
\\
Y \otimes_{\mathcal{O}_F} \mathfrak{p}_i^{-1} &
Y \arrow[l, "id \otimes (\mathfrak{p}_i^{-1} \hookleftarrow \mathcal{O}_F)"] 
\arrow[u, "\phi"] 
 & 
 Y \otimes_{\mathcal{O}_F} \mathfrak{p}_i 
 \arrow[l,"id \otimes (  \mathcal{O}_F \hookleftarrow \mathfrak{p}_i)"]
 \arrow[u, "\phi \otimes id"]
 \\
&
Y \otimes_{\mathcal{O}_F} \mathfrak{p}_i^{-1} 
\arrow[lu, "\xi \otimes id"]
\arrow[u, dashrightarrow] & 
Y
\arrow[l,"id \otimes (\mathfrak{p}_i^{-1} \hookleftarrow \mathcal{O}_F)"]
\arrow[u, dashrightarrow]
\arrow[uu, bend right=90, dashrightarrow, "\phi'"]
\end{tikzcd} 
\]

Moreover, on each strata, $F_{\mathfrak{p}_i}$ induces the morphism
\[
M_n(L^{Z_{\alpha \delta,n}},\langle\cdot, \cdot\rangle^{Z_{\alpha \delta,n}}) \rightarrow  M_n(L^{Z_{\alpha' \delta',n}},\langle\cdot, \cdot\rangle^{Z_{\alpha' \delta',n}})
\]
sending 
$(A, \lambda, i, (\alpha_n, \nu_n))$
to 
$(A', \lambda', i', (\alpha'_n, \nu'_n))$ as in the description before the theorem. For completeness, we summarize the description as follows. Using the above notations, 
$A' := A/ (Ker(F)[\mathfrak{p}_i])$,
$i'$ is induced by the quotient map $\pi_{\mathfrak{p}_i} : A \rightarrow A'$,  $\lambda'$ is characterized by 
$\xi \lambda = \pi_{\mathfrak{p}_i}^{\vee} \circ \lambda' \circ \pi_{\mathfrak{p}_i}$ which defines a prime to $p$ isogeny $\lambda'$, 
$\alpha_n' = \pi_{\mathfrak{p}_i} \circ \alpha_n$
and 
$\nu_n' = \nu_n \circ \kappa $. In other words, restriction of the partial Frobenius to (suitable union of) strata recovers the partial Frobenius on them. 
\end{theorem}

\begin{remark}
The diagram defining $\phi'$ is similar to the diagram defining $\lambda'$, and there is a reason for that. We will see in the proof that the diagram defining $\lambda'$ also defines a polarization for the universal semi-abelian variety over toroidal compactfications, and the diagram for $\phi'$ is the one induced on the (character group of) toric part. 

Moreover, the theorem is proved by first extending $F_{\mathfrak{p}_i}$ to toroidal compactifications, then contracting to a morphism on the minimal compactification. The description of the morphism on strata is obtained by looking at how $F_{\mathfrak{p}_i}$ operates on semi-abelian varieties. In particular, the morphism on indexing sets are obtained by looking at the toric part, and the morphism on strata are determined by the abelian part. See also remark \ref{explaination}.

\end{remark}

\begin{remark}
The description of strata in the minimal compactification shows that minimal compactifications is "plectic", which is the underlying reason that partial Frobenius extends to the minimal compactification. This can be made precise if we take care of the subtlety in dimension zero, which means that taking appropriate unions of the strata to define Shimura varieties of similitude PEL type as in definition \ref{similitude} (the strata we use are of Kottwitz's type as in definition \ref{isomorphism}). 

Another way to see the phenomenon is through Pink's mixed Shimura varieties, where he uses mixed Shimura varieties associated to parabolic subgroups (more precisely, the Levi group) to define strata of the minimal boundary, called rational boundary components in his terminology. In our case, the Shimura variety is associated to $G= Res_{F/\mathbb{Q}} H$, and the parabloics are also of the form $Res_{F/\mathbb{Q}}P$. However, the strata are assocaited to a subgroup $P_1$, in Pink's notation, of the parabolic, which is the Hermitian part in classical language, and this is not necessarily "plectic", i.e. not of the form $Res_{F/\mathbb{Q}} (-) $. The reason is that in \cite{MR1128753} 4.7,  Pink defines $P_1$ as the group satisfying certain minimality property, see \cite{MR1128753} 12.21 for an example how this kills "plecticity". The failure is similar to the difference between $G$ and $G^*$, and the remedy is the same. We can replace $P_1$ by another group in the parabolic, which is different only up to a similitude. Pink's theory still works in this slightly different setting, as already observed by him in remark (ii) of \cite{MR1128753} 4.11.   
\end{remark}

The proof of the theorem is rather technical, and we defer to the last section for details. We first give an application of it on the construction of plectic weight filtration of cohomology of Hilbert modular varieties in the next section. 

\section{Hilbert modular varieties}
\subsection{Basics}

We now specialize discussions in the previous section to the Hilbert modular varieties. The notations in this section will be the same as in the previous one, we simply restrict everything to a special case as follows. 

We take $\mathcal{O} = \mathcal{O}_F$, with $F$ a totally real field of degree 
$[F: \mathbb{Q}] =d$
and $* = id$, which coincides with  notations in the previous section in that $B= F$ and $F=F^c$ is the $*$-invariant part of the center of $B$. Moreover,
$L= \mathcal{O}_F \oplus \mathcal{O}_F$,
$\langle \cdot, \cdot\rangle_F $ is the standard $\mathcal{O}_F$-bilinear alternating pairing defined by the matrix 
$\begin{pmatrix}
0&-1\\
1&0\\
\end{pmatrix}$,
and
$\langle \cdot, \cdot\rangle   = Tr_{\mathcal{O}_F/ \mathbb{Z}} (\langle \cdot, \cdot\rangle_F)$. 
The morphism 
$h : \mathbb{C} \rightarrow End_{\mathcal{O}_{\mathbb{R}}}(L_{\mathbb{R}})$ is defined by 
$h(x+ iy) = \underset{\tau : F \hookrightarrow \mathbb{R}}{\prod} 
\begin{pmatrix}
x & -y\\
y & x\\
\end{pmatrix}$. 
These data defines a type $\mathbf{C}$ PEL datum. It is easy to see that the reflex field $F_0$ is  $\mathbb{Q}$.   
The relevant groups are 
\[
G = Res_{\mathcal{O}_F/\mathbb{Z}} GL_2
\]
and 
\[
G^*= det^{-1}(\textbf{G}_m) \subset G
\]
where $det : G \rightarrow Res_{\mathcal{O}_F/\mathbb{Z}}\textbf{G}_m$ is the similitude map. 

We give a brief account of the moduli problem it defines, which is a special case of definition \ref{similitude}. Let $\alpha \in \Omega$ be as in decomposition (\ref{A}), then $M_{\alpha, K(n)}$ is the moduli space representing the functor associating a locally Noetherian $\mathbb{Z}_{(p)}$-scheme $S$ to the isomorphism classes of tuples 
$(A, \lambda, i, (\eta, u)) $.
Here $A$ is an abelian scheme over $S$, 
$\lambda : A \rightarrow A^{\vee}$ 
is a prime to $p$ polarization, and
$i : \mathcal{O}_F \rightarrow End_S(A)$ 
a ring isomorphism inducing the trivial involution on $\mathcal{O}_F$ through $\lambda$ and  a rank 1 $\mathcal{O}_F$-module structure on $Lie(A)$. Note that the last condition is Kottwitz's determinant condition in this special case. Moreover, the level structure $(\eta, u)$ is an 
$\pi_1(S,\bar{s})$-invariant $K(n)$-orbit of 
$\mathcal{O}_F \otimes \mathbb{A}^{p \infty}$-equivariant isomorphism
\[\eta : L \otimes \mathbb{A}^{p \infty} = (\mathcal{O}_F \otimes \mathbb{A}^{p \infty})^{\oplus 2} \overset{\sim}{\rightarrow} V^pA_{\bar{s}}\]
together with an $\mathcal{O}_F \otimes \hat{\mathbb{Z}}^{p}$-equivariant isomorphism 
\[ u : \mathfrak{d}_F^{-1} \otimes \hat{\mathbb{Z}}^{p}(1) \overset{\sim}{\rightarrow} T^p(\mathfrak{d}_F^{-1} \otimes_{\mathbb{Z}} \textbf{G}_{m,\bar{s}}) \]
such that 
\[\langle \eta(x), \eta(y)\rangle_{\lambda} =Tr_{\mathcal{O}_F/\mathbb{Z}} ( u \circ (\alpha \langle   x,y\rangle_F))\]
where $x,y \in L \otimes \mathbb{A}^{p\infty}$. Since we work only with principal level $n$ structures, the level structure can also be seen as isomorphisms 
\[
(\mathcal{O}_F/n \mathcal{O}_F)^{\oplus 2} \cong A[n]
\]
and 
\[
\mathfrak{d}_F^{-1}/n\mathfrak{d}_F^{-1} \cong \mathfrak{d}_F^{-1} \otimes \mu_n. 
\]

\begin{remark}
In the literature, it is common to use a variant of the above moduli problem. More  precisely, the polarization is defined as an 
$\mathcal{O}_F$-equivariant isomorphism 
\[
(\mathfrak{c}, \mathfrak{c}_+) \cong (Hom_{\mathcal{O}_F}^{Sym}(A, A^{\vee}), Hom_{\mathcal{O}_F}^{Sym}(A, A^{\vee})_+)
\]
where $\mathfrak{c}$ is a fixed prime-to-$p$ fractional ideal representing 
$[\alpha] \in Cl_+(F) = (\mathcal{O}_F \otimes \mathbb{Z}_{(p)})^{\times}_+ \setminus (F \otimes \mathbb{A}^{(p, \infty)})^{\times} /  (\mathcal{O}_F \otimes \hat{\mathbb{Z}}^p)^{\times}$,
$\mathfrak{c}_+$ is the totally positive part (the elements that are positive for all embeddings of $F$ into $\mathbb{R}$), 
$Hom_{\mathcal{O}_F}^{Sym}(A, A^{\vee})$
is the symmetric $\mathcal{O}_F$-equivariant homomorphisms and 
$Hom_{\mathcal{O}_F}^{Sym}(A, A^{\vee})_+$
is the set of polarizations. The level structure is defined as an $\mathcal{O}_F$-equivariant isomorphism 
$(\mathcal{O}_F/n \mathcal{O}_F)^{\oplus 2} \cong A[n]$
together with an isomorphism 
$\mathcal{O}_F/ n \mathcal{O}_F \cong \mu_n \otimes \mathfrak{c}^*$, see \cite{MR3581176} for details. For the equivalence to our definition, see \cite{MR2055355} 4.1.1 for some discussion. 
\end{remark}

Similar to the previous section, we have 
 \[
 Sh_{K(n)}(G,h) = \underset{\alpha \in \Omega}{\coprod} M_{\alpha, K(n)} /\Delta = M_{K(n)} /\Delta
\] 
where 
$M_{K(n)} := \underset{\alpha \in \Omega }{\coprod} M_{\alpha, K(n)}$,
and 
\[
 M_{ K(n)} / \Delta = \underset{\begin{subarray}{c}
  \alpha \in \Omega \\
  \delta \in \Lambda
  \end{subarray}}{\coprod} 
  M_n(L, Tr_{\mathcal{O}_F/ \mathbb{Z}} \circ (\alpha \delta \langle  \cdot, \cdot\rangle_F))
\]

We now describe the minimal compactification in more explicit terms. Recall from theorem \ref{min} that cusp labels are equivalence classes of tuples 
$[(Z_n, \Phi_n, \delta_n)]$, 
where $Z_n$ is an $\mathcal{O}_F$-invariant filtration
\[
0 \subset Z_{n,-2} \subset Z_{n,-1} \subset Z_{n,0} = L/nL
\]
on $L/nL$ satisfying 
$Z_{n,-2} ^{\perp} = Z_{n,-1}$
and some liftablitiy condition,  $\Phi_n$ is a tuple $(X,Y,\phi, \varphi_{-2,n}, \varphi_{0,n}) $,
and 
$\delta_n : \underset{i}{\oplus} Gr^{Z_n}_i \overset{\sim }{\rightarrow} L/nL$
is a splitting with a liftability condition.
In the definition of $\Phi_n$, $X,Y$ are $\mathcal{O}$-lattices that are isomorphic as $B$-modules after tensoring with $\mathbb{Q}$, $\phi : Y \hookrightarrow X$ is an $\mathcal{O}$-invariant embedding. 
\[\varphi_{-2,n} : Gr_{-2}^{Z_n} \overset{\sim}{\rightarrow} Hom(X/nX, (\mathbb{Z}/n \mathbb{Z})(1))\] 
and 
\[\varphi_{0,n} : Gr_{0}^{Z_n} \overset{\sim}{\rightarrow} Y/nY\]
are isomorphisms that are reduction modulo n of $\mathcal{O}$-equivariant isomorphisms
$\varphi_{-2}: Gr_{-2}^Z \overset{\sim}{\rightarrow} Hom_{\hat{\mathbb{Z}}^p} (X \otimes \hat{\mathbb{Z}}^p , \hat{\mathbb{Z}}^p(1))$ and 
$\varphi_0 : Gr_0^Z \overset{\sim}{\rightarrow} Y \otimes \hat{\mathbb{Z}}^p$ such that 
\[\varphi_{-2}(x)(\phi(\varphi_0(y))) = \langle x,y\rangle     
\]
for $x \in Gr_{-2}^Z$ and $y \in Gr_0^Z$.

In our case, $L = \mathcal{O}_F^{\oplus 2}$ and there are essentially two different filtrations on $L/nL$, either 
$Z_{n,-2}=0$ and $Z_{n,-1} = L/nL$,
or $Z_{n,-2}=Z_{n,-1}$ is a $\mathcal{O}_F$-submodule of $L/nL$ being reduction of a rank 1 
$\mathcal{O}_F \otimes \mathbb{A}^{p}$-submodule of $L \otimes \mathbb{A}^p$.
The first case is trivial, the corresponding strata is the open strata in the minimal compactification. We focus on the second case from now on. 

The isomorphisms $\varphi_{0,n}$ and $\varphi_{-2,n}$ force $X$ and $Y$ to be rank 1 $\mathcal{O}_F$-modules, which are isomorphic to fractional ideals of $F$ and classified by $Cl(F)$. We observe that 
$Gr_{-1}^{Z_n}=0$, 
implying that $L^{Z_n} =0 $. Thus the strata 
associated to 
$[(Z_n, \Phi_n, \delta_n)]$
must be 
$Isom( \mathbb{Z}/n \mathbb{Z}(1), \mu_n)$, i.e.
\[
M_n(L^{Z_n}, \langle  \cdot, \cdot\rangle^{Z_n}) = Isom( \mathbb{Z}/n \mathbb{Z}(1), \mu_n)
\]
see remark \ref{r3} for explanations. In other words, the boundary components all have dimension zero, and they are generally referred to as cusps.  

\subsection{The weight spectral sequence}

We now make the spectral sequence in theorem \ref{w.s.s.} more explicit in our special case. With notations as in section \ref{app Sh}, we take $V = \mathbb{Q}(0)$ to be the trivial representation of $G$, then 
$\mathcal{F}V = \mathbb{Q}(0)$
is the constant sheaf in 
$D^b_m(Sh_{K(n)} (G,h))$, 
i.e. $\mathbb{Q}(0)$ is either the constant Hodge module $\mathbb{Q}^H(0)$ or the constant mixed $l$-adic sheaf 
$\overline{\mathbb{Q}_l}(0)$. Let 
$j: Sh_{K(n)} (G,h) \hookrightarrow Sh_{K(n)} (G,h)^{min}$
be the open embedding, then the spectral sequence is 
\begin{equation} \label{s.s.h}
E_1^{p,q}= \mathbb{H}^{p+q}(Sh_{K(n)} (G,h)^{min}, w_{\geq -p} w_{\leq -p} Rj_*(\mathbb{Q}(0))) \Rightarrow H^{p+q}(Sh_{K(n)} (G,h), \mathbb{Q}(0))
\end{equation}
Since 
$\mathbb{Q}(0)$
is pure of weight 0, proposition \ref{functors} and theorem \ref{ie} tells us that the first nontrivial piece is 
\[
w_{\geq 0} w_{\leq 0} Rj_*(\mathbb{Q}(0)) =w_{\leq 0}Rj_*(\mathbb{Q}(0)) = j_{!*}(\mathbb{Q}(0))
\]
as we see in the discussion prior to theorem \ref{w.s.s.},
so 
\[
E^{0, q}_1 = IH^{q}(Sh_{K(n)} (G,h)^{min}, \mathbb{Q}(0))
\]
Similarly, the discussion before theorem \ref{w.s.s.} gives information on the rest of the rest of the pieces. In particular, equation (\ref{b1}) tells that for $k > 0$
\begin{equation} \label{289}
w_{\leq k} w_{\geq k} Rj_*(\mathbb{Q}(0))=
i_* w_{\leq k} w_{\geq k} i^*Rj_*(\mathbb{Q}(0))
\end{equation}
where 
\[
i : Sh_{K(n)} (G,h)^{min} \setminus Sh_{K(n)} (G,h) =
\underset{\begin{subarray}{c}
  \alpha \in \Omega \\
  \delta \in \Lambda
  \end{subarray}}{\coprod} 
  \underset{\begin{subarray}{c}
[(Z_{\alpha \delta,n}, \Phi_{\alpha \delta,n}, \delta_{\alpha \delta,n})] \\
Z_{\alpha \delta, n,-2} \neq 0
\end{subarray}}{\coprod} 
Isom( \mathbb{Z}/n \mathbb{Z}(1), \mu_n) 
\]
\begin{equation} \label{para}
\lhook\joinrel\longrightarrow Sh_{K(n)} (G,h)^{min}
\end{equation}
is the inclusion of the complement of $Sh_{K(n)} (G,h)$, i.e.   the inclusion of finitely many cusps.  Now Pink (\cite{MR1149032}) or Burgos and Wildeshaus’ (\cite{burgos2004hodge}) formula tells us that
\begin{equation} \label{899}
i_{\partial} ^* R^n j_*(\mathbb{Q}(0)) = \underset{a +b = n}{\oplus} \mathcal{F} (H^a (\overline{H}_C , H^b (Lie W_1, \mathbb{Q}(0))))
\end{equation}
where 
$\partial := [(Z_{\alpha \delta,n}, \Phi_{\alpha \delta,n}, \delta_{\alpha \delta,n})]$
and  $i_{\partial} $ is the inclusion of the cusp corresponding to $\partial$. Moreover, 
$W_1$ is the unipotent of the Borel subgroup corresponding to the cusp 
$\partial$, 
and 
$\overline{H}_C$
is an arithmetic subgroup of the linear part of the Levi group determined by the level $K(n)$. By proof of theorem 3.5 in 
\cite{MR2923168}, 
we have that 
  \begin{equation*}
  i_{\partial} ^* R^n j_*(\mathbb{Q}(0)) 
 = \begin{cases}
  \bigwedge^n(\mathbb{Q}(0)^{d-1}) & 0 \leq n \leq d-1, \\
  \bigwedge^{2d-1-n}(\mathbb{Q}(0)^{d-1})(-d) 
 & d \leq n \leq 2d-1 .
  \end{cases}
  \end{equation*}
Note that the author only works with the Hodge module case in \cite{MR2923168}, but the proof works equally well for the $l$-adic case. Indeed, if we view $\mathbb{Q}(0)$ as the trivial representation of $\textbf{G}_m$, which is the group corresponding to the zero dimensional Shimura variety 
$Isom(\mathbb{Z}/n \mathbb{Z}(1), \mu_n)$ 
indexed by $\partial$, and 
$(-d)$ twisting by $d$-th power of the dual of the standard representation, then the proof in \cite{MR2923168} shows that
 \begin{equation} \label{pink}
  i_{\partial} ^* R^n j_*(\mathbb{Q}(0)) 
 = \begin{cases}
  \mathcal{F}(\bigwedge^n(\mathbb{Q}(0)^{d-1})) & 0 \leq n \leq d-1, \\
  \mathcal{F}(\bigwedge^{2d-1-n}(\mathbb{Q}(0)^{d-1})(-d)) 
 & d \leq n \leq 2d-1 .
  \end{cases}
  \end{equation}
  
We will use a different parametrization of the cusps than (\ref{para}). Recall that  $\Lambda$ in (\ref{O}) is chosen such that 
\[
Isom_{\mathcal{O}_F}(\mathfrak{d}_F^{-1}/n\mathfrak{d}_F^{-1}(1),   \mathfrak{d}_F^{-1} \otimes_{\mathbb{Z}} \mu_n) = \underset{\delta \in \Lambda}{\coprod} Isom(\mathbb{Z}/n\mathbb{Z}(1), \mu_n)
\]
We use it to rewrite (\ref{para}) as 
\begin{equation} \label{newpara}
i : Sh_{K(n)} (G,h)^{min} \setminus Sh_{K(n)} (G,h) =
\underset{\begin{subarray}{c}
  \alpha \in \Omega 
  \end{subarray}}{\coprod} 
  \underset{\begin{subarray}{c}
\partial
\end{subarray}}{\coprod} 
Isom_{\mathcal{O}_F}(\mathfrak{d}_F^{-1}/n\mathfrak{d}_F^{-1}(1),   \mathfrak{d}_F^{-1} \otimes_{\mathbb{Z}} \mu_n)
\end{equation}
with a new parametrization set of cusps, which we still denote by $\partial$. For a precise description of $\partial$, see \cite{MR2099078}. For such a parametrization, the boundary is a union of zero-dimensional Shimura varieties associated to $Res_{F/ \mathbb{Q}}\textbf{G}_m$, and a minor modification of the proof in \cite{MR2923168} shows that 
\begin{equation} \label{pink2}
  i_{\partial} ^* R^n j_*(\mathbb{Q}(0)) 
 = \begin{cases}
  \mathcal{F}(\bigwedge^n(\mathbb{Q}(0)^{d-1})) & 0 \leq n \leq d-1, \\
  \mathcal{F}(\bigwedge^{2d-1-n}(\mathbb{Q}(0)^{d-1})(-d)) 
 & d \leq n \leq 2d-1 .
  \end{cases}
  \end{equation}
where $\partial$ denotes the cusps in (\ref{newpara}), and $(-d)$ is twisting by 
the one dimensional representation $Nm_{F/\mathbb{Q}}^{-1} :Res_{F/\mathbb{Q}}\textbf{G}_m \rightarrow \textbf{G}_m$
of $Res_{F/\mathbb{Q}}\textbf{G}_m$. Note that the corresponding sheaf is the $(-d)$-th power of the Tate twisting sheaf, explaining the notation. Further, this is the only new observation one needs in the proof of the above.  

Then together with equation (\ref{289}) we have that for $k > 0$, 
\begin{equation} \label{www}
w_{\geq k} w_{\leq k} R^nj_* (\mathbb{Q}(0)) 
= \begin{cases}
 i_*(\bigwedge^{2d-1-n}(\mathbb{Q}(0)^{d-1})(-d)) 
 & d \leq n \leq 2d-1, k=2d 
 \\
 0 & \text{otherwise}
 \end{cases}
\end{equation}
Thus the spectral sequence (\ref{s.s.h}) becomes

\begin{tikzpicture}
  \matrix (m) [matrix of math nodes,
    nodes in empty cells,nodes={minimum width=12ex,
    minimum height=4ex,outer sep=15pt},
    column sep=1ex,row sep=3ex]{
    \oplus \bigwedge^{0}(\mathbb{Q}(0)^{d-1})(-d) & 0 & \cdots & 0 & 0  & 4d-1
    \\
       \vdots & \vdots & \ddots & \vdots & \vdots & \vdots 
\\
       \oplus \bigwedge^{d-1}(\mathbb{Q}(0)^{d-1})(-d) & 0 & \cdots & 0 & 0  & 3d 
       \\
 0 &   0   &  \cdots   &   0  &          0              & 3d-1
 \\
      \vdots      &   \vdots   &  \ddots   &   \vdots  &          \vdots            & \vdots 
      \\
        0    &   0    &  \cdots &  0   &          0              & 2d+1
        \\
         0   &   0   &  \cdots   &  0   &   IH^{2d}(M^*,\mathbb{Q}(0))  & 2d 
         \\
          \vdots  &  \vdots    &  \ddots  &  \vdots   &  \vdots  & \vdots
          \\
        0  &  0   &  \cdots   &  0  &  IH^{d+1}(M^*,\mathbb{Q}(0))  & d+1
        \\
          \vdots     &  \vdots &  \ddots  & \vdots &  \vdots  & \vdots
          \\
          0     &  0  & \cdots &  0  &  IH^0(M^*,\mathbb{Q}(0)) & 0
          \\
    -2d &   -2d + 1  &  \cdots  &  -1  & 0 &  \strut
    \\};
 
    \draw[-stealth] (m-2-1.north east) -- (m-6-5.south west);
  \draw[-stealth] (m-4-1.north east) -- (m-8-5.south west);
\draw[thick] (m-1-5.east) -- (m-12-5.east) ;
\draw[thick] (m-12-1.north) -- (m-12-6.north) ;
\end{tikzpicture}
where 
$M^* := Sh_{K(n)}(G,h)^{min}$,
$M := Sh_{K(n)}(G,h)$
and 
\[
\oplus \bigwedge^i (\mathbb{Q}(0)^{d-1}) (-d) := \underset{\partial}{\oplus} (\bigwedge^i (\mathbb{Q}(0)^{d-1}) (-d))^{\oplus n}
\]
\[
\ \ \ \ \ \ \ \ \ \ \ \ \ \ \ \ \ \ \  \ \ \ \ \ \ \ \ \ \ \ \ \ \ \ \ \ \ 
= H^{2d-1-i}(M^* \setminus M, i^* Rj^*(\mathbb{Q}(0)))
\]
for $0 \leq i \leq d-1$. 
Note that in either case (Hodge modules or $l$-adic), the cohomology is taken after passing to the algebraic closure of the base field, so 
$M^* \setminus M = \underset{\partial}{\coprod} Isom(\mathbb{Z}/n\mathbb{Z}(1) , \mu_n)$ 
becomes 
$\underset{\partial}{\coprod} \underset{n}{\coprod} \{*\}$, 
explaining the second equality of the above. 

Now we can read off from the above computation that 
\[
0 \longrightarrow E_{\infty}^{-2d, 4d-1} \longrightarrow 
\underset{\partial}{\oplus} (\bigwedge^0 (\mathbb{Q}(0)^{d-1}) (-d))^{\oplus n}
\longrightarrow
IH^{2d}(M^*, \mathbb{Q}(0))
\longrightarrow E^{0,2d}_{\infty} \longrightarrow 0
\]
\[
0 \longrightarrow IH^{2d-1}(M^*, \mathbb{Q}(0)) \longrightarrow
H^{2d-1}(M, \mathbb{Q}(0)) 
\longrightarrow E_{\infty}^{-2d,4d-1} 
\longrightarrow 0
\]
where 
$E^{0,2d}_{\infty} = H^{2d}(M, \mathbb{Q}(0)) = 0$
as $M$ is non-proper of dimension $d$. Moreover, we observe easily that 
\[
H^i(M, \mathbb{Q}(0)) = IH^i(M^*, \mathbb{Q}(0))
\]
for $0 \leq i \leq d-1$, and 
\[
0 \longrightarrow IH^i(M^*, \mathbb{Q}(0)) \longrightarrow
H^i(M, \mathbb{Q}(0)) \longrightarrow 
\underset{\partial}{\oplus} (\bigwedge^{2d-1-i} (\mathbb{Q}(0)^{d-1}) (-d))^{\oplus n}
\longrightarrow 0
\]
for $d \leq i \leq 2d -2$. In the last exact sequence, we use that
\[
E^{-2d, 2d+i}_{\infty} = \underset{\partial}{\oplus} (\bigwedge^{2d-1-i} (\mathbb{Q}(0)^{d-1}) (-d))^{\oplus n}
\]
for $ d \leq i \leq 2d-2$, which follows because the domain and codomain of the differentials in the picture have different weights in this range. 

We observe from the above computation that the spectral sequence (\ref{s.s.h}) gives us the weight filtration on 
$H^*(M, \mathbb{Q}(0))$, which provides a new computation of the weight filtration of the cohomology of Hilbert modular varieties without using the Borel-Serre compactifications as done, for example, in the last section of \cite{nekovar2017plectic}. This is a philosophically better computation as it is performed in the algebraic category, whereas the older computation uses the non-algebraic Borel-Serre compactifications and proceeds in a more indirect way when establishing the mixed Hodge structures. See \cite{ayoub2012relative} for a modern treatment of the motivic meaning of the reductive Borel-Serre compactifications. 

\subsection{The plectic weight filtration}

Now we make use of the spectral sequence (\ref{s.s.h}) to construct the plectic weight filtration. Note that the filtration induced by (\ref{s.s.h}) is a $\mathbb{Z}$-filtration, but the plectic weight filtration we are looking for is a $\mathbb{Z}^d$-filtration. We will use the partial Frobenius to cut out the $\mathbb{Z}$-filtration into a $\mathbb{Z}^d$-filtration, and show that this is the sought-after plectic weight filtration.  

Firstly, we compute the eigenvalues of the partial Frobenius on the boundary cohomology 
$ H^*(M^* \setminus M, i^* Rj_*(\mathbb{Q}(0)))$. 
We denote the canonical PEL (up to similitude) smooth integral model $M_{K(n)}/ \Delta$ of $M$ by $\mathscr{M}$, which is defined over an open dense subset of $Spec(\mathbb{Z})$. Similarly, 
$\mathscr{M}^*$ 
is the integral model of the minimal compactification. 
Now choose a prime $p$ in the open subset such that it is split in $F$, and lies in  the applicable range of theorem \ref{comparison thm}. Then as we have already seen, the Frobenius $Frob_p$ on $\mathscr{M}^*_{\mathbb{F}_p}$
decomposes into 
$Frob_p= \underset{i}{\prod}
F_i$, 
where $F_i$ is the partial Frobenius corresponding to the prime $\mathfrak{p}_i$ in the prime decomposition 
$p=\underset{i}{\prod}\mathfrak{p}_i$
of $p$ in $F$.

Let us recall the construction of the $l$-adic sheaf on a Shimura variety coming from an algebraic representation, following Pink (\cite{MR1149032}). Let $G$ be a reductive group giving rise to a Shimura datum, with associated Shimura variety $Sh_K$, for compact open $K \subset G(\mathbb{A}_f)$. For $K’ \subset K$ normal, there is a natural Galois etale covering 
\[
\pi_{K’}: Sh_{K’} \rightarrow Sh_K
\]
with Galois group $K/K’$. We choose a system of $K’$ such that $K’$ differs from $K$ only in $l$-adic part $K’_l$, i.e. $K/K’ = K_l/K_l’$, 
and their $l$-adic parts $K’_l$ form a basis of $G(\mathbb{Q}_l)$. Let $V$ be an algebraic representation of $G$, then it gives rise to a continuous $l$-adic representation of $G(\mathbb{Q}_l)$, which contains a lattice $\Lambda$ stable by all $K’_l$, and for $K_l’ \subset K_l$, there exists a number $n$ such that the natural action of $K_l’$ and $K_l$ on $\Lambda$ induces a representation of $K_l/ K_l’$ on 
$\Lambda /l^n \Lambda$,
then we have an etale sheaf 
\[
\mathcal{V}_{K’} :=( \pi_{K’*}(\mathbb{Z}/l^n\mathbb{Z}) \otimes_{\mathbb{Z}/l^n\mathbb{Z}} \Lambda /l^n \Lambda)^{K_l /K_l’}
\]
where the action of $K_l/K_l’$ on the first factor is induced by the Galois covering $\pi_{K’}$, and the second factor is induced by the representation we have just constructed. These $\mathcal{V}_{K’}$ form an inverse system, and we define the associated $l$-adic sheaf by
\[
\mathcal{F}V := (\underset{K’}{\varprojlim} \mathcal{V}_{K’})\otimes_{\mathbb{Z}_l} \overline{\mathbb{Q}_l}
\]
This is independent of the choices we have made. Similar to Hecke operators, the partial Frobenius induces natural maps between $\mathcal{F}V$, i.e. 
$\mathcal{F}V \longrightarrow F_{i*}\mathcal{F}V$. The key to it is that the partial Frobenius is compatible with the projections $\pi_{K’}$, i.e.
\[
\begin{tikzcd}
Sh_{K’} \arrow[r,"F_i"] \arrow[d,"\pi_{K’}"] &
Sh_{K’} \arrow[d, "\pi_{K’}"]
\\
Sh_K \arrow[r, "F_i"] &
Sh_K 
\end{tikzcd}
\]
is commutative and equivariant for the Galois group.
It is a general heuristic that the partial Frobenius are amplified Hecke operators in characteristic $p$. 

Moreover, the isomorphism (\ref{899}) is compatible with the partial Frobenius. As in theorem \ref{pFextends} (for PEL Shimura varities), the partial Frobenius $F_i$ extends to the minimal compactification and preserves both the open 
$Sh_K \overset{j}{\hookrightarrow} Sh_K^{min}$
and the boundary 
$Sh_K^{min} \setminus Sh_K \overset{i}{\hookrightarrow} Sh_K^{min}$,
therefore inducing the map 
\[
i^*Rj_* \overline{\mathbb{Q}_l} \overset{i^*Rj_* (-) }{\longrightarrow } 
i^*Rj_* F_{i*} \overline{\mathbb{Q}_l}
= i^* F_{i*}Rj_* \overline{\mathbb{Q}_l}
\overset{\text{b.c.}}{\longrightarrow}  F_{i*}i^* Rj_* \overline{\mathbb{Q}_l}
\]
which under the natural isomorphism (\ref{899}), corresponding to the natural map 
$\mathcal{F}V \rightarrow F_{i*}\mathcal{F}V$
for $V$ specified in (\ref{899}). 

\begin{remark}
The above naturality can be proved with the same proof as in 4.8 of \cite{MR1149032}, where it is proved for the Hecke operators. The key property underlying the proof is the compatibility of Hecke operators with the toroidal compactifications. The same compatibility result holds for the partial Frobenius as we will see in the next section.
\end{remark}

Now we go back to the special case of Hilbert modular varieties. Applying the above functoriality to the isomorphism (\ref{pink}), we can reduce the computation of
$i^*Rj_*\overline{\mathbb{Q}_l} \rightarrow F_{i*}i^*Rj_*\overline{\mathbb{Q}_l}$
to the computation of 
$
\mathcal{F}V \rightarrow F_{i*}\mathcal{F}V
$
for $V$ as in (\ref{pink}). 

We make use of the parametrization (\ref{newpara}). For an arbitrary integer $k$, let \[
\pi_k : Isom_{\mathcal{O}_F}(\mathfrak{d}_F^{-1}/nl^k\mathfrak{d}_F^{-1}(1),   \mathfrak{d}_F^{-1} \otimes_{\mathbb{Z}} \mu_{nl^k}) 
\rightarrow 
Isom_{\mathcal{O}_F}(\mathfrak{d}_F^{-1}/n\mathfrak{d}_F^{-1}(1),   \mathfrak{d}_F^{-1} \otimes_{\mathbb{Z}} \mu_n)
\]
be the natural map, corresponding to the covering map $\pi_{K’}$ as above. Let
$\theta \in Isom_{\mathcal{O}_F}(\mathfrak{d}_F^{-1}/n\mathfrak{d}_F^{-1}(1),   \mathfrak{d}_F^{-1} \otimes_{\mathbb{Z}} \mu_n)$, 
and we suppose that $\theta $ lies in the position
$(\alpha, \delta, \partial)$
of the decomposition
\[
\mathscr{M}^*_{\mathbb{F}_p}\setminus \mathscr{M}_{\mathbb{F}_p} =
\underset{\begin{subarray}{c}
  \alpha \in \Omega 
  \end{subarray}}{\coprod} 
  \underset{\partial}{\coprod} 
Isom_{\mathcal{O}_F}(\mathfrak{d}_F^{-1}/n\mathfrak{d}_F^{-1}(1),   \mathfrak{d}_F^{-1} \otimes_{\mathbb{Z}} \mu_n) 
\]
Recall that $F_i$ maps 
$(\alpha, \delta,\partial) $
to 
$(\alpha_1, \delta_1, \partial_1)$, 
where 
$\alpha_1$ is defined by 
\[
\xi\alpha = \epsilon_1 \alpha_1 \lambda_1
\]
with $\xi \in \mathcal{O}_F$ such that $v_{\mathfrak{p}_i}(\xi)=1$ and $v_{\mathfrak{p}_j}(\xi)=0$ for $j \neq i$, 
$\alpha_1 \in \Omega, \epsilon_1 \in (\mathcal{O}_F \otimes \mathbb{Z}_{(p)})_+^{\times}$
and 
$\lambda_1 \in (\mathcal{O}_F \otimes \hat{\mathbb{Z}}^p)^{\times}$
as in decomposition (\ref{A}). Moreover, $\partial_1$ is defined as in theorem \ref{pFextends} (being a union of $\partial'$ in theorem \ref{pFextends}) , and $F_i$ maps $\theta$ to $\lambda_1 \theta$ as in definition \ref{pFdef}. The vague description of $\partial_1$ here suffices for our purpose. In summary, 
\[
\theta |_{(\alpha, \partial)} \overset{F_i}{\longrightarrow} (\lambda_1 \theta)|_{(\alpha_1,\partial_1)}
\]
with obvious notations.

We can repeat the above procedure and obtain
\[
\xi\alpha_1 = \epsilon_2 \alpha_2 \lambda_2
\]
\begin{equation} \label{jjjj}
\vdots
\end{equation}
\[
\xi\alpha_m = \epsilon_{m+1} \alpha_{m+1} \lambda_{m+1}
\]
where 
$\alpha_j \in \Omega, \epsilon_j \in (\mathcal{O}_F \otimes \mathbb{Z}_{(p)})_+^{\times}$
and 
$\lambda_j \in (\mathcal{O}_F \otimes \hat{\mathbb{Z}}^p)^{\times}$
as in decomposition (\ref{A}). Then 
\[
\theta |_{(\alpha, \partial)} \overset{F_i^m}{\longrightarrow} (\lambda_1 \cdots \lambda_{m} \theta)|_{(\alpha_m,\partial_m)}
\]
As $F_i$ permutes the cusps, we know that there is a minimal integer $N$ such that 
\[
F_i^N (\theta) = \theta.
\]
Note that this means that 
$\lambda_1 \cdots \lambda_{N} \theta = \theta$, $\alpha_N = \alpha$ and 
$\partial_N = \partial$.

We denote by $\widetilde{\mathscr{M}}$ the Hilbert modular variety of principal level $nl^k$, then we have a natural commutative diagram map
\[
\begin{tikzcd}
\widetilde{\mathscr{M}}_{\mathbb{F}_p}^* \setminus \widetilde{\mathscr{M}}_{\mathbb{F}_p} \arrow[r,"F_i"] 
\arrow[d,"\pi_{k}"] &
\widetilde{\mathscr{M}}_{\mathbb{F}_p}^* \setminus \widetilde{\mathscr{M}}_{\mathbb{F}_p}
\arrow[d, "\pi_{k}"]
\\
\mathscr{M}_{\mathbb{F}_p}^* \setminus \mathscr{M}_{\mathbb{F}_p} 
\arrow[r, "F_i"] & 
\mathscr{M}_{\mathbb{F}_p}^* \setminus \mathscr{M}_{\mathbb{F}_p}
\end{tikzcd}
\]
Together with the decomposition 
\[
\widetilde{\mathscr{M}}_{\mathbb{F}_p}^* \setminus \widetilde{\mathscr{M}}_{\mathbb{F}_p} =
\underset{\begin{subarray}{c}
  \alpha \in \Omega 
  \end{subarray}}{\coprod} 
  \underset{\widetilde{\partial}}{\coprod} 
Isom_{\mathcal{O}_F}(\mathfrak{d}_F^{-1}/nl^k\mathfrak{d}_F^{-1}(1),   \mathfrak{d}_F^{-1} \otimes_{\mathbb{Z}} \mu_{nl^k}) 
\]
we are reduced to the situation 
\[
\begin{tikzcd}
Isom_{\mathcal{O}_F}(\mathfrak{d}_F^{-1}/nl^k\mathfrak{d}_F^{-1}(1),   \mathfrak{d}_F^{-1} \otimes_{\mathbb{Z}} \mu_{nl^k})|_{(\alpha, \widetilde{\partial})} 
\arrow[r,"F_i"] 
\arrow[d,"\pi_{k}"] &
Isom_{\mathcal{O}_F}(\mathfrak{d}_F^{-1}/nl^k\mathfrak{d}_F^{-1}(1),   \mathfrak{d}_F^{-1} \otimes_{\mathbb{Z}} \mu_{nl^k})|_{(\alpha_1, \widetilde{\partial}_1)} 
\arrow[d, "\pi_{k}"]
\\
Isom_{\mathcal{O}_F}(\mathfrak{d}_F^{-1}/n\mathfrak{d}_F^{-1}(1),   \mathfrak{d}_F^{-1} \otimes_{\mathbb{Z}} \mu_{n})|_{(\alpha, \partial)} 
\arrow[r, "F_i"] & 
Isom_{\mathcal{O}_F}(\mathfrak{d}_F^{-1}/n\mathfrak{d}_F^{-1}(1),   \mathfrak{d}_F^{-1} \otimes_{\mathbb{Z}} \mu_{n})|_{(\alpha_1, \partial_1)}
\end{tikzcd}
\]
The same description of $F_i$ applies to $\widetilde{\mathscr{M}}_{\mathbb{F}_p}$. In summary, 
\[
\widetilde{\theta} |_{(\alpha, \widetilde{\partial})} \overset{F_i}{\longrightarrow} (\lambda_1 \widetilde{\theta})|_{(\alpha_1,\widetilde{\partial}_1)}
\]
One subtlety here is that there are more than one $\widetilde{\partial}$ lying over $\partial$. However, the cusps they parametrize are canonically isomorphic, and we can choose one $\widetilde{\partial}$ for each $\partial$. 

For simplicity, we assume that $l$ is prime to $n$, then the Galois group for the covering $\pi_k$ is 
$(\mathcal{O}_F / l^k\mathcal{O}_F)^{\times}$.
If we denote by $V$ the one dimensional representation 
$Nm_{F/ \mathbb{Q}}^{-1} : Res_{F/ \mathbb{Q}}\textbf{G}_m \rightarrow \textbf{G}_m$, 
then its $l$-adic points induces the reduced representation
$Nm_{F/\mathbb{Q}}^{-1} : (\mathcal{O}_F / l^k\mathcal{O}_F)^{\times} \longrightarrow (\mathbb{Z} / l^k \mathbb{Z})^{\times}$,
which we denote by $\mathcal{V}_k$.
We fix a non-zero element 
$v_k \in \mathcal{V}_k$ for each $k$, and we assume that they are compatible when $k$ varies. From the description we have just reviewed, we have
\[
\mathcal{F}V =  (\underset{k}{\varprojlim} ( \pi_{k*}(\mathbb{Z}/l^k\mathbb{Z}) \otimes_{\mathbb{Z}/l^k\mathbb{Z}} \mathcal{V}_k)^{(\mathcal{O}_F / l^k\mathcal{O}_F)^{\times}})
\otimes_{\mathbb{Z}_l} \overline{\mathbb{Q}_l}
\]
For a fixed k, if we choose a
\[
\widetilde{\theta} \in Isom_{\mathcal{O}_F}(\mathfrak{d}_F^{-1}/nl^k\mathfrak{d}_F^{-1}(1),   \mathfrak{d}_F^{-1} \otimes_{\mathbb{Z}} \mu_{nl^k})
\]
such that
$\pi_k(\widetilde{\theta} )= 
\theta$,
then 
\[
\pi_{k*}(\mathbb{Z}/l^k\mathbb{Z}) \otimes_{\mathbb{Z}/l^k\mathbb{Z}} \mathcal{V}_k)^{(\mathcal{O}_F / l^k\mathcal{O}_F)^{\times}} |_{\theta}
= (\mathbb{Z}/ \l^k \mathbb{Z}) \cdot
\underset{g \in (\mathcal{O}_F / l^k\mathcal{O}_F)^{\times}}{ \sum} (g \widetilde{\theta}) \otimes ( Nm_{F/ \mathbb{Q}}^{-1}(g) v_k) 
\]
i.e. the choice of $v_k$ and $\widetilde{\theta}$ gives a basis 
$\underset{g \in (\mathcal{O}_F / l^k\mathcal{O}_F)^{\times}}{ \sum} (g \widetilde{\theta}) \otimes ( Nm_{F/ \mathbb{Q}}^{-1}(g) v_k) $
of 
$\pi_{k*}(\mathbb{Z}/l^k\mathbb{Z}) \otimes_{\mathbb{Z}/l^k\mathbb{Z}} \mathcal{V}_k)^{(\mathcal{O}_F / l^k\mathcal{O}_F)^{\times}} |_{\theta}$.

Now using this explicit description, we can compute the natural morphism 
$\mathcal{F}V \rightarrow F_{i*}\mathcal{F}V$ (over 
$\mathscr{M}_{\mathbb{F}_p}^* \setminus \mathscr{M}_{\mathbb{F}_p}
$)
as follows. It is 
($\overline{\mathbb{Q}_l} \otimes (-)$) 
the direct limit of the morphism
\[
\left.
\underset{g \in (\mathcal{O}_F / l^k\mathcal{O}_F)^{\times}}{ \sum} (g \widetilde{\theta}) \otimes ( Nm_{F/ \mathbb{Q}}^{-1}(g) v_k) \right |_{(\alpha, \partial)} 
\rightarrow 
\left. \underset{g \in (\mathcal{O}_F / l^k\mathcal{O}_F)^{\times}}{ \sum} (g \lambda_1 \widetilde{\theta}) \otimes ( Nm_{F/ \mathbb{Q}}^{-1}(g) v_k) \right\vert_{(\alpha_1, \partial_1)} 
\]
For a fixed $\theta$ and the corresponding minimal $N$ as above, we can iterate the process and obtain a basis for the stalk of the sheaf at $F_i^m(\theta |_{(\alpha, \partial)})$
for $m < N$. Note that by the choice of $N$, 
$F_i^m(\theta |_{(\alpha, \partial)})$
are all different for $m < N$. When $m=N$, we have $F_i^N(\theta |_{(\alpha, \partial)}) = \theta |_{(\alpha, \partial)}$, and 
\[
\underset{g \in (\mathcal{O}_F / l^k\mathcal{O}_F)^{\times}}{ \sum} (g \widetilde{\theta}) \otimes ( Nm_{F/ \mathbb{Q}}^{-1}(g) v_k)
\overset{F_i^N}{\longrightarrow}
\underset{g \in (\mathcal{O}_F / l^k\mathcal{O}_F)^{\times}}{ \sum} (g \lambda_1 \cdots \lambda_N \widetilde{\theta}) \otimes ( Nm_{F/ \mathbb{Q}}^{-1}(g) v_k)
\]
\[
\ \ \ \ \ \ \ \ \ \ \ \ \ \ \ \ \ \ \ \ \ \ \ \ \ \ \ \ \ \ \ \ \ \ \ \ \ \ \ \ \\ \\ \ \ \ \ \ \ \ 
= Nm_{F/ \mathbb{Q}}(\lambda_1 \cdots \lambda_N)
\underset{g \in (\mathcal{O}_F / l^k\mathcal{O}_F)^{\times}}{ \sum} (g  \widetilde{\theta}) \otimes ( Nm_{F/ \mathbb{Q}}^{-1}(g) v_k)
\]
This tells us that with the basis we have chosen, 
$F_i$ has a block of the form
\[
\begin{pmatrix} 
 &   &  & &  Nm_{F/ \mathbb{Q}}(\lambda_1 \cdots \lambda_N)  \\
1 &   &     \\
 & 1  &    \\
 &  & \ddots \\
& & &  1
\end{pmatrix}
\]
This is a matrix expression of a morphism between free $\mathbb{Z}/ l^k \mathbb{Z}$-modules, taking the inverse limit over $k$ and tensor with $\overline{\mathbb{Q}_l}$, we have the same matrix (partial expression) for the desired morphism 
$\mathcal{F}V \rightarrow F_{i*}\mathcal{F}V$. 
Now from the equation (\ref{jjjj}) and $\alpha_N = \alpha$, we have 
\[
\xi^N \alpha = (\epsilon_1 \cdots \epsilon_N) \alpha (\lambda_1 \cdots \lambda_N)
\]
with 
$(\epsilon_1 \cdots \epsilon_N) \in  (\mathcal{O}_F \otimes \mathbb{Z}_{(p)})_+^{\times}$
and 
$\lambda_1 \cdots \lambda_N \in (\mathcal{O}_F \otimes \hat{\mathbb{Z}}^p)^{\times}$. Hence 
\[
\lambda_1 \cdots \lambda_N = \xi^N (\epsilon_1 \cdots \epsilon_N)^{-1}
\]
and 
\[
Nm_{F/ \mathbb{Q}}(\lambda_1 \cdots \lambda_N) = 
Nm_{F/ \mathbb{Q}}(c^N) = p^{N}
\]
It is easy to compute that the characteristic polynomial of the matrix  
\[
\begin{pmatrix} 
 &   &  & &  p^{N}  \\
1 &   &     \\
 & 1  &    \\
 &  & \ddots \\
& & &  1
\end{pmatrix}
\]
is 
$x^N- p^{N}$, hence the eigenvalues are of the form $p \zeta_N^i$ with $\zeta_N$ a primitive $N$-th root of unity. Therefore, they are Weil numbers with absolute value $p$. Since every block is of the above form, we see that the eigenvalues are all of absolute value $p$. If we base change everything to $\bar{\mathbb{F}}_p$, then the above computation computes the eigenvalues of the partial Frobenius $F_i$ on 
$
H^*(\mathscr{M}^*_{\bar{\mathbb{F}}_p} \setminus \mathscr{M}_{\bar{\mathbb{F}}_p}, \mathcal{F}V)
$, 
which we see are all of absolute value $p$. 
Then from ($l$-adic realization of) equation (\ref{www}), we have that  
$w_{\geq k} w_{\leq k} R^nj_* \overline{\mathbb{Q}_l} $
is a sum of $i_*\mathcal{F}V$ if $k >0$,
hence the partial Frobenius acts on
\[
H^*(\mathscr{M}^*_{\bar{\mathbb{F}}_p},  w_{\geq k} w_{\leq k} Rj_* \overline{\mathbb{Q}_l})
\]
with eigenvalues of absolute value $p$, if $k>0$. 

To summarize, we have proven the following proposition.

\begin{proposition}
The partial Frobenius $F_i$ acts on the spectral sequence  (\ref{s.s.h}) by proposition \ref{func.s.s}. More precisely, by proposition \ref{func.s.s},  $F_i$ acts on the special fiber variant of  ($l$-adic realization of) the spectral sequence (\ref{s.s.h})
\[
E_1^{p,q}= H^{p+q}(\mathscr{M}^*_{\bar{\mathbb{F}}_p}, w_{\geq -p} w_{\leq -p} Rj_*\overline{\mathbb{Q}_l}) 
\Rightarrow
H^{p+q}(\mathscr{M}_{\bar{\mathbb{F}}_p}, \overline{\mathbb{Q}_l})
\]
which is (at least up to convergence) isomorphic to the Hodge module realization of (\ref{s.s.h}) by theorem \ref{comparison thm} and choice of $p$.
If $p <0$, $F_i$ acts on $E_1^{p,q}$ with eigenvalues of absolute value $p$, hence of partial Frobenius weights $(2, \cdots, 2)$.  

On the other hand, the Hodge module realization of (\ref{s.s.h}) have 
$^H E_1^{p,q} = H^{p+q}(\mathscr{M}^*(\mathbb{C}),w_{\geq -p} w_{\leq -p} Rj_*\mathbb{C}) $,
which is a sum of  
$H^{p+q}(\mathscr{M}^*(\mathbb{C}) \setminus \mathscr{M}(\mathbb{C}), \mathcal{F}V)$
if $p <0$, 
hence of plectic Hodge type $(1, \cdots, 1 ; 1, \cdots, 1)$ (sum of $\mathbb{C}(-1)^{\otimes d}$, the $(-d)$-th power of Tate structure). These are of plectic weight $(2, \cdots, 2)$, and the above computation shows that under the comparison, the partial Frobenius weights is the same as the plectic Hodge weights.  
\end{proposition}

\begin{remark}
It is possible to avoid the comparison theorem \ref{comparison thm} in the special case of Hilbert modular varieties. We have observed that the spectral sequence (\ref{s.s.h}) induces (shifts of) the weight filtration on the open cohomology. Therefore the comparison automatically holds. To spell this out, we note that the identification with the weight filtration gives a motivic meaning of the filtration induced by (\ref{s.s.h}), namely, we can find a smooth projective compactification with smooth normal crossing boundary divisors, and the filtration can be expressed in terms of the cohomology of the natural strata. Then the comparison is reduced to the standard comparison between different cohomology theories. 

Note that in general the filtration induced by the spectral sequence in theoerem \ref{w.s.s.} is not the weight filtration. However, in some sense, it detects the non-trivial extensions of the weight filtration. 
\end{remark}

We have computed the partial Frobenius on 
$E_1^{p,q}$ for $p<0$, and checked the partial Frobenius weights is the same as the plectic Hodge weights. It remains to do the same for the remaining 
$E_1^{0,q}= IH^q(\mathscr{M}^*_{\bar{\mathbb{F}}_p}, \overline{\mathbb{Q}_l})$. 

We note that the Hecke algebra decomposes the cohomology into 
\[
IH^*(\mathscr{M}^*_{\bar{\mathbb{F}}_p}, \overline{\mathbb{Q}_l}) = 
IH^*(\mathscr{M}^*_{\bar{\mathbb{F}}_p}, \overline{\mathbb{Q}_l})_{\text{cusp}} \oplus 
IH^*(\mathscr{M}^*_{\bar{\mathbb{F}}_p}, \overline{\mathbb{Q}_l})_{\text{rest}}
\]
where 
$IH^*(\mathscr{M}^*_{\bar{\mathbb{F}}_p}, \overline{\mathbb{Q}_l})_{\text{cusp}}$
is the subspace on which the Hecke algebra acts with the same type as some cuspidal automorphic representations. Similarly,    
$IH^*(\mathscr{M}^*_{\bar{\mathbb{F}}_p}, \overline{\mathbb{Q}_l})_{\text{rest}}$
is the subspace on which the Hecke algebra acts as a discrete but non-cuspidal automorphic representation.

Note that the corresponding representation is cohomological and we can classify them. The cuspidal part corresponds to holomorphic Hilbert modualr forms of weight $(2, \cdots, 2)$, and the discrete non-cuspidal part corresponds to one-dimensional representations. 

We first compute the cuspidal part. We have
\[
IH^*(\mathscr{M}^*_{\bar{\mathbb{F}}_p}, \overline{\mathbb{Q}_l})_{\text{cusp}} 
= \underset{f}{\oplus} IH^*(\mathscr{M}^*_{\bar{\mathbb{F}}_p}, \overline{\mathbb{Q}_l})_f
\]
where $f$ ranges over  holomorphic Hilbert modualr forms of weight $(2, \cdots, 2)$, see \cite{MR1050763} chapter 3 for example. It is well-known from the $(g, K)$-cohomology computations that 
$IH^*(\mathscr{M}^*_{\bar{\mathbb{F}}_p}, \overline{\mathbb{Q}_l})_f$
is concentrated in degree $d$, and (its complex variant) has plectic Hodge type $((1,0)\oplus (0, 1))^{\otimes d}$, hence of plectic weight $(1, \cdots, 1)$. We want to check that the partial Frobenius weights are again of the same weight, namely, the eigenvalues of the partial Frobenius $F_i$ on 
$IH^*(\mathscr{M}^*_{\bar{\mathbb{F}}_p}, \overline{\mathbb{Q}_l})_f$
have absolute value $p^{\frac{1}{2}}$. 

Recall that Nekov\'a\v r have proved in \cite{JanNekovar2018} that the partial Frobenius satisfies an Eichler-Shimura relation. In the Hilbert modular case, it is 
\[
F_i^2 - (T_i/S_i) F_i + p/S_i =0
\]
where $T_i, S_i$ are standard Hecke operators of the Hecke algebra of $Res_{F/ \mathbb{Q}} GL_2 $ at 
$\mathbb{Q}_p$, i.e.
\[
T_i, S_i \in 
H(Res_{F/ \mathbb{Q}} GL_2(\mathbb{Q}_p) // Res_{F/ \mathbb{Q}} GL_2(\mathbb{Z}_p), \mathbb{Z}) = \otimes_i H( GL_2(\mathbb{Q}_p) // GL_2(\mathbb{Z}_p), \mathbb{Z}) 
\]
indexed by $\{\mathfrak{p}_i\}$, see \cite{JanNekovar2018} A6. The upshot is that this shows that the eigenvalues of the partial Frobenius $F_i$ on 
$IH^*(\mathscr{M}^*_{\bar{\mathbb{F}}_p}, \overline{\mathbb{Q}_l})_f$
is the same as the eigenvalues of the  (geometric) Frobenius $Frob_{\mathfrak{p}_i}$ on the representation $\rho_f^{\vee} (-1)$, where 
$\rho_f: Gal(\bar{\mathbb{Q}}/ F) \rightarrow GL_2(\overline{\mathbb{Q}_l})$
is the Galois representation associated to the Hilbert modular form $f$. We know from \cite{blasius2006hilbert} that the Galois representation $\rho_f$ is pure of weight 1,  so is  
$\rho_f^{\vee} (-1)$, proving that the eigenvalues of the partial Frobenius $F_i$ on 
$IH^*(\mathscr{M}^*_{\bar{\mathbb{F}}_p}, \overline{\mathbb{Q}_l})_f$
have absolute value $p^{\frac{1}{2}}$. 

Finally, we deal with 
$IH^*(\mathscr{M}^*_{\bar{\mathbb{F}}_p}, \overline{\mathbb{Q}_l})_{\text{rest}}$. It is known that it is concentrated in even degrees, and \[
IH^{2k}(\mathscr{M}^*_{\bar{\mathbb{F}}_p}, \overline{\mathbb{Q}_l})_{\text{rest}} 
= \bigwedge^k (IH^0(\mathscr{M}^*_{\bar{\mathbb{F}}_p}, \overline{\mathbb{Q}_l})_{\text{rest}} 
\oplus 
IH^2(\mathscr{M}^*_{\bar{\mathbb{F}}_p}, \overline{\mathbb{Q}_l})_{\text{rest}})
\]
The same holds for the complex variant (\cite{MR1050763} chapter 3), thus it is enough to concentrate on 
$IH^2(\mathscr{M}^*_{\bar{\mathbb{F}}_p}, \overline{\mathbb{Q}_l})_{\text{rest}}$. If we look at a connected component
$\mathscr{M}_{\bar{\mathbb{F}}_p,o }^{*}$
of 
$\mathscr{M}^*_{\bar{\mathbb{F}}_p}$, 
we have 
\[
IH^2(\mathscr{M}^{*}_{\bar{\mathbb{F}}_p, o}, \overline{\mathbb{Q}_l})_{\text{rest}}
= \underset{i}{\oplus} \overline{\mathbb{Q}_l} \cdot c_1(L_i) (-1)
\]
where $L_i$ is a line bundle on 
$\mathscr{M}_{\bar{\mathbb{F}}_p, o}$
to be defined below, and the equality is interpreted as  
$c_1(L_i)(-1) \in H^2(\mathscr{M}_{\bar{\mathbb{F}}_p, o},\overline{\mathbb{Q}_l})$ 
lying in the image of the natural embedding
$IH^2(\mathscr{M}^{*}_{\bar{\mathbb{F}}_p, o}, \overline{\mathbb{Q}_l})_{\text{rest}} \hookrightarrow 
H^2(\mathscr{M}_{\bar{\mathbb{F}}_p, o},\overline{\mathbb{Q}_l}) $. Let 
$p: \mathcal{A} \rightarrow \mathscr{M}_{\mathbb{F}_p,o}$  
be the universal abelian scheme over 
$\mathscr{M}_{\mathbb{F}_p,o}$, 
then 
$Lie^{\vee}_{\mathcal{A}/\mathscr{M}_{\mathbb{F}_p,o}}$
is a coherent sheaf of projective 
$\mathcal{O}_F \otimes_{\mathbb{Z}} \mathbb{F}_p$-module with rank 1. By the choice of $p$, we have 
$\mathcal{O}_F \otimes_{\mathbb{Z}} \mathbb{F}_p = \underset{i}{\prod} \mathbb{F}_p$ 
parametrized by $\{\mathfrak{p}_i\}$, hence 
\[
Lie^{\vee}_{\mathcal{A}/\mathscr{M}_{\mathbb{F}_p,o}} = \underset{i}{\oplus}L_i
\]
where 
$L_i := e_i Lie^{\vee}_{\mathcal{A}/\mathscr{M}_{\mathbb{F}_p,o}}$ 
and $e_i$ is the idempotent of
$\underset{i}{\prod} \mathbb{F}_p$
corresponding to the $i$-th factor. 
Another way to characterize $L_i$ is to note that 
\[
Lie_{\mathcal{A}/\mathscr{M}_{\mathbb{F}_p,o}}^{\vee} = Lie_{\mathcal{A}[p]/\mathscr{M}_{\mathbb{F}_p,o}}^{\vee}
=\underset{i}{\oplus} Lie_{\mathcal{A}[\mathfrak{p}_i]/\mathscr{M}_{\mathbb{F}_p,o}}^{\vee}
\]
and 
\[
L_i=Lie_{\mathcal{A}[\mathfrak{p}_i]/\mathscr{M}_{\mathbb{F}_p,o}}^{\vee}
\]
Now by definition of the partial Frobenius $F_i$, we have a Cartesian diagram 
\[
\begin{tikzcd}
\mathcal{A}^{(\mathfrak{p}_i)} : \arrow[equal]{r} &
\mathcal{A}/(Ker(F)[\mathfrak{p}_i]) \arrow[r,""] \arrow[d,""] &
\mathcal{A} \arrow[d, "p"]
\\
&
\mathscr{M}_{\mathbb{F}_p,o } \arrow[r, "F_i"] &
\mathscr{M}_{\mathbb{F}_p,o' } 
\end{tikzcd}
\]
with a possibly different connected component 
$\mathscr{M}_{\mathbb{F}_p,o' }$. By abuse of the notation, we use the same $\mathcal{A}$ to denote the universal abelian scheme on 
$\mathscr{M}_{\mathbb{F}_p,o' }$,
and similarly for $L_i$. 
The diagram tells us that 
\[
F_i^*(L_i) = e_j Lie_{\mathcal{A}^{(\mathfrak{p}_i)} /\mathscr{M}_{\mathbb{F}_p,o}} 
=Lie_{\mathcal{A}^{(\mathfrak{p}_i)}[\mathfrak{p}_j] /\mathscr{M}_{\mathbb{F}_p,o}}
\]
If $j \neq i$, then clearly 
$\mathcal{A}^{(\mathfrak{p}_i)}[\mathfrak{p}_j] = \mathcal{A}[\mathfrak{p}_j]$, hence 
\[
F_i^*(L_j) = Lie_{\mathcal{A}[\mathfrak{p}_j] /\mathscr{M}_{\mathbb{F}_p,o}} = L_j
\]
If $j=i$, then 
$\mathcal{A}^{(\mathfrak{p}_i)}[\mathfrak{p}_i] = \mathcal{A}^{(p)}[\mathfrak{p}_i] $, 
where 
$\mathcal{A}^{(p)} := \mathcal{A}/Ker(F)$ 
as usual, hence 
\[
F_i^*(L_i) = Lie_{\mathcal{A}^{(p)}[\mathfrak{p}_i] /\mathscr{M}_{\mathbb{F}_p,o}} = 
e_i Lie_{\mathcal{A}^{(p)} /\mathscr{M}_{\mathbb{F}_p,o}} =
e_i Frob^* Lie_{\mathcal{A} /\mathscr{M}_{\mathbb{F}_p,o}}
\]
\[
= e_i Frob^*(\underset{j}{\oplus} L_j)
= e_i (\underset{j}{\oplus} Frob^*L_j) 
= e_i (\underset{j}{\oplus} L_j^{ \otimes p})
= L_i^{ \otimes p}
\]
where 
$Frob:\mathscr{M}_{\mathbb{F}_p,o} \rightarrow \mathscr{M}_{\mathbb{F}_p,o} $
is the absolute Frobenius, and we use that 
$Frob^* L \cong L^{\otimes p}$
for any line bundle $L$ (by looking at the transition function of $L$). 

Now we have proved that 
\begin{equation*}
F_i^*(c_1(L_j)) = c_1(F_i^*L_j)= 
\begin{cases}
c_1(L_j)  & j \neq i, \\
pc_1(L_i) & j=i 
\end{cases}
\end{equation*}
Taking into the subtlety of the connected components, we see that 
\[
IH^2(\mathscr{M}^*_{\bar{\mathbb{F}}_p}, \overline{\mathbb{Q}_l})_{\text{rest}}
= \underset{i}{\oplus} W_i
\]
where $W_i$ is the subspace generated by $c_1(L_i)(-1)$ on each connected component. Then with the modification introduced by base changing to algebraic closure and the Tate twist, $F_i$ acts on  $W_i$ with blocks of the form 
\[
\begin{pmatrix} 
 &   &  & &  p  \\
p &   &     \\
 & p &    \\
 &  & \ddots \\
& & &  p
\end{pmatrix}
\]
hence have eigenvalues $p\zeta$ for $\zeta$ some roots of unity. If $j \neq i$, $F_i$ acts on $W_j$ with blocks of the form 
\[
\begin{pmatrix} 
 &   &  & &  1  \\
1 &   &     \\
 & 1  &    \\
 &  & \ddots \\
& & &  1
\end{pmatrix}
\]
which have eigenvalues roots of unity. This proves that $W_i$ have partial Frobenius weights $(0, \cdots, 0, 2, 0, \cdots 0)$
with 2 at the $i$-th position.

On the other hand, the same process gives line bundles $L_i$ on $\mathscr{M}_{\mathbb{C}}$, where we use that 
$Lie^{\vee}_{\mathcal{A} / \mathscr{M}_{\mathbb{C}}}$ 
is a sheaf of projective 
$\mathcal{O}_F \otimes_{\mathbb{Z}} \mathbb{C} = \underset{i}{\prod} \mathbb{C}$-modules, which are indexed by archimedean places of $F$. The $L_i$ can be further characterized by its transition functions, i.e. its sections corresponds to holomorphic Hilbert modular forms of weight $(0, \cdots, 0, 2, 0, \cdots 0)$ 
with 2 at the $i$-th position. In the comparison between Betti cohomology and $l$-adic cohomology of the special fiber at $p$, we implicitly fix an isomorphism 
$\overline{\mathbb{Q}_p} \cong \mathbb{C}$, which induces an identification between archimedean places  and $p$-adic places of $F$. Thus we can compare the $L_i$ in two different cases, the corresponding 
$W_i \subset IH^2(\mathscr{M}^*(\mathbb{C}), \mathbb{C})$ 
generated by $c_1(L_i)$
is easily seen to be of plectic Hodge type 
\[
(0, \cdots, 0, 1, 0, \cdots 0 ; 0, \cdots, 0, 1, 0, \cdots 0)
\]
with both $1$s in the $i$-th position ($c_1(L_i)$ is represented by 
$dz_i \wedge d\bar{z}_i$ 
with $(z_k)_k \in \mathbb{H}^d$).
Thus $W_i$ have plectic weight 
$(0, \cdots, 0, 2, 0, \cdots 0)$
with 2 at the $i$-th position, which is compatible with the partial Frobenius weights. 

To summarize, we have proved the following theorem.

\begin{theorem}
 The partial Frobenius  $F_i$ acts on the special fiber variant of  ($l$-adic realization of) the spectral sequence (\ref{s.s.h})
\[
E_1^{p,q}= H^{p+q}(\mathscr{M}^*_{\bar{\mathbb{F}}_p}, w_{\geq -p} w_{\leq -p} Rj_*\overline{\mathbb{Q}_l}) 
\Rightarrow
H^{p+q}(\mathscr{M}_{\bar{\mathbb{F}}_p}, \overline{\mathbb{Q}_l})
\]
by proposition \ref{func.s.s}, which is (at least up to convergence) isomorphic to the Hodge module realization of (\ref{s.s.h}) by theorem \ref{comparison thm} and choice of $p$. The Hodge module spectral sequence exhibits plectic Hodge structures on the graded pieces of the filtration induecd by (\ref{s.s.h}) through $(\mathfrak{g},K)$-cohomology, and the partial Frobenius weights are compatible with the exhibited plectic Hodge weights on each graded pieces. 

\end{theorem}

\begin{corollary} (Plectic weight filtration)
There is a natural increasing $\mathbb{Z}^d$-filtration $W_{\underline{a}}$ (defined over $\mathbb{C}$) on 
$H^*(\mathscr{M}(\mathbb{C}), \mathbb{C})$
with $\underline{a}=(a_1, \cdots, a_d) \in \mathbb{Z}^d$, defined by 
\[
W_{\underline{a}}= \underset{\begin{subarray}{c}
  |\beta_i|= p^{\frac{k_i}{2}} \\
  k_i \leq a_i
  \end{subarray}}{\bigoplus}
  V_{(\beta_1, \cdots, \beta_d)}
\]
where $V_{(\beta_1, \cdots, \beta_d)}$ is the generalized eigenspace of $F_i$ with eigenvalue $\beta_i$ for all $i$. The action of $F_i$ on 
$H^*(\mathscr{M}(\mathbb{C}), \mathbb{C})$
is through the natural comparison isomorphism 
$H^*(\mathscr{M}(\mathbb{C}), \mathbb{C}) \cong 
\imath_*H^*(\mathscr{M}_{\bar{\mathbb{F}}_p}, \overline{\mathbb{Q}_l})$ 
for some fixed isomorphism $\imath: \overline{\mathbb{Q}_l} \cong \mathbb{C}$. 

The filtration is plectic in the sense that there is a natural plectic Hodge structure on 
$Gr_{\underline{a}}^W$
with plectic weight $\underline{a}$.
\end{corollary}

\begin{remark}
We have seen that the graded pieces of the constructed plectic weight filtration are motivic, so are independent of the choice of $p$. However, the filtration might still depend on $p$ a priori. We leave the proof of independence of $p$ to future work.
\end{remark}

\section{Toroidal compactifications and the partial Frobenius}

\subsection{Polarized degeneration data}
We begin by recalling the degeneration data of abelian schemes introduced by Faltings-Chai and refined by Kaiwen Lan. It is (almost) a collection of linear algebra objects that characterizes the degeneration of abelian varieties into semi-abelian varieties. It is relatively straightforward to find the parametrization space of the degeneration data, which constitutes the base of a universal degenerating abelian scheme. These are used to glue with the PEL Shimura varieties to form toroidal compactifications. We follow the notations of \cite{lan2013arithmetic} closely, see also \cite{MR3221715} for a minimal summary of definitions.

Let $R$ be a Noetherian normal domain complete with respect to an ideal $I$, with $\sqrt{I} =I $. Let $S := \text{Spec}(R)$, $K:= \text{Frac}(R)$, $\eta := \text{Spec}(K)$ and $S_{\text{for}} := \text{Spf}(R,I)$.

\subsubsection{Definitions and the theorem}

\begin{definition}
The category
$\text{DEG}_{\text{pol}}(R,I)$
has objects $(G,\lambda_{\eta})$,
where 

(1)
$G$ is an semi-abelian scheme over $S$, i.e. a commutative group scheme over $S$ with geometric fibers extensions of abelian varieties by torus, such that the generic fiber $G_{\eta}$ is an abelian variety, and such that $G_0 := G \times_S S_0$ is globally an extension 
\[
0 \longrightarrow T_0 \longrightarrow G_0 \longrightarrow A_0 \longrightarrow 0 
\]
where $T_0$ is an isotrivial torus over $S_0$, i.e. $T_0$ becomes split over a finite \'etale cover of $S_0$, and $A_0$ is an abelian scheme over $S_0$. 

(2) 
$\lambda_{\eta}: G_{\eta} \rightarrow G^{\vee}_{\eta}$
is a polarization of $G_{\eta}$.

\vspace{3mm}

The morphisms in the category are isomorphisms of group schemes over $S$ which respect the polarizations on the generic fibers. 
\end{definition}

Elements of $\text{DEG}_{\text{pol}}(R,I)$
are called degenerating abelian schemes. We will see that they are equivalent to certain datum that is more linear algebraic in nature, called degeneration data, to be defined as follows.

\begin{definition}
The category of degeneration data $\text{DD}_{\text{pol}}(R,I)$ 
has objects 
\[
(A, \lambda_A, \underline{X}, \underline{Y}, \phi, c,c^{\vee}, \tau)
\]
where

(1) $A$ is an abelian scheme over $S$, and 
$\lambda_A : A \rightarrow A^{\vee}$ 
is a polarization.

(2) $\underline{X}$ and $\underline{Y}$ are \'etale sheaves of free commutative groups of the same rank, which can be viewed as \'etale group schemes over $S$, and 
$\phi : \underline{Y} \hookrightarrow \underline{X}$
is an injective homomorphism with finite cokernel.

(3) $c$ and $c^{\vee}$ are homomorphisms 
\[
c: \underline{X} \longrightarrow A^{\vee}
\]
\[
c^{\vee} : \underline{Y} \longrightarrow A
\]
such that 
\[
\lambda_A \circ c^{\vee} = c \circ \phi
\]

(4) $\tau$ is a trivialization 
\[
\tau : \textbf{1}_{\underline{Y} \times_S \underline{X}, \eta} \overset{\sim}{\longrightarrow} 
(c^{\vee} \times c)^* \mathcal{P}_{A,\eta}^{\otimes -1}
\]
of the biextension 
$(c^{\vee} \times c)^* \mathcal{P}_{A,\eta}^{\otimes -1}$
over the \'etale group scheme 
$(\underline{Y} \times_S \underline{X})_{ \eta}$
such that 
$(Id_{\underline{Y}} \times \phi)^* \tau$
is symmetric,
where $\mathcal{P}_A$ is the Poincare line bundle on
$A \times_S A^{\vee}$, 
and  
$\textbf{1}_{\underline{Y} \times_S \underline{X}}$ 
is the structure sheaf of
$\underline{Y} \times_S \underline{X}$. 
See \cite{lan2013arithmetic} 3.2.1.1 for the precise definition of biextension,  $\tau$ being a trivialization of biextensions essentially means that $\tau $ is bilinear in a (the only) reasonable sense, and symmetric means the bilinear form is symmetric. 

\vspace{3mm}
Moreover, $\tau$ is required to satisfy a positivity condition as follows. Taking a finite \'etale base change of $S$ if necessary, we assume that $\underline{X}$ and $\underline{Y}$ are constant with values $X$ and $Y$.   For each $y \in Y$, the isomorphism
\[
\tau(y, \phi(y)): \mathcal{O}_{S,\eta} \overset{\sim}{\longrightarrow}
(c^{\vee}(y) \times c \circ \phi(y))^* \mathcal{P}_{A,\eta}^{\otimes -1}
\]
over the generic fiber extends to a section 
\[
\tau(y, \phi(y)):
\mathcal{O}_{S} \overset{}{\longrightarrow}
(c^{\vee}(y) \times c \circ \phi(y))^* \mathcal{P}_{A}^{\otimes -1}
\]
over $S$, which we still denote by $\tau(y, \phi(y))$. Moreover, for each $y \neq 0$, the induced morphism 
\[
(c^{\vee}(y) \times c \circ \phi(y))^* \mathcal{P}_{A}
\longrightarrow \mathcal{O}_{S} 
\]
factors through $\underline{I}$, where $\underline{I}$ is the subsheaf of $\mathcal{O}_S$ corresponding to the ideal $I \subset R$.

The morphims in the category are defined to be isomorphisms (of $A$, $\underline{X}$ and $\underline{Y}$) over $S$ respecting all the structures. 

\end{definition}

Now we can state the first key result. 

\begin{theorem}
(Faltings-Chai) 
There is a functor 
\[
F_{\text{pol}}(R,I): \text{DEG}_{\text{pol}}(R,I) \longrightarrow \text{DD}_{\text{pol}}(R,I)
\]
which induces an equivalence of categories. 
\end{theorem}

\begin{remark} \label{mumford}
The inverse functor
$\text{DD}_{\text{pol}}(R,I) \longrightarrow \text{DEG}_{\text{pol}}(R,I)$
is called the Mumford quotient construction. We will not describe that in detail.
\end{remark}

\begin{remark}
There are a few variants of the categories $\text{DEG}_{\text{pol}}(R,I)$ and 
$\text{DD}_{\text{pol}}(R,I)$. For example, we can forget about the polarization $\lambda_{\eta}$, or we can rigidify the situation by replacing the polarization by an ample line bundle. The equivalence of categories as in the theorem extends to these variants. This explains why we include the lower index pol in the notations.
\end{remark}

\subsubsection{Motivations}

Now we explain the construction of $F_{\text{pol}}$. Essentially, it is to associate linear algebra data to degenerating abelian varieties that also characterizes it, and a basic model for this kind of construction is to write a complex abelian variety as $\mathbb{C}^n/\Gamma$. However, this is a highly transcendental construction, and it is not obvious how to proceed in our algebraic setting.   

The basic idea is to use the Fourier coefficients of  theta functions to detect the periods of abelian varieties. More precisely, recall that an abelian variety $A$ over $\mathbb{C}$ has the universal covering $\mathbb{C}^n$, and it can be written as $A = \mathbb{C}^n/ \Gamma$ for some period lattice $\Gamma \subset \mathbb{C}^n$. 
A choice of an ample line bundle $L$ on $A$ gives a positive definite Hermitian form on $\mathbb{C}^n$ whose imaginary part $E$ takes integer values on $\Gamma$, and a map 
$\alpha : \Gamma \rightarrow \mathbb{C}^{\times}$ such that 
$\alpha(x +y) = \alpha (x) \alpha(y) \exp{(\pi i E(x,y))}$
. Then the theta functions are sections of $L$, and an element 
$s \in \Gamma(A,L)$
is equivalent to a holomorphic function 
$f: \mathbb{C}^n \rightarrow \mathbb{C}$ such that 
\begin{equation} \label{theta}
f(z+ \gamma)= f(z) \alpha(\gamma) \exp{(\frac{1}{2}\pi  H(\gamma,\gamma) + \pi  H(\gamma,z))} 
\end{equation}
for 
$z \in \mathbb{C}^n$, $\gamma \in \Gamma$. 

Now we can find a rank $n$ sub-lattice $U \subset \Gamma$ isotropic with respect to $E$, such that 
\[
f(z)=  \exp{( l(z) + B(z,z))} \underset{\chi \in Hom(U,\mathbb{Z})}{\sum} c_{\chi} \exp{(2\pi i \chi(z))} 
\]
for some linear form $l : \mathbb{C}^n \rightarrow \mathbb{C}$, and complex bilinear form 
$B: \mathbb{C}^n \times \mathbb{C}^n \rightarrow \mathbb{C}$ which depends only on $L$, and is independent of the section $s$, see the first chapter of \cite{MR0282985} for details. Hence $f$ is essentially a function on 
$\mathbb{C}^n / U \overset{\exp}{\cong} \mathbb{C}^{\times,n}$. 
Note that
$Hom(U,\mathbb{Z})$
can be identified with the character group
$X :=X(\mathbb{C}^{\times,n})$ 
of the algebraic torus 
$\mathbb{C}^{\times,n}$,
and if we write
$q := \exp{(2\pi i z)} \in \mathbb{C}^{n,\times}$,
then 
$\exp{(2\pi i \chi(z))} = \chi(q)$
under the above identification.
Now the essential part of $f$ has a Fourier expansion 
\[
\underset{\chi \in X}{\sum} c_{\chi} \chi(q)
\]
for
$q \in \mathbb{C}^{\times,n}$. 
This expression has a potential to be algebraic. The universal cover $\mathbb{C}^n$ of $A$ is very transcendental, but it seems that the intermediate quotient 
$\mathbb{C}^n / U\cong \mathbb{C}^{\times,n}$
subsumes all the transcendental part through the exponential map, and the factorization 
$\mathbb{C}^{\times, n} \rightarrow A$
is "algebraic" in nature. Moreover, since the theta functions define a projective embedding of A (assume that $L$ is very ample), they determine $A$ completely, and in particular the multiplicative periods
$Y:= \Gamma/U \subset \mathbb{C}^{\times,n}$. Further, the theta functions are determined by the Fourier coefficients $c_{\chi}$, hence in principle we can read off the multiplicative periods from $c_{\chi}$. 

We can make more explicit the procedure to detect the multiplicative periods from $c_{\chi}$. Note that the functional equation (\ref{theta}) gives the relation (for $\gamma \in \Gamma$)
\[
 c_{\chi} = \alpha(\gamma) \cdot \exp{(-l(\gamma))}  \cdot \exp{(-2 \pi i \chi(\gamma))}
 \cdot
 c_{\chi + \phi(\gamma)} 
\]
where 
$\phi :Y \rightarrow X $
is the homomorphism determined by the polarization $E$, namely for $y \in \Gamma$ and $x \in U$, 
$E(y,x) = \phi(y)(x)$
under the identification 
$ X \cong Hom(U, \mathbb{Z})$
($U$ is isotropic with respect to $E$, so it descends to a map on $Y =\Gamma /U$). 
Rewriting it in our new multiplicative notation, we have
\[
 c_{\chi + \phi(\gamma)} = \chi( \imath (\gamma)) a(\gamma) c_{\chi}
\]
where
$\imath : Y \hookrightarrow \mathbb{C}^{n,\times}$
is the inclusion of the multiplicative periods, and 
$a : Y \rightarrow \mathbb{C}^{\times} $
is a function depending on the line bundle $L$. The desired multiplicative periods are then manifested through the ratio between $c_{\chi}$ and $c_{\chi + \phi(y)}$. 

Further, we note that we can give a more direct characterization of the multiplicative periods $Y$, which is useful when we algebraize the above procedure. Recall that $Hom(Y, \mathbb{Z})$ is canonically the multiplicative periods of the dual abelian variety $A^{\vee}$, so $Y$ is naturally the character group of the multiplicative periods of $A^{\vee}$, which is identified with the character group of the associated algebraic torus of $A^{\vee}$. 
\vspace{3mm}

To summarize, for 
$\chi \in X$,
the linear maps 
\[
c_{\chi} : \Gamma(A, L) \rightarrow \mathbb{C}
\]
defined by the Fourier coefficients detect the multiplicative periods
$Y \subset \mathbb{C}^{\times,n}$
of $A$, where $Y$ can be naturally identified with the character group of the algebraic torus associated to the dual abelian variety $A^{\vee}$. More explicitly, 
the relations
\[
 c_{\chi + \phi(\gamma)} = b(\gamma, \chi) a(\gamma) c_{\chi}
\]
characterize a bilinear pairing 
\[
b(\cdot,\cdot) : Y \times X \rightarrow \mathbb{C}^{\times}
\]
such that 
$b( \cdot, \phi(\cdot))$
is symmetric, 
and  the multiplicative periods 
$\imath : Y \hookrightarrow \mathbb{C}^{n,\times}$
is determined by the pairing through
$b(\chi, \gamma) = \chi(\imath (\gamma)) $. 
This is the principle that we aim to algebraize and considerably generalize.  

\subsubsection{Equivalent formulation of polarized degeneration data} \label{equivalentfor}

Before giving the detailed construction of $F_{\text{pol}}(R,I)$, 
we first explain the meaning of the tuple in the degeneration data. 

First, the \'etale sheaf $\underline{X}$ and $\underline{Y}$ can be viewed as the character groups of torus $T$ and $T^{\vee}$ over $S$, and the homomorphisms $c$ and $c^{\vee}$ are equivalent to extensions 
\[
0 \longrightarrow T \longrightarrow G^{\natural} \longrightarrow A \longrightarrow 0 
\]
\[
0 \longrightarrow T^{\vee} \longrightarrow G^{\vee,\natural} \longrightarrow A^{\vee} \longrightarrow 0 
\]
of commutative group schemes over $S$. Passing to a finite \'etale cover of $S$ if necessary, we can assume that $T$ is split, hence 
$\underline{X}$ is constant with value $X$. 
We view $G^{\natural}$ as a $T$-torsor over $A$, then as $G^{\natural}$ is relative affine over $A$, we have
\begin{equation} \label{G}
G^{\natural} \cong 
\underline{\text{Spec}}_{\mathscr{O}_A}(\mathscr{O}_{G^{\natural}}) \cong
\underline{\text{Spec}}_{\mathscr{O}_A}(\underset{\chi \in X}{\oplus} \mathscr{O}_{\chi})
\end{equation}
where $\mathscr{O}_{\chi} := c(\chi) \in Pic^0(A/S)$ is the eigensheaf of 
$\mathscr{O}_{G^{\natural}}$
with weight $\chi$ under the action of $T$. Equivalently, $\mathscr{O}_{\chi}$ is the $\textbf{G}_m$-torsor (viewed as a line bundle) 
$G^{\natural} \times^{T, \chi} \textbf{G}_m$, 
i.e. the pushout of 
$
0 \rightarrow T \rightarrow G^{\natural} \rightarrow A \rightarrow 0 
$
along $\chi: T \rightarrow \textbf{G}_m$. 
This explains the identification of $c$ with the first extension, and similar for $c^{\vee}$. 

Note that $c$ being a group homomorphism equips 
$\underset{\chi \in X}{\oplus} \mathscr{O}_{\chi}$
with an $\mathscr{O}_A$-algebra structure. Further, the $T$-torsor $G^{\natural}$ being a group scheme is equivalent to $c$ taking values in $Pic^0(A)$, which is a consequence of the characterizing property 
$m_A^* \mathscr{L} \cong pr_1^* \mathscr{L} \otimes pr_2^* \mathscr{L}$ 
of $\mathscr{L} \in Pic^0(A/S)$.

Next, the homomorphisms $\phi$ and $\lambda_A$ such that 
\[
\lambda_A \circ c^{\vee} = c \circ \phi
\]
are equivalent to a homomorphism 
\[
\lambda : G^{\natural} \longrightarrow G^{\natural, \vee}
\]
of group schemes over $S$ that induces a polarization $\lambda_A$ on $A$. Note that a homomorphism $\lambda$ induces a homomorphism of the extensions
\[
\begin{tikzcd}
0 \arrow[r] & T \arrow[r] \arrow[d, "\lambda_T"] & G^{\natural} \arrow[r] \arrow[d,"\lambda"] & A \arrow[r] \arrow[d,"\lambda_A"] & 0 
\\
0 \arrow[r] & T^{\vee} \arrow[r] & G^{\natural, \vee} \arrow[r] & A^{\vee} \arrow[r] & 0 
\end{tikzcd}
\]
since there is no non-trivial homomorphism from a torus to an abelian variety. Then $\lambda_A$ is the induced map on $A$, and $\phi $ is the map on characters induced by $\lambda_T$. The relation 
$\lambda_A \circ c^{\vee} = c \circ \phi$
is forced by (and equivalent to) the above commutative diagram.

Lastly and most importantly, the trivialization 
\[
\tau : \textbf{1}_{\underline{Y} \times_S \underline{X}, \eta} \overset{\sim}{\longrightarrow} 
(c^{\vee} \times c)^* \mathcal{P}_{A,\eta}^{\otimes -1}
\]
of the biextension 
$(c^{\vee} \times c)^* \mathcal{P}_{A,\eta}^{\otimes -1}$
is equivalent to a group homomorphism 
\[
\imath : \underline{Y}_{\eta} \longrightarrow G^{\natural}_{\eta}
\]
that lifts $c^{\vee}$ over the generic fiber, i.e. $c^{\vee}_{\eta}$ factorizes as 
\[
c^{\vee}_{\eta} :  \underline{Y}_{\eta} \overset{\imath}{\longrightarrow}  G^{\natural}_{\eta}
\longrightarrow A_{\eta}.
\]
Again, we can assume that both $\underline{X}$ and $\underline{Y}$ are constant with values $X$ and $Y$, and the general case is by \'etale descent. Then $\tau$ is a collection of sections 
$\{ \tau(y, \chi) \}_{y \in Y, \chi \in X}$ 
of the line bundles
$\mathcal{P}_A (c^{\vee}(y),c(\chi))_{\eta}^{\otimes -1}$ 
on the generic fiber of $S$ for each 
$y \in Y, \chi \in X$, 
satisfying bimultiplicative conditions from the biextension structures. Now 
\[
\mathcal{P}_A (c^{\vee}(y),c(\chi))_{\eta}^{\otimes -1} = 
(c^{\vee}(y)^* \circ (id_A \times c(\chi))^* \mathcal{P}_A ^{\otimes -1})_{\eta}
= (c^{\vee}(y)^* \mathscr{O}_{\chi}^{\otimes -1})_{\eta}
\]
by the definition of $\mathscr{O}_{\chi}$, hence 
$\tau(y, \chi) $ is a morphism
\[
\tau(y, \chi) : \mathscr{O}_{\chi}(c^{\vee}(y))_{\eta} \longrightarrow \mathscr{O}_{S, \eta}.
\]
Together with (\ref{G}), we have
\[
c^{\vee}(y)^* \mathscr{O}_{G^{\natural},\eta} = c^{\vee}(y)^* (\underset{\chi \in X}{\oplus} \mathscr{O}_{\chi}) 
\overset{\underset{\chi}{\sum} \tau (y, \chi)}{\longrightarrow}
\mathscr{O}_{S, \eta}
\]
which is a morphism of $\mathcal{O}_{S,\eta}$-algebras by the bimultiplicativity of $\tau$ (more precisely, being an algebra morphism is equivalent to the multiplicativity of the second variable of $\tau$). 
Since 
$G^{\natural} \cong 
\underline{\text{Spec}}_{\mathscr{O}_A}(\mathscr{O}_{G^{\natural}})$
is relative affine over $A$, the algebra morphism is the same as a morphism of 
$\imath(y): \eta \rightarrow G^{\natural}_{\eta}$ 
of schemes over $A$, i.e. 
\[
\begin{tikzcd}
\eta \arrow[rr, "\imath(y)"] \arrow[dr, "c^{\vee}(y)_{\eta}"] &  & G^{\natural}_{\eta} \arrow[dl] \\
 & A_{\eta} &
\end{tikzcd}
\]
Taking all $y \in Y$ together, we obtain the desired morphism 
\[
\imath: Y_{\eta} \rightarrow G^{\natural}_{\eta}
\]
of schemes over $A_{\eta}$. It can be shown that $\imath$ being a group scheme homomorphism is equivalent to the multiplicativity of the first variable of $\tau$. 

In summary, the degeneration data is essentially a commutative group scheme $G^{\natural}$ being an extension of an abelian scheme by a torus over $S$, a period morphism 
$\imath : Y_{\eta} \rightarrow G^{\natural}_{\eta}$
over the generic fiber, 
and some data specifying the polarization. We view $G^{\natural}$ as a "universal cover", and $\imath$ as the period lattice, parallel to the classical complex case 
$Y \subset \mathbb{C}^{\times, n}$. 
Recall that in the definition of degeneration data,  $\tau$ has to satisfy the symmetry and positivity condition, which after translated to the setting
$\imath : Y_{\eta} \rightarrow G^{\natural}_{\eta}$,
is the analogue of the positivity and anti-symmetry of the polarization form $E$ in the classical setting. 

\begin{remark}
In the classical complex setting, the existence of $E$ controls the position of the period lattice, and the positivity is the key (equivalent) to finding enough theta functions with respect to the period lattice to embed the quotient complex torus into a projective space. A similar role is played by the conditions on $\tau$. Indeed, given 
$\imath : Y_{\eta} \rightarrow G^{\natural}_{\eta}$
together with polarization data, 
to construct the quotient "$G^{\natural}/Y_{\eta}$" is a highly non-trivial procedure called Mumford construction as mentioned in remark \ref{mumford}. The positivity condition of $\tau$ is a key ingredient in the construction, the underlying reason seems still to be that the positivity ensures enough theta functions to define a projective embedding.  
\end{remark}

\subsubsection{The construction of $F_{\text{pol}}$}

Now we can explain the construction of $F_{\text{pol}}$ in the theorem. The first step is to functorially associate a “universal cover" of $G$ over $S$, and this will be called the Raynaud extension. 

We take the formal completion $G_{\text{for}}$ of $G$ along the ideal $I$, which is a formal scheme over $S_{\text{for}}:= \text{Spf}(R,I)$. Since the special fiber $G_0 := G \times_S S_0$ is an extension of an abelian scheme by a torus and torus can be uniquely lifted infinitesimally, we see that $G_{\text{for}}$ is an extension
\[
0 \rightarrow T_{\text{for}} \rightarrow G_{\text{for}} \rightarrow A_{\text{for}} \rightarrow 0
\]
where $T_{\text{for}}$ is a formal torus and $A_{\text{for}}$ is a formal abelian variety. There is an ample cubical (see \cite{lan2013arithmetic} 3.2.2.7 for definition) invertible sheaf on $G$ whose formal completion descends to an ample sheaf on $A_{\text{for}}$, then Grothendieck existence theorem implies that $A_{\text{for}}$ is algebrizable, i.e. $A_{\text{for}}$ is the formal completion of an abelian scheme $A$ over $S$. Note that the existence of an ample invertible sheaf on $G$ is a difficult theorem of Grothendieck, where the key ingredient is that the base $S$ is normal. Now we know that $T_{\text{for}}$ is also algebrizable, whose algebrization we denote by $T$. Then the morphism
$X(T_{\text{for}}) \rightarrow A_{\text{for}}^{\vee}$
corresponding to the extension $G_{\text{for}}$ also algebraizes to a unique morphism 
$X(T) \rightarrow A^{\vee}$
because the formal completion of proper schemes is a fully faithful functor. The morphism corresponds to the Raynaud extension 
\[
0 \rightarrow T \rightarrow G^{\natural} \rightarrow A \rightarrow 0
\]
over $S$. 

Now we look at the dual semi-abelian schemes. Since the generic fiber $G_{\eta}$ is an abelian variety, the dual abelian variety $G_{\eta}^{\vee}$ is well-defined, and the problem is whether we can extend it naturally to a semi-abelian scheme over $S$. The hard fact is that the closure in $G$ of the finite group scheme 
$Ker(\lambda_{\eta}) \subset G_{\eta}$ 
is a quasi-finite flat group scheme 
$\overline{Ker(\lambda_{\eta})}$
over $S$, 
and the quotient
$G/\overline{Ker(\lambda_{\eta})}$
is the desired extension of $G_{\eta}^{\vee}$, which we denote by $G^{\vee}$. The semi-abelian extension to $S$ is unique, so $G^{\vee}$ is uniquely defined. Moreover, the polarization 
$\lambda_{\eta}: G_{\eta} \rightarrow G^{\vee}_{\eta}$
extends to a homomorphism
\[
\lambda_S: G \longrightarrow G^{\vee}
\]
over $S$. 

We can apply the previous argument to $G^{\vee}$ and obtain the Raynaud extension
\[
0 \rightarrow T^{\vee} \rightarrow G^{\vee,\natural} \rightarrow A^{\vee} \rightarrow 0
\]
It can be shown that the abelian part of 
$G^{\vee, \natural}$ is naturally identified with the dual abelian variety of $A$, explaining the notation. 
The morphism $\lambda_S$ induces the morphisms
\[
\begin{tikzcd}
0 \arrow[r] & T \arrow[r] \arrow[d, "\lambda_T"] & G^{\natural} \arrow[r] \arrow[d,"\lambda"] & A \arrow[r] \arrow[d,"\lambda_A"] & 0 
\\
0 \arrow[r] & T^{\vee} \arrow[r] & G^{\natural, \vee} \arrow[r] & A^{\vee} \arrow[r] & 0 
\end{tikzcd}
\]
where we can show that $\lambda_A$ is a polarization. By what we have observed, these objects are equivalent to  
\[
(A, \lambda_A, \underline{X}, \underline{Y}, \phi, c,c^{\vee})
\]
in the degeneration data. Hence we have constructed the first seven objects in 
$F_{\text{pol}}((G, \lambda_{\eta}))$. 

\vspace{5mm}

It remains to construct $\tau$ out of $(G, \lambda_{\eta})$. We have seen that $\tau$ essentially corresponds to the periods 
$\imath: Y_{\eta} \hookrightarrow G^{\natural}_{\eta}$
in the "universal cover" $G^{\natural}_{\eta}$, and 
"$G = G^{\natural}/Y_{\eta}$" 
as in the classical case. In particular, as the intuition suggests, $\tau$ is determined by $G$ and is independent of the polarization or ample invertible sheaves used in the construction. Our strategy is to use theta functions to extract the periods, as explained in the motivation part. Indeed, $\tau$ is essentially the analog of the bilinear form $b(\cdot, \cdot)$ that appears in the functional equations of Fourier coefficients of theta functions in the complex case. 

Now we begin to construct $\tau$, following the strategy described in the classical case. Without loss of generality, we assume that $\underline{X}$ and $\underline{Y}$ are constant with values $X$ and $Y$. We choose a cubical ample invertible sheaf $\mathcal{L}$ on $G$, then we can show that its formal completion extends to a cubical ample line bundle $\mathcal{L}^{\natural}$ on $G^{\natural}$. We introduce notations for the maps in the extension by the diagram
\[
0 \rightarrow T \overset{i}{\rightarrow} G^{\natural} \overset{\pi}{\rightarrow} A \rightarrow 0
\]
then we can choose a cubical trivialization
$i^*\mathcal{L}^{\natural} \cong \mathscr{O}_T$, 
which forces $\mathcal{L}^{\natural}$ to descend to an ample invertible sheaf $\mathcal{M}$ on $A$, i.e. 
$\pi^*\mathcal{M} \cong \mathcal{L}^{\natural}$. Further, we assume that $\mathcal{L}_{\eta}$ induces the polarization $\lambda_{\eta}$ on $G_{\eta}$. We can achieve this by possibly replacing $\lambda_{\eta}$ with  $\lambda_{\mathcal{L}_{\eta}}$, the construction of $\tau$ will not depend on the choice of $\lambda_{\eta}$ or $\mathcal{L}$.  

We know that
$G^{\natural} \cong \underline{\text{Spec}}_{\mathscr{O}_A}(\underset{\chi \in X}{\oplus} \mathscr{O}_{\chi})$
as in (\ref{G}), which implies that 
\[
\pi_*\mathcal{L}^{\natural} \cong \underset{\chi \in X}{\oplus} \mathcal{M}_{\chi}
\] 
where 
$\mathcal{M}_{\chi} := \mathcal{M} \otimes_{\mathscr{O}_A} \mathscr{O}_{\chi}$. 
Then by the relative affineness of $G^{\natural}$, 
\[
\Gamma (G^{\natural}, \mathcal{L}^{\natural}) =
\Gamma (A, \pi_* \mathcal{L}^{\natural})=
\underset{\chi \in X}{\oplus} \Gamma (A, \mathcal{M}_{\chi})
\]
this is also true if we base change to $S_i := \text{Spec}(R/I^i)$, which forms a compatible system, hence
\[
\Gamma (G^{\natural}_{\text{for}}, \mathcal{L}^{\natural}_{\text{for}}) 
\cong \hat{\underset{\chi \in X}{\oplus}} \Gamma (A, \mathcal{M}_{\chi})
\]
where the completion is with respect to the $I$-adic topology. Now by the definition of the Raynaud extension, we have that $G$ and $G^{\natural}$ have the same formal completion along $I$, i.e.
$G^{\natural}_{\text{for}} \cong G_{\text{for}}$. 
The canonical pullback map
$\Gamma(G, \mathcal{L}) \rightarrow \Gamma (G_{\text{for}}, \mathcal{L}_{\text{for}})$
becomes
\[
\Gamma(G, \mathcal{L}) \rightarrow \Gamma (G_{\text{for}}, \mathcal{L}_{\text{for}}) \cong 
\Gamma (G^{\natural}_{\text{for}}, \mathcal{L}^{\natural}_{\text{for}}) 
\cong \hat{\underset{\chi \in X}{\oplus}} \Gamma (A, \mathcal{M}_{\chi})
\]
and projecting to the $\chi$-th component we obtain 
\[
\Gamma(G, \mathcal{L}) \longrightarrow \Gamma (A, \mathcal{M}_{\chi}).
\]
Tensoring both sides with $K := \text{Frac}(R)$, we obtain 
\[
\sigma_{\chi}: \Gamma(G_{\eta}, \mathcal{L}_{\eta}) \longrightarrow \Gamma (A_{\eta}, \mathcal{M}_{\chi,\eta})
\]
by flat base change, 
which are the Fourier coefficients of theta functions with respect to $\mathcal{L}$.

Now as in the classical case, we aim to find the functional equation of $\sigma_{\chi}$ and read off the sought-after $\tau$ from it. Let $y \in Y$, and 
$T_{c^{\vee}(y)} : A \rightarrow A$
the translation by the point $c^{\vee}(y)$, then the equation 
$\lambda_A \circ c^{\vee} = c \circ \phi$
applied to $y$ translates into an isomorphism
\[
T_{c^{\vee}(y)}^* \mathcal{M}_{\chi} \cong
\mathcal{M}_{\chi + \phi(y)} \otimes_R \mathcal{M}_{\chi}(c^{\vee}(y))
\]
(using rigidified line bundles to represent elements of $A^{\vee}$, and elements of $A^{\vee}$ are characterized by the identity $T_x^*L \cong L$). This provides us with the natural map
\[
T_{c^{\vee}(y)}^* \circ \sigma_{\chi} : 
\Gamma(G_{\eta}, \mathcal{L}_{\eta}) \rightarrow 
\Gamma (A_{\eta}, T_{c^{\vee}(y)}^*\mathcal{M}_{\chi,\eta}) \cong 
\Gamma (A_{\eta},\mathcal{M}_{\chi + \phi(y),\eta}) \otimes_K 
\mathcal{M}_{\chi}(c^{\vee}(y))_{\eta}.
\]
On the other hand, we have the map
\[
\sigma_{\chi+\phi(y)}: \Gamma(G_{\eta}, \mathcal{L}_{\eta}) \longrightarrow \Gamma (A_{\eta}, \mathcal{M}_{\chi+\phi(y),\eta}). 
\]
The functional equation we are searching for is 
\begin{equation} \label{feq}
\sigma_{\chi + \phi(y)} = \psi (y) \tau(y, \chi) T_{c^{\vee}(y)}^* \circ \sigma_{\chi}
\end{equation}
where 
\[
\psi(y): \mathcal{M}(c^{\vee}(y))_{\eta} \overset{\sim}{\rightarrow} \mathscr{O}_{S, \eta}
\]
is a trivialization of the fiber of $\mathcal{M}$ at $c^{\vee}(y)$, 
and 
\[
\tau(y, \chi) : \mathscr{O}_{\chi}(c^{\vee}(y))_{\eta} \longrightarrow \mathscr{O}_{S, \eta}
\]
is a section of 
$\mathscr{O}_{\chi}(c^{\vee}(y))_{\eta}^{\otimes -1}$
for each $y \in Y$ and $\chi \in X$, so that 
$\psi (y) \tau(y, \chi)$ is a section of $\mathcal{M}_{\chi}(c^{\vee}(y))_{\eta}^{\otimes-1}$
(recall $\mathcal{M}_{\chi} = \mathcal{M} \otimes \mathscr{O}_{\chi}$). 

It is a hard fact that 
\[
\sigma_{\chi} \neq 0
\]
for all $\chi \in X$, hence $\tau$ (and $\psi$) is uniquely characterized by the functional equation (\ref{feq}). The positivity, bilinearity and symmetry of $\tau$ all follow relatively formally from (\ref{feq}) and $\sigma_{\chi} \neq 0$. Further,  $\tau$ is independent of the choice of $\mathcal{L}$ ($\psi$ depends on $\mathcal{L}$ but is independent of the choice of $\mathcal{M}$). 

Equation (\ref{feq}) follows formally if we know that
$\sigma_{\chi + \phi(y)}$ is proportional to 
$T_{c^{\vee}(y)}^* \circ \sigma_{\chi}$,
and this is proved using representations of theta groups. Indeed, we can prove that $\sigma_{\chi}$ factors through an equivariant homomorphism between two irreducible representations with respect to a subgroup of the theta group of $\mathcal{L}$ (isomorphic to the theta group of $\mathcal{M}_{\chi}$ which acts naturally on 
$\Gamma (A_{\eta}, \mathcal{M}_{\chi,\eta})$), and similarly for 
$T_{c^{\vee}(y)}^* \circ \sigma_{\chi}$.
The non-vanishing of $\sigma_{\chi}$ forces that both factorizations are non-zero, so Schur's lemma gives the proportionality.

\subsection{PEL degeneration data}

We want to generalize the polarized degeneration data to include endormorphisms and level structures, so that they characterize degenerations of PEL abelian schemes. It turns out that level structures create substantial technical difficulties, which is one of the main technical contributions of Kaiwen Lan. Following Lan's presentation, we separate the data with and without level structures. We use notations from the previous section, and the notations for PEL datum are as in section \ref{kpelmp}.

\subsubsection{Data without level structures}

We begin by defining the degenerating PE abelian varieties.

\begin{definition}
The category 
$\text{DEG}_{\text{PE}, \mathcal{O}}(R,I)$
has objects 
$(G,\lambda,i)$ 
where $(G, \lambda) \in \text{DEG}_{\text{pol}}(R,I)$, 
and 
\[
i: \mathcal{O} \rightarrow End_S(G)
\]
is a ring homomorphism such that 
\[
i_{\eta}(b)^{\vee} \circ \lambda_{\eta} = \lambda_{\eta} \circ i_{\eta}(b^*)
\]
for every $b \in \mathcal{O}$, where 
$i_{\eta}(b)^{\vee} : G_{\eta}^{\vee} \rightarrow G_{\eta}^{\vee}$ 
is the dual of $i_{\eta}(b)$. The morphisms are isomorphisms respecting all structures. 
\end{definition}

\begin{remark}
We know that the restriction to the generic fiber is a fully faithful functor from the category of degenerating abelian varieties to that of abelian varieties, which implies that
$\lambda_{\eta}: G_{\eta} \rightarrow G_{\eta}^{\vee}$ 
extends uniquely to a morphism 
$\lambda: G \rightarrow G^{\vee}$, 
thus it is unambiguous to write 
$(G, \lambda) \in \text{DEG}_{\text{pol}}(R,I)$. 
Similarly, we have
$End_S(G) \cong End_{\eta}(G_{\eta})$,
so the extra data are determined by their restriction to the generic fiber, and the generic fiber is a PE abelian variety by $\mathcal{O}$. 
\end{remark}

\begin{definition}
The category 
$\text{DD}_{\text{PE},\mathcal{O}}(R,I)$
has objects 
\[
(A, \lambda_A, i_A, \underline{X}, \underline{Y}, \phi, c,c^{\vee}, \tau)
\]
such that 
$(A, \lambda_A, \underline{X}, \underline{Y}, \phi, c,c^{\vee}, \tau) \in \text{DD}_{\text{pol}}(R,I)$
and 
\[
i_A : \mathcal{O} \rightarrow End_S(A)
\]
is a ring homomorphism such that 
$i_A(b)^{\vee} \circ \lambda_A = \lambda_A \circ i_A(b^*)$
for every $b \in \mathcal{O}$. The data are required to the additional $\mathcal{O}$-structures in the sense that

(1) $\underline{X}$ and $\underline{Y}$ are \'etale locally constant sheaf of projective $\mathcal{O}$-modules with structure morphisms 
$i_{\underline{X}}: \mathcal{O} \rightarrow End_S(\underline{X})$
and
$i_{\underline{Y}}: \mathcal{O} \rightarrow End_S(\underline{Y})$.
$\underline{X}$ and $\underline{Y}$ are required to be rationally equivalent as sheaves of
$\mathcal{O}\times_{\mathbb{Z}} \mathbb{Q}$-modules. Moreover, $\phi: \underline{Y} \rightarrow \underline{X}$ 
is $\mathcal{O}$-equivariant.

(2) $c: \underline{X} \rightarrow A^{\vee}$
and $c^{\vee}: \underline{Y} \rightarrow A$
are $\mathcal{O}$-equivariant.

(3) The trivialization 
$\tau : \textbf{1}_{\underline{Y} \times_S \underline{X}, \eta} \overset{\sim}{\longrightarrow} 
(c^{\vee} \times c)^* \mathcal{P}_{A,\eta}^{\otimes -1}$
satisfies 
\[
(i_{\underline{Y}}(b) \times Id_{\underline{X}})^*\tau = 
(Id_{\underline{Y}} \times i_{\underline{X}}(b^*))^*\tau 
\]
for all $b \in \mathcal{O}$, i.e. 
$\tau(by, \chi) = \tau(y, b^* \chi)$ 
for $y \in Y$ and $\chi \in X$ if $\underline{X}$ and $\underline{Y}$ are constant. 

The morphisms are isomorphisms respecting all structures. 
\end{definition}

The following theorem follows directly from the functoriality of $F_{\text{pol}}(R,I) $. 

\begin{theorem}
There is an equivalence of categories
\[
F_{PE,\mathcal{O}}(R,I) : \text{DEG}_{\text{PE}, \mathcal{O}}(R,I) \rightarrow 
\text{DD}_{\text{PE},\mathcal{O}}(R,I)
\]
\end{theorem}

\vspace{6mm}

We can strengthen the theorem by adding the Lie algebra condition on both sides. It is the determinant condition in the definition of PEL moduli problems. 

\begin{definition}
The category 
$\text{DEG}_{\text{PE}_{Lie}, (L \otimes_{\mathbb{Z}} \mathbb{R}, \langle  \cdot,\cdot\rangle,h)}(R,I)$
has objects 
\[
(G,\lambda,i) \in \text{DEG}_{\text{PE}, \mathcal{O}}(R,I)
\] 
such that 
$(G_{\eta},\lambda_{\eta},i_{\eta})$ 
satisfies the determinant condition specified by 
$(L \otimes_{\mathbb{Z}} \mathbb{R}, \langle  \cdot,\cdot\rangle,h)$, 
see \cite{lan2013arithmetic} 1.3.4.1 for definitions. The morphisms are isomorphisms respecting all structures. 
\end{definition}

\begin{definition}
The category 
$\text{DD}_{\text{PE}_{Lie}, (L \otimes_{\mathbb{Z}} \mathbb{R}, \langle  \cdot,\cdot\rangle,h)}(R,I)$
has objects 
\[
(A, \lambda_A, i_A, \underline{X}, \underline{Y}, \phi, c,c^{\vee}, \tau) \in \text{DD}_{\text{PE}, \mathcal{O}}(R,I)
\]
such that 
there exists a totally isotropic embedding 
\[
Hom_{\mathbb{R}}(X \otimes \mathbb{R}, \mathbb{R}(1)) \hookrightarrow L \otimes \mathbb{R}
\]
of $\mathcal{O} \otimes \mathbb{R}$-modules with image denoted by $Z_{-2, \mathbb{R}}$, 
where $X$ is  the underlying $\mathcal{O}$-module of $\underline{X}$, and such that 
$(A_{\eta}, \lambda_{A,\eta}, i_{A, \eta})$ 
satisfies the determinant condition determined by 
$(Z_{-2, \mathbb{R}}^{\perp}/Z_{-2, \mathbb{R}}, \langle  \cdot, \cdot\rangle, h_{-1}) $ induced by the embedding. The morphisms are isomorphisms respecting all the structures. 
\end{definition}

\begin{theorem}
(Lan) There is an equivalence of categories
\[
F_{\text{PE}_{Lie}, (L \otimes_{\mathbb{Z}} \mathbb{R}, \langle  \cdot,\cdot\rangle,h)}(R,I) :
\]
\[
\text{DEG}_{\text{PE}_{Lie}, (L \otimes_{\mathbb{Z}} \mathbb{R}, \langle  \cdot,\cdot\rangle,h)}(R,I) 
\rightarrow 
\text{DD}_{\text{PE}_{Lie}, (L \otimes_{\mathbb{Z}} \mathbb{R}, \langle  \cdot,\cdot\rangle,h)}(R,I)
\]
\end{theorem}

\subsubsection{Data with level structures} \label{dlevel}

We will only work with principal level structures in this paper. The general level structures can be taken as orbits of principal level structures, and the modification with degeneration data is to take the quotient of the data with principal level structures by certain groups. 

We fix an integer $n$ in this section. We assume that the the generic point $\eta = Spec(K)$ is defined over $Spec(\mathcal{O}_{F_0, (\square)})$, 
where $F_0$ is the reflex field and $\square$ is the set of all primes not dividing $nI_{bad}Disc_{\mathcal{O}/\mathbb{Z}} [L^{\#}:L]$, see \cite{lan2013arithmetic} 1.4.1.1 for definitions of these bad primes. In particular, 
$Spec(\mathcal{O}_{F_0, (\square)})$ 
is the maximal base over which the PEL moduli variety is smooth. Moreover, we make the technical assumption that the $\mathcal{O}$-action on $L$ extends to a maximal order in $B$. 

All the morphisms in the category to be defined will be the obvious isomorphisms preserving all the structures, and we omit the description. 

\begin{definition} \label{degpel}
The category 
$\text{DEG}_{\text{PEL},M_n}(R,I)$
has objects 
\[
(G, \lambda, i, (\alpha_n, \nu_n))
\]
where 
\[
(G,\lambda, i) \in \text{DEG}_{\text{PE}_{Lie}, (L \otimes_{\mathbb{Z}} \mathbb{R}, \langle  \cdot,\cdot\rangle,h)}(R,I). 
\] 
Moreover, 
$\alpha_n : L/nL \overset{\sim}{\rightarrow} G[n]_{\eta}$
and 
$ \nu_n : \mathbb{Z}/n \mathbb{Z}(1) \overset{\sim}{\rightarrow} \mu_{n,\eta}$
are isomorphisms such that they
define a level-n structure for $G_{\eta}$ in the sense that
\[
(G_{\eta}, \lambda_{\eta}, i_{\eta}, (\alpha_n, \nu_n)) \in M_n (\eta)
\]
as in definition \ref{isomorphism}. 
\end{definition}

\begin{definition} \label{DDL}
The category 
$\text{DD}_{\text{PEL},M_n}(R,I)$
has objects
\[
(A, \lambda_A, i_A, \underline{X}, \underline{Y}, \phi, c,c^{\vee}, \tau, [ \alpha_n^{\natural}]) 
\]
where 
\[
(A, \lambda_A, i_A, \underline{X}, \underline{Y}, \phi, c,c^{\vee}, \tau) \in \text{DD}_{\text{PE}_{Lie}, (L \otimes_{\mathbb{Z}} \mathbb{R}, \langle  \cdot,\cdot\rangle,h)}(R,I)
\]
and 
\[
\alpha_n^{\natural} := (Z_n, \varphi_{-2,n}, (\varphi_{-1,n}, \nu_{-1,n}), \varphi_{0,n}, \delta_n, c_n, c_n^{\vee}, \tau_n)  
\]
is the level structure data
with objects to be defined as follows: 

(1) $Z_n$ is a filtration
\[
0 \subset Z_{n,-2} \subset Z_{n,-1} \subset Z_{n,0} = L/nL
\]
on $L/nL$, which can be written as the reduction modulo n of a filtration (of $\mathcal{O} \otimes_{\mathbb{Z}} \hat{\mathbb{Z}}^{\square}$-modules)
\[
0 \subset Z_{-2} \subset Z_{-1} \subset Z_{0} = L\otimes_{\mathbb{Z}} \hat{\mathbb{Z}}^{\square}
\]
on 
$L\otimes_{\mathbb{Z}} \hat{\mathbb{Z}}^{\square}$
such that $Z$ extends to a filtration 
$Z_{\mathbb{A}^{\square}}$ on 
$L\otimes_{\mathbb{Z}} \mathbb{A}^{\square}$
which has the property that it is split (as 
$\mathcal{O} \otimes_{\mathbb{Z}} \mathbb{A}^{\square}$-modules),
$Gr_i^{Z_{\mathbb{A}^{\square}}}$ is integral for every i, i.e.  
$Gr_i^{Z_{\mathbb{A}^{\square}}}= M_i \otimes_{\mathbb{Z}} \mathbb{A}^{\square}$ for some torsion-free finitely generated  $\mathcal{O}$-module $M_i$, 
and 
$Z_{\mathbb{A}^{\square}, -2}$ 
is the annihilator of $Z_{\mathbb{A}^{\square}, -1}$ 
under the natural pairing 
$\langle   \cdot, \cdot\rangle_{\mathbb{A}^{\square}}$
on 
$L\otimes_{\mathbb{Z}} \mathbb{A}^{\square}$. 

(2) $\varphi_{-1,n}: Gr_{-1}^{Z_n} \overset{\sim}{\rightarrow} A[n]_{\eta}$
and 
$\nu_{-1,n}:  \mathbb{Z}/n \mathbb{Z}(1) \overset{\sim}{\rightarrow} \mu_{n,\eta}$
are isomorphisms such that 
\[
(A_{\eta}, \lambda_{A, \eta}, i_{A, \eta}, (\varphi_{-1,n}, \nu_{-1,n})) \in M_n(\eta) 
\]
with respect to the PEL datum determined by $Gr_{-1,n}^{Z_n}$, which exists because $Z_n$ satisfying the conditions in (1). 

(3) 
\[
\varphi_{-2,n} : Gr_{-2}^{Z_n} \overset{\sim}{\rightarrow} \underline{Hom}_{\eta}((\underline{X}/n \underline{X})_{\eta}, (\mathbb{Z}/n\mathbb{Z})(1))
\]
and 
\[
\varphi_{0,n} : Gr_{0}^{Z_n} \overset{\sim}{\rightarrow}
(\underline{Y}/ n \underline{Y})_{\eta}
\]
are isomorphisms which are liftable to some isomorphisms 
$\varphi_{-2}: Gr_{-2}^Z  \overset{\sim}{\rightarrow} Hom(X\otimes_{\mathbb{Z}} \hat{\mathbb{Z}}^{\square}, \hat{\mathbb{Z}}^{\square}(1))$ 
and 
$\varphi_0 : Gr_0^Z \overset{\sim}{\rightarrow} Y \otimes_{\mathbb{Z}} \hat{\mathbb{Z}}^{\square}$
over $\bar{\eta}$. Moreover, they are required to satisfy the equation 
\[
\langle  \varphi_{-2,n} (\cdot), \phi \circ \varphi_{0,n} (\cdot)\rangle_{can} = \langle   \cdot, \cdot\rangle_{20,n}
\]
where 
$\langle  \cdot, \cdot\rangle_{can}:\underline{Hom}_{\eta}((\underline{X}/n \underline{X})_{\eta}, (\mathbb{Z}/n\mathbb{Z})(1)) 
\times (\underline{X}/n \underline{X})_{\eta} 
\rightarrow (\mathbb{Z}/n\mathbb{Z})(1)$
is the canonical evaluation pairing, and 
$\langle  \cdot, \cdot\rangle_{20,n} : Gr_{-2}^{Z_n} \times Gr_0^{Z_n} \rightarrow (\mathbb{Z}/n\mathbb{Z})(1)$ 
is the pairing induced by $\langle  \cdot, \cdot\rangle$ on $L$ (using that $Z_{n,-2}$ is the annihilator of $Z_{n,-1}$). 

(4)
\[
\delta_n : \underset{i}{\oplus} Gr_i^{Z_n} \overset{\sim}{\rightarrow} L/nL
\]
is a splitting of the filtration $Z_n$, which can be lifted to a splitting 
$\delta : \underset{i}{\oplus} Gr_i^{Z} \overset{\sim}{\rightarrow} L\otimes_{\mathbb{Z}} \hat{\mathbb{Z}}^{\square} $. 

(5)
\[
c_n : \frac{1}{n} \underline{Y}_{\eta} \rightarrow A_{\eta}
\]
and 
\[
c_n^{\vee} : \frac{1}{n} \underline{X}_{\eta} \rightarrow A_{\eta}^{\vee}
\]
are homomorphisms that lifts $c$ and $c^{\vee}$ over $\eta$, i.e. $c_{\eta} = c_n \circ (\underline{Y}_{\eta} \hookrightarrow \frac{1}{n} \underline{Y}_{\eta} )$
and similarly for $c^{\vee}$. They are required to be compatible with the splitting $\delta_n$ in the sense that 
\[
\langle  \varphi_{-1,n}( \cdot), (\lambda_A\circ c_n^{\vee} - c_n \circ \phi )\circ \varphi_{0,n}( \cdot) \rangle_A = \nu_{-1,n} \circ \langle   \cdot, \cdot\rangle_{10, n}
\]
where 
$\langle  \cdot, \cdot\rangle_A: A[n]_{\bar{\eta}} \times A^{\vee}[n]_{\bar{\eta}} \rightarrow \mu_{n,\bar{\eta}}$
is the Weil pairing of $A_{\bar{\eta}}$, 
and 
\[
\langle  \cdot, \cdot\rangle_{10,n}: Gr_{-1}^{Z_n} \times Gr_0^{Z_n} \rightarrow (\mathbb{Z}/n\mathbb{Z})(1)
\]
is the pairing induced by $\delta_n$ and the natural pairing $\langle  \cdot, \cdot\rangle$ on $L/nL$,  i.e. 
$\langle  \cdot, \cdot\rangle_{10,n} := \langle  \delta_n(\cdot), \delta_n(\cdot)\rangle $
with domain 
$ Gr_{-1}^{Z_n} \times Gr_0^{Z_n}$. 
Moreover, they need to satisfy a level-lifting condition compatible with all the previous lifting, see \cite{lan2013arithmetic} 5.2.7.8 for the precise description. 

(6)
\[
\tau_n: \textbf{1}_{\frac{1}{n}\underline{Y} \times_S \underline{X}, \eta} \overset{\sim}{\longrightarrow} 
(c^{\vee}_n \times c_{\eta})^* \mathcal{P}_{A,\eta}^{\otimes -1}
\]
is a lifting of $\tau$ in the obvious sense. Similar to (5), it is required to be compatible with $\delta_n$ in the sense that 
\[
d_{00,n}(\varphi_{0,n}(\cdot), \varphi_{0,n}(\cdot)) = \nu_{-1,n} \circ \langle  \cdot, \cdot\rangle_{00,n}
\]
where 
$d_{00,n}: \frac{1}{n}Y/Y \times \frac{1}{n}Y/Y \rightarrow \mu_{n,\bar{\eta}}$
is defined by
\[
d_{00,n}(\frac{1}{n}y, \frac{1}{n}y') := \tau_n(\frac{1}{n}y, \phi(y')) \tau_n(\frac{1}{n}y', \phi(y))^{-1}
\]
for 
$\frac{1}{n}y, \frac{1}{n}y' \in \frac{1}{n}Y$, 
and 
$\langle  \cdot, \cdot\rangle_{00,n}: Gr_{0}^{Z_n} \times Gr_0^{Z_n} \rightarrow (\mathbb{Z}/n\mathbb{Z})(1)$
is defined by 
$\langle  \cdot, \cdot\rangle_{00,n} := \langle  \delta_n(\cdot), \delta_n(\cdot)\rangle $. They again have to satisfy a level-lifting condition, see \cite{lan2013arithmetic} 5.2.7.8 for details. Note that we have tacitly used the canonical identification $ \frac{1}{n}Y/Y \cong Y/nY$. 

\vspace{5mm}

The bracket $[\alpha_n^{\natural}]$ means the equivalence class of $\alpha_n^{\natural}$, see \cite{lan2013arithmetic} 5.2.7.11 for the definition. Essentially, taking the equivalence class is to eliminate the choice of the splitting. The subtlety to define the equivalence is that the complicated relations among the data are described using splittings, and changing splittings will introduce modifications into various data. We only remark that the data
$(Z_n, \varphi_{-2,n}, (\varphi_{-1,n}, \nu_{-1,n}), \varphi_{0,n})$ 
is independent of the equivalence class, so the equivalence has effect only on
$(\delta_n, c_n, c_n^{\vee}, \tau_n)$. 
\end{definition}

\begin{remark}
There is redundancy in the above definition, namely, $c$ and $c^{\vee}$ are determined by $c_n$ and $c_n^{\vee}$, and the same is true for $\tau$. \end{remark}

\begin{theorem}
There is an equivalence of categories 
\[
F_{PEL, M_n}(R,I) : \text{DEG}_{\text{PEL},M_n}(R,I) \rightarrow \text{DD}_{\text{PEL},M_n}(R,I).
\]
\end{theorem}

\subsubsection{The construction of $F_{PEL,M_n}$}

We now explain the meaning of the above complicated data and the construction of $F_{PEL,M_n}(R,I)$. 

Without loss of generality, we assume that $\underline{X}$ and $\underline{Y}$ are constant with values $X$ and $Y$ from now on. we have already seen how to associate data that characterize the degenerating abelian scheme $G$, together with its PE structures, we now focus on the level structures. The key point is that Mumford construction tells us that $G[n]_{\eta}$ is naturally an extension 
\[
0 \rightarrow G^{\natural}[n]_{\eta} \rightarrow G[n]_{\eta} \rightarrow \frac{1}{n} Y/Y \rightarrow 0
\]
which further justifies the heuristic "$G = G^{\natural}/Y$". Moreover, $G^{\natural}$ being a global extension of an abelian variety by an algebraic torus implies that $G^{\natural}[n]_{\eta}$ is also an extension 
\[
0 \rightarrow T[n]_{\eta} \rightarrow G^{\natural}[n]_{\eta} \rightarrow A[n]_{\eta} \rightarrow 0
\]
It is clear by naturality that these extensions can be upgraded to extensions in terms of the Tate modules, i.e.
$T^{\square}G_{\eta} := \underset{(m, \square)=1 }{\varprojlim} G[m]_{\eta}$
for example. 

Now if we are given a level-n structure on the generic fiber, we have an isomorphism 
$\alpha_n : L/nL \overset{\sim}{\rightarrow} G[n]_{\eta}$
together with an isomorphism 
$\nu_n :\mathbb{Z}/n \mathbb{Z}(1) \overset{\sim}{\rightarrow} \mu_{n,\eta} $,
which is compatible with the Weil pairing and  liftable to the Tate module. The above two extensions endows a filtration $Z_n$ on $L/nL$ through $\alpha_n$, i.e. 
\[
0 \subset Z_{n,-2} \subset Z_{n,-1} \subset Z_{n,0} = L/nL,
\]
such that $\alpha_n$ identifies $Gr_{-2}^{Z_n}$, $Gr_{-1}^{Z_n}$ and $Gr_{0}^{Z_n}$ with 
$T[n]_{\eta}$, $A[n]_{\eta}$ and $\frac{1}{n}Y/Y$ respectively. Note that
$T[n]_{\eta} = Hom(X/nX, \mu_n) \overset{\nu_n}{\cong} Hom(X/nX,\mathbb{Z}/n \mathbb{Z}(1))$, 
and we denote the corresponding isomorphisms by
\[
\varphi_{-2,n} : Gr_{-2}^{Z_n} \overset{\sim}{\rightarrow} Hom(X/nX, (\mathbb{Z}/n\mathbb{Z})(1)),
\]
\[
\varphi_{-1,n}: Gr_{-1}^{Z_n} \overset{\sim}{\rightarrow} A[n]_{\eta}
\]
and 
\[
\varphi_{0,n} : Gr_{0}^{Z_n} \overset{\sim}{\rightarrow}
(Y/nY)_{\eta}.
\]

This explains where 
$(Z_n, \varphi_{-2,n}, \varphi_{-1,n}, \varphi_{0,n})$ 
come from. The respective liftability conditions in (1), (2) and (3) of definition \ref{DDL} corresponds to the liftability of the level structure $\alpha_n$ and the above extensions. That they satisfy the conditions on Weil pairings in (1), (2) and (3) are general theorems of Grothendieck in SGA 7, where the above two extensions are interpreted as monodromy filtration. The $\nu_{-1,n}$ in the degeneration data is defined to be $\nu_n$, which is forced by the Weil pairing condition in (2) of definition \ref{DDL}.  

We have produced the data 
$(Z_n, \varphi_{-2,n}, (\varphi_{-1,n}, \nu_{-1,n}), \varphi_{0,n})$, 
which characterizes $\alpha_n$ up to graded pieces. Now we aim to find more data from which we can recover the complete $\alpha_n$. The idea is to introduce auxiliary data that corresponds to splittings of the above two extensions, and then take equivalence relations by identifying different splittings. 

First, a splitting of the extension 
\[
0 \rightarrow T[n]_{\eta} \rightarrow G^{\natural}[n]_{\eta} \rightarrow A[n]_{\eta} \rightarrow 0
\]
is the same as a section of 
$G^{\natural}[n]_{\eta} \rightarrow A[n]_{\eta}$, 
which is equivalent to a subgroup scheme $H$ of $G^{\natural}[n]_{\eta}$
that is isomorphic to 
$ A[n]_{\eta}$
through the projection. Let 
$G^{\natural'}_{\eta} := G^{\natural}_{\eta}/H$, 
then the quotient map induces 
\[
\begin{tikzcd}
0 \arrow[r]  & T_{\eta} \arrow[d,equal]  \arrow[r]  & G^{\natural}_{\eta} \arrow[r]  \arrow[d] & A_{\eta} \arrow[r] \arrow[d,"n"]   & 0 
\\
0 \arrow[r]  & T_{\eta} \arrow[r]  & G^{\natural'}_{\eta} \arrow[r]  & A_{\eta} \arrow[r]   & 0
\end{tikzcd}
\]
which can be completed into 
\[
\begin{tikzcd}
0 \arrow[r]  & T_{\eta} \arrow[d,equal]  \arrow[r]  & G^{\natural}_{\eta} \arrow[r]  \arrow[d] \arrow[dd, bend left=30, "n"pos=0.3] & A_{\eta} \arrow[r] \arrow[d,"n"]   & 0 
\\
0 \arrow[r]  & T_{\eta} \arrow[r] \arrow[d,"n "]  & G^{\natural'}_{\eta} \arrow[d] \arrow[r]  & A_{\eta} \arrow[r] \arrow[d,equal]  & 0
\\
0 \arrow[r]  & T_{\eta}   \arrow[r]  & G^{\natural}_{\eta} \arrow[r]   & A_{\eta} \arrow[r]    & 0
\end{tikzcd}
\]
We see that the extension $G^{\natural'}_{\eta}$ together with the isogeny 
\[
\begin{tikzcd}
0 \arrow[r]  & T_{\eta} \arrow[r] \arrow[d,"n "]  & G^{\natural'}_{\eta} \arrow[d] \arrow[r]  & A_{\eta} \arrow[r] \arrow[d,equal]  & 0
\\
0 \arrow[r]  & T_{\eta}   \arrow[r]  & G^{\natural}_{\eta} \arrow[r]   & A_{\eta} \arrow[r]    & 0
\end{tikzcd}
\]
determines the splitting, hence a splitting of 
$0 \rightarrow T[n]_{\eta} \rightarrow G^{\natural}[n]_{\eta} \rightarrow A[n]_{\eta} \rightarrow 0$
is equivalent to a diagram as above, which is the same as a lifting  
$c_n: \frac{1}{n}X \rightarrow A^{\vee}_{\eta} $ 
of 
$c: X \rightarrow A^{\vee}$
over the generic fiber. 

Next, we look at the splitting of 
\[
0 \rightarrow G^{\natural}[n]_{\eta} \rightarrow G[n]_{\eta} \rightarrow \frac{1}{n} Y/Y \rightarrow 0
\]
From the Mumford quotient 
"$G= G^{\natural}_{\eta}/Y$",
it is reasonable to expect that a splitting 
$\frac{1}{n} Y/Y \rightarrow G[n]_{\eta}$
is equivalent to a lifting 
\[
\imath_n : \frac{1}{n}Y \longrightarrow G^{\natural}_{\eta}
\]
of the period homomorphism
$\imath : Y \rightarrow G^{\natural}_{\eta}$, 
and this can be proved rigorously. The composition of $\imath_n$ with projection to $A_{\eta}$ produces
\[
c^{\vee}_{n}: \frac{1}{n}Y \overset{\imath_n}{\longrightarrow} G^{\natural}_{\eta} \overset{\pi}{\longrightarrow} A_{\eta}
\]
which lifts $c_{\eta}^{\vee}$ because $\imath_n$ lifts $\imath$ and 
$c^{\vee}_{\eta} = \pi \circ \imath$.
As we have seen before, such a period homomorphism $\imath_n$ is equivalent to a trivialization of biextensions
\[
\tau_n: \textbf{1}_{\frac{1}{n}\underline{Y} \times_S \underline{X}, \eta} \overset{\sim}{\longrightarrow} 
(c^{\vee}_n \times c_{\eta})^* \mathcal{P}_{A,\eta}^{\otimes -1}
\]
that lifts $\tau$. 

We have seen that a splitting of the monodromy filtration on $G[n]_{\eta}$ is equivalent to the data 
\[
(c_n, c_n^{\vee}, \tau_n)
\]
that lifts 
$(c_{\eta},c^{\vee}_{\eta}, \tau)$.
From the isomorphism 
$\alpha_n: L/nL \overset{\sim}{\rightarrow} G[n]_{\eta}$, 
the splitting on $G[n]_{\eta}$ induces a splitting $\delta_n$ of the filtration $Z_n$ on $L/nL$, and this finishes the construction of the remaining
\[
(\delta_n, c_n, c_n^{\vee}, \tau_n).
\]
Note that the liftability condition is clearly satisfied. 

To summarize, the data 
$(c_n, c_n^{\vee}, \tau_n)$ 
determines a splitting of the monodromy filtration $W_n$ on $G[n]_{\eta}$, i.e. an isomorphism 
$\zeta_n: \underset{i}{\oplus} Gr_i^{W_n} \overset{\sim}{ \rightarrow} G[n]_{\eta}$, 
and the level structure $\alpha_n$ can be recovered as 
\begin{equation} \label{levelstru}
\alpha_n : L/nL \underset{\sim}{\overset{\delta_n^{-1}}{\longrightarrow}}  
\underset{i}{\oplus} Gr_i^{Z_n} 
\overset{\underset{i}{\oplus} \varphi_{i,n}}{\underset{\sim}{\longrightarrow}} 
\underset{i}{\oplus} Gr_i^{W_n}
\underset{\sim}{\overset{\zeta_n}{\longrightarrow}}
G[n]_{\eta}
\end{equation}
which is liftable by the liftability condition on all the intermediate isomorphisms.

The last ingredient is to find characterizing conditions for $\alpha_n$ to be compatible with the Weil pairing. The key is to use the degeneration data to describe the pairing on 
$\underset{i}{\oplus} Gr_i^{W_n}$
induced by the Weil pairing on $G[n]_{\eta}$ and the isomorphism $\zeta_n$. This is the most difficult part of the construction, as well as one of the main technical contributions of Lan. 

We know that the two pairings on
$Gr_{-2}^{W_n} \times Gr_0^{W_n}$ 
and 
$Gr_{-1}^{W_n} \times Gr_{-1}^{W_n}$
are independent of the splitting since $W_{n,-2}$ is the annihilator of $W_{n,-1}$, and has been determined by Grothendieck as we have already remarked. Since the Weil pairing is alternating, the remaining cases to be determined are 
$Gr_{-1}^{W_n} \times Gr_0^{W_n}$
and 
$Gr_{0}^{W_n} \times Gr_0^{W_n}$. The result is as follows, 
the pairing on 
$Gr_{-1}^{W_n} \times Gr_0^{W_n}$
is given by 
\[
Gr_{-1}^{W_n} \times Gr_0^{W_n} = A[n]_{\eta} \times \frac{1}{n} Y/Y \rightarrow \mu_{n,\eta}
\]
that sends $(a,\frac{1}{n}y)$ to
\[
\langle  a, (\lambda_{A,\eta} \circ c_n^{\vee} - c_n \circ \phi)(\frac{1}{n} y)\rangle_{A[n]}  
\]
where 
$\langle  \cdot, \cdot\rangle_{A[n]} : A[n]_{\eta} \times A^{\vee}[n]_{\eta} \rightarrow \mu_{n,\eta}$
is the canonical Weil pairing. 
On the other hand, the pairing 
\[
Gr_{0}^{W_n} \times Gr_0^{W_n} =\frac{1}{n} Y/Y \times \frac{1}{n} Y/Y \rightarrow \mu_{n,\eta}
\]
is given by 
\[
(\frac{1}{n} y_1, \frac{1}{n} y_2) \longrightarrow 
\tau_n(\frac{1}{n} y_1, \phi(y_2)) \tau_n(\frac{1}{n} y_2, \phi(y_1))^{-1}.
\]

We now transform the pairing from
$\underset{i}{\oplus} Gr_i^{W_n}$ 
to  $L/nL$ using $\delta_n$ and $\varphi_{i,n}$, then the compatibility of $\alpha_n$ with the Weil pairing is rephrased in the language of degeneration data, which is exactly the various pairing conditions in definition \ref{DDL}. 

Lastly, the equivalence is defined by identifying different splittings $\zeta_n$ and $\delta_n$ which induce the same $\alpha_n$ through (\ref{levelstru}). This is easily translated into a statement involving only degeneration data, see \cite{lan2013arithmetic} 5.2.7.11 for details. Since $\zeta_n$ is equivalent to 
$(c_n, c_n^{\vee}, \tau_n)$, clearly the equivalence only changes 
$(\delta_n, c_n, c_n^{\vee}, \tau_n)$. 
This concludes the construction of 
$F_{PEL,M_n}(R,I)$. 

\subsection{Toroidal compactifications} \label{tttor}

We now review the construction of toroidal compactifications of PEL Shimura varieties. This is in some sense the universal base of a degenerating PEL abelian scheme. We have already seen that degenerating abelian varieties over a Noetherian normal complete affine base is equivalent to a set of degeneration data. The basic idea of the construction of toroidal compactifications is to find the moduli space of the degenerating data, and glue them to the Shimura variety. 

More precisely, since the degeneration data characterizes the degenerating abelian varieties only over a complete base, the moduli space of degeneration data is the completion of the toroidal compactification along the boundary. To obtain the whole compactification, it is necessary to algebraize the complete situation, which is a subtle procedure that we will not review. 

Let us start with the construction of moduli space of degeneration data. We first construct the moduli space of data without the equivalence relation, i.e. we want to parametrize the tuple
\[
(A, \lambda_A, i_A, \underline{X}, \underline{Y}, \phi, c,c^{\vee}, \tau,  \alpha_n^{\natural}) 
\]
without bracket on $\alpha_n^{\natural}$, 
where 
\[
\alpha_n^{\natural} := (Z_n, \varphi_{-2,n}, (\varphi_{-1,n}, \nu_{-1,n}), \varphi_{0,n}, \delta_n, c_n, c_n^{\vee}, \tau_n).  
\]
The moduli space of the  degeneration data will be the quotient of this parametrizion space by a group action identifying equivalent data. Without loss of generality, we assume that $\underline{X}$ and $\underline{Y}$ are constant with values $X$ and $Y$ as before. 

Since 
$( c,c^{\vee}, \tau)$
is determined by 
$(c_n, c_n^{\vee}, \tau_n)$, 
the data we aim to parametrize is  
\[
(Z_n, (X,Y, \phi, \varphi_{-2,n}, \varphi_{0,n}),  (A, \lambda_A, i_A, (\varphi_{-1,n}, \nu_{-1,n})),\delta_n,(c_n, c_n^{\vee}, \tau_n))
\]
where 
\[
\Phi_n:= (X,Y, \phi, \varphi_{-2,n}, \varphi_{0,n})
\]
describes the torus part of the degeneration and 
\[
(A, \lambda_A, i_A, (\varphi_{-1,n}, \nu_{-1,n}))
\]
characterizes the abelian part, both with level structure specified by $Z_n$. 
Moreover, 
$(c_n, c_n^{\vee}, \tau_n)$
contains the information on the extension between abelian and torus part, the periods, and a splitting of the monodromy filtration, which, together with $\delta_n$, determine the level structure on the generic fiber of the degenerating abelian variety. 

The data 
\[
(Z_n, (X,Y, \phi, \varphi_{-2,n}, \varphi_{0,n}), \delta_n)
\]
is discrete in nature, and the equivalence class of the tuple is called the cusp label. Indeed, two tuples 
$(Z_n, (X,Y, \phi, \varphi_{-2,n}, \varphi_{0,n}), \delta_n)$
and 
$(Z'_n, (X',Y', \phi', \varphi'_{-2,n}, \varphi'_{0,n}), \delta'_n)$
are defined to be equivalent if $Z_n = Z_n'$, and there exists $\mathcal{O}$-equivariant isomorphisms 
$\gamma_X : X' \overset{\sim}{\rightarrow} X$
and 
$\gamma_y : Y \overset{\sim}{\rightarrow} Y'$
such that 
$\phi = \gamma_X \phi' \gamma_Y$,
$\varphi'_{-2,n} = \gamma_X^{\vee} \varphi_{-2,n}$
and
$\varphi'_{0,n} = \gamma_Y \varphi_{0,n}$. 
Note that the equivalence classes is independent of the splitting $\delta_n$. The cusp labels are essentially equivalence classes of PEL torus. Following Lan's notation, we sometimes abbreviate  the notation $(Z_n, \Phi_n, \delta_n)$ to $(\Phi_n, \delta_n)$ for simplicity. This is mostly used in the indexing of various objects.   

The abelian part 
\[
(A, \lambda_A, i_A, (\varphi_{-1,n}, \nu_{-1,n}))
\]
is precisely a point of the moduli space of PEL abelian varieties $M_n$ with PEL data determined by $Gr_{-1}^{Z_n}$, which we denote by $M_n^{Z_n}$. By abuse of notation, we use $A$ to denote the universal PEL abelian variety over $M_n^{Z_n}$. 

Next, the homomorphisms $c_n$ and $c_n^{\vee}$ are parametrized by the group schemes
$\underline{Hom}_{\mathcal{O}}(\frac{1}{n}X, A^{\vee})$
and 
$\underline{Hom}_{\mathcal{O}}(\frac{1}{n}Y, A)$
over $M_n^{Z_n}$. Recall that $c_n$ and $c_n^{\vee}$ lifts $c$ and $c^{\vee}$, and the latter satisfies  the relation 
$\lambda_A \circ c^{\vee} = c \circ \phi$,
which is equivalent to
$(c_n, c_n^{\vee})$ 
lies in the group scheme
\[
\overset{\dots}{C}_{\Phi_n} :=
\underline{Hom}_{\mathcal{O}}(\frac{1}{n}X, A^{\vee}) \underset{\underline{Hom}_{\mathcal{O}}(Y, A^{\vee})}{\times}
\underline{Hom}_{\mathcal{O}}(\frac{1}{n}Y, A)
\]
where the first projection map is
$c_n \rightarrow c_n \circ \phi \circ (Y \hookrightarrow \frac{1}{n}Y)$, 
and the second one is 
$c^{\vee}_n \rightarrow \lambda_A \circ c_n^{\vee} \circ (Y \hookrightarrow \frac{1}{n}Y)$. 

Further, $(c_n, c_n^{\vee})$ are required to satisfy the relation 
\begin{equation} \label{annoying}
\langle  \varphi_{-1,n}( \cdot), (\lambda_A\circ c_n^{\vee} - c_n \circ \phi )\circ \varphi_{0,n}( \cdot) \rangle_A = \nu_{-1,n} \circ \langle   \cdot, \cdot\rangle_{10, n}
\end{equation}
and we want to find the subspace of the parametrization space cut out by this relation. Note that the equation
\[
\langle  \varphi_{-1,n}( \cdot), b_{\Phi_n, \delta_n} \circ \varphi_{0,n}( \cdot) \rangle_A = \nu_{-1,n} \circ \langle   \cdot, \cdot\rangle_{10, n}
\]
defines a liftable homomorphism
\[
b_{\Phi_n, \delta_n}: \frac{1}{n}Y/Y \longrightarrow A^{\vee}[n]
\]
and the relation (\ref{annoying}) is rewritten as 
\[
b_{\Phi_n, \delta_n} = \lambda_A\circ c_n^{\vee} - c_n \circ \phi.
\]
Thus the parametrization space we are searching for is the fiber at $b_{\Phi_n, \delta_n}$ of the homomorphism
\[
\partial_n : \overset{\dots}{C}_{\Phi_n} \longrightarrow 
\underline{Hom}_{\mathcal{O}}(\frac{1}{n}Y/Y, A^{\vee}[n])
\]
that sends 
$(c_n, c_n^{\vee})$ to 
$\lambda_A\circ c_n^{\vee} - c_n \circ \phi$. We denote it by 
\[
\overset{\dots}{C}_{\Phi_n, b_n}: = \partial_n^{-1}(b_{\Phi_n, \delta_n}).
\]

It can be shown that
$\overset{\dots}{C}_{\Phi_n,b_n} $ 
is a trivial torsor with respect to a commutative proper group scheme over $M_n^{Z_n}$, but it is not necessarily geometrically connected. However, the liftability condition on $(c_n, c_n^{\vee})$ singles out a connected component 
$C_{\Phi_n,b_n} $ 
of  
$\overset{\dots}{C}_{\Phi_n,b_n} $, 
which is an abelian scheme. Thus we see that the tuple 
\[
(Z_n, (X,Y, \phi, \varphi_{-2,n}, \varphi_{0,n}),  (A, \lambda_A, i_A, (\varphi_{-1,n}, \nu_{-1,n})),\delta_n,(c_n, c_n^{\vee}))
\]
is parametrized by 
\[
\underset{(Z_n, \Phi_n, \delta_n)}{\coprod} C_{\Phi_n,b_n}
\]
where 
$C_{\Phi_n,b_n}$
is an abelian scheme over $M_n^{Z_n}$. 

The next step is to include $\tau_n$ into the parametrization space. By construction, we have two universal homomorphisms $(c_n,c_n^{\vee}) $ over $C_{\Phi_n, b_n}$.
There is a map 
\[
\frac{1}{n}Y \times X \longrightarrow Pic(C_{\Phi_n, b_n})
\]
defined by 
$(\frac{1}{n}y, \chi) \rightarrow (c_n^{\vee}(\frac{1}{n}y), c_n(\chi))^* \mathcal{P}_A$.
The linearity and $\mathcal{O}$-equivariance of $(c_n,c_n^{\vee}) $ implies that it descends to a morphism
\[
\Psi_n : \overset{\dots}{S}_{\Phi_n} := \frac{1}{n}Y \otimes_{\mathbb{Z}} X/\{
\begin{subarray}{c}
y\otimes \phi(y') - y'\otimes \phi(y) \\
b \frac{1}{n}y \otimes \chi - (\frac{1}{n}y) \otimes (b^*\chi)
\end{subarray}
\}_{
\begin{subarray}{c}
y,y' \in Y \\
\chi \in X, b \in \mathcal{O}
\end{subarray}}
\longrightarrow Pic(C_{\Phi_n, b_n})
\]
and such that 
\[
\underset{l \in \overset{\dots}{S}_{\Phi_n}}{\oplus} \Psi_n(l)
\]
is an $\mathscr{O}_{C_{\Phi_n, b_n}}$-algebra, hence we have 
\[
\overset{\dots}{\Xi}_{\Phi_n, b_n} := \underline{Spec}_{\mathscr{O}_{C_{\Phi_n, b_n}}}(\underset{l \in \overset{\dots}{S}_{\Phi_n}}{\oplus} \Psi_n(l))
\]
which is a 
$\overset{\dots}{E}_{\Phi_n} := \underline{Hom}(\overset{\dots}{S}_{\Phi_n}, \textbf{G}_m)$-torsor. By construction, there is a universal trivialization 
\[
\tau_n: \textbf{1}_{\frac{1}{n}Y \times X} \overset{\sim}{\longrightarrow} 
(c^{\vee}_n \times c)^* \mathcal{P}_{A}^{\otimes -1}
\]
over 
$\overset{\dots}{\Xi}_{\Phi_n, b_n}$.

Note that $\overset{\dots}{E}_{\Phi_n}$ is not necessarily  a torus since 
$\overset{\dots}{S}_{\Phi_n}$
can have torsion elements. However, as explained in \cite{lan2013arithmetic} 6.2.3.17,  the liftability of $\tau_n$ together with the pairing condition in (6) of definition \ref{DDL} cut out a subspace 
$\Xi_{\Phi_n, \delta_n}$
of
$\overset{\dots}{\Xi}_{\Phi_n, b_n}$, 
which is a 
\[
E_{\Phi_n} := \underline{Hom}(S_{\Phi_n}, \textbf{G}_m)
\]-torsor over $C_{\Phi_n, b_n}$
where
$S_{\Phi_n}:= \overset{\dots}{S}_{\Phi_n, free}$
is the free part of 
$\overset{\dots}{S}_{\Phi_n}$. 

\begin{remark}
We have seen that the liftability condition restores connectivity in both $C_{\Phi_n, b_n}$ and $\Xi_{\Phi_n, \delta_n}$. This is a subtlety caused by the non-trivial endomorphism structure $\mathcal{O}$. In particular, it does not appear in the Siegel case treated in Faltings-Chai, where level-n structure  is liftable.   
\end{remark}

Thus we have seen that the tuple 
\[
(Z_n, (X,Y, \phi, \varphi_{-2,n}, \varphi_{0,n}),  (A, \lambda_A, i_A, (\varphi_{-1,n}, \nu_{-1,n})),\delta_n,(c_n, c_n^{\vee}), \tau_n)
\]
is parametrized by 
\[
\underset{(Z_n, \Phi_n, \delta_n)}{\coprod} \Xi_{\Phi_n,\delta_n}
\]
where 
\[
\Xi_{\Phi_n,\delta_n}:= \underline{Spec}_{\mathscr{O}_{C_{\Phi_n, b_n}}}(\underset{l \in S_{\Phi_n}}{\oplus} \Psi_n(l))
\]
is a 
$E_{\Phi_n}$-torsor over the abelian scheme 
$C_{\Phi_n, b_n}$ defined over $M_n^{Z_n}$.
This is almost the parametrization space we are searching for, except that we have not dealt with the positivity condition on $\tau$ (or $\tau_n$). Indeed, $\Xi_{\Phi_n,\delta_n}$ is the moduli space of "degeneration data over the the generic fiber", it remains to construct the boundary on which the data extends and the positivity condition holds universally.  

Recall the positivity condition for $\tau $ is that the morphism
\[
\tau(y, \phi(y)): (c^{\vee}(y) \times c \circ \phi(y))^* \mathcal{P}_{A, \eta}
\longrightarrow \mathcal{O}_{S , \eta} 
\]
extends to $S$ for all $y \in Y$ and that for $y \neq 0$, it factors through the ideal of definition of $S$. We have constructed a universal $\tau_n$ over $\Xi_{\Phi_n,\delta_n}$
by making
$(c_n^{\vee}(\frac{1}{n}y), c_n(\chi))^* \mathcal{P}_A = \Psi(\frac{1}{n}y \otimes \chi)$
part of the structure sheaf of the relatively affine scheme $\Xi_{\Phi_n,\delta_n}$ over $C_{\Phi_n, b_n}$. There is a natural way to compactify the $E_{\Phi_n}$-torsor
$\Xi_{\Phi_n, \delta_n}$, namely the toroidal compactification, which produces directions where $\tau$ can extend. However, this is non-canonical and depends on an auxiliary choice of cone decomposition, since there are infinitely many potential directions and it is necessary to make a choice.  

More precisely, the cocharacters of $E_{\Phi_n}$ is 
$S_{\Phi_n}^{\vee} := Hom(S_{\Phi_n}, \mathbb{Z})$, 
and the corresponding real vector space
$(S_{\Phi_n})_{\mathbb{R}}^{\vee}$
can be naturally  identified with the space of Hermitian pairings 
$(|\cdot, \cdot|) : Y \otimes_{\mathbb{Z}} \mathbb{R} \times Y \otimes_{\mathbb{Z}} \mathbb{R} \rightarrow \mathcal{O}\otimes_{\mathbb{Z}} \mathbb{R}  $. 
Let 
$\textbf{P}_{\Phi_n} \subset (S_{\Phi_n})_{\mathbb{R}}^{\vee}$
be the subset of positive semidefinite Hermitian pairings with admissible radical, i.e. its radical is the $\mathbb{R}$-span of some direct summand $\mathcal{O}$-submodule of $Y$. 
Let $\Sigma_{\Phi_n}= \{\sigma_j\}$ 
be a rational polyhedral cone decomposition of $\textbf{P}_{\Phi_n}$, and let 
$\sigma^{\vee} := \{ v \in S_{\Phi_n} | f(v) \geq 0,  \forall f \in \sigma\}$, 
then we have the natural toroidal compactification
\[
\overline{\Xi}_{\Phi_n, \delta_n, \Sigma_{\Phi_n}}
\]
of $\Xi_{\Phi_n, \delta_n}$, obtained by gluing together the relatively affine toroidal varieties 
\[
\Xi_{\Phi_n, \delta_n}(\sigma_j) := \underline{Spec}_{\mathscr{O}_{C_{\Phi_n, b_n}}}(\underset{l \in \sigma_j^{\vee}}{\oplus} \Psi_n(l)).
\]
Alternatively, we can define 
$\overline{\Xi}_{\Phi_n, \delta_n, \Sigma_{\Phi_n}}$
as 
$\Xi_{\Phi_n, \delta_n} \times ^{E_{\Phi_n}} \overline{E}_{\Phi_n, \Sigma_{\Phi_n}}$, 
where $\overline{E}_{\Phi_n, \Sigma_{\Phi_n}}$
is the classical toric variety associated to the cone $\Sigma_{\Phi_n}$ (viewed as schemes over $C_{\Phi_n, b_n}$).

Note that the toroidal compactification 
$\overline{\Xi}_{\Phi_n, \delta_n, \Sigma_{\Phi_n}}$
has the universal property as follows. 
If $ \overline{S}$ is a Noetherian scheme over
$C_{\Phi_n, b_n}$ with $S \subset \overline{S}$ a dense open subscheme,  and 
$S \rightarrow \Xi_{\Phi_n, \delta_n}$
is a morphism defined over $C_{\Phi_n, b_n}$, then it extends to a morphism 
\[
\overline{S} \longrightarrow \overline{\Xi}_{\Phi_n, \delta_n, \Sigma_{\Phi_n}}
\]
over $C_{\Phi_n, b_n}$ if and only if 
for each geometric point $x$ of $\overline{S}$,  every dominant morphism 
$Spec(V) \rightarrow \overline{S}$ centered at $x$, with $V$ a discrete valuation ring, the associated character 
\[
S_{\Phi_n} \longrightarrow \mathbb{Z}
\]
lies in the closure $\overline{\sigma}$ for some 
$\sigma \in \Sigma_{\Phi_n}$ ($\sigma$ depends only on $x$). The naturally associated character is defined as follows. Let 
$p: \overline{S} \rightarrow C_{\Phi_n, b_n}$
be the structure morphism, then we have the commutative diagram
\[
\begin{tikzcd}
\eta: = \text{Spec}(\text{Frac}(V)) \arrow[r] \arrow[d, hook] & S \arrow[d, hook] \arrow[r] & \Xi_{\Phi_n, \delta_n} \arrow[d, ]
\\
\text{Spec}(V) \arrow[r, "f"] & \overline{S} \arrow[r, "p"]  & C_{\Phi_n, b_n}
\end{tikzcd}
\]
The generic fiber
$(f^* p^* \Psi_n(l))_{\eta}$ 
of the line bundle 
$f^* p^* \Psi_n(l)$ 
is equipped with a natural trivialization since it factorizes thorough $\Xi_{\Phi_n, \delta_n}$
as the top row of the diagram shows, while the line bundles 
$\Psi_n(l)$ on $\Xi_{\Phi_n, \delta_n}$ has a canonical trivialization by construction. 
Now under this trivialization, 
$f^* p^* \Psi_n(l)$
is identified with a $V$-submodule $I_l$ of $\text{Frac}(V)$, we define the desired character
$S_{\Phi_n} \longrightarrow \mathbb{Z}$
by sending $l$ to the lower bound of the valuation of elements of $I_l \subset \text{Frac}(K)$. In other words, $I_l = V \pi^{m_l} \subset \text{Frac}(V)$   with $\pi \in V$ the uniformizer and $m_l \in \mathbb{Z}$, then $l$ is sent to $m_l$.

The universal property follows simply by unravelling the definition of toroidal embedding. This is important because it is the ultimate origin of the universal property of toroidal compactifications of Shimura varieties to be discussed below. The formulation is useful because in the situation we will consider, $I_l$ can be directly read off from the degeneration data of a  degenerating abelian variety over $V$. 

We have now constructed the "moduli space" of the tuple 
\[
(Z_n, (X,Y, \phi, \varphi_{-2,n}, \varphi_{0,n}),  (A, \lambda_A, i_A, (\varphi_{-1,n}, \nu_{-1,n})),\delta_n,(c_n, c_n^{\vee}, \tau_n))
\]
with a specified direction of degeneration, namely
\[
\underset{(Z_n, \Phi_n, \delta_n)}{\coprod} \overline{\Xi}_{\Phi_n, \delta_n, \Sigma_{\Phi_n}},
\]
over which there is a universal degeneration data, and we would like to find the associated degenerating abelian variety. However, the equivalence between degenerating abelian varieties and degeneration data holds only over a complete base, so the correct object to consider is the completion of 
$\overline{\Xi}_{\Phi_n, \delta_n, \Sigma_{\Phi_n}}$
along the boundary, which we denote by 
$\mathfrak{X}_{\Phi_n, \delta_n, \Sigma_{\Phi_n}}$, 
and there is a universal degenerating abelian variety over 
$\mathfrak{X}_{\Phi_n, \delta_n, \Sigma_{\Phi_n}}$.
More precisely, the equivalence between degeneration data and degenerating abelian varieties are only proved over complete affine base, and 
$\overline{\Xi}_{\Phi_n, \delta_n, \Sigma_{\Phi_n}}$
is not affine in general. However, the global degeneration data defines degenerating abelian varieties Zariski locally, and it is not hard to use the functoriality of Mumford's construction to glue them together to obtain a global one.  

The next step is to quotient out the equivalence relation to find the moduli space of degeneration data. Recall that the equivalence classes 
$[(Z_n, \Phi_n, \delta_n)]$
are called cusp labels, which subsumes the ambiguity caused by equivalence classes of $\alpha_n^{\natural}$. Further, we need to take care of the isomorphism classes in the category of degeneration data, and this is described by the action of the automorphism group. We choose a representative 
$(Z_n, \Phi_n, \delta_n)$
for each cusp label, then the automorphism group of the chosen label is 
\[
\Gamma_{\Phi_n} := \{(\gamma_X, \gamma_Y) \in GL_{\mathcal{O}}(X) \times GL_{\mathcal{O}}(Y) | \varphi_{-2,n} = \gamma_X^{\vee} \varphi_{-2,n}, \varphi_{0,n}= \gamma_Y \varphi_{0,n}, \phi = \gamma_X \phi \gamma_Y \}
\]
which acts on $\Xi_{\Phi_n, \delta_n}$ and $\textbf{P}_{\Phi_n}$. We choose the cone decomposition $\Sigma_{\Phi_n}$ to be $\Gamma_{\Phi_n}$-admissible, i.e. 
$\gamma \sigma \in \Sigma_{\Phi_n}$ 
for all $\gamma \in \Gamma_{\Phi_n}$ and $\sigma \in \Sigma_{\Phi_n}$,
and the action of  $\Gamma_{\Phi_n}$ on $\Sigma_{\Phi_n}$  
has finitely many orbits. Under this condition on $\Sigma_{\Phi_n}$, the action of $\Gamma_{\Phi_n}$ extends to 
$\overline{\Xi}_{\Phi_n, \delta_n, \Sigma_{\Phi_n}}$, 
hence also on 
$\mathfrak{X}_{\Phi_n, \delta_n, \Sigma_{\Phi_n}}$,
and the moduli space of degeneration data is 
\[
\underset{[(Z_n, \Phi_n, \delta_n)]}{\coprod} \mathfrak{X}_{\Phi_n, \delta_n, \Sigma_{\Phi_n}}/ \Gamma_{\Phi_n}
\]
where we choose a representative $(Z_n, \Phi_n, \delta_n)$
for each cusp label, and 
$\mathfrak{X}_{\Phi_n, \delta_n, \Sigma_{\Phi_n}}/ \Gamma_{\Phi_n}$
is constructed with respect to this choice.  The degenerating abelian variety on 
$\mathfrak{X}_{\Phi_n, \delta_n, \Sigma_{\Phi_n}}$
descends to 
$\mathfrak{X}_{\Phi_n, \delta_n, \Sigma_{\Phi_n}}/ \Gamma_{\Phi_n}$ 
if a technical condition on $\Sigma_{\Phi_n}$ is satisfied, see \cite{lan2013arithmetic} 6.2.5.25, which we assume from now on.

Now the degenerating PEL abelian variety over 
$\mathfrak{X}_{\Phi_n, \delta_n, \Sigma_{\Phi_n}}/ \Gamma_{\Phi_n}$ 
is a PEL abelian variety over the generic fiber, hence defining a map from the generic fiber of 
$\mathfrak{X}_{\Phi_n, \delta_n, \Sigma_{\Phi_n}}/ \Gamma_{\Phi_n}$
to the moduli space $M_n$. An appropriate algebraization of these attaching maps will provide gluing maps along neighbourhoods of the boundary of the toroidal compactification of $M_n$. Hence we can glue them together to obtain the toroidal compactification. In order for the gluing process to work well, it is necessary to choose the cones $\Sigma_{\Phi_n}$ to be compatible for different $\Phi_n$, see \cite{lan2013arithmetic} 6.3.3.4 for details.

We remark that the algebraization process is very delicate and  not canonical. As a consequence, it is difficult to describe the Zariski neighborhood of the boundary. On the other hand, since the boundary is glued by the algebraization of a formal scheme which we constructed rather explicitly, we have a nice description of the formal neighborhood of the boundary, which is nothing but 
$\mathfrak{X}_{\Phi_n, \delta_n, \Sigma_{\Phi_n}}/ \Gamma_{\Phi_n}$. 
This also tells us what the boundary looks like, which is simply the support of  
$\mathfrak{X}_{\Phi_n, \delta_n, \Sigma_{\Phi_n}}/ \Gamma_{\Phi_n}$. 
Moreover, the universal property that we described for $\overline{\Xi}_{\Phi_n, \delta_n, \Sigma_{\Phi_n}}$
survives  all those completion, algebraization and gluing procedure, and
is transformed to a universal property for the toroidal compactification of $M_n$. 

To summarize, we have the following theorem.

\begin{theorem} \label{lantor}
(Lan \cite{lan2013arithmetic} 6.4.1.1))
To each compatible choice $\Sigma = \{\Sigma_{\Phi_n}\}_{[(\Phi_n,\delta_n)]}$ 
of admissible smooth rational polyhedral cone decomposition as in \cite{lan2013arithmetic} 6.3.3.4, 
there is an associated algebraic stack 
$M_{n, \Sigma}^{tor}$
(which is a scheme when $n > 3$) 
proper and smooth over $Spec(\mathcal{O}_{F_0, (\square)})$, containing $M_n$ as an open dense subspace whose complement consists of normal crossing divisors, together with a degenerating abelian variety 
\[
(G, \lambda, i, (\alpha_n, \nu_n))
\]
over $M_{n, \Sigma}^{tor}$ as in definition \ref{degpel}, such that we have the following:

(1) The restriction of 
$(G, \lambda, i, (\alpha_n, \nu_n))$
to $M_n$ is the universal PEL abelian variety over $M_n$. 

(2)$M_{n, \Sigma}^{tor}$ has a stratification by locally closed subschemes 
\[
M_{n, \Sigma}^{tor} = \underset{[(\Phi_n, \delta_n, \sigma)]}{\coprod } Z_{[(\Phi_n, \delta_n, \sigma)]}
\]
where 
$\sigma \in \Sigma_{\Phi_n} $ 
and 
$[(\Phi_n, \delta_n, \sigma)]$
are the equivalence classes of the tuples 
$(\Phi_n, \delta_n, \sigma)$, 
which are the obvious refinement of the equivalences used to define cusp labels, namely by requiring the isomorphisms to preserve $\sigma$, see \cite{lan2013arithmetic} 6.2.6.1 for details. Note that we suppress $Z_n$ in the notation, following Lan. 

The formal completion
$(M_{n,\Sigma}^{tor})^{\land}_{Z_{[(\Phi_n, \delta_n, \sigma)]}} $
of $M_{n,\Sigma}^{tor}$ 
along
$Z_{[(\Phi_n, \delta_n, \sigma)]}$
is canonically isomorphic to 
$\mathfrak{X}_{\Phi_n, \delta_n, \sigma}/ \Gamma_{\Phi_n,\sigma}$, 
\[
(M_{n,\Sigma}^{tor})^{\land}_{Z_{[(\Phi_n, \delta_n, \sigma)]}} \cong
\mathfrak{X}_{\Phi_n, \delta_n, \sigma}/ \Gamma_{\Phi_n, \sigma}
\]
where 
$\mathfrak{X}_{\Phi_n, \delta_n, \sigma}$
is the formal completion of 
\[
\overline{\Xi}_{\Phi_n, \delta_n}( \sigma) := \underline{Spec}_{\mathscr{O}_{C_{\Phi_n, b_n}}}(\underset{l \in \sigma^{\vee}}{\oplus} \Psi_n(l))
\]
along the boundary 
$\underline{Spec}_{\mathscr{O}_{C_{\Phi_n, b_n}}}(\underset{l \in \sigma^{\perp}}{\oplus} \Psi_n(l))$, 
with 
$\sigma^{\perp} := \{x \in S_{\Phi_n} | \langle  x,y\rangle=0,  \forall y \in \sigma\}$. 
The scheme 
$\overline{\Xi}_{\Phi_n, \delta_n}( \sigma)$
is a relative affine toroidal variety over a $E_{\Phi_n}$-torsor over $C_{\Phi_n, b_n}$, which is an abelian scheme over the PEL moduli space $M_n^{Z_n}$ with PEL data specified by $Gr_{-1}^{Z_n}$. Then the strata 
$Z_{[(\Phi_n, \delta_n, \sigma)]}$
is isomorphic to the support of 
$\mathfrak{X}_{\Phi_n, \delta_n, \sigma}/ \Gamma_{\Phi_n, \sigma}$.
If $n>3$, then the action of $\Gamma_{\Phi_n, \sigma}$
is trivial, and 
$\mathfrak{X}_{\Phi_n, \delta_n, \sigma}/ \Gamma_{\Phi_n, \sigma} \cong \mathfrak{X}_{\Phi_n, \delta_n, \sigma}$.

(3) If $S$ is an irreducible Noetherian normal scheme over $Spec(\mathcal{O}_{F_0, (\square)})$
over which we have a degenerating PEL abelian variety $(G^{\dagger}, \lambda^{\dagger}, i^{\dagger}, (\alpha_n^{\dagger}, \nu_n^{\dagger}))$
as in definition \ref{degpel} (with the same PEL data as that of $M_n$), then there exists a morphism
\[
S \longrightarrow M_{n, \Sigma}^{tor}
\]
over 
$Spec(\mathcal{O}_{F_0, (\square)})$
such that 
$(G^{\dagger}, \lambda^{\dagger}, i^{\dagger}, (\alpha_n^{\dagger}, \nu_n^{\dagger}))$
is the pull back of 
$(G, \lambda, i, (\alpha_n, \nu_n))$
if and only if the following condition is satisfied: 

\vspace{3mm}
For each geometric point $\bar{s}$ of $S$, and any dominant morphism $Spec(V) \rightarrow S$ centered at $\bar{s}$ with $V$ a complete discrete valuation ring, let 
$(G^{\ddagger}, \lambda^{\ddagger}, i^{\ddagger}, (\alpha_n^{\ddagger}, \nu_n^{\ddagger}))
$
be the pullback of 
$(G^{\dagger}, \lambda^{\dagger}, i^{\dagger}, (\alpha_n^{\dagger}, \nu_n^{\dagger}))$
along $Spec(V) \rightarrow S$, 
then the theorem on degeneration data provides us with the degeneration data
\[
(A^{\ddagger}, \lambda_{A^{\ddagger}}^{\ddagger}, i_{A^{\ddagger}}^{\ddagger}, X^{\ddagger}, Y^{\ddagger}, \phi^{\ddagger}, c^{\ddagger},(c^{\vee})^{\ddagger}, \tau^{\ddagger},  [(\alpha_n^{\natural})^{\ddagger}]) 
\]
over $V$. Note that $X$ and $Y$ are automatically constant (hence the notation) since $V$ is a complete discrete valuation ring. Moreover, $(Z_n^{\ddagger}, \Phi_n^{\ddagger})$ 
is determined by $[(\alpha_n^{\natural})^{\ddagger}]$. 
Let $\eta= Spec(K)$ be the generic fiber of $Spec(V)$, the isomorphism 
\[
\tau^{\ddagger}: 
((c^{\vee})^{\ddagger} \times c^{\ddagger})^* \mathcal{P}_{A^{\ddagger}, \eta}  \overset{\sim}{\longrightarrow} 
\textbf{1}_{Y^{\ddagger} \times X^{\ddagger}, \eta}
\]
defines a trivialization of the generic fiber of the invertible sheaf 
$
((c^{\vee})^{\ddagger}(y) \times c^{\ddagger}(\chi))^* \mathcal{P}_{A^{\ddagger}} 
$
on $Spec(V)$ for each 
$y \in Y^{\ddagger}$,
$\chi \in X^{\ddagger}$, 
with which we can identify
$
((c^{\vee})^{\ddagger}(y) \times c^{\ddagger}(\chi))^* \mathcal{P}_{A^{\ddagger}} 
$
with an invertible  $V$-submodule $I_{y, \chi}$ of $K$. This defines a morphism
\[
 Y^{\ddagger} \times X^{\ddagger} \rightarrow \text{Inv}(V)
\]
\[
(y, \chi) \rightarrow I_{y,\chi}
\]
with $\text{Inv}(V)$  the group of invertible $V$-modules (submodules of $K$). 
We can show that it extends to
\[
 \frac{1}{n}Y^{\ddagger} \times X^{\ddagger} \longrightarrow \text{Inv}(V), 
\]
which descends to a homomorphism
\[
B^{\ddagger}: S_{\Phi_n} \longrightarrow \text{Inv}(V)
\]
composed with the natural identification 
$\text{Inv}(V) \cong \mathbb{Z}$ defined by $\pi^m V \longleftrightarrow m$ with $\pi$ the uniformizer, we obtain a homomorphism
\[
v \circ B^{\ddagger}: S_{\Phi_n} \longrightarrow \mathbb{Z}
\]
which is an element of $S_{\Phi_n}^{\vee}$. 
The upshot is that we associate an element
$v \circ B^{\ddagger} \in S_{\Phi_n}^{\vee}$
for each dominant morphism $\text{Spec}(V) \rightarrow S$ centered at $\bar{s}$ with $V$ a complete discrete valuation ring. 

Then the condition is that for some choice of $\delta_n^{\ddagger}$ making 
$(Z_n^{\ddagger}, \Phi_n^{\ddagger}, \delta_n^{\ddagger})$ a representative of a cusp label, there is a cone $\sigma^{\ddagger} \in \Sigma_{\Phi_n^{\ddagger}}$
depending only on $\bar{s}$
such that $v \circ B^{\ddagger} \in \overline{\sigma}^{\ddagger}$
for all those $v \circ B^{\ddagger}$ coming from a dominant $Spec(V) \rightarrow S$ centered at $\bar{s}$ with $V$ a complete discrete valuation ring, 
where 
$\overline{\sigma}^{\ddagger}$
is the closure of 
$\sigma^{\ddagger}$.

\end{theorem}

\subsection{Partial Frobenius extends to toroidal compactifications}

Now we can prove the main technical results on the extension of partial Frobenius to toroidal compactifications. We follow the notations of section \ref{thpartialf}. In particular, we assume that 
\[
p \text{ splits completely in the center } F^c \text{ of } B,
\]
and the moduli problems 
\[
 M_{ K(n)} / \Delta = \underset{\begin{subarray}{c}
  \alpha \in \Omega \\
  \delta \in \Lambda
  \end{subarray}}{\coprod} 
  M_n(L, Tr_{\mathcal{O}_F/ \mathbb{Z}} \circ (\alpha \delta \langle   \cdot, \cdot\rangle_F))
\]
are defined over 
$\mathcal{O}_{F_0} \otimes_{\mathbb{Z}} \mathbb{F}_p$, 
where $\Omega$ and $\Lambda$ are  fixed sets of representatives of 
the double quotients 
\[
(F \otimes \mathbb{A}^{(p\infty)})^{\times} = \underset{\alpha \in \Omega}{\coprod}(\mathcal{O}_F \otimes \mathbb{Z}_{(p)})^{\times}_{+} \alpha (\mathcal{O}_F \otimes \hat{\mathbb{Z}}^p)^{\times}
\]
\[
 (\mathcal{O}_F \otimes \hat{\mathbb{Z}}^p)^{\times} = \underset{\delta \in \Lambda} {\coprod} (\mathcal{O}_F)^{\times}_+ \delta (\nu(K(n)) \hat{\mathbb{Z}}^{p,\times} )
\]

Let $p = \underset{i}{\prod} \mathfrak{p}_i$
be the decomposition of $p$ in $F$, and we will focus on a single $\mathfrak{p}_i$ from now on. 
We fix a $\xi \in F_+^{\times}$ satisfying 
$v_{\mathfrak{p}_i} (\xi)=1$ 
and $v_{\mathfrak{p}_{i'}}(\xi) =0$ for $i' \neq i$. 
Recall that the partial Frobenius 
\[
F_{\mathfrak{p}_i}: M_{ K(n)} / \Delta \longrightarrow M_{ K(n)} / \Delta
\]
is defined by union of the maps
\[
M_n(L, Tr_{\mathcal{O}_F/ \mathbb{Z}} \circ (\alpha \delta \langle   \cdot, \cdot\rangle_F)) \rightarrow 
M_n(L, Tr_{\mathcal{O}_F/ \mathbb{Z}} \circ (\alpha' \delta' \langle   \cdot, \cdot\rangle_F))
\]
with $\alpha' \in \Omega$, $\delta' \in \Lambda$ characterized by
\begin{equation} \label{wwwww}
\xi \alpha  = \epsilon \alpha' \lambda
\end{equation}
\begin{equation} \label{wawawaw}
\lambda \delta = \epsilon_0 \delta' \gamma
\end{equation}
where  
$\epsilon \in (\mathcal{O}_F \otimes \mathbb{Z}_{(p)})^{\times}_{+}$,
$\lambda \in (\mathcal{O}_F \otimes \hat{\mathbb{Z}}^p)^{\times}$, 
$\epsilon_0 \in \mathcal{O}_{F,+}^{\times}$ 
and 
$\gamma \in (\nu(K(n)) \hat{\mathbb{Z}}^{p,\times} )$
as in the above two double quotients of
$(F \otimes \mathbb{A}^{(p\infty)})^{\times}$
and 
$ (\mathcal{O}_F \otimes \hat{\mathbb{Z}}^p)^{\times}$.
The map is defined by
\[
(A, \lambda, i, (\alpha_n, \nu_n)) \longrightarrow
(A', \lambda', i', (\alpha'_n, \nu'_n)), 
\]
where 
\[
A' := A/ (Ker(F)[\mathfrak{p}_i]),
\]
$i'$ is induced by the quotient map $\pi_{\mathfrak{p}_i} : A \rightarrow A'$,  $\lambda'$ is characterized by 
$\xi \lambda = \pi_{\mathfrak{p}_i}^{\vee} \circ \lambda' \circ \pi_{\mathfrak{p}_i}$ which defines a quasi-isogeny $\lambda'$, 
$\alpha_n' = \pi_{\mathfrak{p}_i} \circ \alpha_n$
and 
$\nu_n' = \nu_n \circ \kappa $.

In the last equality, 
we fix a set of representatives of 
$\hat{\mathbb{Z}}^{p,\times} /\nu(\mathcal{U}(n)) \cong (\mathbb{Z}/ n\mathbb{Z})^{\times}$, 
which defines 
\[\nu(K(n)) \hat{\mathbb{Z}}^{p,\times} = \underset{\kappa }{\coprod} \nu(K(n)) \kappa
\]
where 
$\kappa \in \hat{\mathbb{Z}}^{p,\times}$ 
ranges over the chosen representatives. Then the $\kappa$ in the equality   
$\nu_n' = \nu_n \circ \kappa $
is defined by 
\[
\gamma = \beta \kappa
\]
where $\beta \in \nu(K(n))$ 
and $\gamma$ is obtained from the equation 
$\lambda \delta = \epsilon_0 \delta' \gamma$
as above. 

\begin{remark} \label{boring}
The above procedure can be performed to any PEL abelian variety 
$(A, \lambda, i, (\alpha_n, \nu_n))$
defined over a base scheme $S$ over $\mathcal{O}_{F_0} \otimes_{\mathbb{Z}} \mathbb{F}_p$, 
and obtains a new PEL abelian variety
$(A', \lambda', i', (\alpha'_n, \nu'_n))$
over $S$, 
which is nothing but the map $F_{\mathfrak{p}_i}$
on $S$-points.
\end{remark}

Now let 
$\Sigma = \{ \Sigma_{\alpha \delta} \} _{\alpha \in \Omega, \delta \in \Lambda}$, 
where $\Sigma_{\alpha \delta}$ is a compatible choice of admissible smooth rational polyhedral cone decomposition with respect to the PEL moduli variety 
$M_n(L, Tr_{\mathcal{O}_F/ \mathbb{Z}} \circ (\alpha \delta \langle   \cdot, \cdot\rangle_F))$. 
Hence each $\Sigma_{\alpha \delta}$ determines a toroidal compactification 
$M_n(L, Tr_{\mathcal{O}_F/ \mathbb{Z}} \circ (\alpha \delta \langle   \cdot, \cdot\rangle_F))_{\Sigma_{\alpha \delta}}^{tor}$, 
and the union of which defines the toroidal compactifiaction 
\[
(M_{K(n)}/ \Delta )_{\Sigma}^{tor} :=
\underset{\begin{subarray}{c}
  \alpha \in \Omega \\
  \delta \in \Lambda
  \end{subarray}}{\coprod} 
  M_n(L, Tr_{\mathcal{O}_F/ \mathbb{Z}} \circ (\alpha \delta \langle   \cdot, \cdot\rangle_F))_{\Sigma_{\alpha \delta}}^{tor}
\]
Moreover, the union of the strata
\[
M_n(L, Tr_{\mathcal{O}_F/ \mathbb{Z}} \circ (\alpha \delta \langle   \cdot, \cdot\rangle_F))_{\Sigma_{\alpha \delta}}^{tor}=
\underset{[(\Phi_{\alpha\delta,n}, \delta_{\alpha\delta,n}, \sigma_{\alpha\delta})]}{\coprod } Z_{[(\Phi_{\alpha\delta,n}, \delta_{\alpha\delta,n}, \sigma_{\alpha\delta})]}
\]
defines 
\begin{equation} \label{complication}
(M_{K(n)}/ \Delta )_{\Sigma}^{tor} =
\underset{\begin{subarray}{c}
  \alpha \in \Omega \\
  \delta \in \Lambda
  \end{subarray}}{\coprod}
  \underset{[(\Phi_{\alpha\delta,n}, \delta_{\alpha\delta,n}, \sigma_{\alpha\delta})]}{\coprod } Z_{[(\Phi_{\alpha\delta,n}, \delta_{\alpha\delta,n}, \sigma_{\alpha\delta})]}
\end{equation}

We now state the main result of this section.

\begin{theorem}
The partial Frobenius 
$F_{\mathfrak{p}_i}: M_{ K(n)} / \Delta \longrightarrow M_{ K(n)} / \Delta$ 
extends to a map 
\[
F_{\mathfrak{p}_i}: (M_{K(n)}/ \Delta )_{\Sigma}^{tor}  \longrightarrow (M_{K(n)}/ \Delta )_{\Sigma'}^{tor}
\]
where 
$\Sigma' = \{ \Sigma'_{\alpha \delta} \}_{\alpha \in \Omega, \delta \in \Lambda}$
is characterized as follows:  

 \vspace{3mm}
First, for each 
$[(Z_{\alpha \delta,n}, \Phi_{\alpha\delta,n}, \delta_{\alpha\delta,n})]$
we associate another 
$[(Z'_{\alpha' \delta',n}, \Phi'_{\alpha'\delta',n}, \delta'_{\alpha'\delta',n})]$
as follows: 
$\alpha' \in \Omega$, $\delta' \in \Lambda$ 
are determined by $\alpha$ and $\delta$
as in (\ref{wwwww}) and (\ref{wawawaw}), then
\[
Z'_{\alpha' \delta',n} = Z_{\alpha \delta,n}, 
\]
and if 
$ \Phi_{\alpha \delta,n} =(X,Y,\phi, \varphi_{-2,n}, \varphi_{0,n}) $,
we define
\[
\Phi'_{\alpha' \delta',n} = (X \otimes_{\mathcal{O}_F} \mathfrak{p}_i,Y,\phi', \varphi_{-2,n}', \varphi'_{0,n}) \]
where 
\[
\varphi'_{-2,n} : Gr_{-2}^{Z_{\alpha' \delta',n} } = Gr_{-2}^{Z_{\alpha \delta,n} }  \overset{\varphi_{-2,n}}{\longrightarrow} Hom(X/nX, (\mathbb{Z}/n \mathbb{Z})(1)) 
\]
\[
\overset{\sim}{\longrightarrow} Hom(X \otimes \mathfrak{p}_i/n (X \otimes \mathfrak{p}_i), (\mathbb{Z}/n \mathbb{Z})(1))
\]
and 
\[
\varphi_{0,n}' : Gr_{0}^{Z_{\alpha' \delta',n} } = Gr_{0}^{Z_{\alpha \delta,n} }  \overset{\varphi_{0,n}}{\longrightarrow} Y/nY.
\]
Further, $\phi'$ is defined by the diagram 
\[\begin{tikzcd}
[
  arrow symbol/.style = {draw=none,"\textstyle#1" description,sloped},
  isomorphic/.style = {arrow symbol={\simeq}},
  ]
&
X  & 
X \otimes_{\mathcal{O}_F} \mathfrak{p}_i 
\arrow [l, "id \otimes (  \mathcal{O}_F \hookleftarrow \mathfrak{p}_i)"] 
\\
Y \otimes_{\mathcal{O}_F} \mathfrak{p}_i^{-1} &
Y \arrow[l, "id \otimes (\mathfrak{p}_i^{-1} \hookleftarrow \mathcal{O}_F)"] 
\arrow[u, "\phi"] 
 & 
 Y \otimes_{\mathcal{O}_F} \mathfrak{p}_i 
 \arrow[l,"id \otimes (  \mathcal{O}_F \hookleftarrow \mathfrak{p}_i)"]
 \arrow[u, "\phi \otimes id"]
 \\
&
Y \otimes_{\mathcal{O}_F} \mathfrak{p}_i^{-1} 
\arrow[lu, "\xi \otimes id"]
\arrow[u, dashrightarrow] & 
Y
\arrow[l,"id \otimes (\mathfrak{p}_i^{-1} \hookleftarrow \mathcal{O}_F)"]
\arrow[u, dashrightarrow]
\arrow[uu, bend right=90, dashrightarrow, "\phi'"]
\end{tikzcd} 
\]

Now for every
$\sigma \in  \Sigma_{\Phi_{\alpha\delta, n}} $,
we associate 
$\sigma' \in  \Sigma'_{\Phi_{\alpha'\delta', n}} $
by ($\otimes \mathbb{R}$ of) the pullback map  $S_{\Phi_{\alpha \delta}}^{\vee} \rightarrow S_{\Phi_{\alpha' \delta'}}^{\vee}$ 
(recall that $S_{\Phi_{\alpha \delta}}^{\vee} $ is the set of bilinear pairings $Y \times X \rightarrow \mathbb{Z}$ which are $\mathcal{O}$-compatible and becomes symmetric Hermitian once we pullback to $Y \times Y$ along $\phi$) 
induced by the natural map
\[
Y \times (X \otimes \mathfrak{p}_i) \longrightarrow 
Y \times X
\]
More precisely, the pullback map induces an isomorphism
$(S_{\Phi_{\alpha \delta}})^{\vee}_{ \mathbb{R}} \overset{\sim}{\rightarrow}
(S_{\Phi_{\alpha' \delta'}})^{\vee}_{\mathbb{R}}$
preserving positive semi-definite pairings, hence defining an identification 
$\Sigma_{\Phi_{\alpha\delta, n}} \overset{\sim}{\rightarrow}
\Sigma'_{\Phi_{\alpha'\delta', n}}$,
and we define
$\Sigma'_{\alpha'\delta'} = \{ \Sigma'_{\Phi'_{\alpha'\delta', n}} \}_{[(\Phi'_{\alpha'\delta',n}, \delta'_{\alpha'\delta',n})]} $.

\vspace{3mm}

Moreover, with the association as described above, the map 
$F_{\mathfrak{p}_i}$
sends 
$Z_{[(\Phi_{\alpha\delta,n}, \delta_{\alpha\delta,n}, \sigma_{\alpha\delta})]}$
to
$Z_{[(\Phi'_{\alpha'\delta',n}, \delta'_{\alpha'\delta',n}, \sigma'_{\alpha'\delta'})]}$. 

\end{theorem}

\begin{proof}
It is enough to prove that $F_{\mathfrak{p}_i}$ extends to the toroidal compactification on each component, i.e.  
\[
M_n(L, Tr_{\mathcal{O}_F/ \mathbb{Z}} \circ (\alpha \delta \langle   \cdot, \cdot\rangle_F)) \rightarrow 
M_n(L, Tr_{\mathcal{O}_F/ \mathbb{Z}} \circ (\alpha' \delta' \langle   \cdot, \cdot\rangle_F))
\]
extends to a morphism
\[
M_n(L, Tr_{\mathcal{O}_F/ \mathbb{Z}} \circ (\alpha \delta \langle   \cdot, \cdot\rangle_F))^{tor}_{\Sigma_{\alpha\delta}} \rightarrow 
M_n(L, Tr_{\mathcal{O}_F/ \mathbb{Z}} \circ (\alpha' \delta' \langle   \cdot, \cdot\rangle_F))^{tor}_{\Sigma'_{\alpha'\delta'}}
\]
and maps strata to the expected ones. This reduces the question to toroidal compactifications of Kottwitz's PEL moduli varieties, and we can apply the general machinery of Lan, and in particular the universal property in theorem \ref{lantor}. 

The idea is very simple. Let $G$ be the universal semi-abelian variety (with extra structures on the open part) over 
$M_n(L, Tr_{\mathcal{O}_F/ \mathbb{Z}} \circ (\alpha \delta \langle   \cdot, \cdot\rangle_F))^{tor}_{\Sigma_{\alpha\delta}}$. 
Since the partial Frobenius sends 
$A$ to $A/ (Ker(F)[\mathfrak{p}_i])$, 
it is natural to extend the map on semi-abelian varieties by the same formula  $G \rightarrow G/ (Ker(F)[\mathfrak{p}_i])$ (and take care of the extra structures), and this simple idea does indeed work. More precisely, we can define a new semi-abelian scheme 
\[
G/ (Ker(F)[\mathfrak{p}_i])
\]
(with extra structures on the part by definition of partial Frobenius) 
over 
$M_n(L, Tr_{\mathcal{O}_F/ \mathbb{Z}} \circ (\alpha \delta \langle   \cdot, \cdot\rangle_F))^{tor}_{\Sigma_{\alpha\delta}}$,
and we would like it to be the pullback of the universal semi-abelian variety over 
$M_n(L, Tr_{\mathcal{O}_F/ \mathbb{Z}} \circ (\alpha' \delta' \langle   \cdot, \cdot\rangle_F))^{tor}_{\Sigma'_{\alpha'\delta'}}$ 
through a morphism
\[
M_n(L, Tr_{\mathcal{O}_F/ \mathbb{Z}} \circ (\alpha \delta \langle   \cdot, \cdot\rangle_F))^{tor}_{\Sigma_{\alpha\delta}} \rightarrow 
M_n(L, Tr_{\mathcal{O}_F/ \mathbb{Z}} \circ (\alpha' \delta' \langle   \cdot, \cdot\rangle_F))^{tor}_{\Sigma'_{\alpha'\delta'}}.
\]
The universal property of 
$M_n(L, Tr_{\mathcal{O}_F/ \mathbb{Z}} \circ (\alpha' \delta' \langle   \cdot, \cdot\rangle_F))^{tor}_{\Sigma'_{\alpha'\delta'}}$
tells us exactly when this happens, and all we need to do is to verify the semi-abelian variety 
$G/ (Ker(F)[\mathfrak{p}_i])$
satisfies the condition of the universal property. This amounts to finding the period of the degenerating abelian variety, or more precisely the bilinear pairing $v \circ B^{\ddagger} \in S_{\Phi_{\alpha'\delta'}}^{\vee}$
in the notation of theorem \ref{lantor}, which is defined through the assoicated degeneration data. Hence we need to find the degeneration data of
$G/ (Ker(F)[\mathfrak{p}_i])$. More precisely, given the degeneration data of $G$, we aim to write the degeneration data of 
$G/ (Ker(F)[\mathfrak{p}_i])$
in terms of that of $G$, i.e. to translate the map 
$G \rightarrow G/ (Ker(F)[\mathfrak{p}_i])$
to the language of degeneration data (in a suitable formal setting). 

First, note that the restriction  of 
$G/ (Ker(F)[\mathfrak{p}_i])$ to the open stratum
is simply the old $A'$, and the definition of partial Frobenius already tells us that it comes with the PEL structure, i.e. we have a degenerating PEL family 
\[
(G', \lambda', i', (\alpha_n', \nu'_n)) 
\]
over 
$M_n(L, Tr_{\mathcal{O}_F/ \mathbb{Z}} \circ (\alpha \delta \langle   \cdot, \cdot\rangle_F))^{tor}_{\Sigma_{\alpha\delta}}$,
as all the extra data are defined  on the generic open part (although $\lambda'$ and $i'$ extends to the whole base by formal argument). 

Let us specialize the setting of the universal property in theorem \ref{lantor} to our case. Let us fix a geometric point $\bar{s}$ of 
$M_n(L, Tr_{\mathcal{O}_F/ \mathbb{Z}} \circ (\alpha \delta \langle   \cdot, \cdot\rangle_F))^{tor}_{\Sigma_{\alpha\delta}}$, 
and we assume that it lies in the strata 
$Z_{[(\Phi_{\alpha\delta,n}, \delta_{\alpha\delta,n}, \sigma_{\alpha\delta})]}$.
Let $V$ be a complete discrete valuation ring, and we are given a dominant morphism 
$Spec(V) \rightarrow M_n(L, Tr_{\mathcal{O}_F/ \mathbb{Z}} \circ (\alpha \delta \langle   \cdot, \cdot\rangle_F))^{tor}_{\Sigma_{\alpha\delta}}$
centered at $\bar{s}$, 
and
\[
(G^{\dagger}, \lambda^{\dagger}, i^{\dagger}, (\alpha_n^{\dagger}, \nu^{\dagger}_n)) 
\]
the pullback of 
$(G', \lambda', i', (\alpha_n', \nu'_n)) $
to $\text{Spec}(V)$. 

It is an easy observation that 
$G^{\dagger} 
\cong G_{\text{Spec}(V)} / (Ker(F) [\mathfrak{p}_i])$, 
i.e. we can first pullback the universal semi-abelian $G$, 
and then apply the partial Frobenius operation, and similarly for the extra structures. 

Let 
\[
(A, \lambda_A, i_A, X, Y, \phi, c,c^{\vee}, \tau, [ \alpha_n^{\natural}]) 
\]
be the degeneration data associated to 
$(G, \lambda, i, (\alpha_n, \nu_n))_{\text{Spec}(V)}$.
Since $V$ is centered at $\bar{s}$ which lies in 
$Z_{[(\Phi_{\alpha\delta,n}, \delta_{\alpha\delta,n}, \sigma_{\alpha\delta})]}$,
we see that the torus part of the degeneration data is the same as $\Phi_{\alpha\delta,n}$, 
and similarly for $\delta_{\alpha \delta,n}$ (if we choose a representative 
$(Z_{\alpha\delta,n}, \Phi_{\alpha\delta,n}, \delta_{\alpha\delta,n})$
of the cusp label
$[(Z_{\alpha\delta,n},\Phi_{\alpha\delta,n}, \delta_{\alpha\delta,n})]$).
Moreover, $\sigma_{\alpha\delta}$
determines the range of $\tau_n$ in $\alpha_n^{\natural}$. 

We want to write the degeneration data
\[
(A^{\dagger}, \lambda_{A^{\dagger}}^{\dagger}, i_{A^{\dagger}}^{\dagger}, X^{\dagger}, Y^{\dagger}, \phi^{\dagger}, c^{\dagger},(c^{\vee})^{\dagger}, \tau^{\dagger}, [ (\alpha_n^{\natural})^{\dagger}]) 
\]
of 
$(G^{\dagger}, \lambda^{\dagger}, i^{\dagger}, (\alpha_n^{\dagger}, \nu^{\dagger}_n)) $
in terms of that of 
$(G, \lambda, i, (\alpha_n, \nu_n))_{\text{Spec}(V)}$. 
More precisely, our aim is to describe $v \circ B^{\dagger}$ 
in terms of $v \circ B$, and it is enough to describe certain part of the degeneration data for our purpose, as in the next proposition.  

Recall that $B$ is a homomorphism 
$S_{\Phi_n} \rightarrow \text{Inv}(V)$ 
induced by 
\[
\frac{1}{n} Y \times X \longrightarrow \text{Inv}(V)
\]
which is 
\[
(y, \chi) \longrightarrow I_{y, \chi} := (c^{\vee}(y) \times c(\chi))^*P_{A}  \overset{\tau}{\subset} K := \text{Frac}(V)
\]
when restricted to $Y \times X$, and similarly for $B^{\dagger}$. By the proposition below, we have that 
\[
\tau^{\dagger} = \tau |_{Y^{\dagger} \times X^{\dagger}}
\]
under the natural inclusion
\[
Y^{\dagger} \times X^{\dagger} = Y \times (X \otimes_{\mathcal{O}_F} \mathfrak{p}_i) \hookrightarrow Y \times X
\]
and identification
\[
((c^{\vee})^{\dagger} \times c^{\dagger})^*P_{A^{\dagger}} \cong  ((c^{\vee} \times c)^*P_{A} )_{Y^{\dagger} \times X^{\dagger}}.
\]
Therefore, by abuse of notation (viewing $B$  as bilinear forms on $\frac{1}{n} Y \times X$, and similarly for $B^{\dagger}$)
\[
B^{\dagger} |_{Y^{\dagger} \times X^{\dagger}} = B|_{Y^{\dagger} \times X^{\dagger}}
\]
where the second restriction is induced by 
\[
Y^{\dagger} \times X^{\dagger} = Y \times (X \otimes_{\mathcal{O}_F} \mathfrak{p}_i) \hookrightarrow Y \times X.
\]
We now have obtained  
\[
v \circ B^{\dagger} |_{Y^{\dagger} \times X^{\dagger}} = v \circ B|_{Y^{\dagger} \times X^{\dagger}}
\]
which means that
\[
n(v \circ B^{\dagger}) = n( v \circ B|_{Y^{\dagger} \times X^{\dagger}}).
\]
By construction we have (for all dominant morphism 
$Spec(V) \rightarrow M_n(L, Tr_{\mathcal{O}_F/ \mathbb{Z}} \circ (\alpha \delta \langle   \cdot, \cdot\rangle_F))^{tor}_{\Sigma_{\alpha\delta}}$
centered at $\bar{s}$)
\[
v \circ B \in \sigma_{\alpha \delta},
\]
and by the definition of $\sigma'_{\alpha'\delta'}$
as in the statement of the theorem, we have 
\[
n(v \circ B^{\dagger}) = n( v \circ B|_{Y^{\dagger} \times X^{\dagger}}) \in \sigma'_{\alpha'\delta'},
\]
which implies that (for all dominant morphism 
$Spec(V) \rightarrow M_n(L, Tr_{\mathcal{O}_F/ \mathbb{Z}} \circ (\alpha \delta \langle   \cdot, \cdot\rangle_F))^{tor}_{\Sigma_{\alpha\delta}}$
centered at $\bar{s}$)
\[
v \circ B^{\dagger} \in \sigma'_{\alpha'\delta'}
\]
since 
$ \sigma'_{\alpha'\delta'}$ 
is a cone. This finishes the verification of the universal property and also the proof.
\end{proof}

\begin{proposition}
Let $V$ be a complete discrete valuation ring that is defined over $\mathcal{O}_{F_0}\otimes_{\mathbb{Z}} \mathbb{F}_p$ with generic fiber $\eta$,
and 
\[
(G, \lambda, i, (\alpha_n, \nu_n))
\]
is a PEL degenerating abelian variety over $Spec(V)$, i.e. 
$(G, \lambda, i, (\alpha_n, \nu_n)) \in DEG_{PEL,M_n}(V)$ 
as in definition \ref{degpel}. Let 
\[
(G', \lambda', i', (\alpha'_n, \nu'_n)) \in DEG_{PEL,M_n}(V)
\]
be the degenerating abelian variety defined by 
\[
G':= G/(Ker(F)[\mathfrak{p}_i]) 
\]
and the rest of structures obtained by applying the partial Frobenius to the generic fiber
$(G, \lambda, i, (\alpha_n, \nu_n))_{\eta}$ which is a PEL abelian variety, as in remark \ref{boring}. 

Let
\[
(A, \lambda_A, i_A, X, Y, \phi, c,c^{\vee}, \tau, [ \alpha_n^{\natural}])
\]
be the degeneration data of 
$(G, \lambda, i, (\alpha_n, \nu_n))$,
and 
\[
(A', {\lambda'}_{A'}, {i'}_{A'}, X', Y', \phi', c',{c'}^{\vee}, \tau', [ {\alpha'}_n^{\natural}])
\]
be the degeneration data of 
$(G', \lambda', i', (\alpha'_n, \nu'_n)) $. 
Then 
\[
(A', \lambda'_{A'}, i'_{A'}, (\varphi_{-1,n}',\nu_{-1,n}'))
\]
is obtained by applying the partial Frobenius to
$(A, \lambda_A, i_A, (\varphi_{-1,n},\nu_{-1,n}))$
over $V$, 
\[
Z'_{n} = Z_{n}, 
\]
\[
X' = X \otimes_{\mathcal{O}_F} \mathfrak{p}_i,
\]
\[
Y'=Y,
\]
\[
\varphi'_{-2,n} : Gr_{-2}^{Z'_{n} } = Gr_{-2}^{Z_{n} }  \overset{\varphi_{-2,n}}{\longrightarrow} Hom(X/nX, (\mathbb{Z}/n \mathbb{Z})(1)) 
\]
\[
\overset{\sim}{\longrightarrow} Hom(X \otimes \mathfrak{p}_i/n (X \otimes \mathfrak{p}_i), (\mathbb{Z}/n \mathbb{Z})(1)),
\]
\[
\varphi_{0,n}' : Gr_{0}^{Z'_{n} } = Gr_{0}^{Z_{n} }  \overset{\varphi_{0,n}}{\longrightarrow} Y/nY,
\]
and $\phi'$ is defined by the diagram 
\[\begin{tikzcd}
[
  arrow symbol/.style = {draw=none,"\textstyle#1" description,sloped},
  isomorphic/.style = {arrow symbol={\simeq}},
  ]
&
X  & 
X \otimes_{\mathcal{O}_F} \mathfrak{p}_i 
\arrow [l, "id \otimes (  \mathcal{O}_F \hookleftarrow \mathfrak{p}_i)"] 
\\
Y \otimes_{\mathcal{O}_F} \mathfrak{p}_i^{-1} &
Y \arrow[l, "id \otimes (\mathfrak{p}_i^{-1} \hookleftarrow \mathcal{O}_F)"] 
\arrow[u, "\phi"] 
 & 
 Y \otimes_{\mathcal{O}_F} \mathfrak{p}_i 
 \arrow[l,"id \otimes (  \mathcal{O}_F \hookleftarrow \mathfrak{p}_i)"]
 \arrow[u, "\phi \otimes id"]
 \\
&
Y \otimes_{\mathcal{O}_F} \mathfrak{p}_i^{-1} 
\arrow[lu, "\xi \otimes id"]
\arrow[u, dashrightarrow] & 
Y
\arrow[l,"id \otimes (\mathfrak{p}_i^{-1} \hookleftarrow \mathcal{O}_F)"]
\arrow[u, dashrightarrow]
\arrow[uu, bend right=90, dashrightarrow, "\phi'"]
\end{tikzcd} 
\]
Moreover, we have
\[
c': X' = X \otimes_{\mathcal{O}_F} \mathfrak{p}_i \overset{c \otimes id}{\longrightarrow} A^{\vee} \otimes_{\mathcal{O}_F} \mathfrak{p}_i 
\overset{\pi \otimes id}{\longrightarrow} {A^{\vee}}' \otimes_{\mathcal{O}_F} \mathfrak{p}_i \cong {A'}^{\vee} 
\]
\[
{c'}^{\vee} : Y'= Y \overset{c^{\vee}}{\longrightarrow} A
\overset{\pi}{\longrightarrow} A'
\]
where 
$\pi: A \rightarrow A' := A/(Ker(F)[\mathfrak{p}_i])$
is the projection map and the isomorphism 
${A^{\vee}}' \otimes_{\mathcal{O}_F} \mathfrak{p}_i \cong {A'}^{\vee}$
is as in lemma \ref{uuuuyyyyy}. 

Lastly and most importantly, there is a canonical isomorphism
\[
(c'^{\vee} \times c')^*P_{A'} \cong  ((c^{\vee} \times c)^*P_{A} )_{Y' \times X'}
\]
where the pullback to $Y' \times X'$ is through  the natural injection 
\[
Y' \times X' = Y \times (X \otimes_{\mathcal{O}_F} \mathfrak{p}_i) \hookrightarrow Y \times X
\]
induced by
$X \otimes_{\mathcal{O}_F} \mathfrak{p}_i \hookrightarrow X \otimes_{\mathcal{O}_F} \mathcal{O}_F \cong X$.  
Now
$\tau'$ is identified as 
\[
\tau': 1_{Y' \times X', \eta}  \overset{\tau |_{Y' \times X'}}{\longrightarrow} 
((c^{\vee}\times c)^*\mathcal{P}_{A,\eta}^{\otimes -1})_{Y' \times X'} \cong 
(c'^{\vee} \times c')^*P_{A',\eta}^{\otimes -1}.
\]
\end{proposition}

\begin{proof}
We first fix the notation as follows. For any commutative group scheme $H$ over $V$ with an action of $\mathcal{O}$, we denote 
\[
H^{(\mathfrak{p}_i)} := H/(Ker(F)[\mathfrak{p}_i])
\]
where $F$ is the relative Frobenius. 

We begin by showing that taking the partial Frobenius quotient commutes with the Raynaud extension, i.e. we have
\begin{lemma}
\[
G^{(\mathfrak{p}_i), \natural} \cong G^{\natural, (\mathfrak{p}_i)}.
\]
\end{lemma}

\begin{innerproof}[Proof of the Lemma]
Recall that $G^{\natural}$ is characterized as the unique global extension of an abelian variety by a torus whose formal completion along the maximal ideal of $V$ is the same as that of $G$, i.e. $G^{\natural}$ sits in an extension 
\[
0 \rightarrow T \rightarrow G^{\natural} \rightarrow A \rightarrow 0 
\]
with $T$ a torus and $A$ an abelian scheme over $V$, which satisfies 
$G^{\natural}_{\text{for}} \cong G_{\text{for}}$. 
We have a commutative diagram for the relative Frobenius $F$
\[
\begin{tikzcd}
 0  \arrow[r] & T \arrow[r] \arrow[d,"F_T"] & G^{\natural} \arrow[r] \arrow[d,"F"] & A \arrow[r] \arrow[d,"F_A"] & 0 
\\
 0  \arrow[r] & T \arrow[r] & G^{\natural} \arrow[r] & A \arrow[r] & 0
\end{tikzcd}
\]
We observe that the relative Frobenius   is a faithfully flat morphism on smooth schemes, which in particular shows that $F_T$ is surjective as a morphism in the category of fppf sheaves of abelian groups. Then the associated long exact sequence of the diagram tells us that we have a short exact sequence
\[
0 \rightarrow Ker(F_T) \rightarrow Ker(F) \rightarrow Ker(F_A) \rightarrow 0
\]
of finite flat group schemes over $V$. Since the diagram is $\mathcal{O}$-equivariant, so is the short exact sequence of $Ker(F)$. From the observation that $Ker(F)$ is killed by $p$ (and so are the other two groups), we see that 
$Ker(F)= \underset{i}{\prod} Ker(F)[\mathfrak{p}_i]$ 
(and similarly for the other two),
and the above short exact sequence decomposes into a product of short exact sequences
\[
0 \rightarrow Ker(F_T)[\mathfrak{p}_i] \rightarrow Ker(F)[\mathfrak{p}_i] \rightarrow Ker(F_A)[\mathfrak{p}_i] \rightarrow 0
\]
Now the commutative diagram
\[
\begin{tikzcd}
 0  \arrow[r] & Ker(F_T)[\mathfrak{p}_i] \arrow[r] \arrow[d,hook] & Ker(F)[\mathfrak{p}_i] \arrow[r] \arrow[d,hook] & Ker(F_A)[\mathfrak{p}_i] \arrow[r] \arrow[d,hook] & 0 
\\
 0  \arrow[r] & T \arrow[r] & G^{\natural} \arrow[r] & A \arrow[r] & 0
\end{tikzcd}
\]
gives us a short exact sequence
\[
0 \rightarrow T^{(\mathfrak{p}_i)} \rightarrow G^{\natural, (\mathfrak{p}_i)} \rightarrow A^{(\mathfrak{p}_i)} \rightarrow 0 
\]
Let 
$m$ be the maximal ideal of $V$, $k \in \mathbb{Z}$, 
$G^{\natural}_k := G^{\natural}_{\text{Spec}(V/m^k)}$ and similarly for other groups defined over $V$. Then the naturality of $Ker(F)[\mathfrak{p}_i]$ (it commutes with base change) provides isomorphisms
\[
(G^{\natural, (\mathfrak{p}_i)})_k \cong
G^{\natural}_k/(Ker(F)[\mathfrak{p}_i]) 
\cong G_k/(Ker(F)[\mathfrak{p}_i]) =
G^{(\mathfrak{p}_i)}_k,
\]
which are compatible when $k$ varies. This implies that 
\[
G^{\natural, (\mathfrak{p}_i)}_{\text{for}} \cong G^{(\mathfrak{p}_i)}_{\text{for}}
\]
hence we have a canonical isomorphism
\[
G^{\natural, (\mathfrak{p}_i)} \cong G^{ (\mathfrak{p}_i), \natural}
\]
by the characterization of the Raynaud extension. 

\qedhere
\end{innerproof}

\vspace{5mm}

Now by definition of $X'$ and $A'$ in the degeneration data, together with the fact
$G^{ (\mathfrak{p}_i), \natural} \cong 
G^{\natural, (\mathfrak{p}_i)} $
and 
\[
0 \rightarrow T^{(\mathfrak{p}_i)} \rightarrow G^{\natural, (\mathfrak{p}_i)} \rightarrow A^{(\mathfrak{p}_i)} \rightarrow 0 
\]
that we have just proved, 
we have 
\[
A' = A^{(\mathfrak{p}_i)}
\]
\[
X' = Hom(T^{(\mathfrak{p}_i)}, \textbf{G}_m) \cong X \otimes_{\mathcal{O}_F} \mathfrak{p}_i
\]
where the last isomorphism follows since on a torus we have $F=p$, so
\[
Ker(F)[\mathfrak{p}_i] = T[\mathfrak{p}_i] = Ker(T= T\otimes_{\mathcal{O}_F} \mathcal{O}_F \overset{id \otimes \hookrightarrow}{\longrightarrow} T\otimes_{\mathcal{O}_F} \mathfrak{p}_i^{-1})
\]
which implies that 
\begin{equation} \label{ljkjtou}
T^{(\mathfrak{p}_i)} \cong T\otimes_{\mathcal{O}_F} \mathfrak{p}_i^{-1}
\end{equation}
whence the isomorphism on the characters. Note that everything has an $\mathcal{O}$-action, and the isomorphisms are $\mathcal{O}$-equivariant, and in particular the $\mathcal{O}$-structure $i'_{A'}$ on $A'$ is induced from $A$ by the projection $A \rightarrow A^{(\mathfrak{p}_i)}$, which is consistent with partial Frobenius operation on $A$. 

On the other hand, we have a canonical isomorphism 
\[
G^{(\mathfrak{p}_i), \vee}_{\eta} \cong G^{\vee, (\mathfrak{p}_i)}_{\eta} \otimes_{\mathcal{O}_F} \mathfrak{p}_i
\]
as proved in lemma \ref{uuuuyyyyy}, which extends to 
\[
G^{(\mathfrak{p}_i), \vee} \cong G^{\vee, (\mathfrak{p}_i)} \otimes_{\mathcal{O}_F} \mathfrak{p}_i
\]
by formal nonsense (the restriction to the generic fiber is a fully faithful functor from the category of degenerating abelian varieties to that of abelian varieties). We can now take the Raynaud extension of both sides, and obtain
\[
G^{(\mathfrak{p}_i), \vee, \natural} \cong G^{\vee, (\mathfrak{p}_i),\natural} \otimes_{\mathcal{O}_F} \mathfrak{p}_i \cong 
G^{\vee, \natural, (\mathfrak{p}_i)} \otimes_{\mathcal{O}_F} \mathfrak{p}_i
\]
where the first isomorphism follows from the functoriality of Raynaud extensions (which implies that $(-)\otimes_{\mathcal{O}_F} \mathfrak{p}_i$ commutes with the Raynaud extension), and the second isomorphism is the claim we have just proved. From the extension
\[
0 \rightarrow T^{\vee} \rightarrow G^{\vee, \natural} \rightarrow A^{\vee} \rightarrow 0 
\]
and the above isomorphism, we see that 
\[
0 \rightarrow T^{\vee,(\mathfrak{p}_i)} \otimes_{\mathcal{O}_F} \mathfrak{p}_i 
\rightarrow G^{(\mathfrak{p}_i), \vee, \natural} 
\rightarrow A^{\vee, (\mathfrak{p}_i)} \otimes_{\mathcal{O}_F} \mathfrak{p}_i \rightarrow 0 
\]
We have already observed that from (\ref{ljkjtou}) and lemma \ref{uuuuyyyyy} there are natural isomorphisms
$T^{\vee,(\mathfrak{p}_i)}   \cong T^{\vee} \otimes_{\mathcal{O}_F} \mathfrak{p}_i^{-1}$
and 
$A^{(\mathfrak{p}_i),\vee} \cong A^{\vee, (\mathfrak{p}_i)} \otimes_{\mathcal{O}_F} \mathfrak{p}_i$, 
which simplifies the extension to
\[
0 \rightarrow T^{\vee} \rightarrow G^{(\mathfrak{p}_i), \vee, \natural} \rightarrow A^{(\mathfrak{p}_i), \vee} \rightarrow 0 
\]
This tells us that the torus part of the dual Raynaud extension of $G'$ is the same as that of $G$, hence 
\[
Y' = Y
\]
as we expected.

\vspace{5mm}

Now we look at the polarization $\lambda'$ and the associated part of the degeneration data. Recall that
$\lambda'_{\eta}: G'_{\eta} \rightarrow G'^{\vee}_{\eta}$
is characterized by the formula
\[
\xi \lambda_{\eta} = \pi_{\eta}^{\vee} \circ \lambda'_{\eta} \circ \pi_{\eta}
\]
with 
$\pi_{\eta} : G_{\eta} \rightarrow G'_{\eta}$ 
the projection, which extends uniquely to a morphism 
 $\lambda': G' \rightarrow G'^{\vee}$
by formal properties of the degenerating abelian varieties. This extension also satisfies the characterizing relation
\[
\xi \lambda = \pi^{\vee} \circ \lambda' \circ \pi
\]
with 
$\pi : G \rightarrow G'$ 
the projection (note that $\pi^{\vee}$ here has to be interpreted as the unique extension of the $\pi_{\eta}^{\vee}$ being the dual morphism on the dual abelian varieties).
The functoriality of Raynaud extensions provides us with the morphism 
$\lambda'^{\natural}: G'^{\natural} \rightarrow G'^{\vee,\natural}$, 
which fits in a commutative diagram (with change of notation)
\[
\begin{tikzcd}
 0  \arrow[r] & T^{(\mathfrak{p}_i)} \arrow[r] \arrow[d,"\lambda_{T}^{(\mathfrak{p}_i)}"] & G^{(\mathfrak{p}_i),\natural} \arrow[r] \arrow[d,"\lambda^{(\mathfrak{p}_i),\natural}"] & A^{(\mathfrak{p}_i)} \arrow[r] \arrow[d,"\lambda_{A}^{(\mathfrak{p}_i)}"] & 0 
\\
 0  \arrow[r] & T^{(\mathfrak{p}_i), \vee} \arrow[r] &  G^{(\mathfrak{p}_i),\vee, \natural} \arrow[r] & A^{(\mathfrak{p}_i),\vee} \arrow[r] & 0
\end{tikzcd}
\]
The characterizing relation
$\xi \lambda = \pi^{\vee} \circ \lambda^{(\mathfrak{p}_i)} \circ \pi$
extends to
\[
\xi \lambda^{\natural} = \pi^{\vee,\natural} \circ \lambda^{(\mathfrak{p}_i),\natural} \circ \pi^{\natural}
\]
on the Raynaud extension by functoriality, which implies the two relations
\[
\xi \lambda_A = \pi^{\vee,\natural}_A \circ \lambda^{(\mathfrak{p}_i)}_A \circ \pi^{\natural}_A
\]
\[
\xi \lambda_T = \pi^{\vee,\natural}_T \circ \lambda^{(\mathfrak{p}_i)}_T \circ \pi^{\natural}_T
\]
on the abelian and torus part respectively. Note that the relation on the abelian part is exactly the characterizing relation of the partial Frobenius operation on $(A,\lambda_A)$. 

Alternatively, we can write down directly the diagram defining $\lambda'$ on the generic fiber, which extends formally to the whole base as follows
\[\begin{tikzcd}
[
  arrow symbol/.style = {draw=none,"\textstyle#1" description,sloped},
  isomorphic/.style = {arrow symbol={\simeq}},
  ]
&
G \arrow[r,"F^{(\mathfrak{p}_i)}"] \arrow[d,"\lambda" ] & G^{(\mathfrak{p}_i)} \arrow [r, "V^{(\mathfrak{p}_i)}"] \arrow[d,"\lambda^{(\mathfrak{p}_i)}"' ] 
\arrow[ddd, bend left =40, dashrightarrow,  "\lambda'" pos=.2]&  
G  \otimes_{\mathcal{O}_F} \mathfrak{p}_i^{-1} \arrow[d,"\lambda \otimes id" ]
\\
G^{\vee} \otimes_{\mathcal{O}_F} \mathfrak{p}_i 
\arrow[r, "id \otimes ( \mathfrak{p}_i \hookrightarrow \mathcal{O}_F)"] 
\arrow[rd,"\xi \otimes id_{\mathfrak{p}_i}"' ]&
G^{\vee} \arrow[r, "F_{G^{\vee}}^{(\mathfrak{p}_i)}"]  \arrow[rrd, "\xi"' near start] 
\arrow[d, dashrightarrow] & 
 (G^{\vee})^{(\mathfrak{p}_i)} \arrow[r,"V_{G^{\vee}}^{(\mathfrak{p}_i)}"]
 \arrow[d, dashrightarrow]& 
 G^{\vee} \otimes_{\mathcal{O}_F} \mathfrak{p}_i^{-1}
 \arrow[d, dashrightarrow]
 \\
&
G^{\vee} \otimes_{\mathcal{O}_F} \mathfrak{p}_i  
\arrow[r, "(V^{(\mathfrak{p}_i)})^{ \vee}"'] & (G^{\vee})^{(\mathfrak{p}_i)} \otimes_{\mathcal{O}_F} \mathfrak{p}_i  \arrow[r,"(F^{(\mathfrak{p}_i)})^{ \vee}"'] & 
 G^{\vee}
 \\
 & & (G^{(\mathfrak{p}_i)})^{\vee} \arrow[u, isomorphic]&
\end{tikzcd} 
\]
The functoriality of Raynaud extension allows us to draw the same diagram with Raynaud extensions, and so does the abelian and torus part, with which we obtain a rather explicit description of $\lambda_A^{(\mathfrak{p}_i)}$ and $\lambda_T^{(\mathfrak{p}_i)}$. This tells us that $\lambda_A^{(\mathfrak{p}_i)}$ is obtained as in the partial Frobenius operation, and the morphism $\phi'$ on characters induced by 
$\lambda_T^{(\mathfrak{p}_i)}$
is as in the description of the proposition. 

\vspace{5mm}

The next step is to look at the level structures. Recall that the level structure
$\alpha_n' : L/nL \cong G'[n] $ 
on $G'$ is defined by the composition 
\[
\alpha_n': 
L/nL \overset{\alpha_n}{\cong} G[n] \overset{\pi}{\cong} G'[n] 
\]
where 
$\pi: G\rightarrow G'$ is the projection, which induces an isomorphism on $n$-torsion points since $n$ is prime to $p$. Since $\pi$ preserves the monodromy filtration on $G[n]$ and $G'[n]$, we have 
\[
Z_n' = Z_n
\]
be definition. More explicitly, $\pi$ induces isomorphisms of the extensions
\[
\begin{tikzcd}
 0  \arrow[r] & T[n] \arrow[r] \isoarrow{d,"\pi_T"'} & G^{\natural}[n] \arrow[r] \isoarrow{d,"\pi^{\natural}"'}
 & A[n] \arrow[r] \isoarrow{d,"\pi_A"'} & 0 
\\
 0  \arrow[r] & T^{(\mathfrak{p}_i)}[n] \arrow[r] & G^{(\mathfrak{p}_i), \natural}[n] \arrow[r] & A^{(\mathfrak{p}_i)}[n] \arrow[r] & 0
\end{tikzcd}
\]
\[
\begin{tikzcd}
 0  \arrow[r]  & G^{\natural}[n] \arrow[r] \isoarrow{d,"\pi^{\natural}"'}
 & G[n] \arrow[r] \isoarrow{d,"\pi"'} & 
 \frac{1}{n}Y/Y \arrow[d,equal] \arrow[r] & 0 
\\
 0  \arrow[r] & G^{(\mathfrak{p}_i), \natural}[n] \arrow[r] & G^{(\mathfrak{p}_i)}[n] \arrow[r] &  
 \frac{1}{n}Y/Y \arrow[r] & 0
\end{tikzcd}
\]
where we use $Y'=Y$ in the last isomorphism. We have seen that $\pi_T : T \rightarrow T^{(\mathfrak{p}_i)}$ is the natural morphism
$T= T \otimes_{\mathcal{O}_F } \mathcal{O}_F \overset{id \otimes \hookrightarrow}{\longrightarrow} T \otimes_{\mathcal{O}_F} \mathfrak{p}_i^{-1} \cong T^{(\mathfrak{p}_i)}$, 
and the corresponding map on characters is 
$X \otimes_{\mathcal{O}_F} \mathfrak{p}_i \overset{id \otimes \hookrightarrow}{\longrightarrow} X \otimes_{\mathcal{O}_F} \mathcal{O}_F = X$, 
which clearly implies that the degree -2 part of the level-$n$ structure is
\[
\varphi'_{-2,n} : Gr_{-2}^{Z'_{n} } = Gr_{-2}^{Z_{n} }  \overset{\varphi_{-2,n}}{\longrightarrow} Hom(X/nX, (\mathbb{Z}/n \mathbb{Z})(1)) 
\]
\[
\overset{\sim}{\longrightarrow} Hom(X \otimes \mathfrak{p}_i/n (X \otimes \mathfrak{p}_i), (\mathbb{Z}/n \mathbb{Z})(1)).
\]
Similarly, we have
\[
\varphi_{0,n}' : Gr_{0}^{Z'_{n} } = Gr_{0}^{Z_{n} }  \overset{\varphi_{0,n}}{\longrightarrow} Y/nY,
\]
and 
\[
\varphi'_{-1,n}: Gr_{-1}^{Z'_n} = Gr_{-1}^{Z_n} \overset{\varphi_{-1,n}}{\longrightarrow} A[n] \overset{\pi_A}{\cong} A^{(\mathfrak{p}_i)}[n].
\]
Moreover, we have 
\[
\nu'_{-1,n} = \nu'_n = \nu_n \circ \kappa
\]
where the first equality is tautological and the second is part of the definition of the partial Frobenius. 

To summarize, we have proved that 
\[
(A', \lambda'_{A'}, i'_{A'}, (\varphi_{-1,n}',\nu_{-1,n}'))
\]
is exactly the partial Frobenius operation applied to
$(A, \lambda_A, i_A, (\varphi_{-1,n},\nu_{-1,n}))$, 
and have identified the torus argument
\[
(X', Y', \phi', \varphi'_{-2,n}, \varphi'_{0,n})
\]
together with the filtration $Z'_n$. The remaining part to be identified is $(c',c'^{\vee}, \tau')$.

\vspace{5mm}

First, we  note that for any degenerating abelian variety $H$, the canonical morphism
$H=H\otimes_{\mathcal{O}_F} \mathcal{O}_F \overset{id \otimes \hookrightarrow}{\longrightarrow} H\otimes_{\mathcal{O}_F} \mathfrak{p}_i^{-1}$
factors through 
$H \rightarrow H^{(\mathfrak{p}_i)} \rightarrow H\otimes_{\mathcal{O}_F} \mathfrak{p}_i^{-1}$,
and tensoring the factoring morphism with $\mathfrak{p}_i$ we obtain 
\[
V^{(\mathfrak{p}_i)}: H^{(\mathfrak{p}_i)}\otimes_{\mathcal{O}_F} \mathfrak{p}_i \rightarrow H
\]
which is characterized by the commutative diagram
\[
\begin{tikzcd}
 H\otimes_{\mathcal{O}_F} \mathfrak{p}_i \arrow[r,"\pi_H \otimes id_{\mathfrak{p}_i}"] \arrow[dr,"id_H \otimes (\mathfrak{p}_i \hookrightarrow \mathcal{O}_F)"']
 & H^{(\mathfrak{p}_i)} \otimes_{\mathcal{O}_F} \mathfrak{p}_i \arrow[d,"V^{(\mathfrak{p}_i)}"]
 \\
 & H
\end{tikzcd}
\]
We observe that when $H$ is an abelian scheme,
under the identification $H^{(\mathfrak{p}_i),\vee} \cong H^{\vee,(\mathfrak{p}_i)} \otimes_{\mathcal{O}_F} \mathfrak{p}_i$ 
 the morphism
$H^{(\mathfrak{p}_i)}\otimes_{\mathcal{O}_F} \mathfrak{p}_i \rightarrow H$
is the dual of the projection 
$\pi_H: H \rightarrow H^{(\mathfrak{p}_i)}$,
i.e. we have a commutative diagram
\[
\begin{tikzcd}
  H^{(\mathfrak{p}_i),\vee} \arrow[r,"\sim"] 
 \arrow[dr, "(\pi_H)^{\vee}"]
 & H^{\vee,(\mathfrak{p}_i)} \otimes_{\mathcal{O}_F} \mathfrak{p}_i
\arrow[d,"V^{(\mathfrak{p}_i)}"]
 \\
 & H
\end{tikzcd}
\]
Dually, we have that $V^{(\mathfrak{p}_i),\vee} $ is $\pi_{A^{\vee}}$ under canonical isomorphism, i.e.
\[
\begin{tikzcd}
 H^{\vee} \arrow[r,"V^{(\mathfrak{p}_i),\vee}"] 
 \arrow[dr, "\pi_{H^{\vee}}"]
 & (H^{(\mathfrak{p}_i)} \otimes_{\mathcal{O}_F}\mathfrak{p}_i)^{\vee} 
 \isoarrow{d}
 \\
 & H^{\vee,(\mathfrak{p}_i)}
\end{tikzcd}
\]
is commutative.

\vspace{5mm}

\begin{lemma}
\[
c': X' = X \otimes_{\mathcal{O}_F} \mathfrak{p}_i \overset{c \otimes id}{\longrightarrow} A^{\vee} \otimes_{\mathcal{O}_F} \mathfrak{p}_i 
\overset{\pi \otimes id}{\longrightarrow} {A^{\vee}}' \otimes_{\mathcal{O}_F} \mathfrak{p}_i \cong {A'}^{\vee}
\]
\end{lemma}

\begin{innerproof}[Proof of the Lemma]

We have seen that there is a natural morphism
\[
V^{(\mathfrak{p}_i)}: G^{\natural,(\mathfrak{p}_i)} \otimes_{\mathcal{O}_F} \mathfrak{p}_i \rightarrow G^{\natural}
\]
which induces the morphism between extensions
\[
\begin{tikzcd}
 0  \arrow[r] & T \arrow[r] \arrow[d,equal]  
 & G^{\natural,(\mathfrak{p}_i)} \otimes_{\mathcal{O}_F} \mathfrak{p}_i
 \arrow[r] \arrow[d,] 
 & A^{(\mathfrak{p}_i)} \otimes_{\mathcal{O}_F} \mathfrak{p}_i \arrow[r] \arrow[d,"V_A^{(\mathfrak{p}_i)}"] & 0 
\\
 0  \arrow[r] & T \arrow[r] & G^{\natural} \arrow[r] & A \arrow[r] & 0
\end{tikzcd}
\]
where we use the canonical isomorphism
$T^{(\mathfrak{p}_i)}\otimes_{\mathcal{O}_F} \mathfrak{p}_i \cong T$. 
We note that the extension in the first row is determined by the morphism
\[
 X \cong (X\otimes_{\mathcal{O}_F} \mathfrak{p}_i) \otimes_{\mathcal{O}_F} \mathfrak{p}_i^{-1} \overset{c' \otimes id_{\mathfrak{p}_i^{-1}}}{\longrightarrow} 
A^{(\mathfrak{p}_i),\vee} \otimes_{\mathcal{O}_F} \mathfrak{p}_i^{-1} 
\cong A^{\vee,(\mathfrak{p}_i)}
\]
and the extension in the second row is determined by
\[
c: X \longrightarrow A^{\vee}
\]
For $\chi \in X$, we write $\mathscr{O}_{\chi}:= c(\chi) \in Pic^0(A)$  and 
$L_{\chi} := c' \otimes id_{\mathfrak{p}_i^{-1}}(\chi) \in Pic^0(A^{(\mathfrak{p}_i)}\otimes_{\mathcal{O}_F} \mathfrak{p}_i)$. 
By abuse of notation, we will write $\mathscr{O}_{\chi}$ and $L_{\chi}$ for both the line bundle and the $\textbf{G}_m$-torsor.

Recall that $\mathscr{O}_{\chi}$ is defined as the pushout of $G^{\natural}$ along $\chi$, i.e. we have a pushout diagram
\[
\begin{tikzcd}
 0  \arrow[r] & T \arrow[r] \arrow[d,"\chi"]  \arrow[dr, phantom, "\ulcorner", very near start]
 & G^{\natural} \arrow[r] \arrow[d,] & A \arrow[r] \arrow[d,equal] & 0 
\\
 0  \arrow[r] & \textbf{G}_m \arrow[r] & \mathscr{O}_{\chi} \arrow[r] & A \arrow[r] & 0
\end{tikzcd}
\]
and similarly for $L_{\chi}$. We can complete this into the previous diagram and obtain
\[
\begin{tikzcd}
 0  \arrow[r] & T \arrow[r] \arrow[d,equal]  
 & G^{\natural,(\mathfrak{p}_i)} \otimes_{\mathcal{O}_F} \mathfrak{p}_i
 \arrow[r] \arrow[d,] 
 & A^{(\mathfrak{p}_i)} \otimes_{\mathcal{O}_F} \mathfrak{p}_i \arrow[r] \arrow[d,"V_A^{(\mathfrak{p}_i)}"] & 0 
\\
0  \arrow[r] & T \arrow[r] \arrow[d,"\chi"]  \arrow[dr, phantom, "\ulcorner", very near start]
 & G^{\natural} \arrow[r] \arrow[d,] & A \arrow[r] \arrow[d,equal] & 0 
\\
 0  \arrow[r] & \textbf{G}_m \arrow[r] & \mathscr{O}_{\chi} \arrow[r] & A \arrow[r] & 0
\end{tikzcd}
\]
Similarly, $L_{\chi}$ is defined by the pushout of
$G^{\natural,(\mathfrak{p}_i)} \otimes_{\mathcal{O}_F} \mathfrak{p}_i$
along $\chi$, and the universal property of pushout implies that the diagram factorizes through $L_{\chi}$, i.e. we have a commutative diagram
\[
\begin{tikzcd}
 0  \arrow[r] & T \arrow[r] \arrow[d,"\chi"]  
 \arrow[dr, phantom, "\ulcorner", very near start]
 & G^{\natural,(\mathfrak{p}_i)} \otimes_{\mathcal{O}_F} \mathfrak{p}_i
 \arrow[r] \arrow[d,] 
 & A^{(\mathfrak{p}_i)} \otimes_{\mathcal{O}_F} \mathfrak{p}_i \arrow[r] \arrow[d,equal] & 0 
\\
0  \arrow[r] & \textbf{G}_m \arrow[r] \arrow[d,equal]  
 & L_{\chi} \arrow[r] \arrow[d,] & A^{(\mathfrak{p}_i)} \otimes_{\mathcal{O}_F} \mathfrak{p}_i
 \arrow[r] \arrow[d,"V_A^{(\mathfrak{p}_i)}"] & 0 
\\
 0  \arrow[r] & \textbf{G}_m \arrow[r] & \mathscr{O}_{\chi} \arrow[r] & A \arrow[r] & 0
\end{tikzcd}
\]
which tells us that
\[
(V_A^{(\mathfrak{p}_i)})^*\mathscr{O}_{\chi} = L_{\chi}
\]
In other words, we have a commutative diagram
\[
\begin{tikzcd}
 X  \arrow[r,"c' \otimes id_{\mathfrak{p}_i^{-1}}"] \arrow[d,"c"] 
 & 
 A^{(\mathfrak{p}_i),\vee} \otimes_{\mathcal{O}_F} \mathfrak{p}_i^{-1} \isoarrow{d}
\\
A^{ \vee} \arrow[r, "(V_A^{(\mathfrak{p}_i)})^{\vee}"] 
& 
(A^{(\mathfrak{p}_i)} \otimes_{\mathcal{O}_F} \mathfrak{p}_i)^{\vee}
\end{tikzcd}
\]
Now with the help of the commutative diagram
\[
\begin{tikzcd}
 A^{\vee} \arrow[r,"V^{(\mathfrak{p}_i),\vee}"] 
 \arrow[dr, "\pi_{A^{\vee}}"]
 & (A^{(\mathfrak{p}_i)} \otimes_{\mathcal{O}_F}\mathfrak{p}_i)^{\vee} 
 \isoarrow{d}
 \\
 & A^{\vee,(\mathfrak{p}_i)}
\end{tikzcd}
\]
we see that 
\[
\begin{tikzcd}
 X  \arrow[r,"c' \otimes id_{\mathfrak{p}_i^{-1}}"] \arrow[d,"c"] 
 & 
 A^{(\mathfrak{p}_i),\vee} \otimes_{\mathcal{O}_F} \mathfrak{p}_i^{-1} \isoarrow{d}
\\
A^{ \vee} \arrow[r, "\pi_{A^{\vee}}"] 
& 
A^{\vee,(\mathfrak{p}_i)}
\end{tikzcd}
\]
i.e. under the identification $A^{\vee, (\mathfrak{p}_i)} \cong A^{(\mathfrak{p}_i),\vee} \otimes_{\mathcal{O}_F} \mathfrak{p}_i^{-1}$, 
\[
c' \otimes id_{\mathfrak{p}_i^{-1}} = \pi_{A^{\vee}} \circ c
\]
and tensoring with $\mathfrak{p}_i$ we obtain 
\[
c' = (\pi_{A^{\vee}} \circ c) \otimes id_{\mathfrak{p}_i} =(X \overset{c}{\rightarrow} A^{\vee} \overset{\pi_{A^{\vee}}}{\rightarrow} A^{\vee,(\mathfrak{p}_i)}) \otimes id_{\mathfrak{p}_i}
\]
which is what we want to prove. 

\qed
\end{innerproof}

Dually, we have

\begin{lemma}
\[
c'^{\vee} : Y'= Y \overset{c^{\vee}}{\longrightarrow} A
\overset{\pi}{\longrightarrow} A'
\]
\end{lemma}

\begin{innerproof}[Proof of the Lemma]

For $y \in Y$, we write $L_{c^{\vee}(y)}:= c^{\vee}(y) \in Pic^0(A^{\vee})$  and 
$L_{c'^{\vee}(y)} := c'^{\vee}(y) \in Pic^0(A^{(\mathfrak{p}_i),\vee})$. 
We have seen that 
$G^{(\mathfrak{p}_i),\vee,\natural} \cong (G^{\vee,\natural})^{(\mathfrak{p}_i)} \otimes_{\mathcal{O}_F} \mathfrak{p}_i$,
which equips with a natural morphism
\[
(G^{\vee,\natural})^{(\mathfrak{p}_i)} \otimes_{\mathcal{O}_F} \mathfrak{p}_i 
\rightarrow
G^{\vee,\natural}
\]

Recall that 
$G^{(\mathfrak{p}_i),\vee,\natural}$
is an extension
\[
0 \rightarrow T^{\vee} \rightarrow G^{(\mathfrak{p}_i), \vee, \natural} \rightarrow A^{(\mathfrak{p}_i), \vee} \rightarrow 0 
\]
and we observe that the morphism 
$G^{(\mathfrak{p}_i),\vee,\natural} \cong
(G^{\vee,\natural})^{(\mathfrak{p}_i)} \otimes_{\mathcal{O}_F} \mathfrak{p}_i 
\rightarrow
G^{\vee,\natural}$
gives rise to a morphism of the extension 
\[
\begin{tikzcd}
 0  \arrow[r] & T^{\vee} \arrow[r] \arrow[d,equal]  
 & G^{(\mathfrak{p}_i), \vee, \natural} \arrow[r] \arrow[d,] & A^{(\mathfrak{p}_i), \vee} \arrow[r] \arrow[d,"(\pi_A)^{\vee}"] & 0 
\\
 0  \arrow[r] & T^{\vee} \arrow[r] & G^{\vee,\natural} \arrow[r] & A^{\vee} \arrow[r] & 0
\end{tikzcd}
\]
Now, similar as before, $L_{c^{\vee}(y)}$ is the pushout of $G^{\vee,\natural}$ along $y$, and we have a commutative diagram
\[
\begin{tikzcd}
 0  \arrow[r] & T^{\vee} \arrow[r] \arrow[d,equal]  
 & G^{(\mathfrak{p}_i), \vee, \natural} \arrow[r] \arrow[d,] & A^{(\mathfrak{p}_i), \vee} \arrow[r] \arrow[d,"(\pi_A)^{\vee}"] & 0 
\\
 0  \arrow[r] & T^{\vee} \arrow[dr, phantom, "\ulcorner", very near start] \arrow[r] \arrow[d,"y"] & G^{\vee,\natural} \arrow[r] \arrow[d] & A^{\vee} \arrow[r] \arrow[d,equal] & 0
 \\
 0  \arrow[r] & \textbf{G}_m \arrow[r] & L_{c^{\vee}(y)} \arrow[r] & A^{\vee} \arrow[r] & 0
\end{tikzcd}
\]
By the universal property of pushout, we have the factorization
\[
\begin{tikzcd}
 0  \arrow[r] & T^{\vee} \arrow[r] \arrow[d,"y"] \arrow[dr, phantom, "\ulcorner", very near start]  
 & G^{(\mathfrak{p}_i), \vee, \natural} \arrow[r] \arrow[d,] & A^{(\mathfrak{p}_i), \vee} \arrow[r] \arrow[d,equal] & 0 
\\
 0  \arrow[r] & \textbf{G}_m  \arrow[r] \arrow[d,equal] & L_{c'^{\vee}(y)} \arrow[r] \arrow[d] & A^{(\mathfrak{p}_i), \vee} \arrow[r] \arrow[d,"(\pi_A)^{\vee}"] & 0
 \\
 0  \arrow[r] & \textbf{G}_m \arrow[r] & L_{c^{\vee}(y)} \arrow[r] & A^{\vee} \arrow[r] & 0
\end{tikzcd}
\]
which implies that 
\[
(\pi_A)^{\vee *}L_{c^{\vee}(y)} = L_{c'^{\vee}(y)}.
\]
In other words, we have 
\[
c'^{\vee}(y) = ((\pi_A)^{\vee})^{\vee} \circ c^{\vee}(y) =\pi_A \circ c^{\vee}(y)
\]
under the canonical identification $(A^{(\mathfrak{p}_i), \vee})^{\vee} = A^{(\mathfrak{p}_i)}$, 
which means that
\[
c'^{\vee} : Y \overset{c^{\vee}}{\longrightarrow} A \overset{\pi_A}{\longrightarrow} A^{(\mathfrak{p}_i)}.
\]

\qed 
\end{innerproof}

\vspace{5mm}

Lastly, we determine $\tau'$ from $\tau$.

Let us  write
$\mathscr{O}_\chi := c(\chi) \in Pic^0(A)$ 
for $\chi \in X$,
and similarly $\mathscr{O}_{\chi'}:= c'(\chi') \in Pic^0(A')$ 
for $\chi' \in X'$. By abuse of notation, we will write $\mathscr{O}_{\chi}$ for both the line bundle and the $\textbf{G}_m$-torsor. 

We first make an observation on the relation between $\mathscr{O}_{\chi}$ and $\mathscr{O}_{\chi'}$, which can be used to write the canonical morphism $G^{\natural} \rightarrow G^{(\mathfrak{p}_i), \natural}$ in a more explicit way.

Recall that $\mathscr{O}_{\chi}$ is defined as the pushout of $G^{\natural}$ along $\chi$, i.e. we have a pushout diagram
\[
\begin{tikzcd}
 0  \arrow[r] & T \arrow[r] \arrow[d,"\chi"]  \arrow[dr, phantom, "\ulcorner", very near start]
 & G^{\natural} \arrow[r] \arrow[d,] & A \arrow[r] \arrow[d,equal] & 0 
\\
 0  \arrow[r] & \textbf{G}_m \arrow[r] & \mathscr{O}_{\chi} \arrow[r] & A \arrow[r] & 0
\end{tikzcd}
\]
and similarly for $\mathscr{O}_{\chi'}$. We have a diagram
\[
\begin{tikzcd}
 0  \arrow[r] & T \arrow[r] \arrow[d,"\pi_T"]  
 & G^{\natural} \arrow[r] \arrow[d,] & A \arrow[r] \arrow[d,"\pi_A"] & 0 
\\
0  \arrow[r] & T^{(\mathfrak{p}_i)} \arrow[r] \arrow[d,"\chi'"]  \arrow[dr, phantom, "\ulcorner", very near start]
 & G^{\natural,(\mathfrak{p}_i)} \arrow[r] \arrow[d,] & A^{(\mathfrak{p}_i)} \arrow[r] \arrow[d,equal] & 0 
 \\
 0  \arrow[r] & \textbf{G}_m \arrow[r] & \mathscr{O}_{\chi'} \arrow[r] & A^{(\mathfrak{p}_i)} \arrow[r] & 0
\end{tikzcd}
\]
Let
\[
\rho: X'= X \otimes_{\mathcal{O}_F} \mathfrak{p}_i \rightarrow X
\]
be the map induced by $\pi_T$ (being the obvious map induced by $\mathfrak{p}_i \hookrightarrow \mathcal{O}_F$), then $\mathscr{O}_{\rho(\chi')}$ is the pushout along $\pi_T \circ \chi'$, and the universal property of the pushout provides us with a factorization of the short exact sequence
\[
\begin{tikzcd}
 0  \arrow[r] & T \arrow[dr, phantom, "\ulcorner", very near start] \arrow[r] \arrow[d,"\rho(\chi')"]  
 & G^{\natural} \arrow[r] \arrow[d,] & A \arrow[r] \arrow[d,equal] & 0 
\\
0  \arrow[r] & \textbf{G}_m \arrow[r] \arrow[d,equal] 
 & \mathscr{O}_{\rho(\chi')}
 \arrow[r] \arrow[d,] & A \arrow[r] \arrow[d,"\pi_A"] & 0 
 \\
 0  \arrow[r] & \textbf{G}_m \arrow[r] & \mathscr{O}_{\chi'} \arrow[r] & A^{(\mathfrak{p}_i)} \arrow[r] & 0
\end{tikzcd}
\]
This shows 
\begin{equation} \label{bbss}
\pi_A^*\mathscr{O}_{\chi'} = \mathscr{O}_{\rho(\chi')}
\end{equation}
which implies that under the identifications 
$G^{\natural} = \underline{\text{Spec}}_{\mathscr{O}_A} (\underset{\chi \in X}{\oplus} \mathscr{O}_{\chi})$
and 
$G^{(\mathfrak{p}_i), \natural} = \underline{\text{Spec}}_{\mathscr{O}_{A^{(\mathfrak{p}_i)}}} (\underset{\chi' \in X'}{\oplus} \mathscr{O}_{\chi'})$
the morphism
\[
G^{\natural} \rightarrow G^{(\mathfrak{p}_i), \natural}
\]
is induced from the map 
\begin{equation} \label{longgg}
\pi_A^*(\underset{\chi' \in X'}{\oplus} \mathscr{O}_{\chi'})
\cong
\underset{\chi' \in X'}{\oplus} \pi_A^*\mathscr{O}_{\chi'}
\cong 
\underset{\chi' \in X'}{\oplus} \mathscr{O}_{\rho(\chi')}
\hookrightarrow
\underset{\chi \in X}{\oplus} \mathscr{O}_{\chi}
\end{equation}
on relatively affine algebras over $A$.

\vspace{5mm}

Let us recall how we associate $\tau$ to the degenerating abelian variety $G$. We start by choosing an ample invertible cubical sheaf $\mathcal{L}$ on $G$ (whose existence is guaranteed by the normality of the base $\text{Spev}(V)$), then we can show that its formal completion extends canonically to a cubical ample line bundle $\mathcal{L}^{\natural}$ on $G^{\natural}$, which descends to  an ample invertible sheaf $\mathcal{M}$ on $A$, i.e. if we denote by $p: G^{\natural} \rightarrow A$ the projection map, then $p^*\mathcal{M} = \mathcal{L}^{\natural}$. 
Replacing $\lambda_{\eta}$ by $\lambda_{\mathcal{L}_{\eta}}$ if necessary (so $\lambda$ is the unique extension of $\lambda_{\mathcal{L}_{\eta}}$ to $G$), we assume that $\lambda_{\eta}= \lambda_{\mathcal{L}_{\eta}}$. The construction of $\tau$ is independent of the choice of $\lambda$ or $\mathcal{L}$. 

The canonical isomorphism 
$G^{\natural} = \underline{\text{Spec}}_{\mathscr{O}_A} (\underset{\chi \in X}{\oplus} \mathscr{O}_{\chi})$
tells us that 
\[
p_*\mathcal{L}^{\natural} \cong \underset{\chi \in X}{\oplus} \mathcal{M}_{\chi}
\]
with $\mathcal{M}_{\chi} := \mathcal{M} \otimes_{\mathscr{O}_A} \mathscr{O}_{\chi}$, from which we obtain 
\[
\Gamma(G^{\natural},\mathcal{L}^{\natural}) = \underset{\chi \in X}{\oplus} \Gamma(A, \mathcal{M}_{\chi})
\]
and 
\[
\Gamma(G^{\natural}_{\text{for}},\mathcal{L}^{\natural}_{\text{for}}) = \underset{\chi \in X}{\hat{\oplus}} \Gamma(A, \mathcal{M}_{\chi})
\]
where $\hat{}$ denotes the completion with respect to the maximal ideal of $V$.
Now we have the canonical map
\[
\Gamma(G,\mathcal{L}) \rightarrow \Gamma(G_{\text{for}},\mathcal{L}_{\text{for}})
\cong \Gamma(G^{\natural}_{\text{for}},\mathcal{L}^{\natural}_{\text{for}}) \cong 
\underset{\chi \in X}{\hat{\oplus}} \Gamma(A, \mathcal{M}_{\chi}) 
\rightarrow \Gamma(A, \mathcal{M}_{\chi}) 
\]
where the first map is the restriction and the last is the projection  on $\chi$-th component. Tensoring both sides with $\text{Frac}(V)$, we obtain 
\[
\sigma_{\chi}: \Gamma(G_{\eta},\mathcal{L}_{\eta}) \longrightarrow
\Gamma(A_{\eta}, \mathcal{M}_{\chi,\eta})
\]

Let $y \in Y$ and $T_{c^{\vee}(y)}: A \rightarrow A$ the translation by $c^{\vee}(y)$, then 
\[
\lambda_A \circ c^{\vee} = c \circ \phi
\]
tells us that 
we have a canonical isomorphism of rigidified line bundles
\[
T_{c^{\vee}(y)}^* \mathcal{M}_{\chi} \cong
\mathcal{M}_{\chi + \phi(y)} \otimes_R \mathcal{M}_{\chi}(c^{\vee}(y))
\]
and this is the place where we use the assumption 
$\lambda_{\eta} = \lambda_{\mathcal{L}_{\eta}}$
on $\mathcal{L}$. 
Now we have the map
\[
T_{c^{\vee}(y)}^* \circ \sigma_{\chi} : 
\Gamma(G_{\eta}, \mathcal{L}_{\eta}) \rightarrow 
\Gamma (A_{\eta}, T_{c^{\vee}(y)}^*\mathcal{M}_{\chi,\eta}) \cong 
\Gamma (A_{\eta},\mathcal{M}_{\chi + \phi(y),\eta}) \otimes_K 
\mathcal{M}_{\chi}(c^{\vee}(y))_{\eta}.
\]
The desired $\tau$ is obtained by comparing 
$T_{c^{\vee}(y)}^* \circ \sigma_{\chi}$
with the map
\[
\sigma_{\chi+\phi(y)}: \Gamma(G_{\eta}, \mathcal{L}_{\eta}) \longrightarrow \Gamma (A_{\eta}, \mathcal{M}_{\chi+\phi(y),\eta}), 
\]
and the result is  
\[
\sigma_{\chi + \phi(y)} = \psi (y) \tau(y, \chi) T_{c^{\vee}(y)}^* \circ \sigma_{\chi}
\]
where 
\[
\psi(y): \mathcal{M}(c^{\vee}(y))_{\eta} \overset{\sim}{\rightarrow} \mathscr{O}_{S, \eta}
\]
is a trivialization of the fiber of $\mathcal{M}$ at $c^{\vee}(y)$, 
and 
\[
\tau(y, \chi) : \mathscr{O}_{\chi}(c^{\vee}(y))_{\eta} \longrightarrow \mathscr{O}_{S, \eta}
\]
is a section of 
$\mathscr{O}_{\chi}(c^{\vee}(y))_{\eta}^{\otimes -1}$
for each $y \in Y$ and $\chi \in X$, so that 
$\psi (y) \tau(y, \chi)$ is a section of $\mathcal{M}_{\chi}(c^{\vee}(y))_{\eta}^{\otimes-1}$
(recall $\mathcal{M}_{\chi} = \mathcal{M} \otimes \mathscr{O}_{\chi}$). This uniquely characterizes $\tau$ since $\sigma_{\chi} \neq 0$ for every $\chi \in X$.

\vspace{5mm}

\begin{lemma}
There is a canonical isomorphism
\[
(c'^{\vee} \times c')^*P_{A'} \cong  ((c^{\vee} \times c)^*P_{A} )_{Y' \times X'}
\]
where the pullback to $Y' \times X'$ is through  the natural injection 
\[
Y' \times X' = Y \times (X \otimes_{\mathcal{O}_F} \mathfrak{p}_i) \hookrightarrow Y \times X
\]
induced by
$X \otimes_{\mathcal{O}_F} \mathfrak{p}_i \hookrightarrow X \otimes_{\mathcal{O}_F} \mathcal{O}_F \cong X$.  
Now
$\tau'$ is identified as 
\[
\tau': 1_{Y' \times X', \eta}  \overset{\tau |_{Y' \times X'}}{\longrightarrow} 
((c^{\vee}\times c)^*\mathcal{P}_{A,\eta}^{\otimes -1})_{Y' \times X'} \cong 
(c'^{\vee} \times c')^*P_{A',\eta}^{\otimes -1}.
\]
\end{lemma}

\begin{remark}
As in section \ref{equivalentfor}, the lemma is equivalent to the statement that the period of $G^{(\mathfrak{p}_i)}$ is given by 
\[
Y_{\eta} \overset{\imath}{\longrightarrow} G^{\natural}_{\eta} \longrightarrow G^{(\mathfrak{p}_i),\natural}_{\eta}
\]
where $\imath$ is the period of $G$, and 
$G^{\natural}_{\eta} \longrightarrow G^{(\mathfrak{p}_i),\natural}_{\eta}$
is the natural projection map. 
\end{remark}

\begin{innerproof}[Proof of the Lemma]

We have the same description as above for $\tau'$, and we want to compare it with $\tau$. Let us first compare $\sigma_{\chi'}$ and $\sigma_{\chi}$. 

We choose an ample cubical invertible sheaf $\mathcal{L}^{(\mathfrak{p}_i)}$ on $G^{(\mathfrak{p}_i)}$
whose associated line bundle
$\mathcal{L}^{(\mathfrak{p}_i), \natural}$ 
on 
$G^{(\mathfrak{p}_i),\natural}$
descends to an ample invertible sheaf
$\mathcal{M}^{(\mathfrak{p}_i)}$ 
on $A^{(\mathfrak{p}_i)}$.
Let $\mathcal{L}$ (resp. $\mathcal{M}$) be the pullback of 
$\mathcal{L}^{(\mathfrak{p}_i)}$ 
(resp. 
$\mathcal{M}^{(\mathfrak{p}_i)}$)
along the natural map 
$G \rightarrow G^{(\mathfrak{p}_i)}$
(resp. 
$A \rightarrow A^{(\mathfrak{p}_i)}$). 
Note that both $\mathcal{L}$ 
and $\mathcal{M}$ are ample since they are pullback of ample line bundles along finite maps 
$G \rightarrow G^{(\mathfrak{p}_i)}$
and
$A \rightarrow A^{(\mathfrak{p}_i)}$ 
respectively. 

We assume that 
$\mathcal{L}^{(\mathfrak{p}_i)}_{\eta}$ 
induces the polarization 
$\lambda'_{\eta}$ on $G^{(\mathfrak{p}_i)}_{\eta}$,
so we have
\begin{equation} \label{subtlll}
T_{c'^{\vee}(y)}^* \mathcal{M}^{(\mathfrak{p}_i)}_{\chi'} \cong
\mathcal{M}_{\chi' + \phi'(y)}^{(\mathfrak{p}_i)} \otimes_R \mathcal{M}_{\chi'}^{(\mathfrak{p}_i)}(c'^{\vee}(y))
\end{equation}
Let $\pi : G \rightarrow G^{(\mathfrak{p}_i)}$ 
be the projection map,  then
$\mathcal{L} := \pi^*\mathcal{L}^{(\mathfrak{p}_i)}$
and the associated polarization 
\[
\lambda_{\mathcal{L}_{\eta}} = \lambda_{\pi^*\mathcal{L}^{(\mathfrak{p}_i)}_{\eta}}= \pi^{\vee}_{\eta} \circ \lambda_{\mathcal{L}^{(\mathfrak{p}_i)}_{\eta}} \circ
\pi_{\eta} =\pi^{\vee}_{\eta} \circ \lambda'_{\eta} \circ
\pi_{\eta}= \xi \lambda_{\eta}
\]
which has the effect that 
\begin{equation} \label{mistake}
T_{c^{\vee}(y)}^* \mathcal{M}_{\chi} \cong
\mathcal{M}_{\chi + \xi\phi(y)} \otimes_R \mathcal{M}_{\chi}(c^{\vee}(y))
\end{equation}
as we have to replace the relation 
$\lambda_A \circ c^{\vee} = c \circ \phi$
by
$\xi\lambda_A \circ c^{\vee} = c \circ \xi\phi$. 

Let
\[
\rho: X'= X \otimes_{\mathcal{O}_F} \mathfrak{p}_i \rightarrow X
\]
be the  map induced by $\mathfrak{p}_i \hookrightarrow \mathcal{O}_F$ 
as before, then for $\chi' \in X'$,
the natural map 
$G \rightarrow G^{(\mathfrak{p}_i)}$
induces a commutative diagram
\[
\begin{tikzcd}
 \Gamma(G,\mathcal{L})  \arrow[r] 
 & \Gamma(G_{\text{for}},\mathcal{L}_{\text{for}}) \arrow[r,phantom, "\cong"]  
 & \Gamma(G^{\natural}_{\text{for}},\mathcal{L}^{\natural}_{\text{for}}) \arrow[r] 
 & \Gamma(A, \mathcal{M}_{\rho(\chi')})  
\\
\Gamma(G^{(\mathfrak{p}_i)},\mathcal{L}^{(\mathfrak{p}_i)})  \arrow[r] \arrow[u] 
& \Gamma(G^{(\mathfrak{p}_i)}_{\text{for}},\mathcal{L}^{(\mathfrak{p}_i)}_{\text{for}}) \arrow[r,phantom, "\cong"] \arrow[u]
& \Gamma(G^{(\mathfrak{p}_i),\natural}_{\text{for}},\mathcal{L}^{(\mathfrak{p}_i),\natural}_{\text{for}}) \arrow[r] \arrow[u]
& \Gamma(A^{(\mathfrak{p}_i)}, \mathcal{M}^{(\mathfrak{p}_i)}_{\chi'}) \arrow[u]
\end{tikzcd}
\]
where the last two horizontal maps are projections 
\[
\Gamma(G^{\natural}_{\text{for}},\mathcal{L}^{\natural}_{\text{for}}) \cong 
\underset{\chi \in X}{\hat{\oplus}} \Gamma(A, \mathcal{M}_{\chi}) 
\rightarrow \Gamma(A, \mathcal{M}_{\rho(\chi')}) 
\]
and 
\[
\Gamma(G^{(\mathfrak{p}_i),\natural}_{\text{for}},\mathcal{L}^{(\mathfrak{p}_i), \natural}_{\text{for}}) \cong 
\underset{\chi' \in X}{\hat{\oplus}} \Gamma(A^{(\mathfrak{p}_i)}, \mathcal{M}^{(\mathfrak{p}_i)}_{\chi'}) 
\rightarrow \Gamma(A^{(\mathfrak{p}_i)}, \mathcal{M}^{(\mathfrak{p}_i)}_{\chi'}) 
\]
respectively,
and the last vertical map is induced from
\[
\pi_A^*\mathcal{M}^{(\mathfrak{p}_i)}_{\chi'} = \pi_A^*(\mathcal{M}^{(\mathfrak{p}_i)} \otimes_{\mathscr{O}_{A^{(\mathfrak{p}_i)}}} \mathscr{O}_{\chi'})=
\mathcal{M} \otimes_{\mathscr{O}_A} \pi^*\mathscr{O}_{\chi'} \overset{(\ref{bbss})}{\cong}
\mathcal{M} \otimes_{\mathscr{O}_A} \mathscr{O}_{\rho(\chi')}=
\mathcal{M}_{\rho(\chi')}
\]
i.e. it is taking  the global section of the map 
$\mathcal{M}^{(\mathfrak{p}_i)}_{\chi'} \rightarrow \pi_{A*}\pi_A^* \mathcal{M}^{(\mathfrak{p}_i)}_{\chi'} 
\overset{\sim}{\rightarrow} \pi_{A*} \mathcal{M}_{\rho(\chi')}$
with the last isomorphism being $\pi_{A*}$ of the  isomorphism (\ref{bbss}). 
The commutativity of the first two squares is tautological, and that of the last square follows from (\ref{longgg}). Tensoring with $\text{Frac}(V)$ of the above diagram, we obtain a commutative diagram
\[
\begin{tikzcd}
\Gamma(G_{\eta},\mathcal{L}_{\eta})  \arrow[r,"\sigma_{\rho(\chi')}"] 
 & 
 \Gamma(A_{\eta}, \mathcal{M}_{\rho(\chi'),\eta})  
\\
\Gamma(G^{(\mathfrak{p}_i)}_{\eta},\mathcal{L}^{(\mathfrak{p}_i)}_{\eta}) \arrow[r, "\sigma_{\chi'}"] 
\arrow[u]
& 
\Gamma(A_{\eta}^{(\mathfrak{p}_i)}, \mathcal{M}^{(\mathfrak{p}_i)}_{\chi',\eta}) 
\arrow[u]
\end{tikzcd}
\]
for every $\chi' \in X'$. 

For $y \in Y$, we can complete the diagram into \[
\begin{tikzcd}
\Gamma(G_{\eta},\mathcal{L}_{\eta})  \arrow[r,"\sigma_{\rho(\chi')}"] 
 & 
 \Gamma(A_{\eta}, \mathcal{M}_{\rho(\chi'),\eta})  
 \arrow[r,"T_{c^{\vee}(y)}^*"]
 &
 \Gamma(A_{\eta}, (T_{c^{\vee}(y)})^*\mathcal{M}_{\rho(\chi'),\eta}) 
 \arrow[r,equal]
 & \cdots
\\
\Gamma(G^{(\mathfrak{p}_i)}_{\eta},\mathcal{L}^{(\mathfrak{p}_i)}_{\eta}) \arrow[r, "\sigma_{\chi'}"] 
\arrow[u]
& 
\Gamma(A_{\eta}^{(\mathfrak{p}_i)}, \mathcal{M}^{(\mathfrak{p}_i)}_{\chi',\eta}) 
\arrow[u]
\arrow[r,"T_{c'^{\vee}(y)}^*"]
&
\Gamma(A_{\eta}^{(\mathfrak{p}_i)}, (T_{c'^{\vee}(y)})^* \mathcal{M}^{(\mathfrak{p}_i)}_{\chi',\eta})
\arrow[u]
\arrow[r,equal]
& \cdots
\end{tikzcd}
\]
\begin{equation} \label{longdiffic}
\begin{tikzcd}
\cdots \arrow[r,equal,"(\ref{mistake})"]
&
\Gamma(A_{\eta}, \mathcal{M}_{\rho(\chi')+\xi\phi(y),\eta}) \otimes \mathcal{M}_{\rho(\chi')}(c^{\vee}(y))_{\eta}
\\
\cdots \arrow[r,equal,"(\ref{subtlll})"]
&
\Gamma(A_{\eta}^{(\mathfrak{p}_i)}, \mathcal{M}^{(\mathfrak{p}_i)}_{\chi'+\phi'(y),\eta})
\otimes \mathcal{M}_{\chi'}^{(\mathfrak{p}_i)}(c'^{\vee}(y))_{\eta}
\arrow[u]
\end{tikzcd}
\end{equation}
where the last vertical arrow is the tensor product of the morphism
\[
\Gamma(A_{\eta}^{(\mathfrak{p}_i)}, \mathcal{M}^{(\mathfrak{p}_i)}_{\chi'+\phi'(y),\eta})
\rightarrow
\Gamma(A_{\eta}, \mathcal{M}_{\rho(\chi')+\xi\phi(y),\eta})
\]
induced by 
$\pi_A^*\mathcal{M}^{(\mathfrak{p}_i)}_{\chi'+\phi'(y)} \cong
\mathcal{M}_{\rho(\chi'+\phi'(y))} =
\mathcal{M}_{\rho(\chi')+\xi \phi(y)}$
($\rho \circ \phi' = \xi \phi$ by the diagram defining $\phi'$), 
with the isomorphism
\[
\mathcal{M}_{\rho(\chi')}(c^{\vee}(y))_{\eta} = \pi_A^* \mathcal{M}_{\chi'}^{(\mathfrak{p}_i)} (c^{\vee}(y))_{\eta} = \mathcal{M}_{\chi'}^{(\mathfrak{p}_i)} (\pi_A \circ c^{\vee}(y))_{\eta}
= \mathcal{M}_{\chi'}^{(\mathfrak{p}_i)}(c'^{\vee}(y))_{\eta}
\]
where we use 
$c'^{\vee}= \pi_A \circ c^{\vee}$
in the last equality. 
The middle square commutes since we have a commutative diagram \[
\begin{tikzcd}
A 
\arrow[d, "\pi_A"]
&
A \arrow[d, "\pi_A"] \arrow[l,"T_{c^{\vee}(y)}"'] 
\\
A^{(\mathfrak{p}_i)}
&
A^{(\mathfrak{p}_i)}
\arrow[l,"T_{c'^{\vee}(y)}"']
\end{tikzcd}
\]
which follows from 
$c'^{\vee}= \pi_A \circ c^{\vee}$ 
and $\pi_A$ being a group homomorphism. The commutativity of the last square follows from the commutativity of the square
\[
\begin{tikzcd}
(T_{c^{\vee}(y)})^*\mathcal{M}_{\rho(\chi'),\eta}
\arrow[r,"\sim"]
&
\mathcal{M}_{\rho(\chi')+\xi\phi(y),\eta} \otimes \mathcal{M}_{\rho(\chi')}(c^{\vee}(y))_{\eta}
\\
\pi_A^*(T_{c'^{\vee}(y)})^* \mathcal{M}^{(\mathfrak{p}_i)}_{\chi',\eta} 
\isoarrow{u}
\arrow[r,"\sim"]
&
\pi_A^*(\mathcal{M}^{(\mathfrak{p}_i)}_{\chi'+\phi'(y),\eta}
\otimes \mathcal{M}_{\chi'}^{(\mathfrak{p}_i)}(c'^{\vee}(y))_{\eta})
\isoarrow{u}
\end{tikzcd}
\]
and we can prove its commutativity as follows. Recall that the first horizontal map is (using $\lambda_L (x)= T_x^*L \otimes L^{-1} \otimes L_x^{-1}$)
\[
(T_{c^{\vee}(y)})^*\mathcal{M}_{\rho(\chi'),\eta} 
=
(\lambda_{\mathcal{M}} \circ c^{\vee}(y))_{\eta} \otimes \mathcal{M}_{\rho(\chi'),\eta}  \otimes 
\mathcal{M}_{\rho(\chi')}(c^{\vee}(y))_{\eta}
\]
\[
\overset{\lambda_{\mathcal{M}}=\xi \lambda_A}{=\joinrel=\joinrel=\joinrel=\joinrel=}(c \circ \xi\phi(y))_{\eta} 
\otimes \mathcal{M}_{\rho(\chi'),\eta}  \otimes 
\mathcal{M}_{\rho(\chi')}(c^{\vee}(y))_{\eta}
\]
\[
=\joinrel=\joinrel=\joinrel=\joinrel=
\mathscr{O}_{\xi \phi(y),{\eta}}
\otimes \mathcal{M}_{\rho(\chi'),\eta}  \otimes 
\mathcal{M}_{\rho(\chi')}(c^{\vee}(y))_{\eta}
\]
\[
=\joinrel=\joinrel=\joinrel=\joinrel=
\mathcal{M}_{\rho(\chi')+\xi\phi(y),\eta} \otimes \mathcal{M}_{\rho(\chi')}(c^{\vee}(y))_{\eta}
\]
and similarly the second horizontal map is $\pi_A^*$ of
\[
(T_{c'^{\vee}(y)})^* \mathcal{M}^{(\mathfrak{p}_i)}_{\chi',\eta} =
(\lambda'_{A'} \circ c'^{\vee}(y))_{\eta} \otimes \mathcal{M}^{(\mathfrak{p}_i)}_{\chi',\eta}  \otimes
\mathcal{M}_{\chi'}^{(\mathfrak{p}_i)}(c'^{\vee}(y))_{\eta}
\]
\[
=(c' \circ \phi'(y))_{\eta} \otimes \mathcal{M}^{(\mathfrak{p}_i)}_{\chi',\eta}  \otimes
\mathcal{M}_{\chi'}^{(\mathfrak{p}_i)}(c'^{\vee}(y))_{\eta}
\]
\[
= \mathscr{O}_{\phi'(y),\eta} \otimes \mathcal{M}^{(\mathfrak{p}_i)}_{\chi',\eta}  \otimes
\mathcal{M}_{\chi'}^{(\mathfrak{p}_i)}(c'^{\vee}(y))_{\eta}
\]
\[
=\mathcal{M}^{(\mathfrak{p}_i)}_{\chi'+\phi'(y),\eta}
\otimes \mathcal{M}_{\chi'}^{(\mathfrak{p}_i)}(c'^{\vee}(y))_{\eta}
\]
We want to prove that $\pi_A^*$ of the second isomorphism is the first isomorphism under canonical identifications, and the only non-trivial part is to observe that 
\[
\pi_A^*(\lambda'_{A'} \circ c'^{\vee}(y)) = \pi_A^{\vee} \circ \lambda'_{A'} \circ c'^{\vee}(y) \overset{c'^{\vee}=\pi_A \circ c^{\vee}}{=\joinrel=\joinrel=\joinrel=} \pi_A^{\vee} \circ \lambda'_{A'} \circ \pi_A \circ c^{\vee}(y) 
=\xi \lambda_A \circ c^{\vee}(y).
\]

We now want to compare the diagram (\ref{longdiffic}) with 
\begin{equation} \label{shorter}
\begin{tikzcd}
\Gamma(G_{\eta},\mathcal{L}_{\eta})  \arrow[r,"\sigma_{\rho(\chi')+\xi \phi(y)}"] 
 & 
 \Gamma(A_{\eta}, \mathcal{M}_{\rho(\chi')+\xi \phi(y),\eta})  
\\
\Gamma(G^{(\mathfrak{p}_i)}_{\eta},\mathcal{L}^{(\mathfrak{p}_i)}_{\eta}) \arrow[r, "\sigma_{\chi'+\phi'(y)}"] 
\arrow[u]
& 
\Gamma(A_{\eta}^{(\mathfrak{p}_i)}, \mathcal{M}^{(\mathfrak{p}_i)}_{\chi'+\phi'(y),\eta}) 
\arrow[u]
\end{tikzcd}
\end{equation}
Recall that we have 
\begin{equation} \label{painnn}
\sigma_{\rho(\chi') + \xi\phi(y)} = \psi (y) \tau(y, \rho(\chi')) T_{c^{\vee}(y)}^* \circ \sigma_{\rho(\chi')}
\end{equation}
where 
\[
\psi(y): \mathcal{M}(c^{\vee}(y))_{\eta} \overset{\sim}{\rightarrow} \mathscr{O}_{S, \eta}
\]
is a trivialization of the fiber of $\mathcal{M}$ at $c^{\vee}(y)$, 
and 
\[
\tau(y, \chi) : \mathscr{O}_{\chi}(c^{\vee}(y))_{\eta} \longrightarrow \mathscr{O}_{S, \eta}
\]
is a section of 
$\mathscr{O}_{\chi}(c^{\vee}(y))_{\eta}^{\otimes -1}$
for each $y \in Y$ and $\chi \in X$, so that 
$\psi (y) \tau(y, \chi)$ is a section of $\mathcal{M}_{\chi}(c^{\vee}(y))_{\eta}^{\otimes-1}$. 
Note that here we have tacitly changed the polarization of $G$ from $\lambda$ to $\xi \lambda$, which has the effect of replacing $\phi$ by $\xi\phi$. This does not affect $\tau$, but may affect $\psi$, which we use the same notation as before for simplicity.  

Similarly, $\tau'$ is characterized by the equation
\begin{equation} \label{ppppp}
\sigma_{\chi' + \phi'(y)} = \psi' (y) \tau'(y, \chi') T_{c'^{\vee}(y)}^* \circ \sigma_{\chi'}
\end{equation}
with isomorphism
\[
\psi'(y): \mathcal{M}^{(\mathfrak{p}_i)}(c'^{\vee}(y))_{\eta} \overset{\sim}{\rightarrow} \mathscr{O}_{S, \eta}
\]
and 
\[
\tau'(y, \chi') : \mathscr{O}_{\chi'}(c'^{\vee}(y))_{\eta} \overset{\sim}{\rightarrow} \mathscr{O}_{S, \eta}
\]
so that 
$\psi' (y) \tau'(y, \chi')$ defines a section of $\mathcal{M}^{(\mathfrak{p}_i)}_{\chi'}(c'^{\vee}(y))_{\eta}^{\otimes-1}$.

Now (\ref{longdiffic}), (\ref{shorter}), (\ref{painnn}) and (\ref{ppppp}) together implies that we have a commutative diagram
\[
\begin{tikzcd}
\Gamma(A_{\eta}, \mathcal{M}_{\rho(\chi')+\xi\phi(y),\eta}) \otimes \mathcal{M}_{\rho(\chi')}(c^{\vee}(y))_{\eta}  \  \ \  \  
\arrow[r,"{\tau(y,\rho(\chi'))\psi(y)}"] 
 & 
 \  \ \  \  
 \Gamma(A_{\eta}, \mathcal{M}_{\rho(\chi')+\xi \phi(y),\eta}) 
\\
\Gamma(A_{\eta}^{(\mathfrak{p}_i)}, \mathcal{M}^{(\mathfrak{p}_i)}_{\chi'+\phi'(y),\eta})
\otimes \mathcal{M}_{\chi'}^{(\mathfrak{p}_i)}(c'^{\vee}(y))_{\eta}  \  \ \  \  
\arrow[r,"{\tau'(y,\chi')\psi'(y)}"] 
\arrow[u]
& 
\  \ \  \ 
\Gamma(A_{\eta}^{(\mathfrak{p}_i)}, \mathcal{M}^{(\mathfrak{p}_i)}_{\chi'+\phi'(y),\eta}) 
\arrow[u]
\end{tikzcd}
\]
where we use that $\sigma_{\chi} \neq 0$ and $\sigma_{\chi'} \neq 0$ 
for every $\chi \in X$ and $\chi' \in X'$. Observe that the vertical arrows are non-zero, and we obtain
\[
\tau'(y,\chi') = \tau(y, \rho(\chi'))
\]
under the canonical identification 
$\mathscr{O}_{\rho(\chi')}(c^{\vee}(y))_{\eta} \cong 
\pi_A^*\mathscr{O}_{\chi'} (c^{\vee}(y))_{\eta} \cong
\mathscr{O}_{\chi'} (\pi_A \circ c^{\vee}(y))_{\eta} =
\mathscr{O}_{\chi'} (c'^{\vee}(y))$, 
which completes the proof if we take the equivalent formulation of $\tau$ in \ref{equivalentfor}. 
\qed
\end{innerproof}
\qed
\end{proof}

\subsection{Partial Frobenius extends to minimal compactifications}

In this final section, we deduce our main theorem \ref{pFextends} from the theorem proved in the last section. We retain the setting of the last section, so in particular every scheme is defined over 
$\mathcal{O}_{F_0} \otimes_{\mathbb{Z}} \mathbb{F}_p$.

We begin by recalling the construction of the minimal compactifications. In the analytic setting, the minimal compactifications can be constructed directly using rational boundary components. However, in the algebraic settings, the only known method to proceed is to first construct the toroidal compactifications and then contract the boundary to obtain the minimal compactifications. 

More precisely, let 
$\omega^{\text{tor}} := \bigwedge^{\text{top}} \underline{\text{Lie}}^{\vee}_{G^{\text{tor}}/M_{n,\Sigma}^{\text{tor}}}$,
where 
$G^{\text{tor}}$ 
is the universal semi-abelian scheme over the toroidal compactification 
$M_{n,\Sigma}^{\text{tor}}$. Then $\omega^{\text{tor}}$ is an invertible sheaf generated by its global sections, and we define
\[
M_n^{\text{min}} := \text{Proj}(\underset{k \geq 0}{\oplus} \Gamma(M_{n,\Sigma}^{\text{tor}}, (\omega^{\text{tor}})^{\otimes k}))
\]
Alternatively, $M_n^{\text{min}}$ is the Stein factorization of the map 
\[
M_{n,\Sigma}^{\text{tor}} \longrightarrow \mathbb{P}(\Gamma(M_{n,\Sigma}^{\text{tor}}, \omega^{\text{tor}}))
\]
defined by global sections of $\omega^{\text{tor}}$, i.e. it factors through 
\[
\oint : M_{n,\Sigma}^{\text{tor}} \rightarrow M_n^{\text{min}} 
\]
with 
$M_n^{\text{min}} \rightarrow \mathbb{P}(\Gamma(M_{n,\Sigma}^{\text{tor}}, \omega^{\text{tor}}))$
finite, and 
\[
\mathscr{O}_{M_n^{\text{min}}} \overset{\sim}{\rightarrow}  \oint_*\mathscr{O}_{M_{n,\Sigma}^{\text{tor}}}.
\]
It can be shown that $M_n^{\text{min}}$ is independent of the choice of the toroidal compactification. 
Moreover, by construction we have a canonical ample invertible sheaf $\omega^{\text{min}}:= \mathscr{O}(1)$ on $M_n^{\text{min}}$ such that
\[
\oint^*\omega^{\text{min}} \cong \omega^{\text{tor}}
\]

We can show that $M_n^{\text{min}}$ has a stratification
\[
M_n^{\text{min}} =
\underset{[(Z_n, \Phi_n, \delta_n)]}{\coprod}
Z_{[(Z_n, \Phi_n, \delta_n)]}
\]
where 
$Z_{[(Z_n, \Phi_n, \delta_n)]} = M_n^{Z_n}$
as defined in section \ref{tttor}, and the index ranges through all cusp labels. Moreover the map $\oint$ preserves the stratification, and sends $Z_{[(\Phi_n, \delta_n, \sigma)]}$ 
to $Z_{[(Z_n, \Phi_n, \delta_n)]}$.

Similar to the toroidal compactifications, the minimal compactification of $M_{K(n)}/ \Delta$
is defined to be the union of minimal compactifications of 
$M_n(L, Tr_{\mathcal{O}_F/ \mathbb{Z}} \circ (\alpha \delta \langle   \cdot, \cdot\rangle_F)) $. 

\begin{theorem}
$F_{\mathfrak{p}_i}$ extends to a morphism
\[
F_{\mathfrak{p}_i}^{\text{min}} : (M_{K(n)}/ \Delta)^{min} \longrightarrow (M_{K(n)}/ \Delta)^{min}
\]
sending the strata 
$M_n(L^{Z_{\alpha \delta,n}},\langle  \cdot, \cdot\rangle^{Z_{\alpha \delta,n}})$
associated to 
$\alpha \in \Omega,  
  \delta \in \Lambda$
and the cusp label
$[(Z_{\alpha \delta,n}, \Phi_{\alpha \delta,n}, \delta_{\alpha \delta,n})]$ 
to the strata 
$M_n(L^{Z_{\alpha' \delta',n}},\langle  \cdot, \cdot\rangle^{Z_{\alpha' \delta',n}})$
associated to 
$\alpha' \in \Omega,  
  \delta' \in \Lambda$
with the usual notations as before, and the cusp label
$[(Z_{\alpha' \delta',n}, \Phi_{\alpha' \delta',n}, \delta_{\alpha' \delta',n})]$ 
defined as follows:
\[
Z_{\alpha' \delta',n} = Z_{\alpha \delta,n}. 
\]
If 
$ \Phi_{\alpha \delta,n} =(X,Y,\phi, \varphi_{-2,n}, \varphi_{0,n}) $,
then 
\[
\Phi_{\alpha' \delta',n} = (X \otimes_{\mathcal{O}_F} \mathfrak{p}_i,Y,\phi', \varphi_{-2,n}', \varphi'_{0,n}) \]
where 
\[
\varphi'_{-2,n} : Gr_{-2}^{Z_{\alpha' \delta',n} } = Gr_{-2}^{Z_{\alpha \delta,n} }  \overset{\varphi_{-2,n}}{\longrightarrow} Hom(X/nX, (\mathbb{Z}/n \mathbb{Z})(1)) 
\]
\[
\overset{\sim}{\longrightarrow} Hom(X \otimes \mathfrak{p}_i/n (X \otimes \mathfrak{p}_i), (\mathbb{Z}/n \mathbb{Z})(1))
\]
and 
\[
\varphi_{0,n}' : Gr_{0}^{Z_{\alpha' \delta',n} } = Gr_{0}^{Z_{\alpha \delta,n} }  \overset{\varphi_{0,n}}{\longrightarrow} Y/nY.
\]
Lastly, $\phi'$ is defined by the following diagram similar to the diagram defining $\lambda'$,
\[\begin{tikzcd}
[
  arrow symbol/.style = {draw=none,"\textstyle#1" description,sloped},
  isomorphic/.style = {arrow symbol={\simeq}},
  ]
&
X  & 
X \otimes_{\mathcal{O}_F} \mathfrak{p}_i 
\arrow [l, "id \otimes (  \mathcal{O}_F \hookleftarrow \mathfrak{p}_i)"] 
\\
Y \otimes_{\mathcal{O}_F} \mathfrak{p}_i^{-1} &
Y \arrow[l, "id \otimes (\mathfrak{p}_i^{-1} \hookleftarrow \mathcal{O}_F)"] 
\arrow[u, "\phi"] 
 & 
 Y \otimes_{\mathcal{O}_F} \mathfrak{p}_i 
 \arrow[l,"id \otimes (  \mathcal{O}_F \hookleftarrow \mathfrak{p}_i)"]
 \arrow[u, "\phi \otimes id"]
 \\
&
Y \otimes_{\mathcal{O}_F} \mathfrak{p}_i^{-1} 
\arrow[lu, "\xi \otimes id"]
\arrow[u, dashrightarrow] & 
Y
\arrow[l,"id \otimes (\mathfrak{p}_i^{-1} \hookleftarrow \mathcal{O}_F)"]
\arrow[u, dashrightarrow]
\arrow[uu, bend right=90, dashrightarrow, "\phi'"]
\end{tikzcd} 
\]

Moreover, on each strata, $F_{\mathfrak{p}_i}^{\text{min}}$ induces the morphism
\[
M_n(L^{Z_{\alpha \delta,n}},\langle  \cdot, \cdot\rangle^{Z_{\alpha \delta,n}}) \rightarrow  M_n(L^{Z_{\alpha' \delta',n}},\langle  \cdot, \cdot\rangle^{Z_{\alpha' \delta',n}})
\]
sending 
$(A, \lambda, i, (\alpha_n, \nu_n))$
to 
$(A', \lambda', i', (\alpha'_n, \nu'_n))$ as in the description before the theorem. For completeness, we summarize the description as follows. Using the above notations, 
$A' := A/ (Ker(F)[\mathfrak{p}_i])$,
$i'$ is induced by the quotient map $\pi_{\mathfrak{p}_i} : A \rightarrow A'$,  $\lambda'$ is characterized by 
$\xi \lambda = \pi_{\mathfrak{p}_i}^{\vee} \circ \lambda' \circ \pi_{\mathfrak{p}_i}$ which defines a prime to $p$ isogeny $\lambda'$, 
$\alpha_n' = \pi_{\mathfrak{p}_i} \circ \alpha_n$
and 
$\nu_n' = \nu_n \circ \kappa $. In other words, restriction of the partial Frobenius to (suitable union of) strata recovers the partial Frobenius on them. 
\end{theorem}

\begin{proof}
It is enough to prove that $F_{\mathfrak{p}_i}$ extends to the minimal compactification of each component, i.e.  
\[
M_n(L, Tr_{\mathcal{O}_F/ \mathbb{Z}} \circ (\alpha \delta \langle   \cdot, \cdot\rangle_F)) \rightarrow 
M_n(L, Tr_{\mathcal{O}_F/ \mathbb{Z}} \circ (\alpha' \delta' \langle   \cdot, \cdot\rangle_F))
\]
extends to a morphism
\[
M_n(L, Tr_{\mathcal{O}_F/ \mathbb{Z}} \circ (\alpha \delta \langle   \cdot, \cdot\rangle_F))^{\text{min}} \rightarrow 
M_n(L, Tr_{\mathcal{O}_F/ \mathbb{Z}} \circ (\alpha' \delta' \langle   \cdot, \cdot\rangle_F))^{\text{min}}
\]
and maps strata to the expected ones. We are thus reduced to the situation that we are familiar with. 

We have morphisms
\begin{equation} \label{bugb}
\begin{tikzcd}
M_n(L, Tr_{\mathcal{O}_F/ \mathbb{Z}} \circ (\alpha \delta \langle   \cdot, \cdot\rangle_F))^{\text{tor}}_{\Sigma_{\alpha\delta}}
\arrow[r,"F_{\mathfrak{p}_i}^{\text{tor}}"] 
\arrow[d, "\oint"]
&
M_n(L, Tr_{\mathcal{O}_F/ \mathbb{Z}} \circ (\alpha' \delta' \langle   \cdot, \cdot\rangle_F))^{\text{tor}}_{\Sigma'_{\alpha'\delta'}}
\arrow[d, "\oint'"]
\\
M_n(L, Tr_{\mathcal{O}_F/ \mathbb{Z}} \circ (\alpha \delta \langle   \cdot, \cdot\rangle_F))^{\text{min}}
\arrow[r,dashed,"F_{\mathfrak{p}_i}^{\text{min}}"]
&
M_n(L, Tr_{\mathcal{O}_F/ \mathbb{Z}} \circ (\alpha' \delta' \langle   \cdot, \cdot\rangle_F))^{\text{min}}
\end{tikzcd}
\end{equation}
where the horizontal arrow is the extension of the partial Frobenius to toroidal compactifications as we have proved in the previous section, and the dashed arrow is the morphism we are searching for that makes the diagram commute. Once the existence of the dashed arrow is established, the description of $F_{\mathfrak{p}_i}^{\text{min}}$  follows from the commutativity of the diagram and the description of the first horizontal arrow as stated in the last section.   

Let
$G_{\Sigma_{\alpha\delta}}$ 
(resp. $G_{\Sigma'_{\alpha'\delta'}}$)
be the universal semi-abelian scheme over 
$M^{\text{tor}}_{n,\Sigma_{\alpha\delta}}:= 
M_n(L, Tr_{\mathcal{O}_F/ \mathbb{Z}} \circ (\alpha \delta \langle   \cdot, \cdot\rangle_F))^{\text{tor}}_{\Sigma_{\alpha\delta}}$
(resp. 
$M^{\text{tor}}_{n,\Sigma'_{\alpha'\delta'}}:= M_n(L, Tr_{\mathcal{O}_F/ \mathbb{Z}} \circ (\alpha' \delta' \langle   \cdot, \cdot\rangle_F))^{\text{tor}}_{\Sigma'_{\alpha'\delta'}}$).
Recall that 
$F_{\mathfrak{p}_i}^{\text{tor}}$
is characterized by
\[
(F_{\mathfrak{p}_i}^{\text{tor}})^*G_{\Sigma'_{\alpha'\delta'}} \cong 
G_{\Sigma_{\alpha\delta}}^{(\mathfrak{p}_i)} := G_{\Sigma_{\alpha\delta}}/(Ker(F)[\mathfrak{p}_i])
\]
hence
\[
(F_{\mathfrak{p}_i}^{\text{tor}})^*\underline{\text{Lie}}^{\vee}_{G_{\Sigma'_{\alpha'\delta'}}/M^{\text{tor}}_{n,\Sigma'_{\alpha'\delta'}}}
\cong
\underline{\text{Lie}}^{\vee}_{G_{\Sigma_{\alpha\delta}}^{(\mathfrak{p}_i)}/ M^{\text{tor}}_{n,\Sigma_{\alpha\delta}}}.
\]
Since the action of $\mathcal{O}_F$ on 
$\underline{\text{Lie}}^{\vee}_{G_{\Sigma_{\alpha\delta}}/ M^{\text{tor}}_{n,\Sigma_{\alpha\delta}}}$
(resp. 
$\underline{\text{Lie}}^{\vee}_{G_{\Sigma'_{\alpha'\delta'}}/M^{\text{tor}}_{n,\Sigma'_{\alpha'\delta'}}}$)
factors through 
$\mathcal{O}_F/p \cong \underset{i}{\prod} \mathcal{O}_F/\mathfrak{p}_i$,
we have 
\[
\underline{\text{Lie}}^{\vee}_{G_{\Sigma_{\alpha\delta}}/ M^{\text{tor}}_{n,\Sigma_{\alpha\delta}}} 
=
\underset{i}{\oplus} e_i \underline{\text{Lie}}^{\vee}_{G_{\Sigma_{\alpha\delta}}/ M^{\text{tor}}_{n,\Sigma_{\alpha\delta}}}
\]
(resp. 
$
 \underline{\text{Lie}}^{\vee}_{G_{\Sigma'_{\alpha'\delta'}}/M^{\text{tor}}_{n,\Sigma'_{\alpha'\delta'}}}
=
\underset{i}{\oplus} e_i \underline{\text{Lie}}^{\vee}_{G_{\Sigma'_{\alpha'\delta'}}/M^{\text{tor}}_{n,\Sigma'_{\alpha'\delta'}}}
$) with $e_i$ the idempotent of $\mathcal{O}_F/p$ corresponds to the factor $\mathcal{O}_F/\mathfrak{p}_i$. 
Since everything is $\mathcal{O}$-equivaraint, we 
obtain
\begin{equation} \label{badbad}
(F_{\mathfrak{p}_i}^{\text{tor}})^*(e_j\underline{\text{Lie}}^{\vee}_{G_{\Sigma'_{\alpha'\delta'}}/M^{\text{tor}}_{n,\Sigma'_{\alpha'\delta'}}})
\cong
e_j\underline{\text{Lie}}^{\vee}_{G_{\Sigma_{\alpha\delta}}^{(\mathfrak{p}_i)}/ M^{\text{tor}}_{n,\Sigma_{\alpha\delta}}}
= 
\begin{cases}
e_j \underline{\text{Lie}}^{\vee}_{G_{\Sigma_{\alpha\delta}}/ M^{\text{tor}}_{n,\Sigma_{\alpha\delta}}}  
& j \neq i, 
\\
F^*(e_i \underline{\text{Lie}}^{\vee}_{G_{\Sigma_{\alpha\delta}}/ M^{\text{tor}}_{n,\Sigma_{\alpha\delta}}})
& j=i 
\end{cases}
\end{equation}
where the last equality follows from 
\[
e_j\underline{\text{Lie}}^{\vee}_{G_{\Sigma_{\alpha\delta}}^{(\mathfrak{p}_i)}/ M^{\text{tor}}_{n,\Sigma_{\alpha\delta}}} 
=
\underline{\text{Lie}}^{\vee}_{G_{\Sigma_{\alpha\delta}}^{(\mathfrak{p}_i)}[\mathfrak{p}_j]/ M^{\text{tor}}_{n,\Sigma_{\alpha\delta}}}
=
\begin{cases}
\underline{\text{Lie}}^{\vee}_{G_{\Sigma_{\alpha\delta}}[\mathfrak{p}_j]/ M^{\text{tor}}_{n,\Sigma_{\alpha\delta}}}  
& j \neq i, 
\\
F^*(\underline{\text{Lie}}^{\vee}_{G_{\Sigma_{\alpha\delta}}[\mathfrak{p}_i]/ M^{\text{tor}}_{n,\Sigma_{\alpha\delta}}})
& j=i 
\end{cases}
\]
in which we use 
\[
\underline{\text{Lie}}^{\vee}_{G_{\Sigma_{\alpha\delta}}^{(\mathfrak{p}_i)}/ M^{\text{tor}}_{n,\Sigma_{\alpha\delta}}}
=
\underline{\text{Lie}}^{\vee}_{G_{\Sigma_{\alpha\delta}}^{(\mathfrak{p}_i)}[p]/ M^{\text{tor}}_{n,\Sigma_{\alpha\delta}}} 
=
\underset{j}{\oplus}
\underline{\text{Lie}}^{\vee}_{
G_{\Sigma_{\alpha\delta}}^{(\mathfrak{p}_i)}[\mathfrak{p}_j]/ M^{\text{tor}}_{n,\Sigma_{\alpha\delta}}} 
\]
and similarly for 
$\underline{\text{Lie}}^{\vee}_{G_{\Sigma_{\alpha\delta}}/ M^{\text{tor}}_{n,\Sigma_{\alpha\delta}}}$, 
together with
\[
G_{\Sigma_{\alpha\delta}}^{(\mathfrak{p}_i)}[\mathfrak{p}_j]
=
\begin{cases}
G_{\Sigma_{\alpha\delta}}[\mathfrak{p}_j]
& j \neq i, 
\\
G_{\Sigma_{\alpha\delta}}^{(p)}[\mathfrak{p}_i]
& j=i 
\end{cases}
\]
where 
$G_{\Sigma_{\alpha\delta}}^{(p)}
:= (G_{\Sigma_{\alpha\delta}}/Ker(F))$
is the usual base change by the absolute Frobenius $F$ on 
$M^{\text{tor}}_{n,\Sigma_{\alpha\delta}}$.

Let 
$\omega_i := \bigwedge^{\text{top}} 
e_i\underline{\text{Lie}}^{\vee}_{G_{\Sigma_{\alpha\delta}}/ M^{\text{tor}}_{n,\Sigma_{\alpha\delta}}}$
(resp. 
$\omega_i' := \bigwedge^{\text{top}} 
e_i\underline{\text{Lie}}^{\vee}_{G_{\Sigma'_{\alpha'\delta'}}/M^{\text{tor}}_{n,\Sigma'_{\alpha'\delta'}}}
$),
then we have 
\[
\omega= \underset{i}{\otimes} \omega_i
\]
and
\[
\omega'= \underset{i}{\otimes} \omega'_i
\]
with 
$\omega :=  \bigwedge^{\text{top}} 
\underline{\text{Lie}}^{\vee}_{G_{\Sigma_{\alpha\delta}}/ M^{\text{tor}}_{n,\Sigma_{\alpha\delta}}}$
and 
$\omega' := \bigwedge^{\text{top}} 
\underline{\text{Lie}}^{\vee}_{G_{\Sigma'_{\alpha'\delta'}}/M^{\text{tor}}_{n,\Sigma'_{\alpha'\delta'}}}$
as before. Moreover, (\ref{badbad}) tells us that
\[
F^{\text{tor}*}_{\mathfrak{p}_i}\omega'_j = \begin{cases}
\omega_j
& j \neq i, 
\\
F^*\omega_i = \omega_i^p
& j=i 
\end{cases}
\]
hence
\[
F^{\text{tor}*}_{\mathfrak{p}_i}\omega' = 
(\underset{j\neq i}{\otimes} \omega_j) \otimes \omega^p_i
\]
If we can show that each $\omega_i$ descends to a line bundle on 
\[
M_{n, \alpha\delta}^{\text{min}} :=
M_n(L, Tr_{\mathcal{O}_F/ \mathbb{Z}} \circ (\alpha \delta \langle   \cdot, \cdot\rangle_F))^{\text{min}}
\]
through $\oint$, i.e. there exists a line bundle $\omega_i^{\text{min}}$ 
on 
$M_{n, \alpha\delta}^{\text{min}} $
such that 
\[
\oint^*\omega_i^{\text{min}} = \omega_i
\]
then 
$(F^{\text{tor}}_{\mathfrak{p}_i})^*\omega' = 
(\underset{j\neq i}{\otimes} \omega_j) \otimes \omega^p_i$
tells us that we can find a line bundle
$L^{\text{min}} := (\underset{j\neq i}{\otimes} \omega^{\text{min}}_j) \otimes (\omega^{\text{min}}_i)^p$
such that 
\[
\oint^*L^{\text{min}} = (F^{\text{tor}}_{\mathfrak{p}_i})^*\omega'
\]
and the universal property of Proj construction tells us that there exists 
$F^{\text{min}}_{\mathfrak{p}_i}$
which makes the diagram (\ref{bugb}) commutative.

\vspace{5mm}

Indeed,  recall that the universal property of the Proj construction is as follows. Let $\mathcal{A}= \underset{k\geq 0}{\oplus} \mathcal{A}_k$ be a graded $R$-algebra finitely generated by degree 1 elements, and $T$ be a scheme defined over $R$ with structure map $f : T \rightarrow \text{Spec}(R)$. Suppose we are given a line bundle $\mathcal{L}$ on $T$, and
a morphism of graded $R$-algebras 
\[
\psi:
\mathcal{A} \rightarrow f_*(\underset{k \geq 0}{\oplus} \mathcal{L}^{\otimes k} )= \underset{k \geq 0}{\oplus} \Gamma(T,\mathcal{L}^{\otimes k})
\]
whose adjoint morphism at degree 1 
$ f^*\mathcal{A}_1 \rightarrow \mathcal{L}$ 
is surjective (viewing $\mathcal{A}_1$ as a quasi-coherent module on $\text{Spec}(R)$), then there exists a unique morphism 
\[
g: T \longrightarrow \text{Proj}_R(\mathcal{A}) 
\]
of $R$-schemes together with an isomorphism 
\[
\theta: g^*\mathscr{O}(1) \cong \mathcal{L}
\]
such that $\psi$ factorizes as
\[
\psi:
\mathcal{A} \cong \underset{k \geq 0}{\oplus} \Gamma(\text{Proj}_R(\mathcal{A}), \mathscr{O}(1)^{\otimes k}) 
\overset{g^*}{\rightarrow} 
\underset{k \geq 0}{\oplus} \Gamma(T,g^*\mathscr{O}(1)^{\otimes n})
\overset{\theta}{\cong} \underset{k \geq 0}{\oplus} \Gamma(T,\mathcal{L}^{\otimes n}).
\]

\vspace{5mm}

In our setting, 
$\oint' \circ F^{\text{tor}}_{\mathfrak{p}_i}$
is induced by 
\begin{equation} \label{daiaiii}
\underset{k \geq 0}{\oplus} \Gamma(M^{\text{tor}}_{n,\Sigma'_{\alpha'\delta'}}, \omega'^{\otimes k})
\overset{(F^{\text{tor}}_{\mathfrak{p}_i})^*}{\rightarrow} 
\underset{k \geq 0}{\oplus}
\Gamma(M^{\text{tor}}_{n,\Sigma_{\alpha\delta}}, ((F^{\text{tor}}_{\mathfrak{p}_i})^*\omega')^{\otimes k})
\end{equation}
with $\mathcal{L} = (F^{\text{tor}}_{\mathfrak{p}_i})^*\omega'$. 
Assume that we know the existence of $\omega_i^{\text{min}}$, then we have $L^{\text{min}}$ such that 
$\oint^*L^{\text{min}} = (F^{\text{tor}}_{\mathfrak{p}_i})^*\omega'$. 
Since 
$\mathscr{O}_{M_{n, \alpha\delta}^{\text{min}}} 
\overset{\sim}{\rightarrow}  \oint_*\mathscr{O}_{M^{\text{tor}}_{n,\Sigma_{\alpha\delta}}}$,
we have by projection formula
\[
\oint_*\oint^*L^{\text{min}} = L^{\text{min}} \otimes
\oint_*\mathscr{O}_{M^{\text{tor}}_{n,\Sigma_{\alpha\delta}}}
=L^{\text{min}} \otimes \mathscr{O}_{M_{n, \alpha\delta}^{\text{min}}} 
=L^{\text{min}}
\]
(Note that since $L^{\text{min}}$ is locally free, the derived tensor prodoct and left derived pullback is just the usual one), 
which implies that 
\[
\Gamma(M_{n, \alpha\delta}^{\text{min}}, L^{\text{min}})
\overset{\oint^*}{\cong} 
\Gamma(M^{\text{tor}}_{n,\Sigma_{\alpha\delta}}, \oint^*L^{\text{min}})
\cong
\Gamma(M^{\text{tor}}_{n,\Sigma_{\alpha\delta}}, (F^{\text{tor}}_{\mathfrak{p}_i})^*\omega')
\]
and similar for 
$(L^{\text{min}})^{\otimes k}$.
Thus (\ref{daiaiii}) gives us a morphism
\[
\underset{k \geq 0}{\oplus} \Gamma(M^{\text{tor}}_{n,\Sigma'_{\alpha'\delta'}}, \omega'^{\otimes k})
\rightarrow 
\underset{k \geq 0}{\oplus}
\Gamma(M_{n, \alpha\delta}^{\text{min}}, (L^{\text{min}})^{\otimes k})
\]
which by the universal property of Proj construction induces a morphism
\[
F_{\mathfrak{p}_i}^{\text{min}}:
M_{n, \alpha\delta}^{\text{min}}
\longrightarrow
M_{n,\alpha'\delta'}^{\text{min}}
:= 
\text{Proj}(\underset{k \geq 0}{\oplus} \Gamma(M^{\text{tor}}_{n,\Sigma'_{\alpha'\delta'}}, \omega'^{\otimes k}))
\]
which makes the diagram (\ref{bugb}) commutative. 

\vspace{5mm}

Thus we are reduced to show the existence of 
$\omega_i^{\text{min}}$ 
such that 
$\oint^*\omega_i^{\text{min}} = \omega_i$.
Let
$M_{n, \alpha\delta}^1 \subset M_{n, \alpha\delta}^{\text{min}}$
be the union of the open stratum and all the codimension 1 strata, then it follows from \cite{lan2013arithmetic} 7.2.3.13 that 
\[
\oint : \oint^{-1}(M_{n, \alpha\delta}^1) \cong M_{n, \alpha\delta}^1 
\] 
so we can view $M_{n, \alpha\delta}^1$ as an open subscheme of 
$M_{n, \Sigma_{\alpha\delta}}^{\text{tor}}$ 
as well. Let
\[
\omega^{\text{min}}_i := (M_{n, \alpha\delta}^1 \hookrightarrow 
M_{n, \alpha\delta}^{\text{min}})_* (\omega_i|_{M_{n, \alpha\delta}^1})
\]
we will show that
$\omega^{\text{min}}_i$ 
is a line bundle and 
$\oint^*\omega^{\text{min}}_i \cong \omega_i$. This is a direct adaption of the proof of \cite{lan2013arithmetic} 7.2.4.1 in our case.

First observe that $\omega^{\text{min}}_i$  is a  coherent sheaf since 
$M_{n, \alpha\delta}^{\text{min}}$ 
is normal and the complement of 
$M_{n, \alpha\delta}^1$
has codimension at least 2 (\cite{MR0476737} \rom{8} Prop. 3.2). Then to show that it is a line bundle, it is enough to show that its stalk at every point is free of rank 1. By fpqc descent, it is enough to show this for the completions of the strict localizations of 
$M_{n, \alpha\delta}^{\text{min}}$, i.e. 
it is enough to prove that for every geometric point $\bar{x}$ of $M_{n, \alpha\delta}^{\text{min}}$, 
the pullback of $\omega^{\text{min}}_i$ to
$(M_{n, \alpha\delta}^{\text{min}})_{\bar{x}}^{\land}$,
the completions of the strict localization of  
$M_{n, \alpha\delta}^{\text{min}}$
at $\bar{x}$, is free of rank 1. Similarly, it is enough to prove that 
$\oint^*\omega^{\text{min}}_i \cong \omega_i$
holds naturally over 
$(M_{n, \Sigma_{\alpha\delta}}^{\text{tor}})_{\bar{y}}^{\land}$
for every geometric point $\bar{y}$ of 
$M_{n, \Sigma_{\alpha\delta}}^{\text{tor}}$.

Suppose that $\bar{x}$ lies in the stratum 
$Z_{[(Z_n, \Phi_n, \delta_n)]}$, 
and we choose a stratum
$Z_{[(\Phi_n, \delta_n, \sigma)]}$
lying above 
$Z_{[(Z_n, \Phi_n, \delta_n)]}$. 
Then from (2) of theorem \ref{lantor} we have a natural identification
\[
(M_{n, \Sigma_{\alpha\delta}}^{\text{tor}})_{Z_{[(\Phi_n, \delta_n, \sigma)]}}^{\land} 
\cong
\mathfrak{X}_{\Phi_n,\delta_n, \sigma}
\]
where we do not have the quotient by $\Gamma_{\Phi_n,\sigma}$ 
since we assume that $n >3$. We have a canonical map
$\mathfrak{X}_{\Phi_n,\delta_n, \sigma} 
\rightarrow 
(M_{n, \alpha\delta}^{\text{min}})^{\land}_{Z_{[(Z_n, \Phi_n, \delta_n)]}}$
induced by $\oint$, and by abuse of notation we let 
\[
(\mathfrak{X}_{\Phi_n,\delta_n, \sigma})^{\land}_{\bar{x}}
:= \mathfrak{X}_{\Phi_n,\delta_n, \sigma}  
\times_{(M_{n, \alpha\delta}^{\text{min}})^{\land}_{Z_{[(Z_n, \Phi_n, \delta_n)]}}} 
(M_{n, \alpha\delta}^{\text{min}})_{\bar{x}}^{\land}
\]
then by definition we have a morphism
\[
(\mathfrak{X}_{\Phi_n,\delta_n, \sigma})^{\land}_{\bar{x}}
\rightarrow 
(M_{n, \alpha\delta}^{\text{min}})_{\bar{x}}^{\land}
\]
The key point is that we have a morphism 
\[
(M_{n, \alpha\delta}^{\text{min}})_{\bar{x}}^{\land} 
\rightarrow
(M_n^{Z_n})^{\land}_{\bar{x}}
\]
such that the composition
\[
(\mathfrak{X}_{\Phi_n,\delta_n, \sigma})^{\land}_{\bar{x}}
\rightarrow 
(M_{n, \alpha\delta}^{\text{min}})_{\bar{x}}^{\land} 
\rightarrow
(M_n^{Z_n})^{\land}_{\bar{x}}
\]
is induced by the structural morphism 
$p: \mathfrak{X}_{\Phi_n,\delta_n, \sigma}
\rightarrow
M_n^{Z_n}$ 
(recall that 
$\mathfrak{X}_{\Phi_n,\delta_n, \sigma}$
is the formal completion along the boundary of an affine toroidal compactification of a torus torsor over an abelian scheme over $M_n^{Z_n}$),
see \cite{lan2013arithmetic} 7.2.3.16 for details.

We observe that the pull back of the line bundle $\omega_i$ over 
$\mathfrak{X}_{\Phi_n,\delta_n, \sigma}$
is canonically identified with 
$(\bigwedge_{\mathbb{Z}}^{\text{top}} e_iX ) \otimes_{\mathbb{Z}} p^*(\bigwedge^{\text{top}} e_i \underline{\text{Lie}}^{\vee}_{A/M_n^{Z_n}})$,
where $A$ is the universal abelian variety over $M_n^{Z_n}$. 
This is a trivial variant of \cite{lan2013arithmetic} 7.1.2.1, and we briefly recall the proof. By \'etale descent, we can assume that the base is $S=\text{Spf}(R,I)$, with $R$ normal noetherian and $I$-adically complete so that we are in the setting of section \ref{dlevel}. We have 
\[
\underline{\text{Lie}}^{\vee}_{G_{\text{for}}/S} = \underline{\text{Lie}}^{\vee}_{G^{\natural}_{\text{for}}/S} 
\]
hence 
\[
\omega_i = \bigwedge^{\text{top}} e_i \underline{\text{Lie}}^{\vee}_{G_{\text{for}}/S} 
= \bigwedge^{\text{top}} e_i
\underline{\text{Lie}}^{\vee}_{G^{\natural}_{\text{for}}/S} 
=
(\bigwedge_{\mathbb{Z}}^{\text{top}} e_iX ) \otimes_{\mathbb{Z}} \bigwedge^{\text{top}} e_i \underline{\text{Lie}}^{\vee}_{A_{\text{for}}/S}
\]
where the last equality follows from the short exact sequence
\[
0 \rightarrow e_i \underline{\text{Lie}}_{T/S} \rightarrow e_i \underline{\text{Lie}}_{G^{\natural}/S} 
\rightarrow e_i \underline{\text{Lie}}_{A/S}
\rightarrow 0
\]
induced by the global semi-abelian structure 
\[
0\rightarrow T \rightarrow G^{\natural} \rightarrow A \rightarrow 0
\]
of $G^{\natural}$.

From what we have seen, the restriction of $\omega_i$ to 
$(\mathfrak{X}_{\Phi_n,\delta_n, \sigma})^{\land}_{\bar{x}}$
is the pullback of 
$(\bigwedge_{\mathbb{Z}}^{\text{top}} e_iX ) \otimes_{\mathbb{Z}} (\bigwedge^{\text{top}} e_i \underline{\text{Lie}}^{\vee}_{A/M_n^{Z_n}})$
along the composition 
\[
(\mathfrak{X}_{\Phi_n,\delta_n, \sigma})^{\land}_{\bar{x}}
\rightarrow 
(M_{n, \alpha\delta}^{\text{min}})_{\bar{x}}^{\land} 
\rightarrow
(M_n^{Z_n})^{\land}_{\bar{x}}
\]
which in particular shows that it is the pullbcak of some line bundle $L$ on 
$(M_{n, \alpha\delta}^{\text{min}})_{\bar{x}}^{\land}$,
i.e. by abuse of notation 
\begin{equation} \label{ljlkl''''}
(\oint^{\land}_{\bar{x}} )^* L\cong (\omega_i)_{\bar{x}}^{\land}.
\end{equation} 
This implies that both $(\omega_i^{\text{min}})_{\bar{x}}^{\land}$ 
and 
$L$ are extensions of (the completion of the strict localization at $\bar{x}$ of)
$\omega_i|_{M_{n, \alpha\delta}^1}$,
which by Stacks Project 30.12.12 is equivalent ($(\omega_i^{\text{min}})_{\bar{x}}^{\land}$ is reflexive since it is the pushforward of an open embedding and $L$ is reflexive since it is a line bundle on a normal scheme).  

This proves that $(\omega_i^{\text{min}})_{\bar{x}}^{\land}$ 
is free of rank 1 and for every geometric point
$\bar{y}$ of $ Z_{[(\Phi_n, \delta_n, \sigma)]}$
with $\bar{x}=\oint(\bar{y})$, 
then we have a natural map 
$h: (M_{n, \Sigma_{\alpha\delta}}^{\text{tor}})_{\bar{y}}^{\land}
\rightarrow 
(\mathfrak{X}_{\Phi_n,\delta_n, \sigma})^{\land}_{\bar{x}}$
and 
\[
(\oint^*(\omega_i^{\text{min}}))_{\bar{y}}^{\land}
\cong 
h^*(\oint^{\land}_{\bar{x}} )^* (\omega_i^{\text{min}})_{\bar{x}}^{\land} \overset{(\ref{ljlkl''''})}{\cong} h^*((\omega_i)_{\bar{x}}^{\land})
\cong
(\omega_i)_{\bar{y}}^{\land}
\]
proving what we want. 

To be more precise, there are canonical morphisms
\[
\begin{tikzcd}
\omega_i \arrow[rr, hook] 
& &
J_*J^*\omega_i 
\\
\oint^*\omega^{\text{min}}_i 
\arrow[r,equal,"\text{def}"]
&
\oint^*j_*J^*\omega_i
\arrow[r,equal]
&
\oint^*\oint_* J_* J^*\omega_i
\arrow[u]
\end{tikzcd}
\]
where $j$ and $J$ are open embeddings defined by the diagram
\[
\begin{tikzcd}
M_{n, \alpha\delta}^1 \arrow[r, hook, "J"] \arrow[dr, hook, "j"]  &
M_{n, \Sigma_{\alpha\delta}}^{\text{tor}} 
\arrow[d,"\oint"]
\\
& M_{n, \alpha\delta}^{\text{min}} 
\end{tikzcd}
\]
and the two arrows are the adjunction morphism. 
We showed that 
$\oint^*\omega^{\text{min}}_i$
and
$\omega_i$
are naturally identified over 
$(M_{n, \Sigma_{\alpha\delta}}^{\text{tor}})_{\bar{y}}^{\land}$
for every geometric point $\bar{y}$ of 
$M_{n, \Sigma_{\alpha\delta}}^{\text{tor}}$,
and the naturality tells us that after localization and completion, the images of $\omega_i$  and 
$\oint^*\omega^{\text{min}}_i$
in 
$J_*J^*\omega_i$
are identified. Now we can apply the fpqc descent to those two image sheaves and conclude that 
$\oint^*\omega^{\text{min}}_i 
\cong \omega_i$.

\end{proof}

\bibliographystyle{unsrt} 
\bibliography{ref}

\end{document}